\documentclass[12pt]{article}
\usepackage{amsmath}
\usepackage{amssymb}
\usepackage{amsthm}
\usepackage{wrapfig}
\usepackage{comment}
\usepackage{braket}
\usepackage{mathrsfs}
\usepackage{cases}
\numberwithin{equation}{section}
\usepackage[top=18mm,bottom=18mm,left=18mm,right=18mm]{geometry}
\usepackage[abbrev]{amsrefs}
\allowdisplaybreaks[1]

\newtheorem{df}{Definition}[section]
\newtheorem{lem}[df]{Lemma}
\newtheorem{prop}[df]{Proposition}

{\theoremstyle{definition}\newtheorem{rmk}[df]{Remark}}
\newtheorem{thm}[df]{Theorem}

\def\B{\mathbb{B}}

\def\N{\mathbb{N}}
\def\R{\mathbb{R}}

\def\I{\mathbb{I}}

\def\A{\mathscr{A}}
\def\D{\mathscr{D}}

\def\O{\mathscr{O}}

\DeclareMathOperator*{\essinf}{ess\,inf}
\DeclareMathOperator*{\esssup}{ess\,sup}
\def\div{{\rm{div}}\,}
\def\rot{{\rm{rot}}\,}

\begin{document}

\title{\bf Anisotropically weighted $L^q$-$L^r$ estimates of 
the Oseen semigroup in exterior domains, with applications to the 
Navier-Stokes flow past a rigid body}
\author{Tomoki Takahashi}
\date{}
\maketitle
\noindent{\bf Abstract.} We consider 
the spatial-temporal behavior of 
the Navier-Stokes flow past a rigid body in $\R^3$. 
The present paper develops analysis 
in Lebesgue spaces with anisotropic weights 
$(1+|x|)^\alpha(1+|x|-x_1)^\beta$, which naturally arise in 
the asymptotic structure of fluid 
when the translational velocity of the body is parallel to the 
$x_1$-direction. We derive 
anisotropically weighted $L^q$-$L^r$ estimates for the 
Oseen semigroup in exterior domains. 
As applications of those estimates, 
we study the stability/attainability of the Navier-Stokes flow 
in anisotropically weighted $L^q$ spaces 
to get the spatial-temporal behavior of nonstationary solutions. 
\section{Introduction}\label{intro}
\quad 
We consider a viscous incompressible flow past a rigid body 
$\O\subset \R^3$. We suppose that 
$\O$ is translating 
with a velocity $\eta(t)$. Then by taking frame 
attached to the body, the fluid motion which occupies the outside 
of $\O$ obeys 
\begin{align}\label{NS0}
\left\{
\begin{array}{r@{}c@{}ll}
\partial_t u+u\cdot \nabla u&{}={}&
\Delta u+\eta(t)\cdot\nabla u
-\nabla p,&\quad x\in D,~t>0,\\
\nabla\cdot u&{}={}&0,&\quad x\in D,t\geq 0,\\
u|_{\partial D}&{}={}&\eta(t),&\quad t>0,\\
u&{}\rightarrow{}&0&\quad {\rm{as}}~|x|\rightarrow \infty,\\
u(x,0)&{}={}&u_0,&\quad x\in D,
\end{array}\right.
\end{align}
where $D=\R^3\setminus\O$ is the exterior domain with $C^2$ smooth 
boundary $\partial D$ and the origin of coordinate is assumed to be 
contained in the interior of $\O$. The functions 
$u=(u_1(x,t),u_2(x,t),u_3(x,t))^\top$ and $p=p(x,t)$ denote unknown 
velocity and pressure of the fluid, respectively, while 
$u_0$ is a given initial velocity. 
Here and hereafter, $(\cdot)^\top$ denotes the transpose. 
This paper is devoted to the case $\eta(t)=-ae_1$ 
and the case $\eta(t)=-\psi(t)ae_1$ 
(we assume $u_0=0$ in the latter case as mentioned below), 
where $a>0$, 
$e_1=(1,0,\cdots,0)^\top$ and 
$\psi$ is a function on $\R$ describing the transition 
of the translational velocity in such a way that 
\begin{align}\label{psidef}
\psi\in C^1(\R;\R),\quad
|\psi(t)|\leq 1\quad{\rm{for}}~\,t\in\R,\quad
\psi(t)=0\quad{\rm{for}}\,~t\leq 0,\quad 
\psi(t)=1\quad {\rm{for}} \,~t\geq 1. 
\end{align}
Namely, we consider the problem 
\begin{align}\label{NS22}
\left\{
\begin{array}{r@{}c@{}ll}
\partial_t u+u\cdot \nabla u&{}={}&
\Delta u-a\partial_{x_1}u
-\nabla p,&\quad x\in D,~t>0,\\
\nabla\cdot u&{}={}&0,&\quad x\in D,t\geq 0,\\
u|_{\partial D}&{}={}&-ae_1,&\quad t>0,\\
u&{}\rightarrow{}&0&\quad {\rm{as}}~|x|\rightarrow \infty,\\
u(x,0)&{}={}&u_0,&\quad x\in D
\end{array}\right.
\end{align}
and the starting problem 
\begin{align}\label{NS2}
\left\{
\begin{array}{r@{}c@{}ll}
\partial_t u+u\cdot \nabla u&{}={}&
\Delta u-\psi(t)a\,\partial_{x_1}u
-\nabla p,&\quad x\in D,~t>0,\\
\nabla\cdot u&{}={}&0,&\quad x\in D,t\geq 0,\\
u|_{\partial D}&{}={}&-\psi(t)ae_1,&\quad t>0,\\
u&{}\rightarrow{}&0&\quad {\rm{as}}~|x|\rightarrow \infty,\\
u(x,0)&{}={}&0,&\quad  x\in D.
\end{array}\right. 
\end{align}
\par Since $\psi(t)=1$ for $t\geq 1$, 
the large time behavior of solutions 
to (\ref{NS2}) as well as (\ref{NS22}) 
is related to the stationary problem 
\begin{align}\label{sta}
\left\{
\begin{array}{r@{}c@{}ll}
u_s\cdot \nabla u_s&{}={}&\Delta u_s-a\,
\partial_{x_1}u_s-\nabla p_s,
&\quad x\in D,\\
\nabla\cdot u_s&{}={}&0,&\quad x\in D,\\
u_s|_{\partial D}&{}={}&-ae_1,\\
u_s&{}\rightarrow{}& 0&\quad {\rm{as}}~|x|\rightarrow \infty.
\end{array}\right.
\end{align}
The pioneer work due to Leray \cite{leray1933} 
provided the existence theorem for 
weak solution to (\ref{sta}), 
what is called $D$-solution, 
having finite Dirichlet integral 
without smallness of data, however, 
his solution had little information about 
the anisotropic decay structure 
caused by the translation. Later on, 
Finn \cite{finn1959,finn1965s,finn1965o,finn1973} 
succeeded in constructing a stationary solution, 
termed by him physically reasonable solution, exhibiting 
a paraboloidal wake region behind the body, that is, 
\begin{align}\label{wake}
u_s(x)=O((1+|x|)^{-1}(1+|x|-x_1)^{-1})
\end{align} 
if $a$ is small enough. 
The problem (\ref{sta}) have been extensively studied, 
and we also refer to \cite{finn1960,babenko1973,galdi1992,
farwig1992t,faso1998,shibata1999,krnopo2001,galdi2011}. 
\par In \cite{finn1965s}, Finn conjectured that 
(\ref{NS2}) admits a solution 
which tends to a stationary solution 
as $t\rightarrow\infty$ if $a$ is small enough. 
This is called Finn's starting problem and if 
the conjecture is affirmatively solved, then 
the stationary solution is said to be attainable 
by following Heywood \cite{heywood1972} who 
first studied Finn's starting problem, 
but gave a partial answer in a 
special situation that the net force exerted by 
the fluid on the body is identically zero, yielding 
the square summability of stationary solutions. Indeed, 
by the decay structure (\ref{wake}), we have $u_s\in L^q(D)$ for 
$q>2$, but $u_s\notin L^2(D)$ in general 
due to Finn \cite{finn1960}, see also Galdi \cite{galdi2011}. 
From this together with 
$u(t)\in L^2(D)$ for each $t$, which is expected since 
the initial value of (\ref{NS2}) is zero, 
if we seek a solution to (\ref{NS2}) of the form 
\begin{align}\label{perturb1} 
u(x,t)=w(x,t)+\psi(t)u_s,\quad p(x,t)=\theta(x,t)+\psi(t)p_s, 
\end{align} 
then $w(t)\notin L^2(D)$ for each $t$. Thus 
energy methods are far from enough to analyze the starting problem 
and the problem had remained open until Kobayashi and Shibata 
\cite{kosh1998} developed 
the $L^q$ theory of the linearized problem, 
which is called the Oseen problem. 
By applying the $L^q$-$L^r$ 
estimates of the Oseen semigroup established in \cite{kosh1998} 
(see also Enomoto and Shibata \cite{ensh2004,ensh2005}), 
via Kato's approach \cite{kato1984} 
(see also Fujita and Kato \cite{fuka1964}), 
Finn's starting problem was affirmatively solved by 
Galdi, Heywood and Shibata \cite{gahesh1997}. 
The $L^2$ stability of $u_s$, 
that is the problem (\ref{NS22}) with $u_0=u_s+b$ 
($b\in L^2(D):$ perturbation), was proved by 
Heywood \cite{heywood1970,heywood1972} 
(see also \cite{masuda1975,heywood1980,miso1988,bomi1992} 
for the improvement of \cite{heywood1970,heywood1972}) and 
the $L^q$ stability was also proved in Shibata \cite{shibata1999} 
and Enomoto and Shibata \cite{ensh2005} 
by the similar approach to \cite{gahesh1997}. 
Convergence rates in attainability analysis (\cite{gahesh1997}) 
were the same as those in stability analysis (\cite{shibata1999}), 
but the present author \cite{takahashi2021} 
derived convergence rates determined by the summability of $u_s$, 
which are improvement of those in \cite{gahesh1997}. 
\par As in the stationary problem (\ref{sta}), 
we expect that nonstationary solutions 
to (\ref{NS22}) or to (\ref{NS2}) 
exhibit the paraboloidal wake region, but 
the literature for concerning this issue is 
Knightly \cite{knightly1979}, Mizumachi \cite{mizumachi1984} 
and Deuring \cite{deuring2013,deuringtime} only. 
Deuring \cite{deuring2013} used a representation formula for 
the solution to the Oseen system, 
see \cite{deuring2009,deuring2013p,deuring2014}, 
to deduce 
\begin{align*} 
\nabla^iu(x,t)=O((1+|x|)^{-1-\frac{i}{2}}
(1+|x|-x_1)^{-1-\frac{i}{2}}) 
\end{align*}
for $i=0,1$ uniformly in $t$ 
under some assumptions on 
the initial perturbation from the stationary solution 
and on the solution $u$. 
In \cite{deuringtime}, 
by employing another integral representation, 
see \cite{deuring2021,deuringlq}, he also established the estimate 
\begin{align*} 
\nabla^i\big(u(x,t)-u_s\big)=O((1+|x|)^{-\frac{5}{4}-\frac{i}{2}} 
(1+|x|-x_1)^{-\frac{5}{4}-\frac{i}{2}}) 
\end{align*} 
for $t>0$ 
without the boundary condition except the zero-flux condition.  
\par As mentioned above, 
the wake structure uniform with respect to time 
has been investigated by 
\cite{knightly1979,mizumachi1984,deuring2013,deuringtime}, 
while the purpose of the present paper is 
to derive temporal decay rate with the wake structure captured. 
To accomplish our purpose, we develop the theory of 
the Oseen semigroup in $L^q$ spaces 
with the weight $(1+|x|)^\alpha(1+|x|-x_1)^\beta$, in particular, 
derive the anisotropically weighted 
$L^q$-$L^r$ estimates of the Oseen semigroup. 
We then apply those estimates to construct a 
nonstationary solution to 
\begin{align}\label{NS32}\left\{
\begin{array}{r@{}c@{}l}
\partial_t v&{}={}&
\Delta v-a\partial_{x_1}v
-v\cdot\nabla v-v\cdot\nabla u_s
-u_s\cdot\nabla v
-\nabla \phi,
\quad x\in D,\,t>0,\\
\nabla\cdot v&{}={}&0,\quad x\in D,\,t\geq 0,\\
v|_{\partial D}&{}={}&0,\quad t>0,\\
v&{}\rightarrow{}& 0~\quad {\rm{as}}~|x|\rightarrow \infty,\\
v(x,0)&{}={}&b:=u_0-u_s,\quad x\in D
\end{array}\right. 
\end{align}
or to 
\begin{align}\label{NS3}\left\{
\begin{array}{r@{}c@{}l}
\partial_tw&{}={}&
\Delta w-a\partial_{x_1}w
-w\cdot\nabla w-\psi(t)w\cdot\nabla u_s
-\psi(t)u_s\cdot\nabla w
+(1-\psi(t))a\partial_{x_1}w\\
&&\hspace{5cm}+f_1(x,t)+f_2(x,t)
-\nabla \theta,
\quad x\in D,\,t>0,\\
\nabla\cdot w&{}={}&0,\quad x\in D,\,t\geq 0,\\
w|_{\partial D}&{}={}&0,\quad t>0,\\
w&{}\rightarrow{}& 0~\quad {\rm{as}}~|x|\rightarrow \infty,\\
w(x,0)&{}={}&0,\quad x\in D 
\end{array}\right. 
\end{align} 
in the anisotropically weighted $L^q$ framework, 
where $(v,\phi)$ is defined by 
\begin{align*} 
u(x,t)=v(x,t)+u_s,\quad p(x,t)=\phi(x,t)+p_s, 
\end{align*} 
$(w,\theta)$ is given by (\ref{perturb1}) and 
\begin{align} 
&f_1(x,t)=-\psi'(t)u_s,\label{f1}\\ 
&f_2(x,t)=\psi(t)(1-\psi(t)) 
(u_s\cdot\nabla u_s 
+a\partial_{x_1}u_s).\label{f2} 
\end{align} 
Here, we reveal in Theorem \ref{muckenthm} that the condition 
\begin{align}\label{alphabeta} 
-1<\beta<q-1,\quad -3<\alpha+\beta<3(q-1) 
\end{align} 
is the necessary and 
sufficient condition on $\alpha,\beta$ so that 
$(1+|x|)^\alpha(1+|x|-x_1)^\beta$ belongs to the 
Muckenhoupt class $\A_q(\R^3)$, which ensures 
the weighted $L^q$ boundedness of singular integral operators, 
see, for instance, Galc\'{\i}a-Cuerva and Rubio de Francia 
\cite[Chapter IV]{garu1985}, 
Torchinsky \cite[Chapter IX]{torchinsky1986} and 
Stein \cite[Chapter V]{stein1993}. 
We then employ such weights to 
apply the weighted $L^q$ theory for the Stokes resolvent problem 
and the Helmholtz decomposition in weighted $L^q$ spaces 
developed by Farwig and Sohr \cite{faso1997}. 
Theorem \ref{muckenthm} with $q=2$ was already known 
from Farwig \cite{farwig1992t} and 
the sufficiency of (\ref{alphabeta}) was proved by 
Kra\v{c}mer, Novotn\'{y} and Pokorn\'{y} \cite{krnopo2001}, 
but it has not known whether the condition is 
also necessary condition. 
\par To establish the anisotropically weighted 
$L^q$-$L^r$ estimates of the Oseen semigroup $e^{-tA_a}$ 
in the exterior domain $D$, it is important to derive  
the estimates in $\R^3$. In Section 5, 
given $q\leq r\leq\infty~(q\ne\infty)$ and 
$\alpha,\beta>0$ satisfying $\beta<1-1/q,\alpha+\beta<3(1-1/q)$ 
(which ensures 
$(1+|x|)^{\alpha q}(1+|x|-x_1)^{\beta q}\in\A_q(\R^3)$), 
we deduce 
\begin{align}
\|&(1+|x|)^\alpha(1+|x|-x_1)^\beta 
\nabla^j S_{a}(t)f\|_{L^r(\R^3)}
\notag\\
&\leq 
Ct^{-\frac{3}{2}(\frac{1}{q}-\frac{1}{r})-\frac{i}{2}
+\frac{\alpha}{4}+\max\{\frac{\alpha}{4},\frac{\beta}{2}\}
+\varepsilon}\|(1+|x|)^\alpha(1+|x|-x_1)^\beta f\|_{L^q(\R^3)}
\label{lqlrr3large30}
\end{align} 
for $t\geq 1,j=0,1$, where $S_a(t)~(a>0)$ is the Oseen semigroup 
in $\R^3$ and $\varepsilon>0$ is a given positive constant. 
But, it seems difficult to apply (\ref{lqlrr3large30}) to construct 
a solution in the nonlinear problems. 
Therefore, we derive the other estimate 
\begin{align}
\|&(1+|x|)^\alpha(1+|x|-x_1)^\beta 
\nabla^jS_a(t) f\|_{L^r(\R^3)}\notag\\
&\leq C
t^{-\frac{3}{2}(\frac{1}{q_1}-\frac{1}{r})-\frac{j}{2}}
\|(1+|x|)^{\alpha}(1+|x|-x_1)^{\beta} f\|_{L^{q_1}(\R^3)}+C
t^{-\frac{3}{2}(\frac{1}{q_2}-\frac{1}{r})-\frac{j}{2}+\alpha}
\|(1+|x|-x_1)^{\beta} f\|_{L^{q_2}(\R^3)}\notag\\
&\qquad+Ct^{-\frac{3}{2}(\frac{1}{q_3}-\frac{1}{r})-\frac{j}{2}
+\frac{\beta}{2}}
\|(1+|x|)^{\alpha}f\|_{L^{q_3}(\R^3)}
+Ct^{-\frac{3}{2}(\frac{1}{q_4}-\frac{1}{r})-\frac{j}{2}
+\alpha+\frac{\beta}{2}}
\|f\|_{L^{q_4}(\R^3)} 
\label{lqlrr3large300}
\end{align} 
for $t\geq 1,j=0,1$ and $1<q_i\leq r\leq\infty~(q_i\ne\infty)$. 
Note that the exponents $q_i~(i=1,2,3,4)$ 
may not coincide with each other and 
this remark plays an important role when we treat the terms 
$f_1$ and $f_2$ (see (\ref{f1})--(\ref{f2})) 
to analyze the starting problem, which is described in Section 8. 
The estimate (\ref{lqlrr3large30}) is not employed 
in the nonlinear problems, 
but it is seen that the rate in (\ref{lqlrr3large30}) is better than 
the one in (\ref{lqlrr3large300}) with $q_i=q~(i=1,2,3,4)$, 
that is $t^{-3(1/q-1/r)/2-j/2+\alpha+\beta/2}$. We thus use 
(\ref{lqlrr3large30}) in the study of the 
Oseen semigroup $e^{-tA_a}$ in $D$ 
for better understanding of $e^{-tA_a}$. 
With (\ref{lqlrr3large30})--(\ref{lqlrr3large300}) at hand, 
we derive the following, 
which is one of the main results of this paper, 
see Theorem \ref{lqlrd} and Theorem \ref{lqlrd2} 
for the precise statement. 
\begin{thm}\label{thm1}
\begin{enumerate}
\item Let $\varepsilon>0,j=0,1$ and let 
$1<q\leq r\leq\infty~(q\ne\infty)$ and small 
$\alpha,\beta>0$ satisfy $\beta<1-1/q,\alpha+\beta<3(1-1/q)$. 
We suppose 
$q\leq r<r_0$ if $j=1$, where $r_0(>3)$ depends on 
$\alpha,\beta$. Then we have 
\begin{align} 
\|&(1+|x|)^\alpha(1+|x|-x_1)^\beta \nabla^j e^{-tA_a}f\|_{L^r(D)}
\notag\\ 
&\leq 
Ct^{-\frac{3}{2}(\frac{1}{q}-\frac{1}{r})-\frac{i}{2}
+\frac{\alpha}{4}+\max\{\frac{\alpha}{4},\frac{\beta}{2}\}
+\varepsilon}\|(1+|x|)^\alpha(1+|x|-x_1)^\beta f\|_{L^q(D)}
\notag
\end{align} 
for $t\geq 1$. 
\item 
Let $j=0,1$ and let 
$1<q_i\leq r\leq\infty~(q_i\ne\infty,i=1,2,3,4)$ and small 
$\alpha,\beta>0$ satisfy 
$\alpha<3(1-1/q_3),\beta<\min\{1-1/q_1,1-1/q_2\},
\alpha+\beta<3(1-1/q_1)$. 
We suppose $q\leq r<r_1$ if $j=1$, where $r_1(>3)$ depends on 
$\alpha,\beta$. Then we have 
\begin{align}
\|&(1+|x|)^\alpha(1+|x|-x_1)^\beta 
\nabla^je^{-tA_a}f\|_{L^r(D)}\notag\\
&\leq C
t^{-\frac{3}{2}(\frac{1}{q_1}-\frac{1}{r})-\frac{j}{2}}
\|(1+|x|)^{\alpha}(1+|x|-x_1)^{\beta} f\|_{L^{q_1}(D)}+C
t^{-\frac{3}{2}(\frac{1}{q_2}-\frac{1}{r})-\frac{j}{2}+\alpha}
\|(1+|x|-x_1)^{\beta} f\|_{L^{q_2}(D)}\notag\\
&\qquad+Ct^{-\frac{3}{2}(\frac{1}{q_3}-\frac{1}{r})-\frac{j}{2}
+\frac{\beta}{2}}
\|(1+|x|)^{\alpha}f\|_{L^{q_3}(D)}
+Ct^{-\frac{3}{2}(\frac{1}{q_4}-\frac{1}{r})-\frac{j}{2}
+\alpha+\frac{\beta}{2}}
\|f\|_{L^{q_4}(D)}
\label{lqlrdlarge00}
\end{align} 
for $t\geq 1$. 
\end{enumerate}
\end{thm}
The proof of Theorem \ref{thm1} consists of two steps: 
one is the decay estimate near the 
boundary of $D$; the other is the decay estimate near infinity. 
This procedure is employed by Iwashita \cite{iwashita1989}, 
Kobayashi and Kubo \cite{koku2015} for the Stokes semigroup 
and by Kobayashi and Shibata \cite{kosh1998}, 
Enomoto and Shibata \cite{ensh2004,ensh2005} and 
Hishida \cite{hishida2021} for the Oseen semigroup. 
In those papers, they derived the decay rate $t^{-3/2q}$ for 
given $f\in L^q(D)$ 
in the first step by carrying out a cut-off procedure based on 
the $L^q$-$L^r$ estimates of $S_a(t)$ and on the decay estimate 
near the boundary of $D$ for initial velocity with compact support, 
called the local energy decay estimate, see Proposition \ref{local} 
in Section 6. However, since the temporal decay rate 
should be affected by the spatial decay structure of $f$, 
we expect that the decay rate is better than $t^{-3/(2q)}$ 
if $f$ decays faster than $L^q(D)$, for instance, 
$(1+|x|)^\alpha(1+|x|-x_1)^\beta 
f\in L^q(D)~(\alpha>0,\beta\geq 0)$. 
In fact, we adapt the same procedure as in those papers, 
but deduce better decay rate $t^{-3/(2q)-\eta}$ if 
$(1+|x|)^\alpha(1+|x|-x_1)^\beta 
f\in L^q(D)~(\alpha>0,\beta\geq 0)$, 
where $\eta$ is a positive constant dependent on $\alpha,\beta$, 
see the assertion 1 and 2 of Proposition \ref{local3}. 
We in particular find the rate $t^{-3/(2q)-\alpha/2}$  
in the case $\beta=0$. This rate was not found by \cite{koku2015}, 
in which isotropically weighted $L^q$-$L^r$ estimates of   
the Stokes semigroup and its application 
to the Navier-Stokes equations are studied, 
see also \cite{koku2013}. 
\par To get better convergence rates as in \cite{takahashi2021} 
than those in stability analysis, we also derive 
anisotropically weighted $L^q$-$L^r$ estimates of 
the composite operator $e^{-tA_a}P_D\div$ in 
Theorem \ref{thmdual0}, where 
$P_D$ denotes the Fujita-Kato projection on $D$, 
see Subsection \ref{notation}. 
By the duality argument, our main task is reduced to 
deducation of estimates of $\nabla e^{-tA_{-a}}$ in $L^q$ space 
with the weight $(1+|x|)^{-\alpha}(1+|x|-x_1)^{-\beta}$, 
and we perform two steps 
(estimate near the boundary of $D$ and near infinity) 
again to accomplish this task. It is remarkable that 
the decay rate in the first step is slower than $t^{-3/(2q)}$ 
obtained by Kobayashi and Shibata \cite{kosh1998} because the vector 
field $f$ fulfilling $(1+|x|)^{-\alpha}(1+|x|-x_1)^{-\beta}
f\in L^q(D)~(\alpha>0,\beta\geq 0)$ decays slower than $f\in L^q(D)$, 
see the assertion 3 of Proposition \ref{local3}. 
\par Although our main purpose is the Oseen semigroup 
in anisotropically weighted $L^q$ spaces, 
our argument provides new results on the Stokes semigroup 
in isotropically weighted $L^q$ spaces. We prove that 
\begin{align}\label{gradlqlrdlarge3} 
\|(1+|x|)^\alpha \nabla e^{-tA}P_Df\|_{L^r(D)}\leq 
Ct^{-\frac{3}{2}(\frac{1}{q}-\frac{1}{r})-\frac{1}{2}} 
\|(1+|x|)^\alpha f\|_{L^q(D)} 
\end{align} 
for $t>0$ provided $0<\alpha<1,1<q\leq r\leq 3/(1-\alpha)$ 
and that 
\begin{align}\label{lqlrddual2} 
\|(1+|x|)^{-\alpha}\nabla e^{-tA}P_Df\|_{L^r(D)} 
\leq Ct^{-\frac{3}{2}(\frac{1}{q}-\frac{1}{r}) 
-\frac{1}{2}}\|(1+|x|)^{-\alpha}f\|_{L^q(D)} 
\end{align} 
for $t>0$ provided $0\leq \alpha<1,1<q\leq r\leq 3/(1+\alpha)$, 
where $e^{-tA}$ is the Stokes semigroup. 
As for the case $\alpha=0$, we know 
\begin{align}\label{lqlrdstokes} 
\|\nabla e^{-tA}P_Df\|_{L^r(D)} 
\leq Ct^{-\frac{3}{2}(\frac{1}{q}-\frac{1}{r}) 
-\frac{1}{2}}\|f\|_{L^q(D)} 
\end{align} 
for $t>0,1<q\leq r\leq 3$, which  
was proved by Iwashita \cite{iwashita1989} and 
Maremonti and Solonnikov \cite{maso1997}. We note that 
the restriction $1<q\leq r\leq 3$ is optimal in the sense that 
it is impossible to have (\ref{lqlrdstokes}) for $1<q\leq r<q_0$ 
with $q_0>3$. This optimality was first pointed out 
by \cite{maso1997} and another proof 
was provided by Hishida \cite{hishida2011}, in which 
the key is the relation between the restriction of exponent and 
the summability of the steady Stokes flow at spatial infinity. 
As for the case $\alpha>0$, the estimate (\ref{gradlqlrdlarge3}) 
for $q\leq r\leq 3$ 
was obtained by Kobayashi and Kubo \cite{koku2015}, 
see also \cite{koku2013}, 
but the optimality has not been known until now. 
In the present paper, we derive (\ref{gradlqlrdlarge3}) for 
$q\leq r\leq 3/(1-\alpha)$ and prove 
the optimality of the upper bound $3/(1-\alpha)$ by 
adapting the argument due to \cite{hishida2011} 
to weighted $L^q$ spaces, see Theorem \ref{thmoptimal}. 
The present paper also provides 
(\ref{lqlrddual2}) for $q\leq r\leq 3/(1+\alpha)$ as well as 
the optimality of $3/(1+\alpha)$ in Theorem \ref{thmoptimal}. 
The optimality of the gradient estimates for the Oseen semigroup 
in the same sense above has not known and it will be discussed in 
Remark \ref{rmkoseen}. 
\par As applications of anisotropically 
weighted $L^q$-$L^r$ estimates 
of the Oseen semigroup, we construct a solution to (\ref{NS32}) 
or to (\ref{NS3}) in anisotropically weighted $L^q$ spaces to deduce  
the spatial-temporal behavior of the solution. 
The main result on the stability of stationary solutions 
is the following. The precise statement on this result is 
stated in Theorem \ref{stabilitythm}.
\begin{thm}\label{thm2} 
Let $\alpha\geq 0,0\leq \beta<1/3$ satisfy $\alpha+\beta<1$. 
If the velocity $a$ and the $L^3$ norm of initial perturbation $b$, 
which is of class $(1+|x|)^{\alpha}(1+|x|-x_1)^{\beta}b\in L^3(D)$, 
are small enough, 
then the problem $($\ref{NS32}$)$ admits a solution $v$ which enjoys 
\begin{align}
&\|(1+|x|)^{\alpha}(1+|x|-x_1)^{\beta} v(t)\|_{r,D}
=o(t^{-\frac{1}{2}+\frac{3}{2r}+\alpha+\frac{\beta}{2}}), 
\label{anisov0}\\
&\|(1+|x|)^{\alpha}(1+|x|-x_1)^{\beta}\nabla v(t)\|_{3,D}
=o(t^{-\frac{1}{2}+\alpha+\frac{\beta}{2}})\label{anisov20}
\end{align} 
as $t\rightarrow\infty$ for all $r\in[3,\infty]$.
\end{thm} 
The proof of this theorem is accomplished by 
adapting the argument due to Enomoto and Shibata \cite{ensh2005} 
to analyze four norms appeared in the RHS of (\ref{lqlrdlarge00}). 
We note that the smallness of 
$\|(1+|x|)^{\alpha}(1+|x|-x_1)^{\beta}b\|_{L^3(D)}$ is not assumed 
in this theorem. 
The rate in (\ref{anisov0}) with $\beta=0$ is 
$-1/2+3/(2r)=-3(1/3-1/r)/2$, which is same as the one of 
the usual $L^3$-$L^r$ estimate of the Oseen semigroup,  
and the loss $\alpha$. This loss is less than 
the one in Bae and Roh \cite{baro2012} who derived 
\begin{align*}
\|(1+|x|)^\alpha v(t)\|_{L^r(D)}=O(t^{-\frac{3}{2} 
(\frac{1}{q}-\frac{1}{r})+\frac{1+\alpha}{2}+\varepsilon}) 
\end{align*} 
as $t\rightarrow\infty$ for $\varepsilon>0$, where 
$\alpha\in(0,1/2),q\in(2,3),r(\geq q)$ are given some constants. 
Since the proof of Theorem \ref{thm2} works well for 
the problem (\ref{NS3}), we can get 
(\ref{anisov0})--(\ref{anisov20}), 
in which $v$ is replaced by $w$. 
However, because the fluid is initially at 
rest and because the stationary solution $u_s$ belongs to $L^q(D)$ 
with $q<3$, to be precise, $u_s\in L^q(D)$ for all $q>2$, 
we can expect better temporal rate, that is, 
\begin{align}
&\|(1+|x|)^{\alpha}(1+|x|-x_1)^{\beta}w(t)\|_{L^r(D)}
=O(t^{-\frac{1}{2}+\frac{3}{2r}-\frac{1}{4}+\varepsilon
+\alpha+\frac{\beta}{2}}), 
\label{anisow0}\\
&\|(1+|x|)^{\alpha}(1+|x|-x_1)^{\beta} \nabla w(t)\|_{L^3(D)}
=O(t^{-\frac{1}{2}-\frac{1}{4}+\varepsilon
+\alpha+\frac{\beta}{2}})\label{anisow20}
\end{align} 
as $t\rightarrow\infty$ for any $\varepsilon>0$. 
In fact, we prove the following main result on the starting problem,  
which is stated precisely in Theorem \ref{attainthm}. 
\begin{thm}\label{thm3}
Let $0<\alpha,\beta<1/3$. 
Given $\psi$ satisfying $(\ref{psidef})$, 
we set $M=\max_{t\in\R}|\psi'(t)|$. 
If $(M+1)a$ is small enough, then 
the problem $($\ref{NS3}$)$ admits a solution $w$ 
which enjoys $($\ref{anisow0}$)$--$($\ref{anisow20}$)$ 
as $t\rightarrow\infty$ for all $r\in[3,\infty]$ and 
$\varepsilon>0$. 
\end{thm}
We prove Theorem \ref{thm3} 
by employing the idea in the previous study \cite{takahashi2021} 
and by applying (\ref{lqlrdlarge00}) as well as the estimates of 
$e^{-tA_a}P_D\div.$ The key step is to prove the weighted 
$L^3$ estimate of the solution, that is, 
\begin{align}\label{attainrate3}
\|(1+|x|)^{\alpha}(1+|x|-x_1)^{\beta}w(t)\|_{L^3(D)}
=O(t^{-\frac{1}{4}+\varepsilon+\alpha+\frac{\beta}{2}}). 
\end{align} 
Once we have (\ref{attainrate3}), 
the other properties can be derived by the similar 
argument to Enomoto and Shibata \cite{ensh2005}. 
To prove (\ref{attainrate3}), 
we first derive slower rate 
$\|(1+|x|)^{\alpha}(1+|x|-x_1)^{\beta}w(t)\|_{L^3(D)} 
=O(t^{-\mu/2+\alpha+\beta/2})$ 
with some $\mu \in (0,1/4)$, which yields 
better properties of the other norms of the solution. 
With them at hand, we repeat improvement of the estimate 
of weighted $L^3$ norm step by step to find (\ref{attainrate3}). 
\par The paper is organized as follows. In Section 2 we introduce 
the notation and give the main theorems. 
Section 3 is devoted to the proof of Theorem \ref{muckenthm}, 
which provides the necessary and 
sufficiently condition on $\alpha,\beta$ so that 
$(1+|x|)^\alpha(1+|x|-x_1)^\beta\in\A_q(\R^3).$ 
We prove in Section 4 
the anisotropically weighted $L^q$-$L^r$ smoothing action near $t=0$ 
of the Oseen semigroup (Theorem \ref{lqlrd0}). 
To derive the large time behavior of the Oseen semigroup, 
we need the decay estimates in the whole space 
and near the boundary of $D$ in the weighted $L^q$ framework. 
They are discussed in Section 5 and 6, respectively. 
In Section 7, we derive the anisotropically weighted 
$L^q$-$L^r$ estimates of the Oseen semigroup 
(Theorem \ref{lqlrd}, \ref{thmdual0} and \ref{lqlrd2}).   
The optimality of the restriction 
for the gradient estimate of the Stokes semigroup 
(Theorem \ref{thmoptimal}) is also proved. 
The final section is devoted to the proof of 
Theorem \ref{stabilitythm} and Theorem \ref{attainthm}, in which 
we construct nonstationary solutions 
in anisotropically weighted $L^q$ spaces to 
capture the spatial-temporal behavior of solutions.

\section{Main theorems}\label{maintheorem}
\quad In this section, we first introduce some notation 
and after that, we give our main theorems.
\subsection{Notation}\label{notation}
\quad We start with the class of weights. 
The Muckenhoupt class $\A_q(\R^3)$ is defined as follows. 
\begin{df}\label{muckendef}
Let $1\leq q<\infty$. 
A nonnegative function $\rho\in L^1_{\rm loc}(\R^3)$ belongs to 
the Muckenhoupt class $\A_q(\R^3)$ 
if there is a constant $C>0$ such that 
\begin{align*}
\left(\frac{1}{|B_r(x)|}\int_{B_r(x)} \rho(y)\,dy\right)
\left(\frac{1}{|B_{r}(x)|}\int_{B_r(x)}
\rho(y)^{-\frac{1}{q-1}}\,dy\right)^{q-1}\leq C,
\qquad &1<q<\infty,\\
\frac{1}{|B_r(x)|}\int_{B_r(x)} 
\rho(y)\,dy\leq C\essinf_{y\in B_r(x)}\rho(y),\qquad &q=1
\end{align*}
for all balls $B_r(x):=\{y\in\R^3\mid |y-x|<r\}\subset \R^3$, 
where $|B_r(x)|$ denotes the Lebesgue measure of $B_r(x)$. 
\end{df}
\par Following Farwig and Sohr 
\cite[Definition 2.5]{faso1997}, 
we introduce a restricted class of $\A_q(\R^3)$ on exterior 
domain $D$.
\begin{df}\label{defaqd}
Let $1\leq q<\infty$. A nonnegative function 
$\rho\in L^1_{\rm loc}(D)$ belongs to $\A_q(D)$ if 
$\rho=\widetilde{\rho}|_{D}$ with some  
$\widetilde{\rho}\in\A_q(\R^3)$ and if 
there exist a bounded domain $G\subset D$ 
and a constant $\varepsilon>0$ such that 
$\{x\in D \,;\,{\rm dist}(x,\partial D)<\varepsilon\}\subset G$ 
and $\rho\in C(\overline{G}),~\rho|_{\overline{G}}>0$.
\end{df}
\par We define some function spaces.  
Let $\Omega\subset \R^3$ be a domain. 
By $C^{\infty}_0(\Omega)$, we denote the set of all 
$C^{\infty}$ functions 
with compact support in $\Omega$. For $1\leq q\leq \infty$,  
nonnegative integer $m$ and nonnegative function 
$\rho\in L_{\rm loc}^1(\Omega)$, 
weighted Lebesgue and Sobolev spaces are 
defied by
\begin{align*}
&L_{\rho}^q(\Omega):=\left\{u\in L^1_{\rm loc}(\Omega)~\middle|~
\int_{\Omega}|u(x)|^q\rho\,dx<\infty\right\},\quad 1\leq q<\infty,\\
&L_{\rho}^\infty(\Omega):=\{u\in L^1_{\rm loc}(\Omega)\mid
\esssup_{x\in\Omega}|\rho(x)u(x)|<\infty\},\\
&W^{m,q}_{\rho}(\Omega):=\{u\in L^1_{\rm loc}(\Omega)
\mid u,\nabla u,
\cdots, \nabla^mu\in L^q_{\rho}(\Omega)\}
\end{align*}
with norms  
\begin{align*}
&\|u\|_{L^q_\rho(\Omega)}=\left(\int_{\Omega}|u(x)|^q\rho\,dx
\right)^{\frac{1}{q}},\quad
1\leq q<\infty,\qquad \|u\|_{L^\infty_\rho(\Omega)}=
\esssup_{x\in\Omega}|\rho(x)u(x)|,\\ 
&\|u\|_{W^{m,q}_{\rho}(\Omega)}=\left(\sum_{k=0}^m
\|\nabla^k u\|^q_{L^q_\rho(\Omega)}\right)^{\frac{1}{q}},
\end{align*}
respectively. If $\rho\equiv 1$, then we simply write
\begin{align*} 
L^q(\Omega)=L^q_{1}(\Omega),\quad
\|\cdot\|_{q,\Omega}=\|\cdot\|_{L^q_1(\Omega)},\quad  
W^{m,q}(\Omega)=W^{m,q}_{1}(\Omega),\quad\|\cdot\|_{W^{m,q}
(\Omega)}
=\|\cdot\|_{W^{m,q}_1(\Omega)},
\end{align*}
from which it follows that 
\begin{align*}
\|\rho u\|_{q,\Omega}=\|u\|_{L^q_{\rho^q}(\Omega)},
\quad 1\leq q<\infty,\qquad 
\|\rho u\|_{\infty,\Omega}=\|u\|_{L^\infty_{\rho}(\Omega)}.
\end{align*}
For $1\leq q<\infty$ and nonnegative function 
$\rho\in L_{\rm loc}^1(\Omega)$, 
the dual of the space $L^q_\rho(\Omega)$ is 
$L^{q'}_{\rho'}(\Omega)$, where $1/q+1/q'=1$ and 
$\rho'=\rho^{-1/(q-1)}$ if $1<q<\infty$ and $\rho'=\rho^{-1}$ if 
$q=1$. Furthermore, 
$C_0^\infty(\Omega)$ is dense in $L^q_{\rho}(\Omega)$ for 
$1\leq q<\infty$ and 
nonnegative function $\rho\in L^1_{\rm loc}(\Omega)$. 
To conclude this, 
it is enough to prove that given characteristic function 
$\chi_E$, there exists $\{\varphi_n\}_{n=1}^\infty
\subset C_0^\infty(\Omega)$ such that 
\begin{align}\label{phin}
\|\varphi_n-\chi_E\|_{L^q_{\rho}(\Omega)}\rightarrow 0
\end{align}
as $n\rightarrow\infty,$ where $E\subset \Omega$ is a measurable set 
with the finite Lebesgue measure. 
Let $\widetilde{\chi}_E$ be the function on $\R^3$ such that  
$\widetilde{\chi}_E(x)=\chi_E(x)$ for $x\in \Omega$ and 
$\widetilde{\chi}_E(x)=0$ for $x\in\R^3\setminus \Omega$. 
Let $0\leq h\in C_0^\infty(\R^3)$ with 
${\rm supp}~h\subset B_1(0)$ and $\int_{\R^3}h(x)\,dx=1$ and 
set $h_n(x):=n^3h(nx)$ for $n\in\N$. 
Since $\|h_n*\widetilde{\chi}_E-\widetilde{\chi}_E\|_{q,\R^3}
\rightarrow 0$ as $n\rightarrow \infty$, there exists a 
subsequence, again denoted by $h_n*\widetilde{\chi}_E$, such that 
$h_n*\widetilde{\chi}_E(x)-\widetilde{\chi}_E(x)\rightarrow 0$
as $n\rightarrow \infty$ a.e. $x\in\R^3$ and that 
$h_{n}*\widetilde{\chi}_E\in C_0^\infty (\widetilde{E})$ 
for $n\in\N$, where $*$ denotes the convolution,  
$\widetilde{E}(\supset E)$ is a 
finite measure set and the closure of $\widetilde{E}$ 
is contained in $\Omega$. Moreover, we have 
$|h_n*\widetilde{\chi}_E(x)-\chi_E(x)|^q\rho(x)\leq 
2^q\rho(x)$ for $x\in\widetilde{E}$ and 
$2^q\rho\in L^1(\widetilde{E})$ 
because $\rho\in L^1_{\rm loc}(\Omega)$. 
We thus apply Lebesgue's convergence theorem to get 
\begin{align*}
\int_\Omega|h_n*\widetilde{\chi}_E(x)
-\chi_E(x)|^q\rho(x)\,dx=\int_{\widetilde{E}}
|h_n*\widetilde{\chi}_E(x)-\chi_E(x)|^q
\rho(x)\,dx\rightarrow 0
\end{align*}
as $n\rightarrow\infty,$ which yields (\ref{phin}) with 
$\varphi_n=h_n*\widetilde{\chi}_E$.  
The completion of $C_0^\infty(\Omega)$ in $W^{m,q}(\Omega)$ is 
denoted by $W_0^{m,q}(\Omega)$. 
\par We next introduce some solenoidal function spaces. 
Let $\Omega=\R^3$ or exterior domain $D$. 
By $C_{0,\sigma}^\infty(\Omega)$, we denote the 
set of all $C_0^{\infty}$-vector fields $f$ which satisfy $\div f=0$ 
in $\Omega$. For $1<q<\infty$ and $\rho\in \A_q(\Omega)$, 
$L^{q}_{\rho,\sigma}(\Omega)$ denote the 
completion of $C_{0,\sigma}^\infty(\Omega)$ in $L_{\rho}^q(\Omega)$. 
Due to Farwig and Sohr \cite{faso1997}, 
we have the Helmholtz decomposition:
\begin{align}\label{helmholtz}
L_{\rho}^q(\Omega)=L^q_{\rho,\sigma}(\Omega)
\oplus\{\nabla p\in L_{\rho}^q(\Omega) 
\mid p\in L^q_{\rm{loc}}(\overline{\Omega})\}
\end{align}
and the duality relation: 
$L^q_{\rho,\sigma}(\Omega)^*=L^{q'}_{\rho',\sigma}(\Omega)$ 
for $1<q<\infty$ and $\rho\in\A_q(\Omega)$, see also 
Fujiwara and Morimoto \cite{fumo1977}, 
Miyakawa \cite{miyakawa1982} and Simader and Sohr \cite{siso1992} 
for the case $\rho\equiv 1$. 
For simplicity, we write $L^q_\sigma(\Omega)=L^q_{1,\sigma}(\Omega).$ 
Let $P_{q,\rho,\Omega}$ denote the Fujita-Kato projection 
from $L_{\rho}^q(\Omega)$ onto $L^q_{\rho,\sigma}(\Omega)$ 
associated with the decomposition. 
Since $P_{q_1,\rho_1,\Omega}u=P_{q_2,\rho_2,\Omega}u$ 
for $u\in L^{q_1}_{\rho_1}(\Omega)\cap L^{q_2}_{\rho_2}(\Omega)$ 
with $1<q_1,q_2<\infty,\rho_1\in\A_{q_1}(\Omega),
\rho_2\in\A_{q_2}(\Omega)$, 
we simply write $P_\Omega=P_{q,\rho,\Omega}$\,. 
When $\Omega=\R^3$, the Fujita-Kato projection is described as 
$P_{\R^3}=I+\mathcal{R}\otimes\mathcal{R}$, where 
$I$ is the identity map and 
$\mathcal{R}=\nabla (-\Delta)^{-1/2}$ is the Riesz transform. 
We denote various constants by $C$ 
and they may change from line to line. 
The constant dependent on $A,B,\cdots$ is 
denoted by $C(A,B,\cdots)$. 
Finally, if there is no confusion, 
we use the same symbols for denoting spaces of 
scalar-valued functions and those of vector-valued ones.

\subsection{Main theorems}
\quad We are now in a position to give our main results. 
Our first theorem provides the 
necessary and sufficient condition on $\alpha,\beta$ so that 
$(1+|x|)^{\alpha}(1+|x|-x_1)^\beta\in \A_q(\R^3)$. We already know 
the following theorem with $q=2$ from 
Farwig \cite[Theorem 4.1 and Remark 4.7]{farwig1992t}. 
Furthermore, the sufficiency of (\ref{alphabeta}) was proved by 
Kra\v{c}mer, Novotn\'{y} and Pokorn\'{y} 
\cite[Theorem 5.2]{krnopo2001}, 
but it has not known whether the condition is 
also necessary condition.

\begin{thm}\label{muckenthm}
Let $1<q<\infty$. Then 
$\rho_{\alpha,\beta}(x):=(1+|x|)^{\alpha}(1+|x|-x_1)^\beta
\in \A_q(\R^3)$ 
if and only if $($\ref{alphabeta}$)$ is fulfilled.
\end{thm}

Let $\Omega$ be the whole space $\R^3$ or exterior domain $D$.
Given $1<q<\infty$ and $\alpha,\beta$ satisfying 
\begin{align}\label{alphabetaq}
-\frac{1}{q}<\beta<1-\frac{1}{q},
\quad -\frac{3}{q}<\alpha+\beta<3\left(1-\frac{1}{q}\right),
\end{align}
we set 
\begin{align}\label{rho}
\rho(x)=(1+|x|)^{\alpha q}(1+|x|-x_1)^{\beta q}.
\end{align}
By Definition \ref{defaqd} and Theorem \ref{muckenthm}, 
we find that the weight $\rho$ belongs to 
$\A_q(D)$ as well as $\A_q(\R^3)$. 
Therefore, due to Farwig and Sohr \cite{faso1997}, we have 
the Helmholtz decomposition (\ref{helmholtz}) and the bounded 
projection operator 
$P_\Omega:L^q_{\rho}(\Omega)\rightarrow L_{\rho,\sigma}^q(\Omega)$, 
then define the Oseen operator 
$A_a:L^q_{\rho,\sigma}(\Omega)
\rightarrow L^q_{\rho,\sigma}(\Omega)~(a\in\R)$ by
\begin{align*}
\D(A_a)=W^{2,q}_{\rho}(\Omega)
\cap W^{1,q}_0(\Omega)\cap L^{q}_{\rho,\sigma}(\Omega),\quad 
A_au=-P_\Omega[\Delta u-a\partial_{x_1}u].
\end{align*} 
We simply write the Stokes operator $A=A_0$. 
We state in the following theorem that 
$-A_a$ generates an analytic semigroup 
in $L^{q}_{\rho,\sigma}(D)$ possessing 
the $L^q$-$L^r$ smoothing action near the initial time.

\begin{thm}\label{lqlrd0}
Given $a_0>0$ arbitrarily, we assume $|a|\leq a_0$. 
\begin{enumerate}
\item Let $1<q<\infty$ and let $\alpha,\beta$ 
satisfy $($\ref{alphabetaq}$)$.
Then $-A_a$ generates an
analytic $C_0$-semigroup $\{e^{-tA_a}\}_{t\geq 0}$
in $L^{q}_{\rho,\sigma}(D)$, 
where $\rho$ is given by $($\ref{rho}$)$. 
If in particular $a=0$, then 
the Stokes semigroup $\{e^{-tA}\}_{t\geq 0}$ 
is a bounded analytic $C_0$-semigroup in $L^{q}_{\rho,\sigma}(D)$.    
\item Let $1<q\leq r\leq \infty ~(q\ne\infty)$ 
and let $\alpha,\beta$ satisfy $($\ref{alphabetaq}$)$. 
For every multi-index $k~(|k|\leq 1)$, there exists a constant 
$C=C(D,a_0,q,r,\alpha,\beta,k)$, independent of $a$, such that 
\begin{align}\label{lqlrdsmall}
\|(1+|x|)^\alpha(1+|x|-x_1)^\beta\partial^k_{x}e^{-tA_a}P_Df\|_{r,D}
\leq 
Ct^{-\frac{3}{2}(\frac{1}{q}-\frac{1}{r})-\frac{|k|}{2}}
\|(1+|x|)^\alpha(1+|x|-x_1)^\beta f\|_{q,D}
\end{align}
for all $t\leq 3$, $f\in L^{q}_{\rho}(D)$.
\end{enumerate}
\end{thm}

The next theorem asserts the 
large time behavior of the Oseen semigroup. 
We note that the exponents $q_i~(i=1,2,3,4)$ 
in the next theorem may not coincide with each other. 
This remark is important 
when we treat the terms $f_1$ and $f_2$ (see (\ref{f1})--(\ref{f2})) 
to study the starting problem (Theorem \ref{attainthm} below),  
which is performed in Section 8. 
\begin{thm}\label{lqlrd}
Given $a_0>0$ arbitrarily, we assume $a\in[0,a_0]$. 
\begin{enumerate}
\item Let $1<q_i<\infty~(i=1,2,3,4)$, 
$1<r\leq\infty$ and $\alpha,\beta\geq 0$ satisfy 
\begin{align}
&1<q_4\leq q_i\leq q_1\leq r\leq\infty\quad(i=2,3),\label{qr0}\\
&\alpha<\min\left\{3\left(1-\frac{1}{q_3}\right),1\right\},
\quad \beta<\min\left\{1-\frac{1}{q_2},
\frac{1}{3}\right\},\quad 
\alpha+\beta<\min\left\{3\left(1-\frac{1}{q_1}\right),1\right\}.
\label{alphabeta2}
\end{align} 
We set
\begin{align}\label{rhodef}
\rho_1(x)=(1+|x|)^{\alpha q_1}(1+|x|-x_1)^{\beta q_1},\quad 
\rho_2(x)=(1+|x|-x_1)^{\beta q_2},\quad 
\rho_3(x)=(1+|x|)^{\alpha q_3}.
\end{align}
Then there exists a constant 
$C(D,a_0,q_1,q_2,q_3,q_4,r,\alpha,\beta)$, 
independent of $a$, such that 
\begin{align}
\|(1+|x|)^\alpha(1+|x|-x_1)^\beta e^{-tA_a}P_Df\|_{r,D}
\leq 
C\sum_{i=1}^4
t^{-\frac{3}{2}(\frac{1}{q_i}-\frac{1}{r})+\eta_i}
\|(1+|x|)^{\gamma_i}(1+|x|-x_1)^{\delta_i} f\|_{q_i,D}
\label{lqlrdlarge}
\end{align}
for all $t\geq 3$, $f\in \displaystyle\bigcap^3_{i=1}
L^{q_i}_{\rho_i}(D)\cap L^{q_4}(D)$, 
where
$\gamma_i,\delta_i,\eta_i$ are defined by
%\begin{equation}
%\begin{aligned}
\begin{align}\label{gammadeltaeta}
&(\gamma_1,\gamma_2,\gamma_3,\gamma_4)=
(\alpha,0,\alpha,0),
\quad
(\delta_1,\delta_2,\delta_3,\delta_4)=(\beta,\beta,0,0),\notag
\\
&(\eta_1,\eta_2,\eta_3,\eta_4)=
\begin{cases}
\left(0,\alpha,
\displaystyle\frac{\beta}{2},
\alpha+\frac{\beta}{2}\right)&\quad {\rm if}~a>0,\\[13pt]
\left(0,\displaystyle\frac{\alpha}{2},
\frac{\beta}{2},
\frac{\alpha}{2}+\frac{\beta}{2}\right)&\quad {\rm if}~a=0.
\end{cases}.
\end{align}
\item Let $1<r<\infty$, $1<q_i<\infty~(i=1,2,3,4),\alpha,\beta>0$ 
satisfy $($\ref{alphabeta2}$)$. 
If $\alpha<2/3$ $($resp. $\alpha\geq 2/3)$, we suppose 
\begin{align*}
&1<q_4\leq q_i\leq q_1\leq r<\min\left\{\frac{3}{1-\alpha-\beta},
\frac{3}{1-\frac{3\alpha}{2}}\right\}\quad(i=2,3)\\
&\left({\rm resp.}~1<q_4\leq q_i\leq q_1\leq r
<\frac{3}{1-\alpha-\beta}\quad (i=2,3)\right).
\end{align*}
Then there exists a constant
$C(D,a_0,q_1,q_2,q_3,q_4,r,\alpha,\beta)$, independent of $a$,  
such that 
\begin{align}
\|&(1+|x|)^\alpha(1+|x|-x_1)^\beta 
\nabla e^{-tA_a}P_Df\|_{r,D}\notag\\
&\leq C\sum_{i=1}^4
t^{-\frac{3}{2}(\frac{1}{q_i}-\frac{1}{r})-\frac{1}{2}+\eta_i}
\|(1+|x|)^{\gamma_i}(1+|x|-x_1)^{\delta_i} f\|_{q_i,D}
\label{gradlqlrdlarge}
\end{align} 
for all $t\geq 3$ and $f\in \displaystyle\bigcap^3_{i=1}
L^{q_i}_{\rho_i}(D)\cap L^{q_4}(D)$. 
%where
%$\gamma_i,\delta_i,\eta_i$ are given by $($\ref{gammadeltaeta}$)$. 
\item Let $1<q_i<\infty~(i=1,2),1<r<\infty$ and 
$\beta\geq 0$ satisfy $\beta<\min\{1-1/q_1,1/3\}$ and 
$1<q_2\leq q_1\leq r\leq 3$.
Then there exists a constant 
$C(D,a_0,q_1,q_2,r,\alpha,\beta)$ such that 
\begin{align}
\|(1+|x|-x_1)^\beta\nabla e^{-tA_a}P_Df\|_{r,D}
\leq C&t^{-\frac{3}{2}(\frac{1}{q_1}-\frac{1}{r})-\frac{1}{2}}
\|(1+|x|-x_1)^{\beta}f\|_{q_1,D}\notag\\
&+Ct^{-\frac{3}{2}(\frac{1}{q_2}-\frac{1}{r})-\frac{1}{2}
+\frac{\beta}{2}}
\|f\|_{q_2,D}
\label{gradlqlrdlarge42}
\end{align}
for all $t\geq 3$ 
and $f\in L^{q_1}_{(1+|x|-x_1)^{\beta q_1}}(D)\cap L^{q_2}(D)$.
\item Let $1<q_i<\infty~(i=1,2),1<r<\infty$ and 
$\alpha\geq 0$ satisfy $\alpha<\min\{3(1-1/q_1),1\}$ and 
\begin{align*}
1<q_2\leq q_1\leq r\leq \frac{3}{1-\alpha}.
%\left(<\frac{9}{2}\right),
\end{align*}
Then there exists a constant 
$C(D,a_0,q_1,q_2,r,\alpha)$ such that 
\begin{align}
\|(1+|x|)^\alpha\nabla e^{-tA_a}P_Df\|_{r,D}
\leq Ct^{-\frac{3}{2}(\frac{1}{q_1}-\frac{1}{r})-\frac{1}{2}}
\|(1+|x|)^{\alpha}f\|_{q_1,D}+
Ct^{-\frac{3}{2}(\frac{1}{q_2}-\frac{1}{r})-\frac{1}{2}+\eta_2}
\|f\|_{q_2,D}
\label{gradlqlrdlarge4}
\end{align}
for all $t\geq 3$ 
and $f\in L^{q_1}_{(1+|x|)^{\alpha q_1}}(D)\cap L^{q_2}(D)$, where 
$\eta_2$ is given by $($\ref{gammadeltaeta}$)$.
\end{enumerate}
\end{thm}

\begin{rmk}
We suppose that $1<q_i<\infty~(i=1,2,3,4)$, 
$1<r\leq\infty$ and $\alpha,\beta\geq 0$ satisfy the assumption of 
the assertion 1 of Theorem \ref{lqlrd}. 
By the proof of Theorem \ref{lqlrd}, we can find  
\begin{align}
\|&(1+|x|)^\alpha(1+|x|-x_1)^\beta 
e^{-tA_a}P_Df\|_{r,D}\notag\\
&\leq C
t^{-\frac{3}{2}(\frac{1}{q_1}-\frac{1}{r})}
\|(1+|x|)^{\alpha}(1+|x|-x_1)^{\beta} f\|_{q_1,D}+C
t^{-\frac{3}{2}(\frac{1}{q_2}-\frac{1}{r})}\{(1+at)^\alpha
+t^{\frac{\alpha}{2}}\}\|(1+|x|-x_1)^{\beta} f\|_{q_2,D}\notag\\
&\qquad+Ct^{-\frac{3}{2}(\frac{1}{q_3}-\frac{1}{r})+\frac{\beta}{2}}
\|(1+|x|)^{\alpha}f\|_{q_3,D}
+Ct^{-\frac{3}{2}(\frac{1}{q_4}-\frac{1}{r})}
\{(1+at)^\alpha 
t^{\frac{\beta}{2}}+t^{\frac{\alpha}{2}+\frac{\beta}{2}}\}
\|f\|_{q_4,D}
\label{lqlrdlarge4}
\end{align}
for $t\geq 3$ and $a>0$. 
Therefore, the estimate (\ref{lqlrdlarge}) for $a=0$ is recovered 
by passing to the limit $a\rightarrow 0$ in (\ref{lqlrdlarge4}). 
Similarly, the case $a=0$ is recovered 
in the assertion 2 and 4 of Theorem \ref{lqlrd}.
\end{rmk}

\begin{rmk}\label{rmkstokes}
It is not known whether we can obtain  
homogeneous $L^q$-$L^r$ estimates 
if $\beta>0$ or if $a>0$, while 
Kobayashi and Kubo \cite{koku2015} derived 
\begin{align*}
\|(1+|x|)^\alpha e^{-tA}P_Df\|_{r,D}\leq 
Ct^{-\frac{3}{2}(\frac{1}{q}-\frac{1}{r})}
\|(1+|x|)^\alpha f\|_{q,D}
\end{align*}
for $t>0,f\in L^q_{(1+|x|)^{\alpha q}}(D)$
provided $1<q\leq r<\infty$ and (\ref{gradlqlrdlarge3})
for $t>0,f\in L^q_{(1+|x|)^{\alpha q}}(D)$ provided
$1<q\leq r\leq 3$.
\end{rmk}

Our next aim is to derive the estimate of 
$e^{-tA_a}P_D\div$ in 
$L^q_{\rho}(D)$, 
which will be used in the starting problem.
To this end, we employ the duality argument based on the 
estimate of $\nabla e^{-tA_{-a}}$ in
$L^q_{\rho_-}(D)$, where $1<q<\infty,\alpha,\beta\geq 0$ fulfill   
\begin{align}\label{alphabeta4}
\beta<\frac{1}{q},\quad \alpha+\beta<\frac{3}{q} 
\end{align}
and $\rho_-$ is defined by 
\begin{align}\label{rhominus}
\rho_{-}(x)=(1+|x|)^{-\alpha q}(1+|x|-x_1)^{-\beta q}.
\end{align}
\begin{thm}\label{thmdual0}
Given $a_0>0$ arbitrarily, we assume $a\in(0,a_0]$.
\begin{enumerate}
\item Let $1<q\leq r\leq 3$ and $\alpha,\beta\geq 0$ 
satisfy $1/q-1/r<1/3$ and $($\ref{alphabeta4}$)$. 
Then there exists a constant 
$C=C(D,a_0,q,r,\alpha,\beta)$, independent of $a$, such that 
\begin{align}
\|&(1+|x|)^{-\alpha}(1+|x|-x_1)^{-\beta}\nabla 
e^{-tA_{-a}}P_Df\|_{r,D}\nonumber\\
&\leq 
Ct^{-\frac{3}{2}(\frac{1}{q}-\frac{1}{r})-\frac{1}{2}}
(1+t)^{\alpha+\frac{\beta}{2}}
\|(1+|x|)^{-\alpha}(1+|x|-x_1)^{-\beta} f\|_{q,D}
\label{lqlrdgraddual0}
\end{align}
for all $t>0$ and $f\in L^{q}_{\rho_{-}}(D)$, where $\rho_-$ is 
given by $($\ref{rhominus}$)$.
\item Let $3/2\leq q\leq r<\infty$ and $\alpha,\beta\geq 0$ 
satisfy $1/q-1/r<1/3$ and 
\begin{align*}
\beta<1-\frac{1}{r},\quad 
\alpha+\beta<3\left(1-\frac{1}{r}\right). 
\end{align*}  
Then there exists a constant 
$C=C(D,a_0,q,r,\alpha,\beta)$, 
independent of $a$, such that 
\begin{align}
\|&(1+|x|)^{\alpha}(1+|x|-x_1)^{\beta}
e^{-tA_a}P_D\div F\|_{r,D}\nonumber\\
&\leq Ct^{-\frac{3}{2}(\frac{1}{q}-\frac{1}{r})-\frac{1}{2}}
(1+t)^{\alpha+\frac{\beta}{2}}
\|(1+|x|)^{\alpha}(1+|x|-x_1)^{\beta} F\|_{q,D}
\label{lqlrddiv0}
\end{align}
for all $t>0$ and $F\in L^{q}_{\rho}(D),$ where 
$\rho$ is given $($\ref{rho}$)$. 
\end{enumerate}
\end{thm}

In order to obtain the main results on the nonlinear problems, 
it suffices to use Theorem \ref{lqlrd} and Theorem \ref{thmdual0}, 
but the rate in those theorems is improved in the following theorem, 
which might be of independent interest. 
If we take $q_i=q~(i=1,2,3,4)$ in the assertion 1 of 
Theorem \ref{lqlrd}, then 
we find that the rate is $-3(1/q-1/r)/2-i/2+\alpha+\beta/2$ 
if $a>0$. The assertion 1 and 2 below yield that 
this rate is slightly improved 
as long as $1/q-1/r<1/3$ is fulfilled. 
Our approach recovers 
the homogeneous estimates of the Stokes semigroup in 
Remark \ref{rmkstokes} as well as extends the 
restriction $1<q\leq r\leq 3$ for the estimate 
(\ref{gradlqlrdlarge3}) to 
$1<q\leq r\leq 3/(1-\alpha)$, see the assertion 3. 
In the assertion 4--6, 
we derive better rate than the one in (\ref{lqlrdgraddual0}) and 
also deduce the estimates of the Stokes semigroup 
when parameters $r,\alpha,\beta$ are more restricted than those in 
the assertion 1 of Theorem \ref{thmdual0}, namely 
$r$ is smaller than $3$ and 
$\alpha,\beta$ are smaller than (\ref{alphabeta4}),

\begin{thm}\label{lqlrd2}
\begin{enumerate}
\item Given $a_0>0$ arbitrarily, we assume $a\in(0,a_0]$. 
Let $\varepsilon>0,i=0,1$ and let 
$1<q\leq r\leq \infty ~(q\ne\infty)$ 
and $\alpha,\beta$ satisfy $1/q-1/r<1/3,$   
\begin{align}\label{alphabetadef}
\alpha>0,\quad 0<\beta<\min\left\{1-\frac{1}{q},
\frac{1}{3}\right\},\quad 
\alpha+\beta<\min\left\{3\left(1-\frac{1}{q}\right),1\right\}.
\end{align}
If $i=1$ and $2\alpha+\beta<1$, we also suppose 
$\varepsilon<(1-2\alpha-\beta)/2$ and 
$r<3/(1-2\alpha-\beta-2\varepsilon).$ 
Then there exists a constant 
$C=C(D,a_0,\varepsilon,i,q,r,\alpha,\beta)$, 
independent of $a$, such that 
\begin{align}
\|&(1+|x|)^\alpha(1+|x|-x_1)^\beta \nabla^i e^{-tA_a}P_Df\|_{r,D}
\notag\\
&\leq 
Ct^{-\frac{3}{2}(\frac{1}{q}-\frac{1}{r})-\frac{i}{2}
+\frac{\alpha}{4}+\max\{\frac{\alpha}{4},\frac{\beta}{2}\}
+\varepsilon}\|(1+|x|)^\alpha(1+|x|-x_1)^\beta f\|_{q,D}
\label{gradlqlrdlarge5}
\end{align}
for all $t\geq 1$, $f\in L^{q}_{\rho}(D)$, where
$\rho$ is given by $($\ref{rho}$)$.
\item Given $a_0>0$ arbitrarily, we assume $a\in(0,a_0]$. 
Let $1<q\leq r\leq \infty ~(q\ne\infty)$ 
and $\alpha\geq 0$ satisfy $1/q-1/r<1/3$ and 
$0\leq \alpha<\min\{3(1-1/q),1\}.$ Let $i=0,1$ and 
we also suppose $r\leq 3/(1-2\alpha)$ 
if $i=1$ and $\alpha<1/2.$ Then there exists a constant 
$C=C(D,a_0,i,q,r,\alpha)$, 
independent of $a$, such that 
\begin{align}\label{gradlqlrdlarge6}
\|(1+|x|)^\alpha \nabla^i e^{-tA_a}P_Df\|_{r,D}
\leq Ct^{-\frac{3}{2}(\frac{1}{q}-\frac{1}{r})-\frac{i}{2}
+\frac{\alpha}{2}}
\|(1+|x|)^\alpha f\|_{q,D}
\end{align}
for all $t\geq 1$ and $f\in L^{q}_{(1+|x|)^{\alpha q}}(D).$ 
\item Let $a=0$. 
Let $1<q\leq r\leq\infty~(q\ne\infty)$ 
and $\alpha\geq 0$ satisfy
$0\leq \alpha<\min\{3(1-1/q),1\}$. 
Given $i=0,1$, we also suppose 
\begin{align*}
1<q\leq r\leq \frac{3}{1-\alpha}
\end{align*}
if $i=1.$ 
Then there exists a constant $C(D,i,q,r,\alpha)$ such that 
\begin{align}\label{gradlqlrdlarge7}
\|(1+|x|)^\alpha \nabla^i e^{-tA}P_Df\|_{r,D}
\leq Ct^{-\frac{3}{2}(\frac{1}{q}-\frac{1}{r})-\frac{i}{2}}
\|(1+|x|)^\alpha f\|_{q,D}
\end{align}
for all $t>0$ and $f\in L^q_{(1+|x|)^{\alpha q}}(D)$. 
\item Given $a_0>0$ arbitrarily, we assume $a\in(0,a_0]$.
Let $\varepsilon>0,1<q\leq r<\infty$ and $\alpha,\beta>0$  
satisfy $1/q-1/r<1/3$, 
\begin{align}
\beta<\frac{1}{r},\quad 
\alpha+\beta<\min\left\{1,\frac{3}{r}\right\}\label{alphabetar2}
\end{align}
and 
\begin{align}
1<q\leq r\leq \frac{3}{1+\alpha+\min\{\frac{\alpha}{2},\beta\}
-2\varepsilon}.\label{alphabetar3}
\end{align} 
Then there exists a constant 
$C=C(D,a_0,\varepsilon,q,r,\alpha,\beta)$, 
independent of $a$, such that 
\begin{align}
\|&(1+|x|)^{-\alpha}(1+|x|-x_1)^{-\beta}\nabla
e^{-tA_{-a}}P_Df\|_{r,D}\notag\\
&\leq 
Ct^{-\frac{3}{2}(\frac{1}{q}-\frac{1}{r})-\frac{1}{2}
+\frac{\alpha}{4}+\max\{\frac{\alpha}{4},\frac{\beta}{2}\}
+\varepsilon}
\|(1+|x|)^{-\alpha}(1+|x|-x_1)^{-\beta} f\|_{q,D}
\label{lqlrddual10}
\end{align}
for $t\geq 1$ and $f\in L^{q}_{\rho_{-}}(D)$, where $\rho_-$ is 
given by $($\ref{rhominus}$)$.
\item Given $a_0>0$ arbitrarily, we assume $a\in(0,a_0]$. 
Let $1<q\leq r<\infty$ and $\alpha\geq 0$ 
satisfy $1/q-1/r<1/3$, 
\begin{align}
0\leq \alpha<\min\left\{1,\frac{3}{r}\right\}\label{alphar}
\end{align}
and 
\begin{align}
1<q\leq r\leq \frac{3}{1+\alpha}.\label{alphar2}
\end{align}
Then there exists a constant 
$C=C(D,a_0,q,r,\alpha)$, independent of $a$, such that 
\begin{align}\label{lqlrddual3}
\|(1+|x|)^{-\alpha}\nabla e^{-tA_{-a}}P_Df\|_{r,D}
\leq Ct^{-\frac{3}{2}(\frac{1}{q}-\frac{1}{r})
-\frac{1}{2}+\frac{\alpha}{2}}\|(1+|x|)^{-\alpha}f\|_{q,D}
\end{align}
for all $t\geq 1$ and $f\in L^{q}_{(1+|x|)^{-\alpha q}}(D)$. 
\item Let $a=0.$
Let $1<q\leq r<\infty$ and $\alpha\geq 0$ 
satisfy $($\ref{alphar}$)$ and $($\ref{alphar2}$)$.  
Then there exists a constant 
$C=C(D,q,r,\alpha)$ such that $($\ref{lqlrddual2}$)$ 
for all $t>0$ and $f\in L^{q}_{(1+|x|)^{-\alpha q}}(D)$.
\end{enumerate}
\end{thm}

Let us consider the Stokes semigroup. 
Let $0\leq \alpha<1.$ By the assertion 3 and 6 
of Theorem \ref{lqlrd2}, it follows that 
(\ref{gradlqlrdlarge3})
for $t>0,f\in L^{q}_{(1+|x|)^{\alpha q}}(D)$ provided  
\begin{align}\label{rcritical}
\frac{3}{3-\alpha}<q\leq r\leq \frac{3}{1-\alpha}
\end{align} 
and that (\ref{lqlrddual2}) 
for $t>0,f\in L^{q}_{(1+|x|)^{-\alpha q}}(D)$ 
provided (\ref{alphar2}). 
Our next aim is to prove the optimality of 
the restrictions (\ref{rcritical}) and (\ref{alphar2}). 
In the $L^q$ framework, 
it was proved by Maremonti and Solonnikov \cite{maso1997} and 
Hishida \cite{hishida2011} that 
the restriction 
$1<q\leq r\leq 3$ is optimal 
in the sense that we cannot have 
\begin{align*} 
\|\nabla e^{-tA}P_Df\|_{r,D} 
\leq Ct^{-\frac{3}{2}(\frac{1}{q}-\frac{1}{r})-\frac{1}{2}} 
\|f\|_{q,D} 
\end{align*} 
for $1<q\leq r<q_0$ with $q_0>3.$ However, we state in the next 
theorem that the conditions (\ref{rcritical}) and (\ref{alphar2}) 
are optimal, which yields that the range of exponent is 
enlarged in $L^q_{(1+|x|)^{\alpha q}}(D)$ 
(restricted in $L^q_{(1+|x|)^{-\alpha q}}(D)$) 
than the one in $L^q(D)$. 

\begin{thm}\label{thmoptimal}
Let $e^{-tA}$ be the Stokes semigroup and let $0\leq \alpha<1.$ 
We have $($\ref{gradlqlrdlarge3}$)$ 
for $t>0,f\in L^{q}_{(1+|x|)^{\alpha q}}(D)$ 
provided $($\ref{rcritical}$)$. Moreover, 
the condition $($\ref{rcritical}$)$ 
is optimal in the sense that it is impossible 
to have $($\ref{gradlqlrdlarge3}$)$ for $3/(3-\alpha)<q\leq r<q_0$ 
with $q_0>3/(1-\alpha).$ We also have $($\ref{lqlrddual2}$)$ 
for $t>0,f\in L^{q}_{(1+|x|)^{-\alpha q}}(D)$ 
provided $($\ref{alphar2}$)$, and the condition $($\ref{alphar2}$)$ 
is optimal in the sense that it is impossible to have 
$($\ref{lqlrddual2}$)$ for $1<q\leq r<q_0$ with $q_0>3/(1+\alpha).$
\end{thm}

Although the rate in (\ref{gradlqlrdlarge5})  
is better than the one in (\ref{lqlrdlarge}), 
(\ref{gradlqlrdlarge}) with $q_i=q$, that is 
\begin{align}
\|&(1+|x|)^\alpha(1+|x|-x_1)^\beta 
\nabla^je^{-tA_a}P_Df\|_{r,D}\notag\\
&\leq C
t^{-\frac{3}{2}(\frac{1}{q}-\frac{1}{r})-\frac{j}{2}}
\|(1+|x|)^{\alpha}(1+|x|-x_1)^{\beta} f\|_{q,D}+C
t^{-\frac{3}{2}(\frac{1}{q}-\frac{1}{r})-\frac{j}{2}+\alpha}
\|(1+|x|-x_1)^{\beta} f\|_{q,D}\notag\\
&\qquad+Ct^{-\frac{3}{2}(\frac{1}{q}-\frac{1}{r})-\frac{j}{2}
+\frac{\beta}{2}}
\|(1+|x|)^{\alpha}f\|_{q,D}
+Ct^{-\frac{3}{2}(\frac{1}{q}-\frac{1}{r})-\frac{j}{2}
+\alpha+\frac{\beta}{2}}
\|f\|_{q,D},
\label{lqlrdlarge3}
\end{align}
but it seems difficult to apply 
(\ref{gradlqlrdlarge5}) to the nonlinear problems. 
Therefore, we make use of (\ref{lqlrdlarge3}) 
in the nonlinear problems.

\par Let us proceed to the nonlinear problems. To state the main 
theorems precisely, we recall following properties 
of stationary solutions. 
\begin{prop}\label{propsta}
For every $(\alpha_1,\alpha_2,\beta_1,\beta_2)$ satisfying 
\begin{align}\label{parameter1}
2<\alpha_1\leq 4\leq \alpha_2<6,\quad 
\frac{4}{3}<\beta_1\leq 2\leq \beta_2<\frac{12}{5},
\end{align}
there exists a constant 
$\kappa_1=\kappa_1(D,\alpha_1,\alpha_2,\beta_1,\beta_2)\in(0,1)$ 
such that if 
\begin{align*} 
0<a^{\frac{1}{4}}<\kappa_1, 
\end{align*} 
problem $(\ref{sta})$ admits a unique solution $u_s$ along with 
\begin{align}
&\|u_s\|_{\alpha_1,D}+\|u_s\|_{\alpha_2,D}\leq Ca^{\frac{1}{2}},\quad 
\|\nabla u_s\|_{\beta_1,D}+\|\nabla u_s\|_{\beta_2,D}
\leq Ca^{\frac{3}{4}},\label{usest}
\end{align}
where $C>0$ is independent of $a$. Furthermore, there exists 
a constant $\kappa_2\in(0,\kappa_1]$ such that if 
\begin{align*} 
0<a^{\frac{1}{4}}<\kappa_2, 
\end{align*} 
then the solution $u_s$ enjoys 
\begin{align}
&u_s(x)=O((1+|x|)^{-1}(1+|x|-x_1)^{-1}),\quad 
\nabla u_s(x)=O((1+|x|)^{-\frac{3}{2}}(1+|x|-x_1)^{-\frac{3}{2}})
\label{usest2}
\end{align}
as $|x|\rightarrow\infty$.
\end{prop}

\begin{proof}
The estimate (\ref{usest}) was derived in 
the present author's paper \cite[Theorem 1.1]{takahashi2021} 
by employing the $L^q$-theory of 
the Oseen system developed by Galdi \cite{galdi1992oseen,galdi2011}. 
To deduce (\ref{usest2}), which is the same as the one of 
the Oseen fundamental solution, we prove that the solution 
$u_s$ constructed in \cite{takahashi2021} coincides with 
the solution constructed by Finn \cite{finn1965o}, 
which we denote by $u_{s,F}$. This is because 
the leading profile of $u_{s,F}$ is the Oseen fundamental solution 
as long as the net force does not vanish. 
Let $\alpha_1,\alpha_2,\beta_1,\beta_2$ 
satisfy (\ref{parameter1}). We denote the constant 
$C$ in (\ref{usest}) by $\mathcal{C}_1$. Then we have 
\begin{align}
&u_{s}\in N(\alpha_1,\alpha_2,
\beta_1,\beta_2,\mathcal {C}_1,a),\label{N}
\\
&N(\alpha_1,\alpha_2,
\beta_1,\beta_2,\mathcal {C}_1,a)
:=\{u\in L^{\alpha_1}(D)\cap L^{\alpha_2}(D)\mid
\nabla u\in L^{\beta_1}(D)\cap L^{\beta_2}(D),
\|u\|_{N(\alpha_1,\alpha_2,\beta_1,\beta_2,a)}\leq \mathcal{C}_1a\},
\notag\\
&\|u\|_{N(\alpha_1,\alpha_2,
\beta_1,\beta_2,a)}:=
a^\frac{1}{2}\|u\|_{\alpha_1,D}
+a^\frac{1}{2}\|u\|_{\alpha_2,D}
+a^\frac{1}{4}\|\nabla u\|_{\beta_1,D}
+a^\frac{1}{4}\|\nabla u\|_{\beta_2,D}
\notag
\end{align} 
if $a^{1/4}<\kappa_1$. 
It follows from Finn \cite[Theorem 5.1]{finn1965o} and 
Shibata \cite[Theorem 1.1]{shibata1999} 
with $\delta=1/8,\beta=4/3$ that 
\begin{align*}
(1+|x|)(1+|x|-x_1)^{\frac{1}{8}}|u_{s,F}(x)|,
(1+|x|)^{\frac{3}{2}}(1+|x|-x_1)^{\frac{5}{8}}
|\nabla u_{s,F}(x)|\leq a^{\frac{3}{4}}
\end{align*}
for $x\in D$ if $a$ is smaller than $\varepsilon$ and 
$(1/2)^{4/3}$, where 
$\varepsilon$ is given by \cite[Theorem 1.1]{shibata1999} 
with $\delta=1/8,\beta=4/3$. 
By taking 
$\widetilde{\alpha}_1,\widetilde{\beta}_1$ so that 
\begin{align*}
\alpha_1\leq 
\widetilde{\alpha}_1,\quad 
\frac{8}{3}<\widetilde{\alpha}_1\leq 4,\quad 
\beta_1\leq \widetilde{\beta}_1,\quad 
\frac{24}{17}<\widetilde{\beta}_1\leq 2
\end{align*}
and by $u_{s,F},\nabla u_{s,F}\in L^{\infty}(D)$, 
Lemma \ref{finite} in Section 5 yields 
$u_{s,F}\in N(\widetilde{\alpha}_1,\alpha_2,
\widetilde{\beta}_1,\beta_2,\mathcal {C}_2,a)$, 
where 
\begin{align*}
\mathcal{C}_2=
\max\big\{
&\|(1+|x|)^{-1}
(1+|x|-x_1)^{-\frac{1}{8}}\|_{\widetilde{\alpha}_1,\R^3},
\|(1+|x|)^{-1}(1+|x|-x_1)^{-\frac{1}{8}}
\|_{\alpha_2,\R^3},
\\
&\|(1+|x|)^{-\frac{3}{2}}
(1+|x|-x_1)^{-\frac{5}{8}}\|_{\widetilde{\beta}_1,\R^3},
\|(1+|x|)^{-\frac{3}{2}}
(1+|x|-x_1)^{-\frac{5}{8}}\|_{\beta_2,\R^3}
\big\}.
\end{align*}
By the proof of \cite[Theorem 1.1]{takahashi2021}, 
we know that there exists a constant 
$\mathcal{C}_3$ dependent on $\widetilde{\alpha}_1,\alpha_2,
\widetilde{\beta}_1,\beta_2$ such that  
if $C\geq \mathcal{C}_3$ and if $a^{1/4}<1/C^2$, 
then the problem $(\ref{sta})$ 
has at most one solution within 
$N(\widetilde{\alpha}_1,\alpha_2,\widetilde{\beta}_1,
\beta_2,C,a)$. Since we have 
$u_s\in N(\widetilde{\alpha}_1,\alpha_2,\widetilde{\beta}_1,
\beta_2,2\mathcal{C}_1,a)$ due to (\ref{N}),  
$\alpha_1\leq \widetilde{\alpha}_1$ and 
$\beta_1\leq \widetilde{\beta}_1$, 
the uniqueness within 
$N(\widetilde{\alpha}_1,\alpha_2,\widetilde{\beta}_1,
\beta_2,\max\{2\mathcal{C}_1,\mathcal{C}_2,\mathcal{C}_3\},a)$ 
leads us to $u_s=u_{s,F}$ whenever  
\begin{align*}
0<a^{\frac{1}{4}}<\min\left\{\kappa_1,\varepsilon,
\left(\frac{1}{2}\right)^{\frac{4}{3}},
\frac{1}{\max\{2\mathcal{C}_1,\mathcal{C}_2,\mathcal{C}_3\}^2}\right\}
=:\kappa_2
\end{align*}
is satisfied. The proof is complete. 
\end{proof}

To study the stability/attainability of the stationary solution 
$u_s$ given by Proposition \ref{propsta}, it is convenient to set 
\begin{align}\label{staclass2}
\alpha_1=\frac{3}{1+\mu_1},\quad\quad
\alpha_2=\frac{3}{1-\mu_2},\quad\quad 
\beta_1=\frac{3}{2+\mu_3},\quad\quad\beta_2=\frac{3}{2-\mu_4}
\end{align}
with $(\mu_1,\mu_2,\mu_3,\mu_4)$ satisfying
\begin{align}\label{mu1}
0<\mu_1<\frac{1}{2},\quad
\frac{1}{4}\leq \mu_2<\frac{1}{2},\quad 
0<\mu_3<\frac{1}{4},\quad
\frac{1}{2}\leq \mu_4<\frac{3}{4}.
\end{align}
Moreover, if we consider the starting problem (\ref{NS3}), 
then we assume the additional condition 
\begin{align}\label{mu2}
\mu_2+\mu_4>1.
\end{align}
By using the Oseen semigroup $e^{-tA_a}$, the problem 
(\ref{NS32}) and (\ref{NS3}) are converted into
\begin{align}\label{NS42}
v(t)=e^{-tA_a}b-\int_0^t e^{-(t-\tau)A_a}P_D\Big[v\cdot\nabla v
+v\cdot\nabla u_s+u_s\cdot\nabla v\Big]d\tau
\end{align}
and 
\begin{align}\label{NS4}
w(t)=\int_0^t e^{-(t-\tau)A_a}P_D\Big[-w\cdot\nabla w
&-\psi(\tau)w\cdot\nabla u_s-\psi(\tau)u_s\cdot\nabla w\nonumber\\
&+\big(1-\psi(\tau)\big)a\partial_{x_1}w
+f_1(\tau)+f_2(\tau)\Big]d\tau,
\end{align}
respectively. 
\par We are now in a position to give the main result 
on the stability of stationary solutions. 
\begin{thm}\label{stabilitythm}
Let $\alpha,\beta$ satisfy 
\begin{align}\label{alphabeta3}
\alpha\geq 0,\quad 0\leq\beta<\frac{1}{3},\quad \alpha+\beta<1. 
\end{align}
Then there exist constants $\kappa=\kappa(\alpha,\beta)>0$ 
and $\varepsilon=\varepsilon(\alpha,\beta,a)>0$ such that 
if $0<a^{1/4}<\kappa$ and if 
$b\in L^3_{(1+|x|)^{3\alpha}(1+|x|-x_1)^{3\beta}}(D)$ fulfills   
$\|b\|_{3,D}<\varepsilon,$ 
then the problem $($\ref{NS42}$)$ admits a solution $v$ which enjoys 
\begin{align}
&\|(1+|x|)^{\gamma_i}(1+|x|-x_1)^{\delta_i} v(t)\|_{q,D}
=o(t^{-\frac{1}{2}+\frac{3}{2q}+\gamma_i+\frac{\delta_i}{2}}), 
\label{anisov}\\
&\|(1+|x|)^{\gamma_i}(1+|x|-x_1)^{\delta_i} \nabla v(t)\|_{3,D}
=o(t^{-\frac{1}{2}+\gamma_i+\frac{\delta_i}{2}})\label{anisov2}
\end{align} 
as $t\rightarrow\infty$ for all $q\in[3,\infty]$ and 
$i=1,2,3,4,$ where $\gamma_i,\delta_i$ are 
given by $($\ref{gammadeltaeta}$)$.
\end{thm}

In the starting problem, 
since the fluid is initially at 
rest and since the stationary solution $u_s$ 
belongs to $L^q(D)$ with $q<3$, 
we expect that the decay rate in the starting problem 
is better than the one in Theorem \ref{stabilitythm}. 
In fact, we have the following. 
\begin{thm}\label{attainthm}
Let $\alpha,\beta$ satisfy 
\begin{align}\label{alphabeta5}
0\leq\alpha<\frac{1}{3},\quad 0\leq\beta<\frac{1}{3}.
\end{align}
Let $\psi$ be a function on $\R$ satisfying $(\ref{psidef})$ 
and set $M=\max_{t\in\R}|\psi'(t)|$. 
Then there exists constant 
$\widetilde{\kappa}=\widetilde{\kappa}(\alpha,\beta)$ 
$($independent of $\varepsilon$$)$ such that 
if $0<(M+1)a^{1/4}<\widetilde{\kappa}$, 
the problem $($\ref{NS4}$)$ admits a solution $w$ which enjoys 
\begin{align}
&\|(1+|x|)^{\gamma_i}(1+|x|-x_1)^{\delta_i}w(t)\|_{q,D}
=O(t^{-\frac{1}{2}+\frac{3}{2q}-\frac{1}{4}+\varepsilon
+\gamma_i+\frac{\delta_i}{2}}), 
\label{anisow}\\
&\|(1+|x|)^{\gamma_i}(1+|x|-x_1)^{\delta_i} \nabla w(t)\|_{3,D}
=O(t^{-\frac{1}{2}-\frac{1}{4}+\varepsilon
+\gamma_i+\frac{\delta_i}{2}})\label{anisow2}
\end{align} 
as $t\rightarrow\infty$ for all $q\in[3,\infty],i=1,2,3,4$ and 
$\varepsilon>0$, where $\gamma_i,\delta_i$ are 
given by $($\ref{gammadeltaeta}$)$. 
\end{thm}

\begin{rmk}
For the uniqueness of solutions to (\ref{NS4}), 
the present author \cite{takahashi2021} employed the idea due to 
Brezis \cite{brezis1994} and 
established the uniqueness within the class 
\begin{align}\label{unique} 
\{v\in BC([0,\infty);L^3_{\sigma}(D))\mid 
t^{\frac{1}{2}}v\in BC((0,\infty);L^{\infty}(D)),
t^{\frac{1}{2}}\nabla v\in BC((0,\infty);L^3(D))\}.
\end{align} 
Since the argument in \cite{takahashi2021} is also 
valid for the problem (\ref{NS42}), 
a solution to (\ref{NS42}) is also unique within (\ref{unique}).  
\end{rmk}

\section{Proof of Theorem \ref{muckenthm}}
\quad In this section, we first prepare some results on 
the Muckenhoupt class $\A_q(\R^3)$ 
and the estimate of anisotropic weight, 
and then prove Theorem \ref{muckenthm}.
We begin by the following lemma, for the proof, 
see \cite[Chapter IV, Corollary 5.3]{garu1985}, 
\cite[Chapter IX, Proposition 4.3 and Theorem 5.5]{torchinsky1986} 
and \cite[Chapter V, Proposition 9]{stein1993}.
\begin{prop}\label{a1aq}
Let $1<q<\infty$. 
A weight $\rho$ belongs to $\A_q(\R^3)$ if and only if 
there are weights 
$\rho_1,\rho_2\in\A_1(\R^3)$ such that 
\begin{align*}
\rho=\frac{\rho_1}{\rho_2^{q-1}}.
\end{align*}
\end{prop}

\noindent For the anisotropic weight, 
Farwig \cite{farwig1992t} proved the 
following.
\begin{lem}[{\cite[Lemma 4.6]{farwig1992t}}]\label{falem}
For all $\beta\in (-1,0],$ the weight $(1+|x|-x_1)^\beta$ belongs 
to $\A_1(\R^3)$.
\end{lem}

\noindent The following lemma directly follows from 
Proposition \ref{a1aq} and Lemma \ref{falem}.
\begin{lem}\label{betaaq}
Let $1<q<\infty$. 
For all $-1<\beta<q-1$, the weight $(1+|x|-x_1)^\beta$ belongs 
to $\A_q(\R^3)$.
\end{lem}
\begin{proof}
Proposition \ref{a1aq} and Lemma \ref{falem} imply that 
$(1+|x|-x_1)^{\beta_1-\beta_2(q-1)}\in \A_q(\R^3)$ for 
all $-1<\beta_1,\beta_2\leq 0$. Since $\beta_1-\beta_2(q-1)$ 
runs through 
all of the open interval $(-1,q-1)$, we have
$(1+|x|-x_1)^\beta\in \A_q(\R^3)$ for all $-1<\beta<q-1$.
\end{proof}

To prove Theorem \ref{muckenthm}, we also 
need some estimates of weight
$(1+|x|)^\gamma(1+|x|-x_1)^\delta$.
\begin{prop}
\begin{enumerate}
\item Let $\gamma\in\R$ and $\delta>-1$ satisfy 
$\gamma+\delta>-3.$ Then 
there exists a constant $C>0$ such that 
\begin{align}\label{gammadelta1}
\int_{B_r(x)}(1+|y|)^\gamma(1+|y|-y_1)^\delta\,dy
\leq Cr^{\gamma+\delta+3}
\end{align}
for all balls $B_r(x)=\{y\in \R^3\,;\,|y-x|<r\}$ 
subject to 
\begin{align}\label{gammadelta2}
r\geq \frac{|x|}{2},\qquad r+|x|\geq 1.
\end{align} 
\item Let $\gamma\in\R$ and $\delta>-1$ 
satisfy $\gamma+\delta=-3.$ 
Then there exist constants $C>0$ and $R>1$
such that 
\begin{align}\label{gammadelta3}
\int_{B_r(0)}(1+|y|)^\gamma(1+|y|-y_1)^\delta\,dy\geq C\log r
\end{align}
for all $r\geq R$.
\item Let $\gamma\in\R$ and $\delta>-1$.
Then there exist constants 
$C>0$ and $R>1$ such that 
\begin{align}\label{gammadelta4}
\int_{B_r(0)}(1+|y|)^\gamma(1+|y|-y_1)^\delta\,dy
\geq C\max\{r^{\gamma+\delta+3},1\} 
\end{align}
for all $r\geq R$.
\item Let $\gamma\in\R$. Then there exist constants 
$C>0$ and $R>1$
such that 
\begin{align}\label{gammadelta41}
\int_{B_r(0)}(1+|y|)^\gamma(1+|y|-y_1)^{-1}\,dy
\geq Cr^{\gamma+2}\log r
\end{align}
for all $r\geq R$.
\item Let $\gamma\in\R$ and $\delta<-1$. 
Then there exist constants $C>0$ and $R>1$ such that 
\begin{align}\label{gammadelta42}
\int_{B_r(0)}(1+|y|)^\gamma(1+|y|-y_1)^\delta\,dy
\geq Cr^{\gamma+2}
\end{align}
for all $r\geq R$.
\end{enumerate}
\end{prop}
\begin{proof}
Let $\gamma\in\R,\delta>-1$ and $x\in\R^3$ satisfy 
$\gamma+\delta>-3$ and (\ref{gammadelta2}). 
By using the polar coordinates 
$y_1=s\cos\theta,y_2=s\sin\theta\cos\varphi,
y_3=s\sin\theta\sin\varphi$, 
where $(s,\theta,\varphi)\in (0,r+|x|)\times 
(0,\pi)\times (0,2\pi)$ and by changing 
the variable $t=1-\cos\theta$, we have 
\begin{align}
\int_{B_r(x)}(1+|y|)^\gamma(1+|y|-y_1)^\delta\,dy&\leq
\int_{B_{r+|x|}(0)}(1+|y|)^\gamma(1+|y|-y_1)^\delta\,dy\notag\\
&= 2\pi\int_0^{r+|x|}\,ds\int_0^{\pi}
(1+s)^\gamma(1+s-s\cos\theta)^\delta s^2\sin\theta\,d\theta\notag\\
&=2\pi\int_0^{r+|x|}\,ds\,(1+s)^\gamma s^2\int_0^{2}
(1+st)^\delta\,dt\notag\\
&\leq \frac{2\pi}{\delta+1}\int_0^{r+|x|}(1+s)^\gamma s
(1+2s)^{\delta+1}\,ds=\frac{2\pi}{\delta+1}
\left(\int_0^1+\int_1^{r+|x|}\right).\label{polar}
\end{align} 
It follows from $\gamma+\delta+3>0$ that 
\begin{align*}
&\int_0^1\leq \left(\int_0^1(1+s)^\gamma s
(1+2s)^{\delta+1}\,ds\right)
(3r)^{\gamma+\delta+3},\\
&\int_1^{r+|x|}\leq 
C\int_1^{r+|x|}s^{\gamma+\delta+2}\,ds
\leq C(r+|x|)^{\gamma+\delta+3}\leq C(3r)^{\gamma+\delta+3},
\end{align*}
where we have used $3r\geq 1$ and 
$2r\geq |x|$ due to (\ref{gammadelta2}). 
We thus conclude the assertion 1.
\par Let $\gamma\in\R$ and $\delta>-1$ 
satisfy $\gamma+\delta=-3.$ 
By the same calculation as in (\ref{polar}), 
the left hand side of (\ref{gammadelta3}) is estimated as
\begin{align*}
\int_{B_r(0)}(1+|y|)^\gamma(1+|y|-y_1)^\delta\,dy&=
2\pi\int_0^{r}\,ds\,(1+s)^\gamma s^2\int_0^{2}
(1+st)^\delta\,dt\\
&=\frac{2\pi}{\delta+1}\int_0^{r}(1+s)^\gamma s
\{(1+2s)^{\delta+1}-1\}\,ds\\
&\geq C\int_1^r s^{-1}- s^{\gamma+1}\,ds
\geq  C\log r
\end{align*}
for $r\geq R$ with some $C>0$ and $R>1$, 
where we have used $\gamma=-3-\delta<-2$ in the last inequality. 
The proof of the assertion 2 is complete.
\par We next prove the assertion 3. 
Let $\gamma\in \R$ and $\delta>-1.$
If $\gamma+\delta=-3$, then by the assertion 2, 
there exist constants $C_1>0$ and $R_1>1$ such that 
\begin{align}\label{gammadelta5}
\int_{B_r(0)}(1+|y|)^\gamma(1+|y|-y_1)^\delta\,dy
\geq  C_1\log r\geq C_1\log R_1
=(C_1\log R_1)\max\{r^{\gamma+\delta+3},1\}
\end{align}
for all $r\geq R_1$. If $\gamma+\delta<-3$, then it holds that 
\begin{align}\label{gammadelta6}
\int_{B_r(0)}(1+|y|)^\gamma(1+|y|-y_1)^\delta\,dy&\geq 
\int_{B_1(0)}(1+|y|)^\gamma(1+|y|-y_1)^\delta\,dy\nonumber\\
&=\left(\int_{B_1(0)}(1+|y|)^\gamma(1+|y|-y_1)^\delta\,dy\right)
\max\{r^{\gamma+\delta+3},1\}
\end{align}
for all $r\geq 1$. Moreover, if $\gamma+\delta>-3,$ then we get
\begin{align}
\int_{B_r(0)}(1+|y|)^\gamma(1+|y|-y_1)^\delta\,dy
&=\frac{2\pi}{\delta+1}\int_0^{r}(1+s)^\gamma s
\{(1+2s)^{\delta+1}-1\}\,ds\nonumber\\
&\geq C\int_1^{r}s^{\gamma+\delta+2}-s^{\gamma+1}\,ds
\geq C_2r^{\gamma+\delta+3}=C_2\max\{r^{\gamma+\delta+3},1\}
\label{gammadelta7}
\end{align}
for all $r\geq R_2$ with some $C_2>0$ and $R_2>1$. Collecting 
(\ref{gammadelta5})--(\ref{gammadelta7}) 
leads to (\ref{gammadelta4}) with
\begin{align*} 
C=\min\left\{C_1\log R_1,\int_{B_1(0)}(1+|y|)^\gamma
(1+|y|-y_1)^\delta\,dy,C_2\right\},\quad R=\max\{R_1,R_2\},
\end{align*}
which implies the assertion 3.
\par If $\gamma<-2$, then 
by virtue of $r^{\gamma+2}\log r\rightarrow 0$ 
as $r\rightarrow\infty$, 
there exists a constant $R>1$ such that 
\begin{align*}
\int_{B_1(0)}(1+|y|)^\gamma(1+|y|-y_1)^{-1}\,dy
\geq r^{\gamma+2}\log r
\end{align*}
for all $r\geq R.$
we thus get 
\begin{align*}
\int_{B_r(0)}(1+|y|)^\gamma(1+|y|-y_1)^{-1}\,dy\geq 
\int_{B_1(0)}(1+|y|)^\gamma(1+|y|-y_1)^{-1}\,dy
\geq r^{\gamma+2}\log r
\end{align*}
for all $r\geq R$ with some $R>1$, which yields the assertion 4 
with $\gamma<-2$. Moreover, 
if $\gamma\geq -2$, then it follows that 
\begin{align*}
\int_{B_r(0)}(1+|y|)^\gamma(1+|y|-y_1)^{-1}\,dy
&=2\pi\int_0^{r}\,ds\,(1+s)^\gamma s^2
\int_0^2(1+st)^{-1}\,dt\\
&=2\pi \int_0^r (1+s)^\gamma s\log(1+2s)\,ds\\
&\geq C\int_1^rs^{\gamma+1}\log s\,ds\geq Cr^{\gamma+2}\log r 
\end{align*}
for all $r\geq R$ with some $C>0$ and $R>1$, where we have used 
\begin{align*}
\int_1^rs^{\gamma+1}\log s\,ds
&= 
\begin{cases}
\displaystyle\frac{1}{2}(\log r)^2\quad &{\rm if}~\gamma=-2,\\
r^{\gamma+2}\log r\left(\displaystyle\frac{1}{\gamma+2}-
\frac{1}{(\gamma+2)^2\log r}+
\frac{1}{(\gamma+2)^2r^{\gamma+2}\log r}\right)&{\rm if}~\gamma>-2
\end{cases}\\
&\geq Cr^{\gamma+2}\log r
\end{align*}
for all $r\geq R$ with some $C>0$ and $R>1$
in the last inequality. We thus conclude the assertion 4. 
\par Let $\gamma\in\R$ and $\delta<-1.$
The assertion 5 with $\gamma<-2$ is proved 
by the similar argument 
of the assertion 4 with $\gamma<-2$ 
and the assertion 5 with $\gamma\geq -2$ is proved 
by the estimate
\begin{align*}
\int_{B_r(0)}(1+|y|)^\gamma(1+|y|-y_1)^\delta\,dy
&=2\pi\int_0^{r}\,ds\,(1+s)^\gamma s^2
\int_0^2(1+st)^{\delta}\,dt\\
&=\frac{2\pi}{-(\delta+1)}\int_0^r(1+s)^\gamma s
\big\{1-(1+2s)^{\delta+1}\big\}\,ds\\
&\geq \frac{2\pi}{-(\delta+1)}(1-3^{\delta+1})
\int_1^r(1+s)^\gamma s\,ds
\geq Cr^{\gamma+2}
\end{align*}
for all $r\geq R$ with some $C>0$ and $R>1$.
The proof is complete.
\end{proof}

\noindent{\bf Proof of Theorem \ref{muckenthm}.}
~We first prove that if (\ref{alphabeta}) is satisfied, 
then there exists $C>0$ such that
\begin{align}\label{aqweightrho}
\left(\frac{1}{|B|}\int_B \rho_{\alpha,\beta}(y)\,dy\right)
\left(\frac{1}{|B|}\int_B \rho_{\alpha,\beta}
(y)^{-\frac{1}{q-1}}\,dy\right)^{q-1}\leq C
\end{align}
holds for all 
balls $B\subset\R^3$. This asserts that 
(\ref{alphabeta}) is sufficient condition on $\alpha,\beta$ so that 
$\rho_{\alpha,\beta}\in\A_q(\R^3)$. Let $r\leq |x|/2$. Due to 
$|x|/2\leq |y|\leq (3|x|)/2$ for $y\in B_r(x)$, 
it follows from Lemma \ref{betaaq} that 
\begin{align}
&\left(\frac{1}{|B_r(x)|}\int_{B_r(x)}
\rho_{\alpha,\beta}(y)\,dy\right)
\left(\frac{1}{|B_r(x)|}\int_{B_r(x)} \rho_{\alpha,\beta}
(y)^{-\frac{1}{q-1}}\,dy\right)^{q-1}\nonumber\\&\leq C
\left(\frac{1}{|B_r(x)|}\int_{B_r(x)} (1+|y|-y_1)^\beta\,dy\right)
\left(\frac{1}{|B_r(x)|}\int_{B_r(x)}
(1+|y|-y_1)^{-\frac{\beta}{q-1}}\,dy
\right)^{q-1}\leq C\label{rx1}
\end{align}
for all $B_r(x)$ with $r\leq |x|/2$. If $r\geq |x|/2$ and 
$r+|x|\leq 1$, we find 
\begin{align*}
1\leq 1+|y|-y_1\leq 1+2|y|\leq 1+2(r+|x|)\leq 3
\end{align*} 
for all $y\in B_r(x)$, thus 
\begin{align}
&\left(\frac{1}{|B_r(x)|}\int_{B_r(x)}
\rho_{\alpha,\beta}(y)\,dy\right)
\left(\frac{1}{|B_r(x)|}\int_{B_r(x)} \rho_{\alpha,\beta}
(y)^{-\frac{1}{q-1}}\,dy\right)^{q-1}\nonumber\\&\leq C
\left(\frac{1}{|B_r(x)|}\int_{B_r(x)}\,dy\right)
\left(\frac{1}{|B_r(x)|}\int_{B_r(x)}\,dy\label{rx2}
\right)^{q-1}\leq C
\end{align}
for all $B_r(x)$ subject to $r\geq |x|/2$ and $r+|x|\leq 1$. 
If $r\geq |x|/2$ and $r+|x|\geq 1$, 
we can use (\ref{gammadelta1}) with 
$(\gamma,\delta)=(\alpha,\beta),
(-\alpha/(q-1),-\beta/(q-1))$ to see 
\begin{align}
&\left(\frac{1}{|B_r(x)|}\int_{B_r(x)} 
\rho_{\alpha,\beta}(y)\,dy\right)
\left(\frac{1}{|B_r(x)|}\int_{B_r(x)} \rho_{\alpha,\beta}
(y)^{-\frac{1}{q-1}}\,dy\right)^{q-1}\nonumber\\&\leq C
\left(\frac{1}{|B_r(x)|}r^{\alpha+\beta+3}\right)
\left(\frac{1}{|B_r(x)|}r^{-\frac{\alpha}{q-1}-\frac{\beta}{q-1}+3}
\right)^{q-1}\leq C\label{rx3}
\end{align}
for all $B_r(x)$ subject to $r\geq |x|/2$ and $r+|x|\geq 1$. 
Collecting (\ref{rx1})--(\ref{rx3}) leads to (\ref{aqweightrho})
for all balls whenever (\ref{alphabeta}) is satisfied. 
\par We next prove that 
if (\ref{alphabeta}) is not satisfied, then 
\begin{align}\label{infty}
\left(\frac{1}{|B_r(0)|}\int_{B_r(0)} 
\rho_{\alpha,\beta}(y)\,dy\right)
\left(\frac{1}{|B_r(0)|}\int_{B_r(0)} \rho_{\alpha,\beta}
(y)^{-\frac{1}{q-1}}\,dy\right)^{q-1}\rightarrow\infty
\qquad {\rm as}~r\rightarrow\infty,
\end{align}
which implies that there does not exist constant $C$ such that 
$(\ref{aqweightrho})$ holds for all balls.
This combined with the first half completes the proof. Let 
$\beta\leq -1$. 
Since $-\beta/(q-1)>-1$, we use (\ref{gammadelta4}) with 
$(\gamma,\delta)=(-\alpha/(q-1),-\beta/(q-1))$ to see 
\begin{align*}
\left(\frac{1}{|B_r(0)|}\int_{B_r(0)} \rho_{\alpha,\beta}
(y)^{-\frac{1}{q-1}}\,dy\right)^{q-1}
\geq C\left(\frac{1}{r^3}
\max\{r^{-\frac{\alpha}{q-1}-\frac{\beta}{q-1}+3},1\}\right)^{q-1}
\geq Cr^{-\alpha-\beta}
\end{align*}
for all $r\geq R$. From this and 
(\ref{gammadelta41})--(\ref{gammadelta42}) 
with $(\gamma,\delta)=(\alpha,\beta)$, we have 
\begin{align*}
&\left(\frac{1}{|B_r(0)|}\int_{B_r(0)} 
\rho_{\alpha,\beta}(y)\,dy\right)
\left(\frac{1}{|B_r(0)|}\int_{B_r(0)} \rho_{\alpha,\beta}
(y)^{-\frac{1}{q-1}}\,dy\right)^{q-1}\\
&\geq 
\begin{cases}
Cr^{\alpha-1}(\log r)\times r^{-\alpha-\beta}=C\log r
&\quad{\rm if}~\beta=-1,\\[2pt]
Cr^{\alpha-1}\times r^{-\alpha-\beta}=Cr^{-\beta-1}
&\quad{\rm if}~\beta<-1\end{cases}
\end{align*} 
for $r\geq R$, which implies (\ref{infty}) for $\beta\leq -1$. 
If $\beta\geq q-1$, then from (\ref{gammadelta4}) with 
$(\gamma,\delta)=(\alpha,\beta)$ 
and (\ref{gammadelta41})--(\ref{gammadelta42}) with 
$(\gamma,\delta)=(-\alpha/(q-1),-\beta/(q-1))$, it holds that 
\begin{align*}
&\left(\frac{1}{|B_r(0)|}\int_{B_r(0)} 
\rho_{\alpha,\beta}(y)\,dy\right)
\left(\frac{1}{|B_r(0)|}\int_{B_r(0)} \rho_{\alpha,\beta}
(y)^{-\frac{1}{q-1}}\,dy\right)^{q-1}\\
&\geq 
\begin{cases}
Cr^{\alpha+\beta}\times r^{-\alpha-(q-1)}(\log r)^{q-1}
=C(\log r)^{q-1}
&\quad{\rm if}~\beta=q-1,\\[2pt]
Cr^{\alpha+\beta}\times r^{-\alpha-(q-1)}=Cr^{\beta-(q-1)}
&\quad{\rm if}~\beta>q-1\end{cases}
\end{align*} 
for $r\geq R$, thus we  conclude (\ref{infty}) 
for $\beta\geq q-1$. 
Let $\alpha,\beta$ satisfy $\alpha+\beta\leq -3$ 
and $-1<\beta<q-1$. 
If $\alpha+\beta=-3$
(resp. $\alpha+\beta<-3$), then we use 
(\ref{gammadelta3}) (resp. (\ref{gammadelta4})) to have 
\begin{align*}
&\frac{1}{|B_r(0)|}\int_{B_r(0)} \rho_{\alpha,\beta}(y)\,dy
\geq Cr^{-3}\log r\quad
\left({\rm resp.}~
\frac{1}{|B_r(0)|}\int_{B_r(0)} \rho_{\alpha,\beta}(y)\,dy
\geq Cr^{-3}\right)
\end{align*}
for all $r\geq R$. These and (\ref{gammadelta4}) 
with $(\gamma,\delta)=(-\alpha/(q-1),-\beta/(q-1))$ yield 
\begin{align*}
&\left(\frac{1}{|B_r(0)|}\int_{B_r(0)} 
\rho_{\alpha,\beta}(y)\,dy\right)
\left(\frac{1}{|B_r(0)|}\int_{B_r(0)} \rho_{\alpha,\beta}
(y)^{-\frac{1}{q-1}}\,dy\right)^{q-1}\\
&\geq 
\begin{cases}
Cr^{-3}(\log r)\times r^{-\alpha-\beta}
=C\log r
&\quad{\rm if}~\alpha+\beta=-3,\\[2pt]
Cr^{-3}\times r^{-\alpha-\beta}=Cr^{-\alpha-\beta-3}
&\quad{\rm if}~\alpha+\beta<-3
\end{cases}
\end{align*}  
for all $r\geq R$, which leads us to (\ref{infty}) 
for $\alpha+\beta\leq -3$ and $-1<\beta<q-1$.
\par We finally assume $\alpha+\beta\geq 3(q-1)$ 
and $-1<\beta<q-1$. 
Then we see from (\ref{gammadelta3}) with 
$(\gamma,\delta)=(-\alpha/(q-1),-\beta/(q-1))$ 
and (\ref{gammadelta4}) with 
$(\gamma,\delta)=(\alpha,\beta),(-\alpha/(q-1),-\beta/(q-1))$
that 
\begin{align*}
&\left(\frac{1}{|B_r(0)|}\int_{B_r(0)} 
\rho_{\alpha,\beta}(y)\,dy\right)
\left(\frac{1}{|B_r(0)|}\int_{B_r(0)} \rho_{\alpha,\beta}
(y)^{-\frac{1}{q-1}}\,dy\right)^{q-1}\\
&\geq 
\begin{cases}
Cr^{\alpha+\beta}\times (r^{-3}\log r)^{q-1}
=C(\log r)^{q-1}
&\quad{\rm if}~\alpha+\beta=3(q-1),\\[2pt]
Cr^{\alpha+\beta}\times (r^{-3})^{q-1}=Cr^{\alpha+\beta-3(q-1)}
&\quad{\rm if}~\alpha+\beta>3(q-1)\end{cases}
\end{align*}
for all $r\geq R$. 
Consequently, if (\ref{alphabeta}) is not satisfied, 
then (\ref{infty}) holds. The proof is complete.\qed

\section{Proof of Theorem \ref{lqlrd0}}

\quad This section is devoted to the proof of Theorem \ref{lqlrd0}. 
We start with the following interpolation inequality 
derived by Farwig and Sohr \cite[Theorem 3.5]{faso1997}. 
\begin{lem}[\cite{faso1997}]\label{interpo}
Suppose $1<q<\infty$ and $\rho\in \A_q(D)$. 
Given $\varepsilon>0$, there exists a 
constant $C=C(D,q,\rho,\varepsilon)$ such that 
\begin{align*}
\|\rho^{\frac{1}{q}}\nabla u\|_{q,D}\leq 
\varepsilon\|\rho^{\frac{1}{q}}\nabla^2 u\|_{q,D}
+C\|\rho^{\frac{1}{q}} u\|_{q,D}
\end{align*}
for all $u\in W^{2,q}_{\rho}(D)$. 
\end{lem}

\noindent The following lemma follows from the above lemma. 
\begin{lem}\label{oseeninterpo}
Given $a_0>0$ arbitrarily, we assume $|a|\leq a_0$. Let 
$1<q<\infty$ and $\rho\in \A_q(D)$.
For each $\varepsilon>0$, there exists a 
constant $C=C(D,a_0,q,\rho,\varepsilon)$, independent of $a$, 
such that 
\begin{align}\label{perturb}
\|\rho^{\frac{1}{q}}aP_D\partial_{x_1}u\|_{q,D}\leq 
\varepsilon\|\rho^{\frac{1}{q}}A u\|_{q,D}
+C\|\rho^{\frac{1}{q}} u\|_{q,D}
\end{align}
for $u\in\D(A)$, where $A=A_0=-P_D\Delta.$
\end{lem}
\begin{proof}
Let $\delta>0$. By applying Lemma \ref{interpo}, we have  
\begin{align*}
\|\rho^{\frac{1}{q}}aP_D\partial_{x_1}u\|_{q,D}\leq 
\delta a_0\|\rho^{\frac{1}{q}}\nabla^2 u\|_{q,D}
+Ca_0\|\rho^{\frac{1}{q}} u\|_{q,D}
\end{align*}
for $u\in\D(A)$. This combined with 
\begin{align*}
\|\rho^{\frac{1}{q}} \nabla^2 u\|_{q,D} 
\leq C\|\rho^{\frac{1}{q}} (1+A)u\|_{q,D}
\leq C\|\rho^{\frac{1}{q}} u\|_{q,D}
+C\|\rho^{\frac{1}{q}} Au\|_{q,D}
\end{align*} 
due to Farwig and Sohr \cite[Theorem 5.5 (i)]{faso1997} 
asserts (\ref{perturb}) for all $\varepsilon>0$. 
The proof is complete.
\end{proof}

We also mention the resolvent estimate for resolvent parameter 
$|\lambda|>>1$. 
In view of Lemma \ref{oseeninterpo} and the resolvent estimate 
for the Stokes operator derived by Farwig and Sohr 
\cite[Theorem 5.5 (i)]{faso1997}, the proof of the following lemma 
is essentially same as the one of 
Kobayashi and Shibata \cite[Lemma 4.5]{kosh1998}, thus we omit it.
\begin{lem}\label{lambdalarge}
Given $a_0>0$ arbitrarily, we assume $|a|\leq a_0$. Let  
$1<q<\infty$ and $\rho\in \A_q(D)$.
For $\varepsilon\in(0,\pi/2)$, there exist 
constants $C=C(D,a_0,q,\rho,\varepsilon)$ and 
$R=R(D,a_0,q,\rho,\varepsilon)$ such that
\begin{align*}
|\lambda|\|\rho^{\frac{1}{q}}u\|_{q,D}+
\|\rho^{\frac{1}{q}}\nabla^2u\|_{q,D}\leq 
C\|\rho^{\frac{1}{q}}(\lambda+A_a)u\|_{q,D}
\end{align*}
for all $u\in\D(A_a)$ and $|\lambda|\geq R$ 
with $|{\rm arg} \lambda|<\pi-\varepsilon$.
\end{lem}

\noindent{\bf Proof of Theorem \ref{lqlrd0}.} 
\quad The analyticity of $e^{-tA_a}$ directly follows from Lemma 
\ref{lambdalarge}. Furthermore, due to 
Farwig and Sohr \cite[Theorem 1.5 (ii)]{faso1997}, 
we can prove that 
$e^{-tA}$ is bounded analytic $C_0$-semigroup 
in $L^{q}_{\rho,\sigma}(D)$. In fact, let $-1/q<\beta<1-1/q$. 
For $\alpha_0\in(-3/q-\beta+2,-3/q-\beta+3)$ and 
$\alpha_1\in(-3/q-\beta,-3/q-\beta+1)$, we set 
$\rho_0(x):=(1+|x|)^{\alpha_0 q}(1+|x|-x_1)^{\beta q}$ and
$\rho_1(x):=(1+|x|)^{\alpha_1 q}(1+|x|-x_1)^{\beta q}$. 
Then it holds that 
\begin{align}\label{rho0rho1}
\rho_0,\rho_1,\rho_0(1+|x|)^{-2q},\rho_1(1+|x|)^{2q}\in\A_q(D).
\end{align} 
We note that $\alpha_0$ and $\alpha_1$ can be written as
$\alpha_0=-3/q-\beta+(2+\eta_0)$ and 
$\alpha_1=-3/q-\beta+\eta_1$ with some $\eta_0,\eta_1\in(0,1)$. 
Given $\alpha=-3/q-\beta+3\eta$ with $\eta\in(0,1)$, we take 
$\eta_0$ and $\eta_1$ so that $\eta_1<3\eta<2+\eta_0$, then 
there exists $\theta\in(0,1)$ 
such that $3\eta=(1-\theta)(2+\eta_0)+\theta\eta_1$, thus 
\begin{align}\label{rhotheta}
\rho=(1+|x|)^{\alpha q}
(1+|x|-x_1)^{\beta q}=\rho_0^{1-\theta}\rho_1^\theta.
\end{align} 
From (\ref{rho0rho1}) and (\ref{rhotheta}) together with 
Farwig and Sohr \cite[Theorem 1.5 (ii)]{faso1997}, 
$\{e^{-tA}\}_{t\geq 0}$ is a bounded analytic 
$C_0$-semigroup in $L^{q}_{\rho,\sigma}(D)$, 
which yields the assertion 1. 
\par We next prove the assertion 2. 
By Lemma \ref{lambdalarge}, we know
\begin{align}
&\|(1+|x|)^\alpha(1+|x|-x_1)^\beta e^{-tA_a}P_Df\|_{q,D}
\leq C\|(1+|x|)^\alpha(1+|x|-x_1)^\beta
f\|_{q,D},\label{lqlqsmall}\\
&\|(1+|x|)^\alpha(1+|x|-x_1)^\beta \nabla^2 e^{-tA_a}P_Df\|_{q,D}
\leq Ct^{-1}
\|(1+|x|)^\alpha(1+|x|-x_1)^\beta f\|_{q,D}\label{lqlqsmall2}
\end{align}
for all $f\in L^q_{\rho}(D)$ and $0<t\leq 3$. 
These with Lemma \ref{interpo} yields 
\begin{align*}
\|(1+|x|)^\alpha(1+|x|-x_1)^\beta \nabla e^{-tA_a}P_Df\|_{q,D}
\leq Ct^{-1}
\|(1+|x|)^\alpha(1+|x|-x_1)^\beta f\|_{q,D}
\end{align*} 
for all $f\in L^q_{\rho}(D)$ and $0<t\leq 3$. 
Furthermore, in view of $0\in \R^3\setminus \overline{D}$ and 
\begin{align*}
&\partial_{x_i}\big\{(1+|x|)^\alpha(1+|x|-x_1)^\beta\big\}\\
&\qquad=(1+|x|)^\alpha(1+|x|-x_1)^\beta\left(\frac{1}{1+|x|}\cdot
\frac{\alpha x_i}{|x|}+\frac{1}{1+|x|-x_1}\cdot
\frac{\beta(x_i-|x|\delta_{1i})}{|x|}\right),\\
&\partial^2_{x_ix_j}\big\{
(1+|x|)^\alpha(1+|x|-x_1)^\beta\big\}\\
&\qquad=(1+|x|)^\alpha(1+|x|-x_1)^\beta\left[
\frac{1}{1+|x|}\left(\frac{\alpha \delta_{ij}}{|x|}+
\frac{\alpha x_ix_j}{|x|^3}\right)+
\frac{1}{(1+|x|)^2}\cdot\frac{(\alpha^2+\alpha)x_ix_j}{|x|^2}\right.\\
&\qquad\qquad\quad+\frac{1}{(1+|x|)(1+|x|-x_1)}\left\{
\frac{\beta}{|x|}\left(\delta_{ij}-\frac{x_j}{|x|}\delta_{1i}\right)
+\frac{\beta x_j}{|x|^3}(x_i-|x|\delta_{1i})\right\}\\
&\left.\quad\qquad\qquad+\frac{1}{(1+|x|-x_1)^2}\left\{
\frac{\beta^2}{|x|^2}\left(\frac{x_j}{|x|}-\delta_{1i}\right)
(x_i-|x|\delta_{1i})
+\frac{\beta}{|x|^2}(x_i-|x|\delta_{1i})(x_j-|x|\delta_{1j})
\right\}\right],
\end{align*}
where $\delta_{ij}$ denotes the Kronecker delta, we find
\begin{align}\label{bound}
\big|\partial_{x_i}\big\{(1+|x|)^\alpha(1+|x|-x_1)^\beta\big\}\big|,
\big|\partial^2_{x_ix_j}\big\{
(1+|x|)^\alpha(1+|x|-x_1)^\beta\big\}\big|
\leq C(1+|x|)^\alpha(1+|x|-x_1)^\beta
\end{align}
for all $x\in D$ and $i,j=1,2,3$. We thus get 
\begin{align}
&\|(1+|x|)^\alpha(1+|x|-x_1)^\beta e^{-tA_a}P_Df\|_{W^{2,q}(D)}
\nonumber\\&\leq C\sum_{m=0}^2
\|(1+|x|)^\alpha(1+|x|-x_1)^\beta \nabla^m e^{-tA_a}P_Df\|_{q,D}
\leq Ct^{-1}\|(1+|x|)^\alpha(1+|x|-x_1)^\beta f\|_{q,D}
\label{nabla2lqlqsmall}
\end{align}
for all $f\in L^q_{\rho}(D)$ and $0<t\leq 3$, thereby, 
\begin{align}
&\|\nabla\{(1+|x|)^\alpha(1+|x|-x_1)^\beta e^{-tA_a}P_Df\}\|_{q,D}
\nonumber\\
&\leq \|(1+|x|)^\alpha(1+|x|-x_1)^\beta 
e^{-tA_a}P_Df\|^{\frac{1}{2}}_{W^{2,q}(D)}
\|(1+|x|)^\alpha(1+|x|-x_1)^\beta
e^{-tA_a}P_Df\|^{\frac{1}{2}}_{q,D}\nonumber\\
&\leq Ct^{-\frac{1}{2}}
\|(1+|x|)^\alpha(1+|x|-x_1)^\beta f\|_{q,D}\label{nablalqlqsmall0}
\end{align}
for all $f\in L^q_{\rho}(D)$ and $0<t\leq 3$. 
From (\ref{lqlqsmall}), (\ref{bound}), (\ref{nablalqlqsmall0}) 
together with 
\begin{align}
(1+|x|)^\alpha(1+|x|-x_1)^\beta \nabla e^{-tA_a}P_Df=
\nabla&\{(1+|x|)^\alpha(1+|x|-x_1)^\beta e^{-tA_a}P_Df\}\nonumber\\
&-\nabla\{(1+|x|)^\alpha(1+|x|-x_1)^\beta\} 
e^{-tA_a}P_Df,\label{leibniz}
\end{align}
we obtain 
\begin{align}\label{nablalqlqsmall}
\|(1+|x|)^\alpha(1+|x|-x_1)^\beta \nabla e^{-tA_a}P_Df\|_{q,D}
\leq Ct^{-\frac{1}{2}}
\|(1+|x|)^\alpha(1+|x|-x_1)^\beta f\|_{q,D}
\end{align}
for all $f\in L^q_{\rho}(D)$ and $0<t\leq 3$.
\par Let $1<q\leq r\leq\infty~(q\ne\infty)$ 
satisfy $1/q-1/r<1/3$. 
We use the Gagliard-Nirenberg inequality to see 
\begin{align}
&\|(1+|x|)^\alpha(1+|x|-x_1)^\beta e^{-tA_a}P_Df\|_{r,D}\nonumber
\\&\leq C\|\nabla\{(1+|x|)^\alpha(1+|x|-x_1)^\beta 
e^{-tA_a}P_Df\}\|^{3(\frac{1}{q}-\frac{1}{r})}_{q,D}
\|(1+|x|)^\alpha(1+|x|-x_1)^\beta
e^{-tA_a}P_Df\|^{1-3(\frac{1}{q}-\frac{1}{r})}
_{q,D}\nonumber\\
&\leq Ct^{-\frac{3}{2}(\frac{1}{q}-\frac{1}{r})}
\|(1+|x|)^\alpha(1+|x|-x_1)^\beta f\|_{q,D}\notag
\end{align}
for all $f\in L^q_{\rho}(D)$ and $0<t\leq 3$, where 
we have taken (\ref{lqlqsmall}), 
(\ref{nablalqlqsmall0}) into account. 
Given $1<q\leq r\leq\infty~(q\ne\infty)$, by
taking $q_0,\cdots,q_{j}$ so that 
$q=q_0\leq q_1\leq \cdots\leq q_j=r$ and 
$1/q_{i-1}-1/q_i<1/3~(i=1,\cdots, j)$, 
and by using the semigroup property, 
the condition $1/q-1/r<1/3$ is eliminated, 
which implies (\ref{lqlrdsmall}) with $|k|=0$. 
If $r<\infty$, then due to (\ref{nablalqlqsmall}) and 
(\ref{lqlrdsmall}) with $|k|=0$, we get 
\begin{align}
&\|(1+|x|)^\alpha(1+|x|-x_1)^\beta 
\nabla e^{-tA_a}P_Df\|_{r,D}\nonumber\\
&\leq Ct^{-\frac{1}{2}}
\|(1+|x|)^\alpha(1+|x|-x_1)^\beta e^{-\frac{t}{2}A_a}P_Df\|_{r,D}
\leq Ct^{-\frac{3}{2}(\frac{1}{q}-\frac{1}{r})-\frac{1}{2}}
\|(1+|x|)^\alpha(1+|x|-x_1)^\beta f\|_{q,D}
\notag
\end{align}
for all $f\in L^q_{\rho}(D)$ and $0<t\leq 3$. 
If $3<q<\infty$, then by (\ref{lqlrdsmall}) with $|k|=0$, 
(\ref{lqlqsmall})--(\ref{nabla2lqlqsmall}), (\ref{leibniz})
and the Gagliard-Nirenberg inequality, it holds that 
\begin{align*}
&\|(1+|x|)^\alpha(1+|x|-x_1)^\beta 
\nabla e^{-tA_a}P_Df\|_{\infty,D}\\
&\leq 
\|\nabla\{(1+|x|)^\alpha(1+|x|-x_1)^\beta 
e^{-tA_a}P_Df\}\|_{\infty,D}
+C\|(1+|x|)^\alpha(1+|x|-x_1)^\beta e^{-tA_a}P_Df\|_{\infty,D}\\
&\leq C\|(1+|x|)^\alpha(1+|x|-x_1)^\beta 
e^{-tA_a}P_Df\|^{\frac{3}{2q}+\frac{1}{2}}_{W^{2,q}(D)}
\|(1+|x|)^\alpha(1+|x|-x_1)^\beta 
e^{-tA_a}P_Df\|^{-\frac{3}{2q}+\frac{1}{2}}_{q,D}\\
&\qquad+Ct^{-\frac{3}{2q}}
\|(1+|x|)^\alpha(1+|x|-x_1)^\beta f\|_{q,D}\\
&\leq C t^{-\frac{3}{2q}-\frac{1}{2}}
\|(1+|x|)^\alpha(1+|x|-x_1)^\beta f\|_{q,D}
\end{align*} 
for all $f\in L^q_{\rho}(D)$ and $0<t\leq 3$. If $1<q\leq 3$, 
then in view of this and (\ref{lqlrdsmall}) with $|k|=0$, we get 
\begin{align*}
\|(1+|x|)^\alpha(1+|x|-x_1)^\beta \nabla e^{-tA_a}P_Df\|_{\infty,D}
&\leq Ct^{-\frac{3}{8}-\frac{1}{2}}
\|(1+|x|)^\alpha(1+|x|-x_1)^\beta e^{-\frac{t}{2}A_a}P_Df\|_{4,D}\\
&\leq C t^{-\frac{3}{2q}-\frac{1}{2}}
\|(1+|x|)^\alpha(1+|x|-x_1)^\beta f\|_{q,D}
\end{align*} 
for all $f\in L^q_{\rho}(D)$ and $0<t\leq 3$. 
Therefore, 
we conclude (\ref{lqlrdsmall}) with $|k|=1$. 
The proof is complete.   
\qed

\section{Anisotropically weighted
 $L^q$-$L^r$ estimates of the Oseen semigroup in $\R^3$}
\label{section5}
\quad In this section, we derive
the anisotropically weighted $L^q$-$L^r$ estimates of 
solutions to the Oseen equation in $\R^3$\,:
\begin{align}\label{oseenproblemr3}
&\partial_t u-\Delta u+a\partial_{x_1} u+\nabla p=0,\quad 
\nabla\cdot u=0,
\quad x\in\R^3,t>0,\qquad 
u(x,0)=g,\quad x\in\R^3. 
\end{align}
Given $g\in L^q(\R^3),$ 
we denote a solution to the heat equation by $e^{t\Delta}g,$ which
is given by the formula: 
\begin{align*}
(e^{t\Delta}g)(x)=
\left(\frac{1}{4\pi t}\right)^{\frac{3}{2}}\int_{\R^3}
e^{-\frac{|x-y|^2}{4t}}g(y)\,dy.
\end{align*}
If in particular, $g\in L^q_\sigma(\R^3)$, 
then we see that $\nabla p=0$ and that  
\begin{align}\label{oseensemir3}
u(x,t)=(S_a(t)g)(x):=(e^{t\Delta}g)(x-ate_1)
\end{align}
solves the problem (\ref{oseenproblemr3}). 
For later use, it is convenient to prepare the integrability of 
anisotropic weights. 
\begin{lem}\label{finite}
Let $\gamma,\delta>0$. If either
\begin{align}\label{gammadeltas}
2\delta<\gamma, \quad 
3/(\gamma+\delta)<s<1/\delta
\end{align}
or 
\begin{align*}
s>\max\{1/\delta,2/\gamma\},
\end{align*}
then
\begin{align*}
\int_{\R^3}(1+|x|)^{-\gamma s}(1+|x|-x_1)^{-\delta s}\,dx<\infty.
\end{align*}
\end{lem}
\begin{proof}
The polar coordinates imply 
\begin{align*}
\int_{\R^3}(1+|x|)^{-\gamma s}(1+|x|-x_1)^{-\delta s}\,dx&=
2\pi\int_0^\infty\,dr\int_0^\pi
(1+r)^{-\gamma s}(1+r-r\cos\theta)^{-\delta s}
r^2\sin\theta\,d\theta\\
&=2\pi \int_0^\infty (1+r)^{-\gamma s}r^2\,dr
\int_0^2(1+rt)^{-\delta s}\,dt\\
&=\frac{2\pi}{-\delta s+1}\int_0^\infty(1+r)^{-\gamma s}r
\{(1+2r)^{-\delta s+1}-1\}\,dr,
\end{align*}
thus we have 
\begin{align*}
\int_{\R^3}(1+|x|)^{-\gamma s}(1+|x|-x_1)^{-\delta s}\,dx
\leq \frac{2\pi}{-\delta s+1}\int_0^\infty
(1+r)^{-\gamma s}r(1+2r)^{-\delta s+1}\,dr<\infty
\end{align*}
if (\ref{gammadeltas}) and 
\begin{align*}
\int_{\R^3}(1+|x|)^{-\gamma s}(1+|x|-x_1)^{-\delta s}\,dx
\leq \frac{2\pi}{\delta s-1}\int_0^\infty(1+r)^{-\gamma s}r\,dr
<\infty
\end{align*}
if $s>\max\{1/\delta,2/\gamma\}$. The proof is complete. 
\end{proof}

\noindent Lemma \ref{finite} and (Lorentz-)H\"{o}lder inequality 
yield the following.

\begin{lem}\label{holder}
Let $\Omega=\R^3$ or $D$. 
\begin{enumerate}
\item Let $1<q<\infty$ and let $\alpha,\beta>0$ 
satisfy $\alpha+\beta<3(1-1/q)$. 
Suppose $f\in L^q_{\rho}(\Omega)$, 
where $\rho$ is given by $($\ref{rho}$)$.
Then we have $f\in L^r(\Omega)$ for all 
\begin{align*}
r\in \left(\max\Big\{
\frac{3q}{3+\alpha q+\beta q},\frac{2q}{2+\alpha q}\Big\},q\right]
\end{align*}
with the estimate
\begin{align*}
\|f\|_{r,\Omega}
\leq\|(1+|x|)^{-\alpha}(1+|x|-x_1)^{-\beta}\|_{\frac{qr}{q-r},\Omega}
\|(1+|x|)^{\alpha}(1+|x|-x_1)^{\beta}f\|_{q,\Omega}.
\end{align*}
\item Let $1<q<\infty,0<\alpha<3(1-1/q)$. Then 
$f\in L^q_{(1+|x|)^{\alpha q}}(\Omega)$ belongs to 
$L^{(3q)/(3+\alpha q),q}(\Omega)$ with 
\begin{align*}
\|f\|_{L^{\frac{3q}{3+\alpha q},q}(\Omega)}
\leq\|(1+|x|)^{-\alpha}\|_{L^{\frac{3}{\alpha},\infty}(\Omega)}
\|(1+|x|)^{\alpha}f\|_{q,\Omega},
\end{align*} 
where $L^{q,r}(\Omega)$ denotes the Lorentz space.
\end{enumerate}
\end{lem}

\begin{proof}
By H\"{o}lder inequality and Lemma \ref{finite}, we have 
\begin{align*}
\int_\Omega|f|^r\,dx \leq
\left(\int_\Omega(1+|x|)^{-\alpha\frac{qr}{q-r}}
(1+|x|-x_1)^{-\beta\frac{qr}{q-r}}\,dx\right)^{1-\frac{r}{q}}
\|(1+|x|)^{\alpha}
(1+|x|-x_1)^{\beta}f\|_{q,\Omega}^r<\infty
\end{align*}
if $r<q$, $2\beta<\alpha$ and 
\begin{align}\label{qs}
\frac{3}{\alpha+\beta}<\frac{qr}{q-r}<\frac{1}{\beta}
\end{align}
or if $r<q$ and 
\begin{align}\label{qs2}
\frac{qr}{q-r}>\max\left\{\frac{1}{\beta},
\frac{2}{\alpha}\right\}.
\end{align}
Because the condition (\ref{qs}) 
is equivalent to 
$(3q)/(3+\alpha q+\beta q)<r<q/(1+\beta q)$ and because 
$f\in L^q(\Omega)$, we have $f\in L^r(\Omega)$ for all 
$r\in((3q)/(3+\alpha q+\beta q),q]$ if $2\beta<\alpha$. 
On the other hand, 
if $2\beta\geq \alpha$, then we find that 
(\ref{qs2}) is equivalent to 
$r>q/(2+\alpha q)$, thus $f\in L^r(\Omega)$ for all 
$r\in((2q)/(2+\alpha q),q]$. 
Since the condition $2\beta<\alpha$ (resp. $2\beta\geq\alpha$) 
is equivalent to 
\begin{align*}
\max\Big\{
\frac{3q}{3+\alpha q+\beta q},\frac{2q}{2+\alpha q}\Big\}
=\frac{3q}{3+\alpha q+\beta q}\quad 
\left(\text{resp.}~\max\Big\{
\frac{3q}{3+\alpha q+\beta q},\frac{2q}{2+\alpha q}\Big\}
=\frac{2q}{2+\alpha q}\right), 
\end{align*}
the above observation yields the assertion 1. 
The assertion 2 directly follows from 
$(1+|x|)^{-\alpha}\in L^{3/\alpha,\infty}(\Omega)$ and the 
Lorentz-H\"{o}lder inequality.
The proof is complete.
\end{proof}

By Lemma \ref{holder} and by 
the explicit formula (\ref{oseensemir3}), 
we have the following estimates. 
\begin{prop}\label{lqlrr3}
\begin{enumerate}
\item Given $a_0>0$ arbitrarily, we assume $a\in[0,a_0]$. 
Let $1<q\leq r\leq \infty ~(q\ne\infty)$ 
and $\alpha,\beta\geq 0$ satisfy 
\begin{align}\label{alphabetadef0}
\beta<1-\frac{1}{q},\quad 
\alpha+\beta<3\left(1-\frac{1}{q}\right). 
\end{align}
For multi-index $k$ satisfying $|k|\leq 1$, 
there exists a constant 
$C=C(a_0,q,r,\alpha,\beta,k)$, independent of $a$, such that 
\begin{align}\label{lqlrr3small}
\|(1+|x|)^\alpha(1+|x|-x_1)^\beta\partial^k_{x}
S_a(t)P_{\R^3}f\|_{r,\R^3}\leq 
Ct^{-\frac{3}{2}(\frac{1}{q}-\frac{1}{r})-\frac{|k|}{2}}
\|(1+|x|)^\alpha(1+|x|-x_1)^\beta f\|_{q,\R^3}
\end{align}
for all $t\leq 1$, $f\in L^{q}_{\rho}(\R^3)$, where
$\rho$ is given by $($\ref{rho}$)$. 
\item Given $a_0>0$ arbitrarily, we assume $a\in[0,a_0]$. 
Let $1<q_i<\infty~(i=1,2,3,4)$, $1<r\leq\infty$  
and $\alpha,\beta$ satisfy $q_i\leq r~(i=1,2,3,4)$ and 
\begin{align*}
0\leq \alpha<3\left(1-\frac{1}{q_3}\right),
\quad0\leq 
\beta<\min\left\{1-\frac{1}{q_1},1-\frac{1}{q_2}\right\},
\quad\alpha+\beta<3\left(1-\frac{1}{q_1}\right).
\end{align*}  
For multi-index $k$ satisfying $|k|\leq 1$, there exists a constant 
$C=C(a_0,q_1,q_2,q_3,q_4,r,\alpha,\beta,k)$, 
independent of $a$, such that 
\begin{align}
\|&(1+|x|)^\alpha(1+|x|-x_1)^\beta\partial^k_{x}S_a(t)P_{\R^3}f
\|_{r,\R^3}
\nonumber\\
&\leq 
C\sum_{i=1}^4
t^{-\frac{3}{2}(\frac{1}{q_i}-\frac{1}{r})-\frac{|k|}{2}+\eta_i}
\|(1+|x|)^{\gamma_i}(1+|x|-x_1)^{\delta_i} f\|_{q_i,\R^3}
\label{lqlrr3large}
\end{align}
for all $t\geq 1$, $f\in \displaystyle\bigcap^3_{i=1}
L^{q_i}_{\rho_i}(\R^3)\cap L^{q_4}(\R^3)$,
where $\gamma_i,\delta_i,\eta_i$ and $\rho_i$ are 
given by $($\ref{gammadeltaeta}$)$ and $($\ref{rhodef}$)$.
\item Given $a_0>0$ arbitrarily, we assume $a\in(0,a_0]$. 
Let $1<q\leq r\leq\infty~(q\ne\infty)$ 
and $\alpha,\beta>0$ satisfy $($\ref{alphabetadef0}$)$. 
For multi-index $k$ satisfying $|k|\leq 1$ and $\varepsilon>0,$ 
there exists a constant 
$C=C(a_0,q,r,\alpha,\beta,k,\varepsilon)$, 
independent of $a$, such that  
\begin{align}
\|&(1+|x|)^\alpha(1+|x|-x_1)^\beta\partial^k_{x}
S_a(t)P_{\R^3}f\|_{r,\R^3}\notag\\&\leq 
Ct^{-\frac{3}{2}(\frac{1}{q}-\frac{1}{r})-\frac{|k|}{2}+
\frac{\alpha}{4}+\max\{\frac{\alpha}{4},\frac{\beta}{2}\}
+\varepsilon}
\|(1+|x|)^\alpha(1+|x|-x_1)^\beta f\|_{q,\R^3}\label{losslessr31}
\end{align}
for all $t\geq 1$ and $f\in L^{q}_{\rho}(\R^3)$, 
where $\rho$ is given by $($\ref{rho}$).$
\item Given $a_0>0$ arbitrarily, we assume $a\in(0,a_0]$. 
Let $1<q\leq r\leq\infty~(q\ne\infty)$ 
and $0\leq\alpha<3(1-1/q)$. 
For multi-index $k$ satisfying $|k|\leq 1$, 
there exists a constant 
$C=C(a_0,q,r,\alpha,k)$, 
independent of $a$, such that 
\begin{align}
\|(1+|x|)^\alpha\partial^k_{x}
S_a(t)P_{\R^3}f\|_{r,\R^3}\leq 
Ct^{-\frac{3}{2}(\frac{1}{q}-\frac{1}{r})-\frac{|k|}{2}+
\frac{\alpha}{2}}
\|(1+|x|)^\alpha f\|_{q,\R^3}\label{losslessr32}
\end{align}
for all $t\geq 1$ and $f\in L^{q}_{(1+|x|)^{\alpha q}}(\R^3)$.
\item Let $a=0$. Let $1<q\leq r\leq\infty~(q\ne\infty)$ 
and $0\leq\alpha<3(1-1/q)$. 
For multi-index $k$ satisfying $|k|\leq 1$, 
there exists a constant 
$C=C(q,r,\alpha,k)$, such that 
\begin{align}
\|(1+|x|)^\alpha\partial^k_{x}
S_0(t)P_{\R^3}f\|_{r,\R^3}\leq 
Ct^{-\frac{3}{2}(\frac{1}{q}-\frac{1}{r})-\frac{|k|}{2}}
\|(1+|x|)^\alpha f\|_{q,\R^3}\label{losslessr33}
\end{align}
for all $t>0$ and $f\in L^{q}_{(1+|x|)^{\alpha q}}(\R^3)$.
\end{enumerate}
\end{prop}

\begin{proof}
We first prove the assertion 1 and 2. From 
\begin{align*}
(1+|x|)^\alpha(1+|x|-x_1)^\beta
\leq C\big\{&(1+|y|)^\alpha(1+|y|-y_1)^\beta
+(1+|x-y|)^\alpha(1+|y|-y_1)^\beta\\
&+(1+|y|)^\alpha(1+|x-y|-(x_1-y_1))^\beta\\
&+(1+|x-y|)^\alpha(1+|x-y|-(x_1-y_1))^\beta\big\},
\end{align*}
we have 
\begin{align}
(1+|x|)^\alpha&(1+|x|-x_1)^\beta|\partial^k_{x}S_a(t)P_{\R^3}f(x)|
\nonumber\\
&\leq C\Big\{G_{1,k}*\big((1+|y|)^\alpha
(1+|y|-y_1)^\beta|P_{\R^3}f|\big)(x)
+G_{2,k}*\big((1+|y|-y_1)^\beta|P_{\R^3}f|\big)(x)\nonumber\\
&\qquad\quad+G_{3,k}*\big((1+|y|)^\alpha|P_{\R^3}f|\big)(x)
+G_{4,k}*|P_{\R^3}f|(x)\Big\},\label{g1g4}
\end{align}
where 
\begin{align}
&G_{1,k}(x,t)=\left(\frac{1}{4\pi t}\right)^{\frac{3}{2}}
\left(\frac{1}{2\sqrt{t}}\right)^{|k|}
\left|h_k
\left(\frac{x-ate_1}{2\sqrt{t}}\right)\right|,\label{g1def}\\
&G_{2,k}(x,t)=\left(\frac{1}{4\pi t}\right)^{\frac{3}{2}}
\left(\frac{1}{2\sqrt{t}}\right)^{|k|}
\left|
h_k\left(\frac{x-ate_1}{2\sqrt{t}}\right)
\right|
(1+|x|)^\alpha,\label{g2def}\\
&G_{3,k}(x,t)=\left(\frac{1}{4\pi t}\right)^{\frac{3}{2}}
\left(\frac{1}{2\sqrt{t}}\right)^{|k|}
\left|
h_k\left(\frac{x-ate_1}{2\sqrt{t}}\right)
\right|
(1+|x|-x_1)^\beta,\label{g3def}\\
&G_{4,k}(x,t)=\left(\frac{1}{4\pi t}\right)^{\frac{3}{2}}
\left(\frac{1}{2\sqrt{t}}\right)^{|k|}
\left|
h_k\left(\frac{x-ate_1}{2\sqrt{t}}\right)
\right|
(1+|x|)^\alpha(1+|x|-x_1)^\beta,\label{g4def}\\
&h_k(z)=\partial_z^ke^{-|z|^2}\label{hdef}
\end{align}
By changing variables $z=(x-ate_1)/(2\sqrt{t})$, we get 
\begin{align*}
\|G_{1,k}(t)\|_{s,\R^3}^s=
\left(\frac{1}{4\pi t}\right)^{\frac{3}{2}s}
\left(\frac{1}{2\sqrt{t}}\right)^{|k|s}(2\sqrt{t})^3
\int_{\R^3}|\partial_z^ke^{-|z|^2}|^s\,dz
\end{align*}
for all $t>0$, $s\in [1,\infty)$, thus 
\begin{align}\label{g1s}
\|G_{1,k}(t)\|_{s,\R^3}
\leq Ct^{-\frac{3}{2}-\frac{|k|}{2}+\frac{3}{2s}}
\end{align}
for  $t>0$ and $s\in [1,\infty)$. Moreover, it holds that 
\begin{align}
\|G_{2,k}(t)\|_{s,\R^3}^s
&\leq \left(\frac{1}{4\pi t}\right)^{\frac{3}{2}s}
\left(\frac{1}{2\sqrt{t}}\right)^{|k|s}(2\sqrt{t})^3
\int_{\R^3} |\partial_z^ke^{-|z|^2}|^s
\big(1+at+2\sqrt{t}|z|\big)^{\alpha s}\,dz\notag\\
&\leq Ct^{-\frac{3}{2}s-\frac{|k|}{2}s+\frac{3}{2}}
\left(\int_{|z|\leq \frac{1+at}{2\sqrt{t}}}
|\partial_z^ke^{-|z|^2}|^s\cdot2^{\alpha s}
(1+at)^{\alpha s}\,dz\right.\notag\\
&\hspace{5cm}\left.+\int_{|z|\geq \frac{1+at}{2\sqrt{t}}}
|\partial_z^ke^{-|z|^2}|^s\cdot2^{\alpha s}
(2\sqrt{t}|z|)^{\alpha s}\,dz
\right)\notag\\
&\leq Ct^{-\frac{3}{2}s-\frac{|k|}{2}s+\frac{3}{2}}
\left((1+at)^{\alpha s}\int_{\R^3}|\partial_z^ke^{-|z|^2}|^s\,dz
+t^{\frac{\alpha}{2}s}
\int_{\R^3}|\partial_z^ke^{-|z|^2}|^s\,|z|^{\alpha s}\,dz\right)
\notag\\
&\leq Ct^{-\frac{3}{2}s-\frac{|k|}{2}s+\frac{3}{2}}
\{(1+at)^{\alpha s}+t^{\frac{\alpha}{2}s}\},\label{g2est0}
\end{align}
and that 
\begin{align*}
\|G_{3,k}(t)\|^s_{s,\R^3}&=
\left(\frac{1}{4\pi t}\right)^{\frac{3}{2}s}
\left(\frac{1}{2\sqrt{t}}\right)^{|k|s}(2\sqrt{t})^3
\int_{\R^3} |\partial_z^ke^{-|z|^2}|^s
\big(1+|ate_1+2\sqrt{t}z|-(at+2\sqrt{t}z_1)\big)^{\beta s}\,dz\\
&\leq \left(\frac{1}{4\pi t}\right)^{\frac{3}{2}s}
\left(\frac{1}{2\sqrt{t}}\right)^{|k|s}(2\sqrt{t})^3
\int_{\R^3} |\partial_z^ke^{-|z|^2}|^s
\big(1+4\sqrt{t}|z|\big)^{\beta s}\,dz\\
&\leq Ct^{-\frac{3}{2}s-\frac{|k|}{2}s+\frac{3}{2}}\left(
\int_{\R^3} |\partial_z^ke^{-|z|^2}|^s\,dz+
(\sqrt{t})^{\beta s}\int_{\R^3}
|\partial_z^ke^{-|z|^2}|^s|z|^{\beta s}\,dz\right)\\
&\leq Ct^{-\frac{3}{2}s-\frac{|k|}{2}s+\frac{3}{2}}
(1+t^{\frac{\beta}{2}s}),
\end{align*}
thereby,
\begin{align}\label{g2s}
\|G_{2,k}(t)\|_{s,\R^3}\leq 
\begin{cases}
 C(1+a_0)^{\alpha }t^{-\frac{3}{2}-\frac{|k|}{2}+\frac{3}{2s}}
&\quad{\rm if}~t\leq 1,a\in[0,a_0],\\
Ca_0^{\alpha }t^{-\frac{3}{2}-\frac{|k|}{2}+\frac{3}{2s}+\alpha}
&\quad{\rm if}~t\geq 1,a\in(0,a_0],\\
Ct^{-\frac{3}{2}-\frac{|k|}{2}+\frac{3}{2s}+\frac{\alpha}{2}}
&\quad{\rm if}~t\geq 1,a=0 
\end{cases}
\end{align}
and 
\begin{align}\label{g3s}
\|G_{3,k}(t)\|_{s,\R^3}\leq 
\begin{cases}
 Ct^{-\frac{3}{2}-\frac{|k|}{2}+\frac{3}{2s}}
&\quad{\rm if}~t\leq 1,a\in[0,a_0],\\
Ct^{-\frac{3}{2}-\frac{|k|}{2}+\frac{3}{2s}+\frac{\beta}{2}}
&\quad{\rm if}~t\geq 1,a\in[0,a_0]
\end{cases}
\end{align}
for all $s\in[1,\infty)$, where constant $C$ is 
independent of $a$. From
\begin{align*}
\|G_{4,k}(t)\|_{s,\R^3}^s&
\leq Ct^{-\frac{3}{2}s-\frac{|k|}{2}s+\frac{3}{2}}
\int_{\R^3} |\partial_z^ke^{-|z|^2}|^s
(1+at+2\sqrt{t}|z|)^{\alpha s}
\big(1+4\sqrt{t}|z|\big)^{\beta s}\,dz\\
&=Ct^{-\frac{3}{2}s-\frac{|k|}{2}s+\frac{3}{2}}
\left(\int_{|z|\leq \frac{1}{4\sqrt{t}}}
+\int_{\frac{1}{4\sqrt{t}}\leq |z|\leq\frac{1+at}{2\sqrt{t}}}
+\int_{\frac{1+at}{2\sqrt{t}}\leq |z|}\right)\\
&\leq Ct^{-\frac{3}{2}s-\frac{|k|}{2}s+\frac{3}{2}}
\left\{(1+at)^{\alpha s}\int_{\R^3}
|\partial_z^ke^{-|z|^2}|^s\,dz+(1+at)^{\alpha s}(\sqrt{t})^{\beta s}
\int_{\R^3}
|\partial_z^ke^{-|z|^2}|^s\,|z|^{\beta s}\,dz\right.\\
&\hspace{4cm}\left.+(\sqrt{t})^{(\alpha+\beta)s}
\int_{\R^3}
|\partial_z^ke^{-|z|^2}|^s\,|z|^{\alpha s+\beta s}\,dz\right\}\\
&\leq Ct^{-\frac{3}{2}s-\frac{|k|}{2}s+\frac{3}{2}}
\big\{(1+at)^{\alpha s}(1+t^{\frac{\beta}{2}s})
+t^{\frac{\alpha}{2}s+\frac{\beta}{2}s}\big\},
\end{align*}
we find
\begin{align}
\|G_{4,k}(t)\|_{s,\R^3}\leq 
\begin{cases}
C(1+a_0)^{\alpha}t^{-\frac{3}{2}-\frac{|k|}{2}+\frac{3}{2s}}
&\quad{\rm if}~t\leq 1,a\in[0,a_0],\\
Ca_0^{\alpha}t^{-\frac{3}{2}-\frac{|k|}{2}+\frac{3}{2s}+\alpha
+\frac{\beta}{2}}
&\quad{\rm if}~t\geq 1,a\in(0,a_0],\\
Ct^{-\frac{3}{2}-\frac{|k|}{2}+\frac{3}{2s}+
\frac{\alpha}{2}+\frac{\beta}{2}}
&\quad{\rm if}~t\geq 1,a=0 
\end{cases}
\end{align} 
for all $s\in[1,\infty)$, where constant $C$ is 
independent of $a$. Given $q,q_i\leq r~(i=1,2,3,4)$, 
we set 
\begin{align}\label{si}
\frac{1}{s_0}:=1-\frac{1}{q}+\frac{1}{r}\in (0,1],\quad 
\frac{1}{s_i}:=1-\frac{1}{q_i}+\frac{1}{r}\in (0,1]
\quad (i=1,2,3,4).
\end{align}
Then collecting (\ref{g1g4})--(\ref{si}) implies that
\begin{align}
\|&(1+|x|)^\alpha(1+|x|-x_1)^\beta
\partial^k_{x}S_a(t)P_{\R^3}f(x)\|_{r,\R^3}\nonumber\\
&\leq C\Big(\|G_{1,k}(t)\|_{s_0,\R^3}
\|\big((1+|y|)^\alpha(1+|y|-y_1)^\beta P_{\R^3}f\|_{q,\R^3}
+\|G_{2,k}(t)\|_{s_0,\R^3}\|(1+|y|-y_1)^\beta P_{\R^3}f
\|_{q,\R^3}\nonumber\\
&\qquad\quad+\|G_{3,k}(t)\|_{s_0,\R^3}
\|(1+|y|)^\alpha P_{\R^3}f\|_{q,\R^3}
+\|G_{4,k}(t)\|_{s_0,\R^3}\|P_{\R^3}f\|_{q,\R^3}\Big)\nonumber\\
&\leq Ct^{-\frac{3}{2}(\frac{1}{q}-\frac{1}{r})-\frac{|k|}{2}}
\|(1+|y|)^\alpha(1+|y|-y_1)^\beta f
\|_{q,\R^3}\label{lqlrr3small2}
\end{align}
for $t\leq 1$, $f\in L^{q}_{\rho}(\R^3)$ and that  
\begin{align}
\|&(1+|x|)^\alpha(1+|x|-x_1)^\beta
\partial^k_{x}S_a(t)P_{\R^3}f(x)\|_{r,\R^3} \nonumber\\
&\leq C\Big(\|G_{1,k}(t)\|_{s_1,\R^3}
\|\big((1+|y|)^\alpha(1+|y|-y_1)^\beta P_{\R^3}f\|_{q_1,\R^3}
+\|G_{2,k}(t)\|_{s_2,\R^3}
\|(1+|y|-y_1)^\beta P_{\R^3}f\|_{q_2,\R^3}\nonumber\\
&\qquad\quad+\|G_{3,k}(t)\|_{s_3,\R^3}
\|(1+|y|)^\alpha P_{\R^3}f\|_{q_3,\R^3}
+\|G_{4,k}(t)\|_{s_4,\R^3}\|P_{\R^3}f\|_{q_4,\R^3}\Big)\notag\\
&\leq C\sum_{i=1}^4
t^{-\frac{3}{2}(\frac{1}{q_i}-\frac{1}{r})-\frac{|k|}{2}+\eta_i}
\|(1+|x|)^{\gamma_i}(1+|x|-x_1)^{\delta_i}f\|_{q_i,\R^3}
\label{lqlrlr3large2}
\end{align}
for all $t\geq 1$, $f\in \displaystyle\bigcap^3_{i=1}
L^{q_i}_{\rho_i}(\R^3)\cap L^{q_4}(\R^3)$. 
The proof of the assertion 1 and 2 is complete. 
\par We next prove the assertion 3. 
To derive (\ref{losslessr31}), we use (\ref{g1g4}). 
It follows from (\ref{g1s}) and (\ref{g3s}) that 
\begin{align}
\|G_{1,k}*\big((1+|y|)^\alpha(1+|y|-y_1)^\beta|P_{\R^3}f|
\big)\|_{r,\R^3}&\leq 
\|G_{1,k}(t)\|_{s_0,\R^3}\|(1+|y|)^\alpha
(1+|y|-y_1)^\beta P_{\R^3}f\|_{q,\R^3}\notag\\
&\leq Ct^{-\frac{3}{2}(\frac{1}{q}-\frac{1}{r})-\frac{|k|}{2}}
\|(1+|x|)^\alpha(1+|x|-x_1)^\beta f\|_{q,\R^3}\label{g1est}
\end{align}
and
\begin{align}
\|G_{3,k}*\big((1+|y|)^\alpha|P_{\R^3}f|\big)\|_{r,\R^3}&\leq 
\|G_{3,k}(t)\|_{s_0,\R^3}\|(1+|y|)^\alpha P_{\R^3}f\|_{q,\R^3}\notag\\
&\leq Ct^{-\frac{3}{2}(\frac{1}{q}-\frac{1}{r})-\frac{|k|}{2}
+\frac{\beta}{2}}\|(1+|y|)^{\alpha}(1+|y|-y_1)^\beta
f\|_{q,\R^3}\label{g3est}
\end{align}
for all $t\geq 1$, $f\in L^{q}_{\rho}(\R^3)$, 
where $s_0$ is defined by (\ref{si}). 
If $f\in L^q_{\rho}(\R^3)$, then 
$(1+|y|-y_1)^\beta P_{\R^3}f\in L^q_{(1+|x|)^{\alpha q}}(\R^3)$, thus 
the assertion 2 of Lemma \ref{holder} yields 
$(1+|y|-y_1)^\beta P_{\R^3}f\in L^{(3q)/(3+\alpha q),q}(\R^3)$ with 
\begin{align}\label{beta}
\|(1+|y|-y_1)^\beta P_{\R^3}f\|_{L^{\frac{3q}{3+\alpha q},q}(\R^3)}
\leq\|(1+|y|)^{-\alpha}\|_{L^{\frac{3}{\alpha},\infty}(\R^3)}
\|(1+|y|)^{\alpha}(1+|y|-y_1)^\beta P_{\R^3}f\|_{q,\R^3}.
\end{align}
Moreover, we define $s_5$ by  
\begin{align*}
\frac{1}{s_5}=-\frac{1}{q}+\frac{1}{r}-\frac{\alpha}{3}+1,
\end{align*}
which satisfies 
$1<s_5<\infty$ and $1<(3s_5)/(3+\alpha s_5)<\infty$ 
due to $\alpha<3(1-1/q)$ and $q\leq r$. 
Then we get 
$G_{2,k}(t)\in L^{s_5,(3s_5)/(3+\alpha s_5)}(\R^3)$ with 
\begin{align}
\|G_{2,k}(t)\|_{L^{s_5,\frac{3s_5}{3+\alpha s_5}}(\R^3)}
\leq C\|G_{2,k}(t)\|^{1-\theta}_{\kappa_1,\R^3}
\|G_{2,k}(t)\|^\theta_{\kappa_2,\R^3}\label{interpo0}
\end{align}
by virtue of 
\begin{align*}
L^{s_5,\frac{3s_5}{3+\alpha s_5}}(\R^3)
=(L^{\kappa_1}(\R^3),L^{\kappa_2}(\R^3))_{\theta,
\frac{3s_5}{3+\alpha s_5}},
\end{align*} 
where $(\cdot,\cdot)_{\theta,q}$ denotes 
the real interpolation functor and 
$1<\kappa_1<s_5<\kappa_2<\infty,0<\theta<1$ satisfy 
$1/s_5=(1-\theta)/\kappa_1+\theta/\kappa_2$, see 
Bergh and L\"{o}fstr\"{o}m \cite{belobook1976}. 
From (\ref{beta}), (\ref{interpo0}), (\ref{g2s})  
together with Young's inequality for convolution in 
Lorentz spaces due to O'neil 
\cite[Theorem 2.6 and Theorem 3.6]{oneil1963}, we have 
\begin{align}
\|&G_{2,k}*\big((1+|y|-y_1)^\beta|P_{\R^3}f|\big)\|_{r,\R^3}\notag\\
&\leq C\|G_{2,k}(t)\|_{L^{s_5,\frac{3s_5}{3+\alpha s_5}}(\R^3)}
\|(1+|y|-y_1)^\beta P_{\R^3}f\|
_{L^{\frac{3q}{3+\alpha q},q}(\R^3)}\label{g2est2}\\
&\leq Ct^{-\frac{3}{2}-\frac{|k|}{2}+\frac{3}{2s_5}+\alpha} 
\|(1+|y|)^{-\alpha}\|_{L^{\frac{3}{\alpha},\infty}(\R^3)}
\|(1+|y|)^{\alpha}(1+|y|-y_1)^\beta P_{\R^3}f\|_{q,\R^3}\notag\\
&=Ct^{-\frac{3}{2}(\frac{1}{q}-\frac{1}{r})-\frac{|k|}{2}
+\frac{\alpha}{2}} 
\|(1+|y|)^{\alpha}(1+|y|-y_1)^\beta f\|_{q,\R^3}\label{g2est}
\end{align}
for all $t\geq 1$, $f\in L^{q}_{\rho}(\R^3)$. 
Let $\widetilde{q}$ satisfy 
$\max\{
(3q)/(3+\alpha q+\beta q),(2q)/(2+\alpha q)\}<\widetilde{q}<q$ and 
set  
\begin{align}
\frac{1}{\widetilde{q}}=
\left(\max\Big\{
\frac{3q}{3+\alpha q+\beta q},
\frac{2q}{2+\alpha q}\Big\}\right)^{-1}-\frac{2}{3}\varepsilon,
\notag
\end{align}
where $\varepsilon\in(0,(\alpha+\beta)/2)$ (resp. 
$\varepsilon\in(0,(3\alpha)/4)$) 
if $2\beta<\alpha$ (resp. $2\beta\geq \alpha$). 
Then the condition (\ref{alphabetadef0}) ensures 
$\widetilde{q}\in(1,\infty)$ 
and $0<1-(1/\widetilde{q}-1/r)<1$, thus we use 
the assertion 1 in Lemma \ref{holder} to obtain
\begin{align}
\|G_{4,k}*|P_{\R^3}f|\|_{r,\R^3}
&\leq \|G_{4,k}(t)\|_{(1-(1/\widetilde{q}-1/r))^{-1},\R^3}
\|P_{\R^3}f\|_{\widetilde{q},\R^3}\notag\\
&\leq Ct^{-\frac{3}{2}(\frac{1}{\widetilde{q}}-\frac{1}{r})
-\frac{|k|}{2}+\alpha+\frac{\beta}{2}}
\|(1+|y|)^{\alpha}(1+|y|-y_1)^\beta f\|_{q,\R^3}\notag\\
&\leq 
Ct^{-\frac{3}{2}(\frac{1}{q}-\frac{1}{r})
-\frac{|k|}{2}+\frac{\alpha}{4}+
\max\{\frac{\alpha}{4},\frac{\beta}{2}\}+\varepsilon}
\|(1+|y|)^{\alpha}(1+|y|-y_1)^\beta f\|_{q,\R^3}\label{g4est}
\end{align}
for all $t\geq 1$ and $f\in L^{q}_{\rho}(\R^3)$. 
Collecting (\ref{g1g4}), (\ref{g1est}), 
(\ref{g3est}), (\ref{g2est})--(\ref{g4est}) 
yields (\ref{losslessr31}) for $t\geq 1$ and 
$f\in L^{q}_{\rho}(\R^3)$, from which we conclude 
the assertion 3. 
\par Under the condition $0\leq \alpha<3(1-1/q)$, we see 
from the same calculation as in (\ref{g2est}) that 
\begin{align*}
\|G_{2,k}*|P_{\R^3}f|\|_{r,\R^3}\leq 
\begin{cases}
Ct^{-\frac{3}{2}(\frac{1}{q}-\frac{1}{r})-\frac{|k|}{2}
+\frac{\alpha}{2}} 
\|(1+|y|)^{\alpha}f\|_{q,\R^3}&{\rm if~}a>0,\\
Ct^{-\frac{3}{2}(\frac{1}{q}-\frac{1}{r})-\frac{|k|}{2}} 
\|(1+|y|)^{\alpha}f\|_{q,\R^3}&{\rm if~}a=0
\end{cases}
\end{align*}
for $t\geq 1$, $f\in L^{q}_{\rho}(\R^3)$. 
By applying this estimate and 
(\ref{g1est}) with $\beta=0$ to 
\begin{align}
(1+|x|)^\alpha|\partial^k_{x}S_a(t)P_{\R^3}f(x)|
\leq C\big\{G_{1,k}*\big((1+|y|)^\alpha|P_{\R^3}f|\big)(x)
+G_{2,k}*|P_{\R^3}f|(x)\big\},\label{g1g2}
\end{align}
we have (\ref{losslessr32}) and (\ref{losslessr33}) for 
$t\geq 1,f\in L^q_{(1+|x|)^{\alpha q}}(D)$, 
which combined with the assertion 1 with $a=\beta=0$ 
completes the proof.
\end{proof}

\begin{rmk}
By the similar argument to the proof of Proposition \ref{lqlrr3}, 
Kobayashi and Kubo \cite{koku2015} also derived 
the homogeneous estimate (\ref{losslessr33}), 
but it is not clear whether we can obtain 
homogeneous one if $\beta>0$ or if $a>0$. 
For the relation between (\ref{losslessr32}) and (\ref{losslessr33}),
we can recover (\ref{losslessr33}) by taking the 
limit $a\rightarrow 0$. In fact, 
from (\ref{g1g2}), (\ref{g2est0}), 
(\ref{g1est}) with $\beta=0$, (\ref{interpo0}) and 
(\ref{g2est2}) with $\beta=0$, we can obtain
\begin{align}
\|&(1+|x|)^\alpha\partial^k_{x}S_a(t)P_{\R^3}f\|_{r,\R^3}\notag\\
&\leq C\|G_{1,k}*\big((1+|y|)^\alpha|P_{\R^3}f|\big)\|_{r,\R^3}
+C\|G_{2,k}(t)\|_{L^{s_5,
\frac{3s_5}{3+\alpha s_5}}(\R^3)}\|P_{\R^3}f\|
_{L^{\frac{3q}{3+\alpha q},q}(\R^3)}\notag\\
&\leq Ct^{-\frac{3}{2}(\frac{1}{q}-\frac{1}{r})-\frac{|k|}{2}}
\|(1+|y|)^\alpha f\|_{q,\R^3}+
Ct^{-\frac{3}{2}-\frac{|k|}{2}+\frac{3}{2s_5}}
\{(1+at)^{\alpha}+t^{\frac{\alpha}{2}}\}
\|(1+|y|)^\alpha f\|_{q,\R^3}\notag\\
&=Ct^{-\frac{3}{2}(\frac{1}{q}-\frac{1}{r})-\frac{|k|}{2}}
\big\{t^{-\frac{\alpha}{2}}(1+at)^\alpha+1\big\}
\|(1+|y|)^\alpha f\|_{q,\R^3}\label{lqlrr3iso}
\end{align}
for all $a>0$, $t\geq 1$ and 
$f\in L^{q}_{(1+|x|)^{\alpha q}}(\R^3)$.
\end{rmk}

We next derive the estimates of $S_{-a}(t)$ in 
$L^q_{\rho_-}(D)$ by employing the 
duality argument based on the following lemma, 
where $\rho_-$ is given by (\ref{rhominus}), thus it is important to 
derive the $L^1$-$L^r$ estimates in the assertion 1.

\begin{lem}\label{lemlosslessr3}
\begin{enumerate}
\item Given $a_0>0$ arbitrarily, we assume $a\in[0,a_0]$. 
Let $1\leq q\leq r<\infty$ and 
let $\alpha,\beta\geq 0$. 
Then there exists a constant 
$C=C(a_0,q,r,\alpha,\beta)$, independent of $a$, such that 
\begin{align*}
\|(1+|x|)^\alpha(1+|x|-x_1)^\beta
S_a(t)g\|_{r,\R^3}\leq 
Ct^{-\frac{3}{2}(\frac{1}{q}-\frac{1}{r})}
(1+t)^{\eta_4}
\|(1+|x|)^\alpha(1+|x|-x_1)^\beta g\|_{q,\R^3}
\end{align*}
for all $t>0,g\in C_0^\infty(\R^3)^3$ and that 
\begin{align*}
\|&(1+|x|)^\alpha(1+|x|-x_1)^\beta
S_a(t)\partial_{y_j}g\|_{r,\R^3}\\
&\leq 
Ct^{-\frac{3}{2}(\frac{1}{q}-\frac{1}{r})-\frac{1}{2}}
(1+t)^{\eta_4}
\|(1+|x|)^\alpha(1+|x|-x_1)^\beta g\|_{q,\R^3}
\end{align*}
for all $t>0,g\in C_0^\infty(\R^3)^3,j=1,2,3$, where 
$\eta_4$ is given by $($\ref{gammadeltaeta}$)$. 
\item Given $a_0>0$ arbitrarily, we assume $a\in(0,a_0]$. 
Let $1<q\leq r<\infty$ 
and $\alpha,\beta>0$ satisfy $($\ref{alphabetadef0}$)$. 
For $\varepsilon>0,$ there exists a constant 
$C=C(a_0,q,r,\alpha,\beta,\varepsilon)$, 
independent of $a$, such that  
\begin{align*}
\|&(1+|x|)^\alpha(1+|x|-x_1)^\beta
S_a(t)g\|_{r,\R^3}\notag\\&\leq 
Ct^{-\frac{3}{2}(\frac{1}{q}-\frac{1}{r})+
\frac{\alpha}{4}+\max\{\frac{\alpha}{4},\frac{\beta}{2}\}
+\varepsilon}
\|(1+|x|)^\alpha(1+|x|-x_1)^\beta g\|_{q,\R^3}
\end{align*}
for all $t\geq 1$, $g\in C^\infty_{0}(\R^3)^3$ and that 
\begin{align*}
\|&(1+|x|)^\alpha(1+|x|-x_1)^\beta
S_a(t)\partial_{y_j}g\|_{r,\R^3}\notag\\&\leq 
Ct^{-\frac{3}{2}(\frac{1}{q}-\frac{1}{r})-\frac{1}{2}+
\frac{\alpha}{4}+\max\{\frac{\alpha}{4},\frac{\beta}{2}\}
+\varepsilon}
\|(1+|x|)^\alpha(1+|x|-x_1)^\beta g\|_{q,\R^3}
\end{align*}
for all $t\geq 1,g\in C^\infty_{0}(\R^3)^3,j=1,2,3$.
\item Given $a_0>0$ arbitrarily, we assume $a\in(0,a_0]$. 
Let $1<q\leq r<\infty$ 
and $0\leq\alpha<3(1-1/q)$. Then there exists a constant 
$C=C(a_0,q,r,\alpha)$, 
independent of $a$, such that 
\begin{align*}
\|&(1+|x|)^\alpha
S_a(t)g\|_{r,\R^3}\leq 
Ct^{-\frac{3}{2}(\frac{1}{q}-\frac{1}{r})+
\frac{\alpha}{2}}
\|(1+|x|)^\alpha g\|_{q,\R^3}
\end{align*}
for all $t\geq 1$, $g\in C^\infty_{0}(\R^3)^3$ and that 
\begin{align*}
\|(1+|x|)^\alpha S_a(t)\partial_{y_j}g\|_{r,\R^3}\leq 
Ct^{-\frac{3}{2}(\frac{1}{q}-\frac{1}{r})-\frac{1}{2}+
\frac{\alpha}{2}}
\|(1+|x|)^\alpha g\|_{q,\R^3}
\end{align*}
for all $t\geq 1,g\in C^\infty_{0}(\R^3)^3,j=1,2,3$.
\item Let $a=0$. Let $1<q\leq r<\infty$ 
and $0\leq\alpha<3(1-1/q)$. Then there exists a constant 
$C=C(q,r,\alpha)$ such that 
\begin{align*}
\|&(1+|x|)^\alpha
S_0(t)g\|_{r,\R^3}\leq 
Ct^{-\frac{3}{2}(\frac{1}{q}-\frac{1}{r})}
\|(1+|x|)^\alpha g\|_{q,\R^3}
\end{align*}
for all $t>0$, $g\in C^\infty_{0}(\R^3)^3$ and that 
\begin{align*}
\|(1+|x|)^\alpha S_0(t)\partial_{y_j}g\|_{r,\R^3}\leq 
Ct^{-\frac{3}{2}(\frac{1}{q}-\frac{1}{r})-\frac{1}{2}}
\|(1+|x|)^\alpha g\|_{q,\R^3}
\end{align*}
for all $t>0,g\in C^\infty_{0}(\R^3)^3,j=1,2,3$.
\end{enumerate}
\end{lem}

\begin{proof}
It follows that  
\begin{align}
(1+|x|)^\alpha&(1+|x|-x_1)^\beta|S_a(t)
g(x)|\nonumber\\
&\leq C\Big\{G_{1,0}*\big((1+|y|)^\alpha
(1+|y|-y_1)^\beta|g|\big)(x)
+G_{2,0}*\big((1+|y|-y_1)^\beta|g|\big)(x)\nonumber\\
&\qquad\quad+G_{3,0}*\big((1+|y|)^\alpha|g|\big)(x)
+G_{4,0}*|g|(x)\Big\}\label{g1g42}
\end{align} 
for $g\in C_0^\infty(\R^3)^3$ and that 
\begin{align}
(1+|x|)^\alpha&(1+|x|-x_1)^\beta|S_a(t)
\partial_{y_j}g(x)|\nonumber\\
&\leq C\Big\{G_{1,1}*\big((1+|y|)^\alpha
(1+|y|-y_1)^\beta|g|\big)(x)
+G_{2,1}*\big((1+|y|-y_1)^\beta|g|\big)(x)\nonumber\\
&\qquad\quad+G_{3,1}*\big((1+|y|)^\alpha|g|\big)(x)
+G_{4,1}*|g|(x)\Big\}\label{g1g43}
\end{align} 
for $g\in C_0^\infty(\R^3)^3$ due to
\begin{align*}
S_a(t)\partial_{y_j}g(x)
=-\left(\frac{1}{4\pi t}\right)^{\frac{3}{2}}
\left(\frac{1}{2\sqrt{t}}\right)
\int_{\R^3}h_j
\left(\frac{x-ate_1-y}{2\sqrt{t}}\right)g(y)\,dy,
\end{align*} 
where $h_j(z)=\partial_{z_j}e^{-|z|^2}$, 
see also (\ref{g1def})--(\ref{hdef}). 
By (\ref{g1g42})--(\ref{g1g43}), 
we conclude the assertions from the same calculation as in 
the proof of Proposition \ref{lqlrr3}. 
We note that each 
$s_0,\cdots,s_4$ with $q=q_i=1$, see (\ref{si}), coincides with $r$, 
which leads to $1<s_i<\infty~(i=0,1,2,3,4)$ if $r<\infty$. 
Therefore, we have the assertion 1 with $q=1$ 
by the same calculations as in 
(\ref{lqlrr3small2})--(\ref{lqlrlr3large2}) with $q=q_i=1$. 
\end{proof}

\noindent By taking Lemma \ref{lemlosslessr3} into account, 
the duality argument yields the following. 
In order to analyze the nonlinear problems, it is enough to 
derive the assertion 1 below. 
However, the assertion 2--4 below imply that  
the rate in the assertion 1 can be improved 
if $r<\infty$ and if $\alpha,\beta$ are smaller than 
those in the assertion 1, which is of independent interest. 
\begin{prop}\label{proplqlrr3dual}
\begin{enumerate}
\item Given $a_0>0$ arbitrarily, we assume $a\in[0,a_0]$. 
Let $i=0,1$ and let $1<q\leq r\leq \infty~(q\ne\infty),
\alpha,\beta\geq 0$ satisfy $($\ref{alphabeta4}$)$.
Then there exists a constant 
$C=C(a_0,i,q,r,\alpha,\beta)$, independent of $a$, such that 
\begin{align}
\|&(1+|x|)^{-\alpha}(1+|x|-x_1)^{-\beta}\nabla^i 
S_{-a}(t)P_{\R^3}f\|_{r,\R^3}
\nonumber\\
&\leq 
Ct^{-\frac{3}{2}(\frac{1}{q}-\frac{1}{r})-\frac{i}{2}}
(1+t)^{\eta_4}
\|(1+|x|)^{-\alpha}(1+|x|-x_1)^{-\beta} f\|_{q,\R^3}
\label{lqlrr3dual3}
\end{align}
for $t>0$ and $f\in L^{q}_{\rho_{-}}(\R^3)$, where 
$\eta_4$ and $\rho_-$ are given by $($\ref{gammadeltaeta}$)$ and 
$($\ref{rhominus}$)$, respectively. 
\item Given $a_0>0$ arbitrarily, we assume $a\in(0,a_0]$. 
Let $i=0,1$ and let $1<q\leq r<\infty,\alpha,\beta>0$ 
satisfy $\beta<1/r$ and $\alpha+\beta<3/r$. 
For $\varepsilon>0,$ there exists a constant 
$C=C(a_0,i,q,r,\alpha,\beta,\varepsilon)$, 
independent of $a$, such that 
\begin{align}
\|&(1+|x|)^{-\alpha}(1+|x|-x_1)^{-\beta}\nabla^i 
S_{-a}(t)P_{\R^3}f\|_{r,\R^3}
\nonumber\\
&\leq 
Ct^{-\frac{3}{2}(\frac{1}{q}-\frac{1}{r})
-\frac{i}{2}+
\frac{\alpha}{4}+\max\{\frac{\alpha}{4},\frac{\beta}{2}\}
+\varepsilon}
\|(1+|x|)^{-\alpha}(1+|x|-x_1)^{-\beta} f\|_{q,\R^3}
\label{lqlrr3dual}
\end{align}
for $t\geq 1$ and $f\in L^{q}_{\rho_{-}}(\R^3)$, where 
$\rho_-$ is given by $($\ref{rhominus}$)$. 
\item Given $a_0>0$ arbitrarily, we assume $a\in(0,a_0]$. 
Let $i=0,1$ and let $1<q\leq r<\infty,0\leq \alpha<3/r.$ 
Then there exists a constant 
$C=C(a_0,i,q,r,\alpha)$, 
independent of $a$, such that 
\begin{align}
\|(1+|x|)^{-\alpha}\nabla^i S_{-a}(t)P_{\R^3}f\|_{r,\R^3}
\leq 
Ct^{-\frac{3}{2}(\frac{1}{q}-\frac{1}{r})
-\frac{i}{2}+\frac{\alpha}{2}}
\|(1+|x|)^{-\alpha}f\|_{q,\R^3}
\label{lqlrr3dual2}
\end{align}
for $t\geq 1$ and $f\in L^{q}_{\rho_{-}}(\R^3)$, where 
$\rho_-$ is given by $($\ref{rhominus}$)$. 
\item Let $a=0,i=0,1$ and let $1<q\leq r<\infty,0\leq \alpha<3/r$.   
Then there exists a constant 
$C=C(i,q,r,\alpha)$ such that 
\begin{align}
\|(1+|x|)^{-\alpha}\nabla^i S_{0}(t)P_{\R^3}f\|_{r,\R^3}
\leq Ct^{-\frac{3}{2}(\frac{1}{q}-\frac{1}{r})
-\frac{i}{2}}\|(1+|x|)^{-\alpha}f\|_{q,\R^3}
\label{lqlrr3dual4}
\end{align}
for $t>0$ and $f\in L^{q}_{(1+|x|)^{-\alpha q}}(\R^3)$.
\end{enumerate}
\end{prop}

\begin{proof}
If (\ref{alphabeta4}) is fulfilled, then it turns out that 
\begin{align*}
|(S_{-a}(t)P_{\R^3}f,\varphi)_{\R^3}|
&=\left|(P_{\R^3}f,S_a(t)\varphi)_{\R^3}\right|\\
&\leq \|(1+|x|)^{-\alpha}(1+|x|-x_1)^{-\beta}P_{\R^3}f\|_{q,\R^3}
\|(1+|x|)^{\alpha}(1+|x|-x_1)^{\beta}S_a(t)\varphi\|_{q',\R^3}\\
&\leq \|(1+|x|)^{-\alpha}(1+|x|-x_1)^{-\beta}f\|_{q,\R^3}
\|(1+|x|)^{\alpha}(1+|x|-x_1)^{\beta}S_a(t)\varphi\|_{q',\R^3}
\end{align*} 
for $t>0,\varphi\in C_{0}^\infty(\R^3)^3$ and that 
\begin{align*}
|(\nabla S_{-a}(t)P_{\R^3}f,\varphi)_{\R^3}|
&=\left|\sum_{i=1}^3(P_{\R^3}f,S_a(t)\partial_{y_i}
(\varphi_{i1},\varphi_{i2},\varphi_{i3}))_{\R^3}\right|\\
&\leq \|(1+|x|)^{-\alpha}(1+|x|-x_1)^{-\beta}P_{\R^3}f\|_{q,\R^3}\\
&\qquad \times \sum_{i=1}^3
\|(1+|x|)^{\alpha}(1+|x|-x_1)^{\beta}S_a(t)
\partial_{y_i}(\varphi_{i1},\varphi_{i2},\varphi_{i3})\|_{q',\R^3}\\
&\leq \|(1+|x|)^{-\alpha}(1+|x|-x_1)^{-\beta}f\|_{q,\R^3}\\
&\qquad \times \sum_{i=1}^3
\|(1+|x|)^{\alpha}(1+|x|-x_1)^{\beta}S_a(t)
\partial_{y_i}(\varphi_{i1},\varphi_{i2},\varphi_{i3})\|_{q',\R^3}
\end{align*}
for $t>0,\varphi=\{\varphi_{ij}\}_{i,j=1,2,3}
\in C_0^\infty(\R^3)^{3\times 3},$ 
where $1/q'=1-1/q$ and  
$(\cdot,\cdot)_{\R^3}$ denotes inner product on $\R^3$. 
Therefore, we conclude the assertions  
by applying Lemma \ref{lemlosslessr3} with $q=r',r=q'$ and by 
$L^{r'}_{(1+|x|)^{\alpha r'}
(1+|x|-x_1)^{\beta r'}}(\R^3)^*
=L^r_{(1+|x|)^{-\alpha r}(1+|x|-x_1)^{-\beta r}}(\R^3)$. The proof is 
complete.
\end{proof}

\begin{rmk}
\par Concerning the relation between (\ref{lqlrr3dual2}) 
(resp. (\ref{lqlrr3dual3}) with $\beta=0,a>0$)
and (\ref{lqlrr3dual4}) 
(resp. (\ref{lqlrr3dual3}) with $\beta=0,a=0$), we can find that 
(\ref{lqlrr3dual4}) and (\ref{lqlrr3dual3}) with $\beta=0,a=0$, 
that is 
\begin{align}
\|(1+|x|)^{-\alpha}\nabla^i S_{0}(t)P_{\R^3}f\|_{\infty,\R^3}
\leq Ct^{-\frac{3}{2q}-\frac{i}{2}}(1+t)^{\frac{\alpha}{2}}
\|(1+|x|)^{-\alpha}f\|_{q,\R^3}
\label{lqlrr3dual5}
\end{align}
are recovered by taking the limit $a\rightarrow 0$. 
In fact, by applying (\ref{g2est0}), (\ref{g1est}) with $\beta=0$, 
(\ref{interpo0}) and (\ref{g2est2}) with $\beta=0$ to 
\begin{align*}
&(1+|x|)^\alpha|S_a(t)g(x)|
\leq C\big\{G_{1,0}*\big((1+|y|)^\alpha|g|\big)(x)
+G_{2,0}*|g|(x)\big\}\\
&(1+|x|)^\alpha|S_a(t)
\partial_{y_j}g(x)|\leq C\big\{G_{1,1}*\big((1+|y|)^\alpha|g|\big)(x)
+G_{2,1}*|g|(x)\big\}, 
\end{align*} 
it holds that 
\begin{align*}
&\|(1+|x|)^\alpha S_a(t)g\|_{r,\R^3}
\leq 
Ct^{-\frac{3}{2}(\frac{1}{q}-\frac{1}{r})}
\big\{t^{-\frac{\alpha}{2}}(1+at)^{\alpha}+1\big\}
\|(1+|x|)^\alpha g\|_{q,\R^3},\\
&\|(1+|x|)^\alpha S_a(t)\partial_{y_j}g\|_{r,\R^3}\leq 
Ct^{-\frac{3}{2}(\frac{1}{q}-\frac{1}{r})-\frac{1}{2}}
\big\{t^{-\frac{\alpha}{2}}(1+at)^{\alpha}+1\big\}
\|(1+|x|)^\alpha g\|_{q,\R^3},\\
&\|(1+|x|)^\alpha S_a(t)g\|_{r,\R^3}
\leq Ct^{-\frac{3}{2}(1-\frac{1}{r})}
\big\{(1+at)^{\alpha}+t^{\frac{\alpha}{2}}\big\}
\|(1+|x|)^\alpha g\|_{1,\R^3},\\
&\|(1+|x|)^\alpha S_a(t)\partial_{y_j}g\|_{r,\R^3}\leq 
Ct^{-\frac{3}{2}(1-\frac{1}{r})-\frac{1}{2}}
\big\{(1+at)^{\alpha}+t^{\frac{\alpha}{2}}\big\}
\|(1+|x|)^\alpha g\|_{1,\R^3}
\end{align*}
for $a>0,t\geq 1,g\in C_0^\infty(\R^3)^3$ 
and $j=1,2,3$ provided $1<q\leq r<\infty.$ 
Therefore, as in the proof of Proposition \ref{proplqlrr3dual}, 
the duality argument leads us to 
\begin{align}
&\|(1+|x|)^{-\alpha}\nabla^i S_{-a}(t)P_{\R^3}f\|_{r,\R^3}
\leq Ct^{-\frac{3}{2}(\frac{1}{q}-\frac{1}{r})-\frac{i}{2}}
\big\{t^{-\frac{\alpha}{2}}(1+at)^{\alpha}+1\big\}
\|(1+|x|)^{-\alpha}f\|_{q,\R^3}\label{lqlinftyr30}
\\
&\|(1+|x|)^{-\alpha}\nabla^i S_{-a}(t)P_{\R^3}f\|_{\infty,\R^3}
\leq Ct^{-\frac{3}{2q}-\frac{i}{2}}
\big\{(1+at)^{\alpha}+t^{\frac{\alpha}{2}}\big\}
\|(1+|x|)^{-\alpha}f\|_{q,\R^3}\label{lqlinftyr3}
\end{align}
for $a>0,t\geq 1,f\in L^{q}_{(1+|x|)^{-\alpha q}}(\R^3)$ 
provided $1<q\leq r<\infty$. 
We also note that it is impossible to 
get the decay rate $t^{-3/(2q)-i/2}$ in (\ref{lqlrr3dual5}), 
thus the $L^q$-$L^\infty$ estimates  
with the weight $(1+|x|)^{-\alpha}$ 
are essentially inhomogeneous, 
see Remark \ref{rmks0optimal} in Section 7.
\end{rmk}

\section{Decay estimates near the boundary}
\quad This section is devoted to the 
decay estimates of $e^{-tA_a}f$ near the boundary of $D$. 
Given $f\in L^q(D)$, Kobayashi and Shibata \cite{kosh1998} 
derived the decay estimate with the rate $t^{-3/(2q)}$, however, 
given $f\in L^q_{\rho}(D)$, 
we use the better spatial decay structure of $f$ at infinity 
to get the better decay rate $t^{-3/(2q)-\varepsilon}$, where 
$\rho$ is given by (\ref{rho}). 
We also derive the slower rate if $f\in L^q_{\rho_-}(D)$, 
where $\rho_-$ is given by (\ref{rhominus}). 
\par Let us start with the so-called 
local energy decay estimates derived by 
\cite{kosh1998}, which are the decay
estimates of $e^{-tA_a}f$ near the boundary 
when the initial data $f$ has compact support.

\begin{prop}[{\cite[Theorem 1.1]{kosh1998}}]\label{local}
\par Let $R>0$ such that $\R^3\setminus D\subset B_{R-1}(0)$. 
We set $D_R:=D\cap B_{R}(0)$. Let 
$1<q<\infty$, $a_0>0$ and assume $|a|\leq a_0$. Then 
there exists a constant $C>0$, independent of $a$, such that  
\begin{align*}
\|\partial_t e^{-tA_a}f\|_{q,D_R}
+\|e^{-tA_a}f\|_{W^{2,q}(D_R)}
\leq Ct^{-\frac{3}{2}}\|f\|_{q,D_R}
\end{align*}
for all $t\geq 1$ and 
$f\in \{f\in L_{\sigma}^q(D)\mid f(x)=0~{\rm for }~|x|\geq R\}$.
\end{prop}
\noindent Combining Proposition \ref{local} and 
the $L^q$-$L^r$ estimates of $S_a(t)$, 
see (\ref{oseensemir3}),   
implies the following estimates 
of $e^{-tA_a}f$ near the boundary when $f\in L^q_{\sigma}(D)$.  
\begin{prop}[{\cite[(6.18)]{kosh1998}}]\label{local2}
\par Let $R>0$ such that $\R^3\setminus D\subset B_{R-1}(0)$ 
and set $D_R=D\cap B_{R}(0)$. Let 
$1<q<\infty$, $a_0>0$ and assume $|a|\leq a_0$. Then 
there exists a constant $C>0$, independent of $a$, such that  
\begin{align*}
\|\partial_t e^{-tA_a}f\|_{q,D_R}
+\|e^{-tA_a}f\|_{W^{2,q}(D_R)}
\leq Ct^{-\frac{3}{2q}}\|f\|_{q,D}
\end{align*}
for all $t\geq 1$ and $f\in L_{\sigma}^q(D)$.
\end{prop}

Lemma \ref{holder} tells us that $f\in L^r(D)$ with 
some $r<q$ if $f\in L^q_{\rho}(D)$. By making use of this, 
we next derive the better decay rate 
$t^{-3/2q-\varepsilon}$ than the one in Proposition \ref{local2}. 
We also deduce the slower decay rate if $f\in L^q_{\rho_-}(D)$.
\begin{prop}\label{local3} 
Let $1<q<\infty.$ 
We take $R>0$ such that $\R^3\setminus D\subset B_{R-1}(0)$ 
and set $D_R=D\cap B_{R}(0)$. Fix $a_0>0$ and assume $a\in[0,a_0]$. 
\begin{enumerate}
\item Let $\alpha,\beta>0$ satisfy $\alpha+\beta<3(1-1/q)$ and let 
$s\in(\max\{3q/(3+\alpha q+\beta q),2q/(2+\alpha q)\},q]$. 
Then there exists a constant $C(D,a_0,q,s,\alpha,\beta)$ 
such that 
\begin{align}\label{loc20}
\|\partial_t e^{-tA_a}P_Df\|_{q,D_R}
+\|e^{-tA_a}P_Df\|_{W^{2,q}(D_R)}
\leq Ct^{-\frac{3}{2s}}\|(1+|x|)^\alpha
(1+|x|-x_1)^\beta f\|_{q,D}
\end{align}
for all $t\geq 3$ and $f\in L^q_{\rho}(D)$,
where $\rho$ is given by $($\ref{rho}$)$. 
\item Let $0\leq \alpha<3(1-1/q)$. 
Then there exists
a constant $C(D,a_0,q,\alpha)$ such that
\begin{align*}
\|\partial_t e^{-tA_a}P_Df\|_{q,D_R}
+\|e^{-tA_a}P_Df\|_{W^{2,q}(D_R)}
\leq Ct^{-\frac{3}{2q}-\frac{\alpha}{2}}
\|(1+|x|)^\alpha f\|_{q,D}
\end{align*}
for all $t\geq 3$ and $f\in L^q_{(1+|x|)^{\alpha q}}(D)$.
\item For $\alpha,\beta\geq 0$ satisfying $($\ref{alphabeta4}$)$, 
there exists a constant $C(D,a_0,q,\alpha,\beta)$ such that 
\begin{align}
\|&\partial_t
e^{-tA_{-a}}P_Df\|_{q,D_R}+\|e^{-tA_{-a}}P_Df\|_{W^{2,q}(D_R)}
\leq Ct^{-\frac{3}{2q}+\eta_4}
\|(1+|x|)^{-\alpha}(1+|x|-x_1)^{-\beta}f\|_{q,D}\label{loc2}
\end{align}
for all $t\geq 3$ and $f\in L^q_{\rho_-}(D)$, where
$\eta_4$ and $\rho_-$ are 
given by $($\ref{gammadeltaeta}$)$ and 
$($\ref{rhominus}$)$, respectively. 
\end{enumerate}
\end{prop}

\begin{proof}
In order to prove the assertions, we use the estimates
\begin{align}\label{lqlrr33}
\|\partial_x^kS_a(t)f\|_{r,\R^3}
\leq Ct^{-\frac{3}{2}(\frac{1}{q}-\frac{1}{r})-\frac{|k|}{2}}
\|f\|_{q,\R^3}
\end{align}
for $|a|\leq a_0$, $|k|\leq 3,f\in L^q_\sigma(\R^3),t>0,$ 
provided $1<q\leq r\leq \infty~(q\ne\infty)$, which are derived by 
the same calculation as in the proof of Proposition \ref{lqlrr3}. 
We first prove the assertion 1. Under the assumption in 
the assertion 1, Lemma \ref{holder} yields 
\begin{align}\label{holderf}
\|f\|_{s,D}
\leq\|(1+|x|)^{-\alpha}(1+|x|-x_1)^{-\beta}\|_{\frac{qs}{q-s},D}
\|(1+|x|)^{\alpha}(1+|x|-x_1)^{\beta}f\|_{q,D}
\end{align}
for $f\in L^q_{\rho}(D).$ Let $\zeta$ 
be a function on $\R^3$ satisfying 
\begin{align}\label{zetadef}
\zeta\in C^\infty(\R^3),\quad 
\zeta(x)=0\quad |x|\leq R-1,\quad  
\zeta(x)=1\quad |x|\geq R
\end{align}
and denote the Bogovski\u{\i} operator on 
$A_{R-1,R}=\{x\in\R^3\,;\,R-1<|x|<R\}$ by $\B_{A_{R-1,R}}$, 
see Bogovski\u{\i} \cite{bogovskii1979}, 
Borchers and Sohr \cite{bosh1990} and Galdi \cite{galdi2011}.  
Set 
\begin{align*}
g:=\zeta e^{-A_a}P_Df
-\B_{A_{R-1,R}}[(\nabla\zeta)\cdot e^{-A_a}P_Df]. 
\end{align*}
We then find $\nabla \cdot g=0$ in $\R^3$ and 
\begin{align}
&\|g\|_{W^{2,q}(\R^3)}\leq C\|e^{-A_a}P_Df\|_{W^{2,q}(D)}
\leq C\|f\|_{q,D}\leq
C\|(1+|x|)^{\alpha}(1+|x|-x_1)^{\beta}f\|_{q,D},\label{gest2}\\
&\|g\|_{s,\R^3}\leq 
C\|e^{-A_a}P_Df\|_{s,D}\leq C\|f\|_{s,D}\leq
C\|(1+|x|)^{\alpha}(1+|x|-x_1)^{\beta}f\|_{q,D}\label{gest}
\end{align}
due to (\ref{holderf}).  
In view of (\ref{lqlrr33}), (\ref{gest2}), (\ref{gest}) and 
the formula (\ref{oseensemir3}), we get 
\begin{align}\label{sag}
\|S_a(t)g\|_{W^{3,q}(D_R)}\leq Ct^{-\frac{1}{2}}
(1+t)^{\frac{1}{2}-\frac{3}{2s}}
\|(1+|x|)^{\alpha}(1+|x|-x_1)^{\beta}f\|_{q,D}
\end{align}
for all $t>0$. In fact, it holds that
\begin{align*}
\|S_a(t)g\|_{W^{3,q}(D_R)}
\leq C\|S_a(t)g\|_{W^{3,\infty}(\R^3)}\leq 
Ct^{-\frac{3}{2s}}\|g\|_{s,\R^3}
\leq Ct^{-\frac{3}{2s}}
\|(1+|x|)^{\alpha}(1+|x|-x_1)^{\beta}f\|_{q,D}
\end{align*}
for $t\geq 1$ and that 
\begin{align}
\|S_a(t)g\|_{W^{3,q}(D_R)}
&\leq
\|S_a(t)g\|_{W^{1,q}(D_R)}
+\sum_{i,j=1}^3
\|S_a(t)(\partial_{y_i}\partial_{y_j}g)\|_{W^{1,q}(D_R)}\notag\\
&\leq Ct^{-\frac{1}{2}}\|g\|_{W^{2,q}(\R^3)}
\leq Ct^{-\frac{1}{2}}
\|(1+|x|)^{\alpha}(1+|x|-x_1)^{\beta}f\|_{q,D}\label{sagsmall}
\end{align}
for $t<1$. We thus obtain (\ref{sag}). 
We take a function $\widetilde{\zeta}$ such that  
\begin{align*}
\widetilde{\zeta}\in C^\infty(\R^3),\quad \widetilde{\zeta}(x)=0
\quad |x|\leq R,\quad  
\widetilde{\zeta}(x)=1\quad |x|\geq R+1 
\end{align*} 
and set
\begin{align*}
v(t):=u(t)-\widetilde{\zeta}S_a(t)g+\B_{A_{R,R+1}}
[(\nabla \widetilde{\zeta})\cdot S_a(t)g],\quad 
u(t):=e^{-(t+1)A_a}P_Df,
\end{align*}
then the pair $(v(t),p)$, where $p(t)$ is the pressure 
associated with $u(t)$, obeys 
\begin{align*}
&\partial_tv-\Delta v+a\partial_{x_1}v+\nabla p=K,
\quad \nabla\cdot v=0,\quad x\in D,t>0\qquad v|_{\partial D}=0,
\quad t>0,\\
&v(x,0)=(1-\widetilde{\zeta})e^{-A_a}P_Df+\B_{A_{R,R+1}}
[(\nabla\widetilde{\zeta})\cdot e^{-A_a}P_Df],\quad x\in D,
\end{align*} 
where 
\begin{align*} 
K(t)&=(\partial_t-\Delta+a\partial_{x_1}) 
\{-\widetilde{\zeta}S_a(t)g 
+\B_{A_{R,R+1}}[(\nabla \widetilde{\zeta})\cdot S_a(t)g]\}\notag\\  
&=(\Delta \widetilde{\zeta})S_a(t)g 
+2(\nabla \widetilde{\zeta}\cdot\nabla)S_a(t)g 
-(a\partial_{x_1}\widetilde{\zeta})S_a(t)g 
+(\partial_t-\Delta+a\partial_{x_1}) 
\B_{A_{R,R+1}}[(\nabla\widetilde{\zeta})\cdot S_a(t)g].
\end{align*}
From the estimate of the Bogovski\u{\i} operator, 
$\partial_t=-\Delta+a\partial_{x_1}$ and (\ref{sag}), 
we have 
\begin{align}
\|K(t)\|_{W^{2,q}(D)}
&\leq C\|S_a(t)g\|_{W^{3,q}(D_{R+1})}+
C\|\B_{A_{R,R+1}}[(\nabla\widetilde{\zeta})\cdot 
\partial_tS_a(t)g]\|_{W^{2,q}(A_{R+1})}\nonumber\\
&\qquad+C\|\B_{A_{R,R+1}}[(\nabla\widetilde{\zeta})\cdot 
S_a(t)g]\|_{W^{4,q}(A_{R,R+1})}\nonumber\\
&\leq C\|S_a(t)g\|_{W^{3,q}(D_{R+1})}\leq 
Ct^{-\frac{1}{2}}(1+t)^{\frac{1}{2}-\frac{3}{2s}}
\|(1+|x|)^{\alpha}(1+|x|-x_1)^\beta f\|_{q,D}\label{lest2}
\end{align}
for all $t>0$. 
By Duhamel's principle, $v$ satisfies the integral equation 
\begin{align}\label{wintegral}
v(t)=e^{-tA_a}v(0)+\int_0^te^{-(t-\tau)A_a}K(\tau)\,d\tau.
\end{align}
Since supp~$v(0)\subset D_{R+1}$, applying
Proposition \ref{local} leads to 
\begin{align}\label{localv0}
\|\partial_te^{-tA_a}v(0)\|_{q,D_{R+1}}
+\|e^{-tA_a}v(0)\|_{W^{2,q}(D_{R+1})}
&\leq Ct^{-\frac{3}{2}}\|v(0)\|_{q,D_{R+1}}\nonumber\\
&\leq Ct^{-\frac{3}{2}}
\|(1+|x|)^{\alpha}(1+|x|-x_1)^\beta f\|_{q,D}
\end{align}
for all $t\geq 1$. On the other hand, 
due to Proposition \ref{local} and 
\begin{align*}
\|&\partial_te^{-(t-\tau)A_a}K(\tau)\|_{q,D_{R+1}}+
\|e^{-(t-\tau)A_a}K(\tau)\|_{W^{2,q}(D_{R+1})}
\\&\leq C(\|e^{-(t-\tau)A_a}A_aK(\tau)\|_{q,D}
+\|K(\tau)\|_{q,D})
\leq C\|K(\tau)\|_{W^{2,q}(D)}
\end{align*}
for $0<\tau<t$ with $t-\tau<1,$ we have  
\begin{align*}
\|\partial_te^{-(t-\tau)A_a}K(\tau)\|_{q,D_{R+1}}+
\|e^{-(t-\tau)A_a}K(\tau)\|_{W^{2,q}(D_{R+1})}
\leq C(1+t-\tau)^{-\frac{3}{2}}\|K(\tau)\|_{W^{2,q}(D)}
\end{align*}
for all $0<\tau<t$, which combined with (\ref{lest2}) yields  
\begin{align}
\int_0^t&\|\partial_te^{-(t-\tau)A_a}K(\tau)\|_{q,D_{R+1}}
\,d\tau+
\int_0^t\|e^{-(t-\tau)A_a}K(\tau)\|_{W^{2,q}(D_{R+1})}\,d\tau
\nonumber\\
&\leq C\int_0^t(1+t-\tau)^{-\frac{3}{2}}\tau^{-\frac{1}{2}}
(1+\tau)^{\frac{1}{2}-\frac{3}{2s}}\,d\tau\,
\|(1+|x|)^{\alpha}(1+|x|-x_1)^\beta f\|_{q,D}\nonumber\\
&\leq C\left\{t^{-\frac{3}{2}}\int_0^1\tau^{-\frac{1}{2}}
(1+\tau)^{\frac{1}{2}-\frac{3}{2s}}\,d\tau+
\left(1+\frac{t}{2}\right)^{-\frac{3}{2}}
\int_1^{\frac{t}{2}}(1+\tau)^{\frac{1}{2}-\frac{3}{2s}}\,d\tau
\right.\notag\\
&\left.\qquad\qquad +\left(\frac{t}{2}\right)^{-\frac{1}{2}}
(1+t)^{\frac{1}{2}}
\left(1+\frac{t}{2}\right)^{-\frac{3}{2s}}
\int_{\frac{t}{2}}^t(1+t-\tau)^{-\frac{3}{2}}\,d\tau\right\}
\|(1+|x|)^{\alpha}(1+|x|-x_1)^\beta f\|_{q,D}\nonumber\\
&\leq Ct^{-\frac{3}{2s}}
\|(1+|x|)^{\alpha}(1+|x|-x_1)^\beta f\|_{q,D}\label{drs}
\end{align}
for $t\geq 2$. Due to 
\begin{align*}
\partial_tv(t)=\partial_te^{-tA_a}v(0)+\int_0^t
\partial_te^{-(t-\tau)A_a}K(\tau)\,d\tau+K(t)
\end{align*} 
and $v|_{D_R}=u(t)$, 
collecting (\ref{lest2})--(\ref{drs}) 
completes the proof of the assertion 1. 
\par We next prove the assertion 2. 
We know from Lemma \ref{holder} that 
$f\in L^{(3q)/(3+\alpha q),q}(D)$ with 
$\|f\|_{L^{(3q)/(3+\alpha q),q}(D)}
\leq C\|(1+|x|)^{\alpha}f\|_{q,D}$ for 
$f\in L^q_{(1+|x|)^{\alpha q}}(D)$. 
Moreover, it holds that 
\begin{align*}
L^{\frac{3q}{3+\alpha q},q}(\Omega)=(L^{q_0}(\Omega),
L^{q_1}(\Omega))_{\theta,q},\qquad \Omega=\R^3~{\rm or }~D,
\end{align*} 
where $1<q_0<(3q)/(3+\alpha q)<q_1<\infty$ 
and $0<\theta<1$ satisfy 
$(3+\alpha q)/(3q)=(1-\theta)/q_0+\theta/q_1$.
From this fact  
and (\ref{lqlrr33}) with $r=\infty,$ we have 
\begin{align*}
%\|S_a(t)Pf\|_{L^{\frac{3q}{3+\alpha q},q}(\R^3)}
%\leq C\|f\|_{L^{\frac{3q}{3+\alpha q},q}(\R^3)},\quad
&\|\partial_x^kS_a(t)P_{\R^3}h\|_{\infty,\R^3}
\leq Ct^{-\frac{3}{2q}-\frac{\alpha}{2}-\frac{|k|}{2}}
\|h\|_{L^{\frac{3q}{3+\alpha q},q}(\R^3)}
\end{align*}
for all $t>0$, $|k|\leq 3$ and 
$h\in L^{3q/(3+\alpha q),q}(\R^3)$, 
which combined with (\ref{sagsmall}) with $\beta=0$ implies that 
(\ref{sag}) is replaced by 
\begin{align*}
\|S_a(t)g\|_{W^{3,q}(D_R)}\leq
Ct^{-\frac{1}{2}}(1+t)^{\frac{1}{2}-\frac{3}{2q}-\frac{\alpha}{2}}
\|(1+|x|)^{\alpha}f\|_{q,D}.
\end{align*}
By applying this estimate,  
we can obtain (\ref{lest2}) and (\ref{drs}) with 
$\beta=0,s=(3q)/(3+\alpha q)$. We thus conclude the assertion 2. 
\par To prove the assertion 3, as in the proof of the assertion 1, 
we set 
\begin{align*}
v_-(t)&:=u_-(t)-\widetilde{\zeta}S_{-a}(t)g_-+\B_{A_{R,R+1}}
[(\nabla \widetilde{\zeta})\cdot S_{-a}(t)g_-],\quad 
u_-(t):=e^{-(t+1)A_{-a}}P_Df,\\
g_-&:=\zeta e^{-A_{-a}}P_Df-
\B_{A_{R-1,R}}[(\nabla\zeta)\cdot e^{-A_{-a}}P_Df].
\end{align*}
Then the pair $(v_-,p_-),$ where $p_-(t)$ is 
the pressure associated with $u_-(t),$ obeys  
\begin{align*}
&\partial_tv_--\Delta v_--a\partial_{x_1}v_-+\nabla p_-=K_-,
\quad \nabla\cdot
v_-=0,\quad x\in D,t>0,\qquad v_-|_{\partial D}=0,\quad t>0,\\
&v_-(x,0)=(1-\widetilde{\zeta})e^{-A_{-a}}P_Df+
\B_{A_{R,R+1}}
[(\nabla\widetilde{\zeta})\cdot e^{-A_{-a}}P_Df],
\quad x\in D,
\end{align*} 
where 
\begin{align*}
K_-(t)=(\Delta \widetilde{\zeta})S_{-a}(t)g_- 
&+2(\nabla \widetilde{\zeta}\cdot\nabla)S_{-a}(t)g_- 
+(a\partial_{x_1}\widetilde{\zeta})S_{-a}(t)g_-\\
&+(\partial_t-\Delta-a\partial_{x_1}) 
\B_{A_{R,R+1}}[(\nabla\widetilde{\zeta})\cdot S_{-a}(t)g_-].
\end{align*}
Therefore, it suffices to 
consider the integral equation 
\begin{align}\label{wintegral2}
v_-(t)=e^{-tA_{-a}}v_-(0)+\int_0^te^{-(t-\tau)A_{-a}}K_-(\tau)
\,d\tau.
\end{align}
Proposition \ref{local} yields 
\begin{align}
\|\partial_te^{-tA_{-a}}v_-(0)\|_{q,D_{R+1}}
+&\|e^{-tA_{-a}}v_-(0)\|_{W^{2,q}(D_{R+1})}\notag\\
\leq Ct^{-\frac{3}{2}}\|v_-(0)\|_{q,D_{R+1}}
&\leq Ct^{-\frac{3}{2}}\|(1+|x|)^{-\alpha}
(1+|x|-x_1)^{-\beta}e^{-A_{-a}}P_Df\|_{q,D}\notag\\
&\leq Ct^{-\frac{3}{2}}
\|(1+|x|)^{-\alpha}(1+|x|-x_1)^{-\beta}f\|_{q,D}\label{localv02}
\end{align}
for all $t\geq 1$. By (\ref{oseensemir3}) and
(\ref{lqlrr3dual3}), we have 
\begin{align}
\|S_{-a}(t)g_-\|_{W^{3,q}(D_{R+1})}&\leq 
\|S_{-a}(t)g_-\|_{W^{1,q}(D_{R+1})}+\sum_{|k|=2}
\|S_{-a}(t)\partial^k_{y}g_-\|_{W^{1,q}(D_{R+1})}
\notag\\
&\leq C\|(1+|x|)^{-\alpha}(1+|x|-x_1)^{-\beta}
S_{-a}(t)g_-\|_{q,\R^3}\notag\\
&\qquad+C
\|(1+|x|)^{-\alpha}(1+|x|-x_1)^{-\beta}
\nabla S_{-a}(t)g_-\|_{q,\R^3}\notag\\
&\qquad+C\sum_{|k|=2}
\|(1+|x|)^{-\alpha}(1+|x|-x_1)^{-\beta}
S_{-a}(t)\partial^k_{y}g_-\|_{q,\R^3}\notag\\
&\qquad+C\sum_{|k|=2}
\|(1+|x|)^{-\alpha}(1+|x|-x_1)^{-\beta}
\nabla S_{-a}(t)\partial^k_{y}g_-\|_{q,\R^3}\notag\\
&\leq Ct^{-\frac{1}{2}}\sum_{i=0,2}
\|(1+|x|)^{-\alpha}(1+|x|-x_1)^{-\beta}\nabla^ig_-\|_{q,\R^3}
\label{sminusa1}
\end{align}
for $t\leq 1$ and 
\begin{align}
\|S_{-a}(t)g_-\|_{W^{3,q}(D_{R+1})}
&\leq \|S_{-a}(t)g_-\|_{W^{3,\infty}(D_{R+1})}\notag\\
&\leq C\|(1+|x|)^{-\alpha}(1+|x|-x_1)^{-\beta}
S_{-a}(t)g_-\|_{\infty,\R^3}\notag\\
&\qquad+C
\|(1+|x|)^{-\alpha}(1+|x|-x_1)^{-\beta}
\nabla S_{-a}(t)g_-\|_{\infty,\R^3}\notag\\
&\qquad+C\sum_{|k|=2}
\|(1+|x|)^{-\alpha}(1+|x|-x_1)^{-\beta}
S_{-a}(t)\partial^k_{y}g_-\|_{\infty,\R^3}\notag\\
&\qquad+C\sum_{|k|=2}
\|(1+|x|)^{-\alpha}(1+|x|-x_1)^{-\beta}
\nabla S_{-a}(t)\partial^k_{y}g_-\|_{\infty,\R^3}
\notag\\
&\leq Ct^{-\frac{3}{2q}}(1+t)^{\eta_4}
\sum_{i=0,2}
\|(1+|x|)^{-\alpha}(1+|x|-x_1)^{-\beta}\nabla^ig_-\|_{q,\R^3}
\label{sminusa2}
\end{align}
for $t\geq 1$. Moreover, due to (\ref{lqlqsmall}), 
(\ref{lqlqsmall2}) and (\ref{nablalqlqsmall}), 
we have 
\begin{align}
\|(1+|x|)^{-\alpha}(1+|x|-x_1)^{-\beta}\partial_x^k g_-\|_{q,\R^3}
&\leq C\sum_{i=0}^2
\|(1+|x|)^{-\alpha}(1+|x|-x_1)^{-\beta}\nabla^i
e^{-A_{-a}}P_Df\|_{q,D}\notag\\
&\qquad+C\|e^{-A_{-a}}P_Df\|_{W^{1,q}(A_{R-1,R})}\notag\\
&\leq C\sum_{i=0}^2
\|(1+|x|)^{-\alpha}(1+|x|-x_1)^{-\beta}\nabla^i
e^{-A_{-a}}P_Df\|_{q,D}\notag\\
&\leq C\|(1+|x|)^{-\alpha}(1+|x|-x_1)^{-\beta}f\|_{q,D}\notag
\end{align}
for $|k|\leq 2$, which together with 
(\ref{sminusa1})--(\ref{sminusa2}) yields 
\begin{align*}
\|S_{-a}(t)g_-\|_{W^{3,q}(D_{R+1})}\leq Ct^{-\frac{1}{2}}
(1+t)^{\frac{1}{2}-\frac{3}{2q}+\eta_4}
\|(1+|x|)^{-\alpha}(1+|x|-x_1)^{-\beta}f\|_{q,D}
\end{align*} 
for $t>0$, thereby,
\begin{align}
\|K_-(t)\|_{W^{2,q}(D)}\leq C\|S_{-a}(t)g_-\|_{W^{3,q}(D_{R+1})}
\leq Ct^{-\frac{1}{2}}
(1+t)^{\frac{1}{2}-\frac{3}{2q}+\eta_4}
\|(1+|x|)^{-\alpha}(1+|x|-x_1)^{-\beta}f\|_{q,D}.
\label{lminus}
\end{align}
for $t>0$. 
With (\ref{lminus}) at hand, 
we carry out the same calculation as in (\ref{drs}) to obtain 
\begin{align}
\int_0^t&\|\partial_te^{-(t-\tau)A_{-a}}K_-(\tau)\|_{q,D_{R+1}}
\,d\tau+
\int_0^t\|e^{-(t-\tau)A_{-a}}K_-(\tau)\|_{W^{2,q}(D_{R+1})}\,d\tau
\nonumber\\
&\leq C(1+t)^{\eta_4}
\int_0^t(1+t-\tau)^{-\frac{3}{2}}\tau^{-\frac{1}{2}}
(1+\tau)^{\frac{1}{2}-\frac{3}{2q}}\,d\tau\,
\|(1+|x|)^{-\alpha}(1+|x|-x_1)^{-\beta} f\|_{q,D}\nonumber\\
&\leq Ct^{-\frac{3}{2q}}(1+t)^{\eta_4}
\|(1+|x|)^{-\alpha}(1+|x|-x_1)^{-\beta}f\|_{q,D}\notag
\end{align}
for $t\geq 2$, which combined with 
(\ref{wintegral2})--(\ref{localv02}) and (\ref{lminus}) leads 
to the assertion 3. The proof is complete. 
\end{proof}

\section{Proof of Theorem \ref{lqlrd}, 
\ref{thmdual0}, \ref{lqlrd2} and \ref{thmoptimal}}

\quad This section is devoted to the anisotropically weighted  
$L^q$-$L^r$ decay estimates of the Oseen semigroup 
in the exterior domain. We begin by proving 
Theorem \ref{lqlrd} and \ref{thmdual0}.\\
\noindent{\bf Proof of Theorem \ref{lqlrd} and \ref{thmdual0}.} 
\quad To prove Theorem \ref{lqlrd}, 
we derive the estimate on $\R^3\setminus B_R(0)$.
Let $\zeta$ be a function on $\R^3$ satisfying (\ref{zetadef}). 
Given $f\in\displaystyle\bigcap^3_{i=1}
L^{q_i}_{\rho_i}(D)\cap L^{q_4}(D)$, we define 
\begin{align}\label{uvpdef}
u(t):=e^{-(t+1)A_a}P_Df,\quad
w(t):=\zeta u(t)-\B_{A_{R-1,R}}
[\nabla \zeta\cdot u(t)],\quad \pi(t):=\zeta p(t),
\end{align}
where $\B_{A_{R-1,R}}$ denotes the 
Bogovski\u{\i} operator on $A_{R-1,R}=\{x\in\R^3\,;\,R-1<|x|<R\}$ 
and $p(t)$ is the pressure associated 
with $u(t)$ that satisfies
\begin{align}\label{p}
\int_{A_{R-1,R}}p(t)\,dx=0
\end{align}
for all $t>0$. Then the pair $(w,\pi)$ fulfills 
\begin{align*}
&\partial_t w-\Delta w+a\partial_{x_1} w+\nabla \pi=L,\quad  
\nabla \cdot w=0,\quad x\in\R^3,t>0,\\
&w(x,0)=\zeta e^{-A_a}P_Df-\B_{A_{R-1,R}}
[\nabla \zeta\cdot e^{-A_a}P_Df], \quad x\in\R^3,
\end{align*}
where 
\begin{align}\label{kdef}
L(x,t)=-2(\nabla\zeta\cdot \nabla)u
-(\Delta\zeta)u+a(\partial_{x_1}\zeta)u
-(\partial_t-\Delta+a\partial_{x_1})\B_{A_{R-1,R}}
[\nabla \zeta\cdot u(t)]
+(\nabla \zeta)p,
\end{align}
thereby,  Duhamel's principle implies 
\begin{align}\label{duhamelv}
w(t)=S_a(t)w(0)+\int_0^tS_a(t-\tau)P_{\R^3}L(\tau)\,d\tau.
\end{align}
Here, $S_a(t)$ is the Oseen semigroup in $\R^3$, 
see (\ref{oseensemir3}). 
Since $e^{-A_a}:L^{q_i}_{\rho_i,\sigma}(D)\rightarrow 
L^{q_i}_{\rho_i,\sigma}(D)$ is bounded for $i=1,2,3,4$, 
it follows from (\ref{lqlrr3large}) that 
\begin{align}
\|(1+|x|)^\alpha(1+|x|-x_1)^\beta S_a(t)w(0)\|_{r,\R^3}
\leq C\sum_{i=1}^4
t^{-\frac{3}{2}(\frac{1}{q_i}-\frac{1}{r})+\eta_i}
\|(1+|x|)^{\gamma_i}(1+|x|-x_1)^{\delta_i} f\|_{q_i,D}
\label{salqlr}
\end{align}
for $t>0$ and $f\in \displaystyle\bigcap^3_{i=1}
L^{q_i}_{\rho_i}(D)\cap L^{q_4}(D)$, 
where $\gamma_i,\delta_i,\eta_i$ are 
given by (\ref{gammadeltaeta}). 
By (\ref{p}) and Proposition \ref{local2}, 
the Poincar\'{e} inequality implies that 
$(\nabla\zeta)p(t)\in L^{\kappa}(\R^3)$ for all $t>0,$  
$\kappa=q_i~(i=1,2,3,4)$ and that 
\begin{align}
\|(\nabla\zeta)p(t)\|_{\kappa,\R^3}
\leq C\|p(t)\|_{\kappa,A_{R-1,R}}
&\leq C\|\nabla p(t)\|_{\kappa,A_{R-1,R}}\nonumber\\
&=C\|\partial_t u-\Delta u+a\partial_{x_1}u\|_{\kappa,A_{R-1,R}}
\leq C(1+t)^{-\frac{3}{2\kappa}}\|f\|_{\kappa,D}\label{pkappa}
\end{align}
for all $t>0$ and $\kappa=q_i~(i=1,2,3,4).$
Given $1<q_i<\infty~(i=1,2,3,4)$, $1<r\leq \infty$
and $\alpha,\beta$ 
subject to (\ref{qr0})--(\ref{alphabeta2}), 
we take $\lambda_i~(i=1,2,3,4)$ so that 
\begin{align*}
1<\lambda_i<\min\left\{\frac{3}{2},q_i\right\},\quad
\beta<1-\frac{1}{\lambda_1},\quad
\beta<1-\frac{1}{\lambda_2},\quad 
\alpha<3\left(1-\frac{1}{\lambda_3}\right),\quad 
\alpha+\beta<3\left(1-\frac{1}{\lambda_1}\right)
\end{align*} 
for $i=1,2,3,4.$ Then due to $L^{\lambda_i}$ boundedness of 
$\B_{A_{R-1,R}}$, 
(\ref{pkappa}), Proposition \ref{local2} 
and ${\rm supp}\,L(t)\subset A_{R-1,R}$, we find 
$L(t)\in L^{\lambda_i}_{(1+|x|)^{\gamma_i\lambda_i}
(1+|x|-x_1)^{\delta_i\lambda_i}}(\R^3)$ with  
%$1<\lambda\leq\max\{q_i\,;\,i=1,2,3,4\}$. 
\begin{align}
\|&(1+|x|)^{\gamma_i}(1+|x|-x_1)^{\delta_i}
P_{\R^3}L(t)\|_{\lambda_i,\R^3}\notag\\
&\leq C\|(1+|x|)^{\gamma_i}(1+|x|-x_1)^{\delta_i}
L(t)\|_{\lambda_i,\R^3}\notag\\
&\leq C\|L(t)\|_{q_i,A_{R-1,R}}\leq C(1+t)^{-\frac{3}{2q_i}}
\|(1+|x|)^{\gamma_i}(1+|x|-x_1)^{\delta_i}f\|_{q_i,D}\label{lest}
\end{align}
for $t>0$ and $i=1,2,3,4$. Similarly, we have 
\begin{align}
\|(1+|x|)^{\alpha}(1+|x|-x_1)^{\beta}
P_{\R^3}L(t)\|_{q_1,\R^3}
\leq C(1+t)^{-\frac{3}{2q_1}}
\|(1+|x|)^{\alpha}(1+|x|-x_1)^{\beta}f\|_{q_1,D}\label{lest1}
\end{align}
\par For $t\geq 2$, 
the integral in (\ref{duhamelv}) is splitted into 
\begin{align}\label{split}
\int_0^t\|(1+|x|)^{\alpha}(1+|x|-x_1)^{\beta}
S_a(t-\tau)P_{\R^3}L(\tau)\|_{r,\R^3}\,d\tau=\int_0^{\frac{t}{2}}+
\int^{t-1}_{\frac{t}{2}}+\int_{t-1}^t.
\end{align}
By applying (\ref{lqlrr3large}) and (\ref{lest}), we see 
\begin{align*}
\int_0^{\frac{t}{2}}
&\leq C\int_0^{\frac{t}{2}}\Big[
\sum_{i=1}^4(t-\tau)^{-\frac{3}{2}
(\frac{1}{\lambda_i}-\frac{1}{r})+\eta_i}
(1+\tau)^{-\frac{3}{2q_i}}\,
\|(1+|x|)^{\gamma_i}(1+|x|-x_1)^{\delta_i} f\|_{q_i,D}\Big]\,d\tau\\
&\leq C\sum_{i=1}^4t^{-\frac{3}{2}
(\frac{1}{\lambda_i}-\frac{1}{r})+\eta_i}
\|(1+|x|)^{\gamma_i}(1+|x|-x_1)^{\delta_i} f\|_{q_i,D}
\int_0^{\frac{t}{2}}(1+\tau)^{-\frac{3}{2q_i}}\,d\tau.
\end{align*}
This combined with 
\begin{align*}
\int_0^{\frac{t}{2}}(1+\tau)^{-\frac{3}{2q_i}}\,d\tau\leq 
\begin{cases}
Ct^{-\frac{3}{2q_i}+1}&\quad{\rm if}~
\displaystyle q_i>\frac{3}{2},\\[10pt]
C\log t&\quad{\rm if}~\displaystyle q_i=\frac{3}{2},\\[10pt]
C&\quad{\rm if}~\displaystyle q_i<\frac{3}{2}
\end{cases}
\end{align*}
for $t\geq 2$ implies
\begin{align}
\int_0^{\frac{t}{2}}\leq C\sum_{i=1}^4t^{-\frac{3}{2}
(\frac{1}{q_i}-\frac{1}{r})+\eta_i}
\|(1+|x|)^{\gamma_i}(1+|x|-x_1)^{\delta_i}f\|_{q_i,D}
\label{0totover2}
\end{align}
for $t\geq 2$, $f\in \displaystyle\bigcap^3_{i=1}
L^{q_i}_{\rho_i}(D)\cap L^{q_4}(D)$. Similarly, we have 
\begin{align}
\int_{\frac{t}{2}}^{t-1}\leq C\sum_{i=1}^4t^{-\frac{3}{2}
(\frac{1}{q_i}-\frac{1}{r})+\eta_i}
\|(1+|x|)^{\gamma_i}(1+|x|-x_1)^{\delta_i} f\|_{q_i,D}
\label{tover2tot-1}
\end{align}
due to 
\begin{align*}
\int_{\frac{t}{2}}^{t-1}
&\leq C\int_{\frac{t}{2}}^{t-1}\Big[
\sum_{i=1}^4(t-\tau)^{-\frac{3}{2}
(\frac{1}{\lambda_i}-\frac{1}{r})+\eta_i}
(1+\tau)^{-\frac{3}{2q_i}}\,
\|(1+|x|)^{\gamma_i}(1+|x|-x_1)^{\delta_i} f\|_{q_i,D}\Big]\,d\tau\\
&\leq C\sum_{i=1}^4
t^{-\frac{3}{2q_i}+\eta_i}\|(1+|x|)^{\gamma_i}
(1+|x|-x_1)^{\delta_i} f\|_{q_i,D}
\int_{\frac{t}{2}}^{t-1}
(t-\tau)^{-\frac{3}{2}(\frac{1}{\lambda_i}
-\frac{1}{r})}\,d\tau\\
&\leq C\sum_{i=1}^4
t^{-\frac{3}{2}(\frac{1}{q_i}-\frac{1}{r})+\eta_i}
\|(1+|x|)^{\gamma_i}
(1+|x|-x_1)^{\delta_i} f\|_{q_i,D}
\end{align*}
for $t\geq 2$, $f\in \displaystyle\bigcap^3_{i=1}
L^{q_i}_{\rho_i}(D)\cap L^{q_4}(D)$. 
Moreover, under the assumption $1/q_1-1/r<2/3$, 
it follows from (\ref{lqlrr3small}) and (\ref{lest1}) that 
\begin{align}
\int_{t-1}^t&\leq C\int_{t-1}^{t}
(t-\tau)^{-\frac{3}{2}(\frac{1}{q_1}
-\frac{1}{r})}
(1+\tau)^{-\frac{3}{2q_1}}\,
\|(1+|x|)^{\alpha}(1+|x|-x_1)^{\beta} f\|_{q_1,D}\,d\tau
\nonumber\\
&\leq Ct^{-\frac{3}{2q_1}} 
\|(1+|x|)^\alpha(1+|x|-x_1)^\beta f\|_{q_1,D}
\int_{t-1}^t(t-\tau)^{-\frac{3}{2}(\frac{1}{q_1}
-\frac{1}{r})}\,d\tau\nonumber\\
&\leq Ct^{-\frac{3}{2q_1}} 
\|(1+|x|)^\alpha(1+|x|-x_1)^\beta f\|_{q_1,D}
\leq Ct^{-\frac{3}{2}(\frac{1}{q_1}-\frac{1}{r})} 
\|(1+|x|)^\alpha(1+|x|-x_1)^\beta f\|_{q_1,D}.\label{t-1tot}
\end{align}
Collecting (\ref{duhamelv}), (\ref{salqlr}), 
(\ref{split})--(\ref{t-1tot}) together with 
$w|_{\R^3\setminus B_R(0)}=e^{-(t+1)A_a}P_Df$ yields
\begin{align*}
\|(1+|x|)^\alpha(1+|x|-x_1)^\beta
e^{-tA_a}P_Df\|_{r,\R^3\setminus B_R(0)}
\leq C\sum_{i=1}^4
t^{-\frac{3}{2}(\frac{1}{q_i}-\frac{1}{r})+\eta_i}
\|(1+|x|)^{\gamma_i}(1+|x|-x_1)^{\delta_i} f\|_{q_i,D}
\end{align*}
for $t\geq 3$, $f\in \displaystyle\bigcap^3_{i=1}
L^{q_i}_{\rho_i}(D)\cap L^{q_4}(D)$ if $1/q_1-1/r<2/3.$ 
On the other hand, if $1/q_1-1/r<1/3$, then we also have 
\begin{align}\label{lqlrdr}
\|&(1+|x|)^\alpha(1+|x|-x_1)^\beta
e^{-tA_a}P_Df\|_{r,D_R}\notag\\
&\leq C\|e^{-tA_a}P_Df\|_{W^{1,q_1}(D_R)}\leq 
Ct^{-\frac{3}{2q_1}}\|f\|_{q_1,D}\leq 
Ct^{-\frac{3}{2}(\frac{1}{q_1}-\frac{1}{r})}
\|(1+|x|)^{\alpha}(1+|x|-x_1)^{\beta} f\|_{q_1,D}
\end{align}  
by the Sobolev embedding
and Proposition \ref{local2}, thus
(\ref{lqlrdlarge}) holds if $1/q_1-1/r<1/3$. 
To eliminate the restriction $1/q_1-1/r<1/3$, 
we take $\{s_i\}^j_{i=0}$ 
such that $q_1=s_0\leq s_1\leq \cdots\leq s_j=r,~ 
1/s_{i-1}-1/s_i<1/3$ for $i=1,\cdots,j.$ Then 
\begin{align*}
\|&(1+|x|)^\alpha(1+|x|-x_1)^\beta e^{-tA_a}P_Df\|_{r,D}\\
&\leq 
Ct^{-\frac{3}{2}(\frac{1}{s_{j-1}}-\frac{1}{r})}
\|(1+|x|)^\alpha(1+|x|-x_1)^\beta
e^{-\frac{(j-1)}{j}tA_a}P_Df\|_{s_{j-1},D}\nonumber\\
&\qquad+C
\sum_{i=2}^4t^{-\frac{3}{2}(\frac{1}{q_i}-\frac{1}{r})
+\eta_i}
\|(1+|x|)^{\gamma_i}(1+|x|-x_1)^{\delta_i}
e^{-\frac{(j-1)}{j}tA_a}P_Df\|_{q_i,D}\\
&\leq Ct^{-\frac{3}{2}(\frac{1}{s_{j-2}}-\frac{1}{r})}
\|(1+|x|)^\alpha(1+|x|-x_1)^\beta e^{-\frac{(j-2)}{j}tA_a}P_Df
\|_{s_{j-2},D}\\
&\qquad+C
\sum_{i=2}^4t^{-\frac{3}{2}(\frac{1}{q_i}-\frac{1}{r})+\eta_i}
\|(1+|x|)^{\gamma_i}(1+|x|-x_1)^{\delta_i}
e^{-\frac{(j-2)}{j}tA_a}P_Df\|_{q_i,D}\\
&\qquad+C\sum_{i=2}^4t^{-\frac{3}{2}(\frac{1}{q_i}
-\frac{1}{r})+\eta_i}
\|(1+|x|)^{\gamma_i}(1+|x|-x_1)^{\delta_i}f\|_{q_i,D}\\
&\leq \cdots\leq C\sum_{i=1}^4t^{-\frac{3}{2}(\frac{1}{q_i}
-\frac{1}{r})+\eta_i}
\|(1+|x|)^{\gamma_i}(1+|x|-x_1)^{\delta_i}f\|_{q_i,D}.
\end{align*}
The proof of the assertion 1 of Theorem \ref{lqlrd} is complete. 
\par Let $\alpha,\beta>0$ and $1<q_4\leq q_i\leq r~(i=2,3)$ satisfy 
\begin{align}\label{alphabetaqr}
\alpha<\min\left\{3\left(1-\frac{1}{q_3}\right),1\right\},
\quad \beta<\min\left\{1-\frac{1}{q_2},
\frac{1}{3}\right\},\quad 
\alpha+\beta<\min\left\{3\left(1-\frac{1}{r}\right),1\right\}
\end{align}
and we also suppose 
\begin{align}
&1<q_4\leq q_i\leq r<\min\left\{\frac{3}{1-\alpha-\beta},
\frac{3}{1-\frac{3\alpha}{2}}\right\}\quad(i=2,3)\\
&\left({\rm resp.}~1<q_4\leq q_i\leq r
<\frac{3}{1-\alpha-\beta}\quad (i=2,3)\right)\label{alphabetaqr2}
\end{align}
if $\alpha<2/3$ (resp. $\alpha\geq 2/3)$. 
In view of the semigroup property and the assertion 1, 
to prove the assetion 2, it is enough to derive 
\begin{align}
\|&(1+|x|)^\alpha(1+|x|-x_1)^\beta\nabla e^{-tA_a}P_Df\|_{r,D}
\nonumber\\
&\leq Ct^{-\frac{1}{2}}
\|(1+|x|)^{\alpha}(1+|x|-x_1)^{\beta} f\|_{r,D}
+C\sum_{i=2}^4
t^{-\frac{3}{2}(\frac{1}{q_i}-\frac{1}{r})-\frac{1}{2}+\eta_i}
\|(1+|x|)^{\gamma_i}(1+|x|-x_1)^{\delta_i} f\|_{q_i,D}
\label{gradlrlr2}
\end{align}
for $t\geq 3$ and $f\in \displaystyle\bigcap^3_{i=2}
L^{q_i}_{\rho_i}(D)\cap L^{r}_{\widetilde{\rho}}(D)\cap L^{q_4}(D)$, 
where 
\begin{align}\label{rhotilde} 
\widetilde{\rho}=(1+|x|)^{\alpha r}(1+|x|-x_1)^{\beta r}. 
\end{align} 
Under the conditions 
(\ref{alphabetaqr})--(\ref{alphabetaqr2}), we find 
\begin{align*} 
1<\max\left\{\frac{3r}{3+\alpha r+\beta r}, 
\frac{2r}{2+\alpha r}\right\}<\min\{3,r\}, 
\end{align*} 
thus take $s_0$ so that 
\begin{align*} 
\max\left\{\frac{3r}{3+\alpha r+\beta r}, 
\frac{2r}{2+\alpha r}\right\}<s_0<\min\{3,r\}. 
\end{align*} 
It follows from (\ref{loc20}) with $s=s_0,q=r$ that 
\begin{align} 
\|(1+|x|)^{\alpha}(1+|x|-x_1)^{\beta}\nabla e^{-tA_a}P_Df\|_{r,D_R}
&\leq Ct^{-\frac{3}{2s_0}}
\|(1+|x|)^{\alpha}(1+|x|-x_1)^{\beta} f\|_{r,D}\notag\\
&\leq Ct^{-\frac{1}{2}}
\|(1+|x|)^{\alpha}(1+|x|-x_1)^{\beta} f\|_{r,D}\label{graddr}
\end{align}
for all $t\geq 3$ and $f\in L^r_{\widetilde{\rho}}(D)$. We use 
\begin{align}\label{duhamelgradv}
\nabla w(t)=\nabla S_a(t)w(0)
+\int_0^t\nabla S_a(t-\tau)P_{\R^3}L(\tau)\,d\tau
\end{align}
to derive the estimate on $\R^3\setminus B_R(0)$, 
see (\ref{duhamelv}). 
Applying (\ref{lqlrr3large}) leads us to  
\begin{align}
\|&(1+|x|)^\alpha(1+|x|-x_1)^\beta\nabla S_a(t)w(0)\|_{r,\R^3}
\nonumber\\
&\leq Ct^{-\frac{1}{2}}
\|(1+|x|)^{\alpha}(1+|x|-x_1)^{\beta} f\|_{r,D}
+C\sum_{i=2}^4
t^{-\frac{3}{2}(\frac{1}{q_i}-\frac{1}{r})-\frac{1}{2}+\eta_i}
\|(1+|x|)^{\gamma_i}(1+|x|-x_1)^{\delta_i} f\|_{q_i,D}
\label{salqlr2}
\end{align}
for $t\geq 1$ 
and $f\in \displaystyle\bigcap^3_{i=2}
L^{q_i}_{\rho_i}(D)\cap L^{r}_{\widetilde{\rho}}(D)
\cap L^{q_4}(D)$. 
Let $\{\widetilde{\lambda}_i\}_{i=1}^4$ 
satisfy   
\begin{align*}
&1<\widetilde{\lambda}_1<\min\left\{\frac{3}{2},r\right\}, \quad
1<\widetilde{\lambda}_i<\min\left\{\frac{3}{2},q_i\right\}\quad 
{\rm for}~i=2,3,4,\quad
 \frac{1}{\widetilde{\lambda}_i}-\frac{1}{r}\ne\frac{1}{3}
\quad{\rm for}~i=1,2,3,4,\\
&\beta<1-\frac{1}{\widetilde{\lambda}_1},\quad
\beta<1-\frac{1}{\widetilde{\lambda}_2},\quad 
\alpha<3\left(1-\frac{1}{\widetilde{\lambda}_3}\right),\quad 
\alpha+\beta<3\left(1-\frac{1}{\widetilde{\lambda}_1}\right).
\end{align*}
Then by (\ref{loc20}), we get 
\begin{align*}
&\|(\nabla\zeta)p(t)\|_{\widetilde{\lambda}_1,\R^3},\|(\nabla\zeta)p(t)\|_{r,\R^3}
\leq C\|\nabla p(t)\|_{r,A_{R-1,R}}
\leq C(1+t)^{-\frac{3}{2s_0}}
\|(1+|x|)^\alpha(1+|x|-x_1)^{\beta}f\|_{r,D}
\end{align*}
for $t>0$, see the calculation as in (\ref{pkappa}). 
From this estimate and (\ref{pkappa}) with $\kappa=q_i~(i=2,3,4)$, 
we carry out the same calculation as in (\ref{lest}) to get 
\begin{align*}
\|&(1+|x|)^{\alpha}(1+|x|-x_1)^{\beta}
P_{\R^3}L(t)\|_{\widetilde{\lambda}_1,\R^3}\leq 
C(1+t)^{-\frac{3}{2s_0}}\|(1+|x|)^\alpha(1+|x|-x_1)^{\beta}f\|_{r,D},
\\
\|&(1+|x|)^{\alpha}(1+|x|-x_1)^{\beta}
P_{\R^3}L(t)\|_{r,\R^3}\leq 
C(1+t)^{-\frac{3}{2s_0}}\|(1+|x|)^\alpha(1+|x|-x_1)^{\beta}f\|_{r,D},
\\
\|&(1+|x|)^{\gamma_i}(1+|x|-x_1)^{\delta_i}
P_{\R^3}L(t)\|_{\widetilde{\lambda}_i,\R^3}
\leq C(1+t)^{-\frac{3}{2q_i}}
\|(1+|x|)^{\gamma_i}(1+|x|-x_1)^{\delta_i}f\|_{q_i,D}
\end{align*}
for $t>0,i=2,3,4$. These estimates together with 
(\ref{lqlrr3small}), (\ref{lqlrr3large}) assert
\begin{align}
\int_0^t&\|(1+|x|)^{\alpha}(1+|x|-x_1)^{\beta}
\nabla
 S_a(t-\tau)P_{\R^3}L(\tau)\|_{r,\R^3}\,d\tau=\int_0^{\frac{t}{2}}+
\int^{t-1}_{\frac{t}{2}}+\int_{t-1}^t,\label{gradsplit}\\
\int_0^{\frac{t}{2}}
&\leq Ct^{-\frac{3}{2}
(\frac{1}{\widetilde{\lambda}_1}-\frac{1}{r})-\frac{1}{2}}
\|(1+|x|)^{\alpha}(1+|x|-x_1)^{\beta} f\|_{r,D}
\int_0^{\frac{t}{2}}(1+\tau)^{-\frac{3}{2s_0}}\,d\tau\nonumber\\
&\qquad+C\sum_{i=2}^4t^{-\frac{3}{2}
(\frac{1}{\widetilde{\lambda}_i}-\frac{1}{r})-\frac{1}{2}+\eta_i}
\|(1+|x|)^{\gamma_i}(1+|x|-x_1)^{\delta_i} f\|_{q_i,D}
\int_0^{\frac{t}{2}}(1+\tau)^{-\frac{3}{2q_i}}\,d\tau\nonumber\\
&\leq Ct^{-\frac{1}{2}}
\|(1+|x|)^{\alpha}(1+|x|-x_1)^{\beta} f\|_{r,D}
+C\sum_{i=2}^4
t^{-\frac{3}{2}(\frac{1}{q_i}-\frac{1}{r})-\frac{1}{2}+\eta_i}
\|(1+|x|)^{\gamma_i}(1+|x|-x_1)^{\delta_i} f\|_{q_i,D},
\label{grad0totover2}\\
\int_{\frac{t}{2}}^{t-1}&\leq Ct^{-\frac{3}{2s_0}}
\|(1+|x|)^{\alpha}(1+|x|-x_1)^{\beta} f\|_{r,D}
\int_{\frac{t}{2}}^{t-1}
(t-\tau)^{-\frac{3}{2}(\frac{1}{\widetilde{\lambda}_1}
-\frac{1}{r})-\frac{1}{2}}\,d\tau
\nonumber\\
&\qquad+C\sum_{i=2}^4
t^{-\frac{3}{2q_i}+\eta_i}\|(1+|x|)^{\gamma_i}
(1+|x|-x_1)^{\delta_i} f\|_{q_i,D}
\int_{\frac{t}{2}}^{t-1}
(t-\tau)^{-\frac{3}{2}(\frac{1}{\widetilde{\lambda}_i}
-\frac{1}{r})-\frac{1}{2}}\,d\tau\nonumber\\
&\leq Ct^{-\frac{1}{2}}
\|(1+|x|)^{\alpha}(1+|x|-x_1)^{\beta} f\|_{r,D}
+C\sum_{i=2}^4
t^{-\frac{3}{2}(\frac{1}{q_i}-\frac{1}{r})-\frac{1}{2}+\eta_i}
\|(1+|x|)^{\gamma_i}(1+|x|-x_1)^{\delta_i} f\|_{q_i,D}
\end{align}
and 
\begin{align}
\int_{t-1}^t
&\leq Ct^{-\frac{3}{2s_0}} 
\|(1+|x|)^\alpha(1+|x|-x_1)^\beta f\|_{r,D}
\int_{t-1}^t(t-\tau)^{-\frac{1}{2}}\,d\tau\nonumber\\
&\leq Ct^{-\frac{3}{2s_0}} 
\|(1+|x|)^\alpha(1+|x|-x_1)^\beta f\|_{r,D}
\leq Ct^{-\frac{1}{2}} 
\|(1+|x|)^\alpha(1+|x|-x_1)^\beta f\|_{r,D}\label{gradt-1tot}
\end{align} 
for  all $t\geq 2$ and $f\in \displaystyle\bigcap^3_{i=2}
L^{q_i}_{\rho_i}(D)\cap L^{r}_{\widetilde{\rho}}(D)
\cap L^{q_4}(D)$.
Collecting (\ref{graddr})--(\ref{gradt-1tot}), we obtain 
(\ref{gradlrlr2}), from which the assertion 2 follows. 
Similarly, given $q_2,r,\beta$ satisfying
$1<q_2\leq r\leq 3,0\leq \beta<\min\{1-1/r,1/3\}$, 
we have 
\begin{align*}
\|(1+|x|-x_1)^\beta\nabla e^{-tA_a}P_Df\|_{r,D}
\leq Ct^{-\frac{1}{2}}
\|(1+|x|-x_1)^{\beta} f\|_{r,D}
+Ct^{-\frac{3}{2}(\frac{1}{q_2}-\frac{1}{r})-\frac{1}{2}
+\frac{\beta}{2}}\|f\|_{q_2,D}
\end{align*}
for $t\geq 3,f\in L^r_{(1+|x|-x_1)^{\beta r}}(D)\cap L^{q_2}(D)$ 
by using Proposition \ref{local2} and  
\begin{align*}
&\|(1+|x|-x_1)^{\beta}
P_{\R^3}L(t)\|_{\widetilde{\lambda}_1,\R^3}
\leq C(1+t)^{-\frac{3}{2r}}\|
(1+|x|-x_1)^{\beta}f\|_{r,D},\\
&
\|(1+|x|-x_1)^{\beta}
P_{\R^3}L(t)\|_{r,\R^3}
\leq C(1+t)^{-\frac{3}{2r}}\|
(1+|x|-x_1)^{\beta}f\|_{r,D},\\
&\|P_{\R^3}L(t)\|_{\widetilde{\lambda}_2,\R^3}
\leq C(1+t)^{-\frac{3}{2q_2}}
\|f\|_{q_2,D},
\end{align*}
where we have taken $\widetilde{\lambda}_1,\widetilde{\lambda}_2$ 
so that 
\begin{align*}
1<\widetilde{\lambda}_1<\min\left\{\frac{3}{2},r\right\}, \quad
1<\widetilde{\lambda}_2<\min\left\{\frac{3}{2},q_2\right\},\quad 
\frac{1}{\widetilde{\lambda}_i}-\frac{1}{r}\ne\frac{1}{3},\quad 
\beta<1-\frac{1}{\widetilde{\lambda}_i} \quad (i=1,2).
\end{align*}
And also, given $q_2,r,\alpha$ satisfying   
$1<q_2\leq r\leq 3/(1-\alpha),0\leq\alpha<\min\{3(1-1/r),1\}$, 
we can obtain 
\begin{align*}
\|(1+|x|)^\alpha\nabla e^{-tA_a}P_Df\|_{r,D}
\leq Ct^{-\frac{1}{2}}
\|(1+|x|)^{\alpha} f\|_{r,D}
+Ct^{-\frac{3}{2}(\frac{1}{q_2}-\frac{1}{r})-\frac{1}{2}
+\eta_2}\|f\|_{q_2,D}
\end{align*}
by using the assertion 2 of Proposition \ref{local3} and
\begin{align*}
&\|(1+|x|)^{\alpha}
P_{\R^3}L(t)\|_{\widetilde{\lambda}_1,\R^3}
\leq C(1+t)^{-\frac{3}{2r}-\frac{\alpha}{2}}
\|(1+|x|)^{\alpha}f\|_{r,D},\\
&\|(1+|x|)^{\alpha}
P_{\R^3}L(t)\|_{r,\R^3}
\leq C(1+t)^{-\frac{3}{2r}-\frac{\alpha}{2}}
\|(1+|x|)^{\alpha}f\|_{r,D},\quad 
\|P_{\R^3}L(t)\|_{\widetilde{\lambda}_2,\R^3}
\leq C(1+t)^{-\frac{3}{2q_2}}\|f\|_{q_2,D},
\end{align*}
where $\eta_2$ is given by (\ref{gammadeltaeta}) and 
we have taken $\widetilde{\lambda}_1,\widetilde{\lambda}_2$ so that 
\begin{align*}
1<\widetilde{\lambda}_1<\min\left\{\frac{3}{2},r\right\}, \quad
1<\widetilde{\lambda}_2<\min\left\{\frac{3}{2},q_2\right\},\quad 
\frac{1}{\widetilde{\lambda}_i}-\frac{1}{r}\ne\frac{1}{3},\quad 
\alpha<3\left(1-\frac{1}{\widetilde{\lambda}_i}\right)\quad 
(i=1,2). 
\end{align*}
We thus conclude the assertion 3 and 4, 
which completes the proof of Theorem \ref{lqlrd}. 
\par Let us proceed to the proof of Theorem \ref{thmdual0}. 
Let $a\in(0,a_0]$. To prove the assertion 1, we define 
\begin{align}\label{uvpdef2}
u_-(t):=e^{-(t+1)A_{-a}}P_Df,\quad
w_-(t):=\zeta u_-(t)-\B_{A_{R-1,R}}
[\nabla \zeta\cdot u_-(t)],\quad \pi_-(t):=\zeta p_-(t),
\end{align}
where $\zeta$ is given by (\ref{zetadef}) and 
$p_-(t)$ is the pressure associated 
with $u_-(t)$ that satisfies
\begin{align}\label{p2}
\int_{A_{R-1,R}}p_-(t)\,dx=0
\end{align}
for all $t>0$. Then the pair $(w_-,\pi_-)$ fulfills 
\begin{align*}
&\partial_t w_--\Delta w_--a\partial_{x_1} w_-
+\nabla \pi_-=L_-,\quad  
\nabla \cdot w_-=0,\quad x\in\R^3,t>0,\\
&w_-(x,0)=\zeta e^{-A_{-a}}P_Df-\B_{A_{R-1,R}}
[\nabla \zeta\cdot e^{-A_{-a}}P_Df],
\quad x\in\R^3,
\end{align*}
where 
\begin{align*}
L_-(x,t)=-2(\nabla\zeta\cdot \nabla)u_-
-(\Delta\zeta)u_--a(\partial_{x_1}\zeta)u_-
-(\partial_t-\Delta-a\partial_{x_1})\B_{A_{R-1,R}}
[\nabla \zeta\cdot u_-(t)]
+(\nabla \zeta)p_-,
\end{align*}
thereby, it suffices to consider the integral equation 
\begin{align}\label{duhamelv2}
\nabla w_-(t)=\nabla S_{-a}(t)w_-(0)+
\int_0^t\nabla S_{-a}(t-\tau)P_{\R^3}L_-(\tau)\,d\tau.
\end{align}
Due to boundedness of
$e^{-A_{-a}}:L^{q}_{\rho_-,\sigma}(D)\rightarrow 
L^{q}_{\rho_-,\sigma}(D)$ and (\ref{lqlrr3dual3}), we have 
\begin{align}
\|&(1+|x|)^{-\alpha}(1+|x|-x_1)^{-\beta}\nabla 
S_{-a}(t)w_-(0)\|_{r,\R^3}
\notag\\&
\leq Ct^{-\frac{3}{2}(\frac{1}{q}-\frac{1}{r})-\frac{1}{2}}
(1+t)^{\alpha+\frac{\beta}{2}}
\|(1+|x|)^{-\alpha}(1+|x|-x_1)^{-\beta} f\|_{q,D}
\label{sminusalqlr}
\end{align}
for $f\in L^{q}_{\rho_-}(D),$ where 
$\rho_-$ is given by (\ref{rhominus}).  
Let $1<\nu<\min\{3/2,q\}$ satisfy $1/\nu-1/r\ne 1/3.$
From (\ref{loc2}), 
the same argument as in (\ref{pkappa}) implies that 
$(\nabla\zeta)p_-(t)\in L^{\kappa}(\R^3)$ for $t>0,$ 
$\kappa=\nu,q$ and that 
\begin{align*}
\|(\nabla\zeta)p_-(t)\|_{\kappa,\R^3}
\leq C(1+t)^{-\frac{3}{2q}+\alpha+\frac{\beta}{2}}
\|(1+|x|)^{-\alpha}(1+|x|-x_1)^{-\beta}f\|_{q,D}
\end{align*}
for all $t>0$ and $\kappa=\nu,q.$ We thus obtain 
$L_-(t)\in L^{\kappa}(\R^3)$ with the estimate
\begin{align}
\|P_{\R^3}L_-(t)\|_{\kappa,\R^3}
\leq C(1+t)^{-\frac{3}{2q}+\alpha+\frac{\beta}{2}}
\|(1+|x|)^{-\alpha}(1+|x|-x_1)^{-\beta}f\|_{q,D}\label{kminusest}
\end{align}
for $t>0$ and $\kappa=\nu,q$. If $t\geq 2$, then 
the integral in (\ref{duhamelv2}) is splitted into 
\begin{align}\label{split2}
\int_0^t&\|(1+|x|)^{-\alpha}(1+|x|-x_1)^{-\beta}
\nabla S_{-a}(t-\tau)P_{\R^3}L_-(\tau)\|_{r,\R^3}\,d\tau\notag\\
&\leq \int_0^t\|\nabla 
S_{-a}(t-\tau)P_{\R^3}L_-(\tau)\|_{r,\R^3}\,d\tau
=\int_0^{\frac{t}{2}}+
\int^{t-1}_{\frac{t}{2}}+\int_{t-1}^t.
\end{align}
We apply (\ref{lqlrr33}) and (\ref{kminusest}) to obtain 
\begin{align}
\int_0^{\frac{t}{2}}
&\leq C\int_0^{\frac{t}{2}}
(t-\tau)^{-\frac{3}{2}
(\frac{1}{\nu}-\frac{1}{r})-\frac{1}{2}}
(1+\tau)^{-\frac{3}{2q}+\alpha+\frac{\beta}{2}}\,
\|(1+|x|)^{-\alpha}(1+|x|-x_1)^{-\beta}f\|_{q,D}\,d\tau
\notag\\
&\leq Ct^{-\frac{3}{2}
(\frac{1}{\nu}-\frac{1}{r})-\frac{1}{2}}
\left(1+\frac{t}{2}\right)^{\alpha+\frac{\beta}{2}}
\|(1+|x|)^{-\alpha}(1+|x|-x_1)^{-\beta} f\|_{q,D}
\int_0^{\frac{t}{2}}(1+\tau)^{-\frac{3}{2q}}\,d\tau\notag\\
&\leq Ct^{-\frac{3}{2}
(\frac{1}{q}-\frac{1}{r})-\frac{1}{2}+\alpha+\frac{\beta}{2}}
\|(1+|x|)^{-\alpha}(1+|x|-x_1)^{-\beta} f\|_{q,D},
\label{0totover22}\\
\int_{\frac{t}{2}}^{t-1}
&\leq C\int_{\frac{t}{2}}^{t-1}(t-\tau)^{-\frac{3}{2}
(\frac{1}{\nu}-\frac{1}{r})-\frac{1}{2}}
(1+\tau)^{-\frac{3}{2q}+\alpha+\frac{\beta}{2}}\,
\|(1+|x|)^{-\alpha}(1+|x|-x_1)^{-\beta}f\|_{q,D}\,d\tau \notag\\
&\leq C\left(1+\frac{t}{2}\right)^{-\frac{3}{2q}}
t^{\alpha+\frac{\beta}{2}}\,
\|(1+|x|)^{-\alpha}(1+|x|-x_1)^{-\beta}f\|_{q,D}
\int_{\frac{t}{2}}^{t-1}
(t-\tau)^{-\frac{3}{2}(\frac{1}{\nu}
-\frac{1}{r})-\frac{1}{2}}\,d\tau\notag\\
&\leq Ct^{-\frac{3}{2}
(\frac{1}{q}-\frac{1}{r})-\frac{1}{2}+\alpha+\frac{\beta}{2}}
\|(1+|x|)^{-\alpha}(1+|x|-x_1)^{-\beta} f\|_{q,D}
\end{align}
and 
\begin{align}
\int_{t-1}^t&\leq C\int_{t-1}^{t}
(t-\tau)^{-\frac{3}{2}(\frac{1}{q}
-\frac{1}{r})-\frac{1}{2}}
(1+\tau)^{-\frac{3}{2q}+\alpha+\frac{\beta}{2}}\,
\|(1+|x|)^{-\alpha}(1+|x|-x_1)^{-\beta} f\|_{q,D}\,d\tau
\nonumber\\
&\leq Ct^{-\frac{3}{2q}}(1+t)^{\alpha+\frac{\beta}{2}} 
\|(1+|x|)^{-\alpha}(1+|x|-x_1)^{-\beta} f\|_{q,D}
\int_{t-1}^t(t-\tau)^{-\frac{3}{2}(\frac{1}{q}
-\frac{1}{r})-\frac{1}{2}}\,d\tau\nonumber\\
&\leq Ct^{-\frac{3}{2q}+\alpha+\frac{\beta}{2}}
\|(1+|x|)^{-\alpha}(1+|x|-x_1)^{-\beta} f\|_{q,D}\notag\\
&\leq Ct^{-\frac{3}{2}
(\frac{1}{q}-\frac{1}{r})-\frac{1}{2}+\alpha+\frac{\beta}{2}}
\|(1+|x|)^{-\alpha}(1+|x|-x_1)^{-\beta} f\|_{q,D},\label{t-1tot3}
\end{align}
where we have used $1/q-1/r<1/3$ and $r\leq 3$.  
Collecting (\ref{duhamelv2}), (\ref{sminusalqlr}), 
(\ref{split2})--(\ref{t-1tot3}) together with 
$w_-|_{\R^3\setminus B_R(0)}=e^{-(t+1)A_{-a}}P_Df$ yields
\begin{align*}
\|&(1+|x|)^{-\alpha}(1+|x|-x_1)^{-\beta}
\nabla e^{-tA_{-a}}P_Df\|_{r,\R^3\setminus B_R(0)}\\
&\leq Ct^{-\frac{3}{2}
(\frac{1}{q}-\frac{1}{r})-\frac{1}{2}+\alpha+\frac{\beta}{2}}
\|(1+|x|)^{-\alpha}(1+|x|-x_1)^{-\beta} f\|_{q,D}
\end{align*}
for $t\geq 3$ and $f\in L^{q}_{\rho_-}(D)$ whenever 
$r\leq 3$, (\ref{alphabeta4}) and $1/q-1/r<1/3$ are 
fulfilled. Moreover, if $1/q-1/r<1/3$, then we also have 
\begin{alignat}{2}
\|&(1+|x|)^{-\alpha}(1+|x|-x_1)^{-\beta}
\nabla&&e^{-tA_{-a}}P_Df\|_{r,D_R}\notag\\
&\leq C\|\nabla e^{-tA_{-a}}P_Df\|_{W^{1,q}(D_R)}&&\leq 
Ct^{-\frac{3}{2q}+\alpha+\frac{\beta}{2}}\,
\|(1+|x|)^{-\alpha}(1+|x|-x_1)^{-\beta}f\|_{q,D}\notag\\
&&&\leq Ct^{-\frac{3}{2}(\frac{1}{q}-\frac{1}{r})-\frac{1}{2}
+\alpha+\frac{\beta}{2}}\,
\|(1+|x|)^{-\alpha}(1+|x|-x_1)^{-\beta}f\|_{q,D}\label{lqlrdr2}
\end{alignat}  
by the Sobolev embedding
and (\ref{loc2}), thus (\ref{lqlrdgraddual0}) holds for $t\geq 1$. 
Since (\ref{lqlrdgraddual0}) 
with $t\leq 1$ directly follows from (\ref{lqlrdsmall}), 
we conclude (\ref{lqlrdgraddual0}) 
for $t>0,$ which completes the proof of the assertion 1. 
Under the assumption in the assertion 2, we can apply the 
assertion 1 with $q=r',r=q'$ to get
\begin{align*}
|(e^{-tA_a}P_D\div F,\varphi)|&=
|(F,\nabla e^{-tA_{-a}}\varphi)|\\&
\leq \|(1+|x|)^{\alpha}(1+|x|-x_1)^{\beta}F\|_{q,D}
\|(1+|x|)^{-\alpha}(1+|x|-x_1)^{-\beta}
\nabla e^{-tA_{-a}}\varphi\|_{q',D}\\
&\leq Ct^{-\frac{3}{2}(\frac{1}{q}-\frac{1}{r})-\frac{1}{2}}
(1+t)^{\alpha+\frac{\beta}{2}}
\|(1+|x|)^{\alpha}(1+|x|-x_1)^{\beta}F\|_{q,D}\\
&\qquad\qquad
\times\|(1+|x|)^{-\alpha}(1+|x|-x_1)^{-\beta}\varphi\|_{r',D}
\end{align*}
for $t>0,\varphi\in C_{0,\sigma}^\infty(D)$, 
which combined with $L^{r'}_{(1+|x|)^{-\alpha r'}
(1+|x|-x_1)^{-\beta r'},\sigma}(D)^*
=L^r_{(1+|x|)^{\alpha r}(1+|x|-x_1)^{\beta r},\sigma}(D)$ 
yields the assertion 2. The proof is complete. \qed

In the proof of Theorem \ref{lqlrd2}, we 
make use of the assertion 3--5 of 
Proposition \ref{lqlrr3} and 
the assertion 2--4 of Proposition \ref{proplqlrr3dual} to 
improve the decay rate in Theorem \ref{lqlrd} and \ref{thmdual0}.

\noindent{\bf Proof of Theorem \ref{lqlrd2}.}
We start with the proof of the assertion 1. 
For simplicity, let us set  
\begin{align*}
\eta_0:=\frac{\alpha}{4}+\max\left\{
\frac{\alpha}{4},\frac{\beta}{2}\right\}
+\varepsilon.
\end{align*}
We use (\ref{duhamelv}) to derive
the estimate on $\R^3\setminus B_R(0).$ 
It follows from (\ref{losslessr31}) that 
\begin{align}
\|(1+|x|)^\alpha(1+|x|-x_1)^\beta\nabla^iS_a(t)v(0)\|_{r,\R^3}
&\leq Ct^{-\frac{3}{2}(\frac{1}{q}-\frac{1}{r})-\frac{i}{2}+\eta_0}
\|(1+|x|)^{\alpha}(1+|x|-x_1)^{\beta} 
v(0)\|_{q,\R^3}\notag\\
&\leq Ct^{-\frac{3}{2}(\frac{1}{q}-\frac{1}{r})-
\frac{i}{2}+\eta_0}
\|(1+|x|)^{\alpha}(1+|x|-x_1)^{\beta} f\|_{q,D}
\label{v0}
\end{align}
for $t\geq 1$ and $f\in L^{q}_{\rho}(D).$ 
Let $\lambda$ satisfy 
\begin{align*}
1<\lambda<\min\left\{\frac{3}{2},q\right\},\quad 
\beta<1-\frac{1}{\lambda},\quad 
\alpha+\beta<3\left(1-\frac{1}{\lambda}\right),  
\end{align*}
then in view of (\ref{loc20}), 
the same calculation as in (\ref{lest}) leads to 
\begin{align}
\|(1+|x|)^{\alpha}(1+|x|-x_1)^{\beta}
P_{\R^3}L(t)\|_{\kappa,\R^3}
&\leq C\|(1+|x|)^{\alpha}(1+|x|-x_1)^{\beta}
L(t)\|_{\kappa,\R^3}\notag\\
&\leq C(1+t)^{-\frac{3}{2s}}
\|(1+|x|)^{\alpha}(1+|x|-x_1)^{\beta}f\|_{q,D}\label{kest}
\end{align}
for $t>0$ and $\kappa=\lambda,q$, where 
\begin{align}\label{sdef}
\max\left\{
\frac{3q}{3+\alpha q+\beta q},\frac{2q}{2+\alpha q}
\right\}<s<q. 
\end{align} 
We apply (\ref{losslessr31}) and (\ref{kest}) to get
\begin{align}
\int_0^{\frac{t}{2}}&\|(1+|x|)^{\alpha}(1+|x|-x_1)^{\beta}\nabla^i
S_a(t-\tau)P_{\R^3}L(\tau)\|_{r,\R^3}\,d\tau\notag\\
&\leq C\int_0^{\frac{t}{2}}(t-\tau)^{-\frac{3}{2}
(\frac{1}{\lambda}-\frac{1}{r})-\frac{i}{2}+\eta_0}
(1+\tau)^{-\frac{3}{2s}}\,d\tau\,
\|(1+|x|)^{\alpha}(1+|x|-x_1)^{\beta} f\|_{q,D}\notag\\
&\leq Ct^{-\frac{3}{2}
(\frac{1}{\lambda}-\frac{1}{r})-\frac{i}{2}+\eta_0}
\|(1+|x|)^{\alpha}(1+|x|-x_1)^{\beta} f\|_{q,D}
\int_0^{\frac{t}{2}}(1+\tau)^{-\frac{3}{2s}}\,d\tau\notag\\
&\leq Ct^{-\frac{3}{2}
(\frac{1}{q}-\frac{1}{r})-\frac{i}{2}+\eta_0}
\|(1+|x|)^{\alpha}(1+|x|-x_1)^{\beta} f\|_{q,D},
\label{0totover20}\\
\int_{\frac{t}{2}}^{t-1}&
\|(1+|x|)^{\alpha}(1+|x|-x_1)^{\beta}\nabla^i
S_a(t-\tau)P_{\R^3}L(\tau)\|_{r,\R^3}\,d\tau\notag\\
&\leq C\int_{\frac{t}{2}}^{t-1}
(t-\tau)^{-\frac{3}{2}
(\frac{1}{\lambda}-\frac{1}{r})-\frac{i}{2}+\eta_0}
(1+\tau)^{-\frac{3}{2s}}\,d\tau\,
\|(1+|x|)^{\alpha}(1+|x|-x_1)^{\beta} f\|_{q,D}\notag\\
&\leq Ct^{-\frac{3}{2s}}\|(1+|x|)^{\alpha}
(1+|x|-x_1)^{\beta} f\|_{q,D}
\int_{\frac{t}{2}}^{t-1}
(t-\tau)^{-\frac{3}{2}(\frac{1}{\lambda}
-\frac{1}{r})-\frac{i}{2}+\eta_0}\,d\tau\notag\\
&\leq 
\begin{cases}
Ct^{-\frac{3}{2}(\frac{1}{s}
-\frac{1}{r})-\frac{i}{2}+\eta_0+1-\frac{3}{2\lambda}}
\|(1+|x|)^{\alpha}
(1+|x|-x_1)^{\beta} f\|_{q,D}
\quad &{\rm if}~
-\displaystyle\frac{3}{2}\left(\frac{1}{\lambda}
-\frac{1}{r}\right)-\frac{i}{2}+\eta_0>-1,\\[13pt]
Ct^{-\frac{3}{2s}}(\log t)
\|(1+|x|)^{\alpha}(1+|x|-x_1)^{\beta} f\|_{q,D}
&{\rm if}~-\displaystyle\frac{3}{2}\left(\frac{1}{\lambda}
-\frac{1}{r}\right)-\frac{i}{2}+\eta_0=-1,\\[13pt]
Ct^{-\frac{3}{2s}}\|(1+|x|)^{\alpha}(1+|x|-x_1)^{\beta} f\|_{q,D}
&{\rm if}~
-\displaystyle\frac{3}{2}\left(\frac{1}{\lambda}
-\frac{1}{r}\right)-\frac{i}{2}+\eta_0<-1
\end{cases}\notag\\
&\leq Ct^{-\min\{\frac{3}{2s},\frac{3}{2}\left(\frac{1}{q}
-\frac{1}{r}\right)+\frac{i}{2}-\eta_0\}}
\|(1+|x|)^{\alpha}(1+|x|-x_1)^{\beta} f\|_{q,D}
\label{tover2tot-12}
\end{align}
for $t\geq 2,f\in L^{q}_{\rho}(D),i=0,1$, where we have used 
\begin{align*}
t^{-\frac{3}{2s}}\log t\leq t^{-\frac{3}{2q}}
=t^{-\frac{3}{2q}-\frac{3}{2}\left(\frac{1}{\lambda}
-\frac{1}{r}\right)-\frac{i}{2}+\eta_0+1}\leq 
t^{-\frac{3}{2}\left(\frac{1}{q}
-\frac{1}{r}\right)-\frac{i}{2}+\eta_0}
\end{align*}
for $t\geq 1$ in the case $-3(1/\lambda-1/r)/2-i/2+\eta_0=-1.$ 
Moreover, it follows from (\ref{lqlrr3small}), (\ref{kest}) that
\begin{align}   
\int_{t-1}^t&\|(1+|x|)^\alpha(1+|x|-x_1)^\beta
\nabla^i S_a(t-\tau)P_{\R^3}L(\tau)\|_{r,\R^3}\,d\tau\notag\\
&\leq Ct^{-\frac{3}{2s}}\int_{t-1}^t
(t-\tau)^{-\frac{3}{2}(\frac{1}{q}
-\frac{1}{r})-\frac{i}{2}}\,d\tau\,
\|(1+|x|)^{\alpha}(1+|x|-x_1)^{\beta} f\|_{q,D}\notag\\
&\leq Ct^{-\frac{3}{2s}}
\|(1+|x|)^{\alpha}(1+|x|-x_1)^{\beta} f\|_{q,D},\label{t-1tot2}
\end{align}
where we have used $1/q-1/r<1/3$ in the last inequality. 
If $1/q-1/r<1/3,$ then by (\ref{loc20}), we also have  
\begin{align}
\|(1+|x|)^\alpha(1+|x|-x_1)^\beta\nabla^i e^{-tA_a}P_Df\|_{r,D_R}
&\leq C\|\nabla^i e^{-tA_a}P_Df\|_{W^{1,q}(D_R)}\notag\\&\leq 
Ct^{-\frac{3}{2s}}\|(1+|x|)^\alpha(1+|x|-x_1)^\beta f\|_{q,D}
\label{locals}
\end{align}
for $t\geq 3,f\in L^q_\rho(D),i=0,1.$
Collecting (\ref{v0}), 
(\ref{0totover20})--(\ref{locals}) yields
\begin{align}\label{lqlrds}
\|(1+|x|)^\alpha(1+|x|-x_1)^\beta 
\nabla e^{-tA_a}P_Df\|_{r,D}\leq 
Ct^{-\min\{\frac{3}{2s},
\frac{3}{2}(\frac{1}{q}-\frac{1}{r})+\frac{i}{2}-\eta_0\}}
\|(1+|x|)^\alpha(1+|x|-x_1)^\beta f\|_{q,D}
\end{align} 
for $t\geq 3,f\in L^q_\rho(D)$ 
if (\ref{alphabetadef}) and $1/q-1/r<1/3$ are satisfied. 
Due to $3/(2s)\geq 3/(2q)\geq 3(1/q-1/r)/2-\eta_0,$ 
we have (\ref{gradlqlrdlarge5}) with $i=0$ 
by (\ref{lqlrds}) with $i=0$. 
To get the case $i=1$, from (\ref{sdef}), 
we put 
\begin{align*}
\frac{3}{2s}=
\frac{3}{2q}+\frac{\alpha}{2}+\min\left\{
\frac{\alpha}{4},\frac{\beta}{2}\right\}-\varepsilon_*
\end{align*}
with some $\varepsilon_*>0,$ then the condition 
\begin{align}\label{condis}
\frac{3}{2}\left(\frac{1}{q}-\frac{1}{r}\right)
+\frac{1}{2}-\eta_0\leq \frac{3}{2s}
\end{align}
is equivalent to 
$1-2\alpha-\beta\leq 3/r+2(\varepsilon-\varepsilon_*)$.
Hence, under the assumption 
\begin{align*} 
1-2\alpha-\beta<\frac{3}{r}+2\varepsilon, 
\end{align*}
the condition (\ref{condis}) is accomplished 
by taking $\varepsilon_*$ small (by taking $s$ small). 
From this observation and (\ref{lqlrds}), 
we obtain (\ref{gradlqlrdlarge5})
with $r\leq \infty$ 
(resp. $r\leq 3/(1-2\alpha-\beta-2\varepsilon))$
if $2\alpha+\beta\geq 1$ (resp. $2\alpha+\beta<1$), which asserts
the assertion 1. 
\par If $\beta=0$, 
we can see from the assertion 2 of Proposition 
\ref{local3} that (\ref{kest}) with 
$\beta=0,s=(3q)/(3+\alpha q)$ holds for $a\in[0,a_0]$. 
From this together with 
(\ref{losslessr32}) or (\ref{losslessr33}), 
the same calculation above leads us to 
\begin{align*}
\|(1+|x|)^\alpha\nabla^i e^{-tA_a}P_Df\|_{r,D}\leq
\begin{cases}
Ct^{-\min\{\frac{3}{2q}+\frac{\alpha}{2},
\frac{3}{2}(\frac{1}{q}-\frac{1}{r})+\frac{i}{2}-\frac{\alpha}{2}\}}
\|(1+|x|)^\alpha f\|_{q,D}&\text{if}~a>0,\\
Ct^{-\min\{\frac{3}{2q}+\frac{\alpha}{2},
\frac{3}{2}(\frac{1}{q}-\frac{1}{r})+\frac{i}{2}\}}
\|(1+|x|)^\alpha f\|_{q,D}&\text{if}~a=0
\end{cases}
\end{align*} 
for $t\geq 3,f\in L^q_{(1+|x|)^{\alpha q}}(D)$ and $i=0,1$ provided 
that $1/q-1/r<1/3.$
From this with $i=0$, we obtain 
\begin{align*}
\|(1+|x|)^\alpha e^{-tA_a}P_Df\|_{r,D}\leq
\begin{cases}
Ct^{-\frac{3}{2}(\frac{1}{q}-\frac{1}{r})+\frac{\alpha}{2}}
\|(1+|x|)^\alpha f\|_{q,D}&\text{if}~a>0,\\
Ct^{-\frac{3}{2}(\frac{1}{q}-\frac{1}{r})}
\|(1+|x|)^\alpha f\|_{q,D}&\text{if}~a=0
\end{cases}
\end{align*} 
for $t\geq 3$ provided that $1/q-1/r<1/3.$ 
In the case of the Stokes semigroup $(a=0)$, 
since the estimate is homogeneous (see also (\ref{lqlrdsmall}) with 
$a=0,\beta=0,|k|=0$), the restriction $1/q-1/r<1/3$ is eliminated 
by the semigroup property, which 
yields the assertion 2 and 3 with $i=0$. 
The case $i=1$ also holds because the condition 
\begin{align*}
\frac{3}{2}\left(\frac{1}{q}-\frac{1}{r}\right)
+\frac{1}{2}-\frac{\alpha}{2}
\leq \frac{3}{2q}+\frac{\alpha}{2}\quad 
\left(\text{resp.}~\frac{3}{2}\left(\frac{1}{q}-\frac{1}{r}\right)
+\frac{1}{2}
\leq \frac{3}{2q}+\frac{\alpha}{2}\right)
\end{align*} 
is accomplished by 
\begin{align*}
1-2\alpha\leq \frac{3}{r}\quad 
\left(\text{resp.}~1-\alpha\leq \frac{3}{r}\right). 
\end{align*}
The proof of the assertion 2 and 3 is complete. 
\par We next prove the assertion 4 
by using (\ref{duhamelv2}). 
Due to boundedness of 
$e^{-A_{-a}}:L^{q}_{\rho_-,\sigma}(D)\rightarrow 
L^{q}_{\rho_-,\sigma}(D)$ and (\ref{lqlrr3dual}), we have 
\begin{align}
\|&(1+|x|)^{-\alpha}(1+|x|-x_1)^{-\beta}\nabla
S_{-a}(t)v_-(0)\|_{r,\R^3}
\notag\\&
\leq Ct^{-\frac{3}{2}(\frac{1}{q}-\frac{1}{r})-\frac{1}{2}}
(1+t)^{\frac{\alpha}{4}
+\max\{\frac{\alpha}{4},\frac{\beta}{2}\}+\varepsilon}
\|(1+|x|)^{-\alpha}(1+|x|-x_1)^{-\beta} f\|_{q,D}\notag
\end{align}
for $t>0,i=0,1$ and $f\in L^{q}_{\rho_-}(D).$ Let $\widetilde{\nu}$ 
satisfy 
\begin{align}\label{lambdaminus}
1<\widetilde{\nu}<\min\left\{\max\left\{\frac{3}{2+\frac{3}{2}\alpha},
\frac{3}{2+\alpha+\beta}\right\},q\right\}
\end{align}
and $1/\widetilde{\nu}-1/r\ne 1/3.$ Here, 
we note that 
the condition (\ref{alphabetar2}) ensures
\begin{align}\label{alphabetageq1}
1<\max\left\{\frac{3}{2+\frac{3}{2}\alpha},
\frac{3}{2+\alpha+\beta}\right\}.
\end{align} 
In fact, if
\begin{align*}
\max\left\{\frac{3}{2+\frac{3}{2}\alpha},
\frac{3}{2+\alpha+\beta}\right\}=\frac{3}{2+\frac{3}{2}\alpha}, 
\end{align*} 
then we have $\alpha\leq 2\beta.$ This together 
with $\alpha+\beta<1$ implies $\alpha<2/3,$ thus 
\begin{align*}
\max\left\{\frac{3}{2+\frac{3}{2}\alpha},
\frac{3}{2+\alpha+\beta}\right\}=\frac{3}{2+\frac{3}{2}\alpha}>1.
\end{align*} 
The other case 
\begin{align*}
\max\left\{\frac{3}{2+\frac{3}{2}\alpha},
\frac{3}{2+\alpha+\beta}\right\}=\frac{3}{2+\alpha+\beta}  
\end{align*} 
directly follows from $\alpha+\beta<1$, 
which asserts (\ref{alphabetageq1}). Because 
\begin{align*}
-\frac{3}{2q}+\alpha+\frac{\beta}{2}
\leq -\frac{3}{2}
\left(\frac{1}{q}-\frac{1}{r}\right)-\frac{1}{2}+\frac{\alpha}{4}
+\max\left\{\frac{\alpha}{4},\frac{\beta}{2}\right\}+\varepsilon
\end{align*}
and 
\begin{align*}
-\frac{3}{2}
\left(\frac{1}{\widetilde{\nu}}-\frac{1}{r}\right)
-\frac{3}{2q}+\alpha+\frac{\beta}{2}+\frac{1}{2}
\leq -\frac{3}{2}
\left(\frac{1}{q}-\frac{1}{r}\right)-\frac{1}{2}+\frac{\alpha}{4}
+\max\left\{\frac{\alpha}{4},\frac{\beta}{2}\right\}+\varepsilon
\end{align*}
are accomplished by 
(\ref{alphabetar2})--(\ref{alphabetar3}) and (\ref{lambdaminus}), 
we see from the same calculation 
as in (\ref{kminusest})--(\ref{lqlrdr2}) that 
\begin{align*}
\int_0^t&\|(1+|x|)^{-\alpha}(1+|x|-x_1)^{-\beta}
\nabla S_{-a}(t-\tau)P_{\R^3}L_-(\tau)\|_{r,\R^3}\,d\tau\notag\\
&\leq \int_0^t\|\nabla 
S_{-a}(t-\tau)P_{\R^3}L_-(\tau)\|_{r,\R^3}\,d\tau
=\int_0^{\frac{t}{2}}+
\int^{t-1}_{\frac{t}{2}}+\int_{t-1}^t,\\
\int_0^{\frac{t}{2}}
&\leq Ct^{-\frac{3}{2}
(\frac{1}{\widetilde{\nu}}-\frac{1}{r})-\frac{1}{2}}
\left(1+\frac{t}{2}\right)^{\alpha+\frac{\beta}{2}}
\|(1+|x|)^{-\alpha}(1+|x|-x_1)^{-\beta} f\|_{q,D}
\int_0^{\frac{t}{2}}(1+\tau)^{-\frac{3}{2q}}\,d\tau\notag\\
&\leq Ct^{-\frac{3}{2}
(\frac{1}{q}-\frac{1}{r})-\frac{1}{2}+\frac{\alpha}{4}
+\max\left\{\frac{\alpha}{4},\frac{\beta}{2}\right\}+\varepsilon}
\|(1+|x|)^{-\alpha}(1+|x|-x_1)^{-\beta} f\|_{q,D},\\
\int_{\frac{t}{2}}^{t-1}
&\leq C\left(1+\frac{t}{2}\right)^{-\frac{3}{2q}}
t^{\alpha+\frac{\beta}{2}}\,
\|(1+|x|)^{-\alpha}(1+|x|-x_1)^{-\beta}f\|_{q,D}
\int_{\frac{t}{2}}^{t-1}
(t-\tau)^{-\frac{3}{2}(\frac{1}{\widetilde{\nu}}
-\frac{1}{r})-\frac{1}{2}}\,d\tau\notag\\
&\leq Ct^{-\frac{3}{2}
(\frac{1}{q}-\frac{1}{r})-\frac{1}{2}+\frac{\alpha}{4}
+\max\left\{\frac{\alpha}{4},\frac{\beta}{2}\right\}+\varepsilon}
\|(1+|x|)^{-\alpha}(1+|x|-x_1)^{-\beta} f\|_{q,D},\\
\int_{t-1}^t
&\leq Ct^{-\frac{3}{2q}+\alpha+\frac{\beta}{2}}
\|(1+|x|)^{-\alpha}(1+|x|-x_1)^{-\beta} f\|_{q,D}\notag\\
&\leq Ct^{-\frac{3}{2}
(\frac{1}{q}-\frac{1}{r})-\frac{1}{2}+\frac{\alpha}{4}
+\max\left\{\frac{\alpha}{4},\frac{\beta}{2}\right\}+\varepsilon}
\|(1+|x|)^{-\alpha}(1+|x|-x_1)^{-\beta} f\|_{q,D}
\end{align*}
and that 
\begin{align*}
\|&(1+|x|)^{-\alpha}(1+|x|-x_1)^{-\beta}\nabla 
e^{-tA_{-a}}P_Df\|_{r,D_R}
\\&
\leq Ct^{-\frac{3}{2}(\frac{1}{q}-\frac{1}{r})-\frac{1}{2}+
\max\left\{\frac{\alpha}{4},\frac{\beta}{2}\right\}+\varepsilon}
\|(1+|x|)^{-\alpha}(1+|x|-x_1)^{-\beta}f\|_{q,D}.
\end{align*}  
We thus conclude the assertion 4. 
Similarly, we can prove the assertion 5 
by taking $\widetilde{\nu}$ so that (\ref{lambdaminus}) with 
$\beta=0,1/\widetilde{\nu}-1/r\ne 1/3$ and by 
using (\ref{lqlrr3dual2}), (\ref{loc2}) with $\beta=0,a>0,$ and 
 (\ref{kminusest}) with 
$\beta=0,\kappa=\widetilde{\nu},q$. 
This is in fact, the conditions 
\begin{align}\label{alphaq}
-\frac{3}{2q}+\alpha
\leq -\frac{3}{2}
\left(\frac{1}{q}-\frac{1}{r}\right)-\frac{1}{2}+\frac{\alpha}{2},
\qquad 
-\frac{3}{2}
\left(\frac{1}{\widetilde{\nu}}-\frac{1}{r}\right)
-\frac{3}{2q}+\alpha+\frac{1}{2}
\leq -\frac{3}{2}
\left(\frac{1}{q}-\frac{1}{r}\right)-\frac{1}{2}+\frac{\alpha}{2}
\end{align}
are accomplished by (\ref{alphar})--(\ref{alphar2}) and 
(\ref{lambdaminus}) with $\beta=0$. 
\par Finally, if $\beta=a=0$, then we take 
$\widetilde{\nu}$ so that (\ref{lambdaminus}) with 
$\beta=0,1/\widetilde{\nu}-1/r\ne 1/3$, then 
due to (\ref{loc2}) with $\beta=a=0$, 
it turns out that 
(\ref{kminusest}) is replaced by 
\begin{align*}
\|P_{\R^3}L_-(t)\|_{\kappa,\R^3}
\leq C(1+t)^{-\frac{3}{2q}+\frac{\alpha}{2}}
\|(1+|x|)^{-\alpha}f\|_{q,D}
\end{align*}
for $\kappa=\widetilde{\nu},q$. 
This combined with (\ref{lqlrr3dual4}), 
(\ref{loc2}) with $\beta=a=0$ and (\ref{alphaq})
yields 
(\ref{lqlrddual2}) whenever (\ref{alphar})--(\ref{alphar2}) and 
$1/q-1/r<1/3$ are satisfied. 
The restriction 
$1/q-1/r<1/3$ is eliminated 
by the semigroup property and by the estimate 
\begin{align}\label{stokeshom}
\|(1+|x|)^{-\alpha}e^{-tA}P_Df\|_{r,D}
\leq Ct^{-\frac{3}{2}(\frac{1}{q}-\frac{1}{r})}
\|(1+|x|)^{-\alpha}f\|_{q,D},
\end{align} 
which can be derived in the same way. 
We thus conclude the assertion 6, which completes the proof. \qed

\begin{rmk}\label{rmkrelation}
Fix $a_0>0$ and assume $a\in(0,a_0].$ 
Let $i=0,1$ and let $1<q\leq r\leq \infty~(q\ne\infty),\alpha\geq 0$ 
satisfy $1/q-1/r<1/3$ and $0\leq \alpha<3(1-1/q).$ 
If $i=1,$ we also suppose $r\leq 3/(1-\alpha).$
By using (\ref{lqlrr3iso}) and by carrying out 
the same calculation as in the proof of 
Theorem \ref{lqlrd2}, we have 
\begin{align}\label{lqlrdiso}
\|(1+|x|)^\alpha \nabla^ie^{-tA_a}P_Df\|_{r,D}\leq
Ct^{-\frac{3}{2}(\frac{1}{q}-\frac{1}{r})-\frac{i}{2}}
(1+a^\alpha t^{\frac{\alpha}{2}})
\|(1+|x|)^\alpha f\|_{q,D}
\end{align}
for $t\geq 3,f\in L^q_{(1+|x|)^{\alpha q}}(D)$ and $i=0,1$. 
Therefore, (\ref{gradlqlrdlarge7}) can be recovered by 
passing to the limit $a\rightarrow 0$ in (\ref{lqlrdiso}). 
\par By (\ref{lqlinftyr3}),
the same argument as in the proof of (\ref{loc2}) leads to   
\begin{align*}
\|&\partial_t
e^{-tA_{-a}}P_Df\|_{q,D_R}+\|e^{-tA_{-a}}P_Df\|_{W^{2,q}(D_R)}
\leq Ct^{-\frac{3}{2q}}(a^\alpha t^{\alpha}+t^{\frac{\alpha}{2}})
\|(1+|x|)^{-\alpha}f\|_{q,D}
\end{align*}
for $t\geq 1,f\in L^q_{(1+|x|)^{-\alpha q}}(D)$ provided that 
$0\leq \alpha<3/q.$
Let $\alpha\geq 0,1<q\leq r<\infty$ satisfy $1/q-1/r<1/3$, 
(\ref{alphar}) and (\ref{alphar2}). Then 
by taking this estimate, (\ref{lqlinftyr30}) and 
(\ref{lqlinftyr3}) into account and by the same calculation as in 
the proof of Theorem \ref{lqlrd2}, we get 
\begin{align}\label{lqlrddual4}
\|(1+|x|)^{-\alpha}\nabla e^{-tA_{-a}}P_Df\|_{r,D}
\leq Ct^{-\frac{3}{2}(\frac{1}{q}-\frac{1}{r})-\frac{1}{2}}
(a^\alpha t^{\frac{\alpha}{2}}+1)
\|(1+|x|)^{-\alpha}f\|_{q,D}.
\end{align}
Therefore, (\ref{lqlrddual2}) can also be recovered by 
passing to the limit $a\rightarrow 0$ in (\ref{lqlrddual4}).
\end{rmk}

Let us proceed to the optimality of the restrictions 
(\ref{rcritical}) and (\ref{alphar2}) for the Stokes semigroup. 
In $L^q$ framework, 
it is known from Hishida \cite{hishida2011} that 
this issue is related to summability 
of the steady Stokes flow near spatial infinity. 
To accomplish our purpose, 
we extend the argument due to \cite{hishida2011} 
to the weighted $L^q$ space.
\par Let $F(t)$ satisfy the following assumption.\medskip\\ 
{\bf (A)} $F(t)\in 
C_0^\infty(\overline{D})^{3\times 3}:=\{\widetilde{u}|_{D}\mid
\widetilde{u}\in C_0^\infty(\R^3)^{3\times 3}\}$ 
for each 
$t\in\R$ and $\sup_{t\in\R}\|F(t)\|_{\infty,D}<\infty.$ 
Furthermore, there exist constants $r_0>0$ and  
$\theta\in (0,1)$ such that 
supp\,$F(t)\subset B_{r_0}(0)=\{y\in \R^3\,;\,|y|<r_0\}$ 
for all $t\in \R$ and 
$\div F(t)\in 
C^{\theta}_{\rm loc}(\R;L^{\infty}(D)^3).$\medskip\\ 
We consider the unsteady problem
\begin{align*}
\partial_t v-\Delta v+\nabla \pi=\nabla\cdot F(t),
\quad \nabla\cdot v=0,
\quad x\in D,t\in\R,\qquad 
v|_{\partial D}=0,\quad t\in\R. 
\end{align*}
This problem is rewritten as 
\begin{align}\label{abstract}
\frac{dv}{dt}+Av=P_D\nabla\cdot F(t),\quad t\in\R
\end{align}
in $L^r_\sigma(D),r\in (1,\infty).$ 
For the problem (\ref{abstract}), in view of 
\begin{align*}
\|\nabla^ie^{-tA}P_Df\|_{r,D}
\leq Ct^{-\frac{3}{2}(\frac{1}{q}-\frac{1}{r})-\frac{i}{2}}
\|f\|_{q,D}
\end{align*}
for $t>0$ and $f\in L^q(D)$ provided that 
$1<q\leq r\leq\infty~(q\ne\infty)$ if $i=0$ and 
$1<q\leq r\leq 3$ if $i=1$,  
we know the following from \cite[Theorem 3.1]{hishida2011}.
\begin{lem}[\cite{hishida2011}]
We suppose {\bf (A)}. Then $($\ref{abstract}$)$ admits a solution 
\begin{align}\label{vrep}
v(t)=\int_{-\infty}^t e^{-(t-\tau)A}P_D\nabla\cdot F(\tau)\,d\tau,
\end{align}
which is of class 
\begin{align*}
v(t)\in L^s_{\sigma}(D)\quad{\rm for}~ t\in\R,\qquad
\sup_{t\in\R}\|v(t)\|_{s,D}<\infty
\end{align*}
for all $s\in(3,\infty).$
\end{lem}
Furthermore, by the similar argument 
to the proof of \cite[Theorem 3.1]{hishida2011}, 
we have the following, which implies
the relation between weighted $L^q$ summability 
of the solution (\ref{vrep}) 
and the weighted $L^q$-$L^r$ estimates of the first derivative of 
$e^{-tA}$. 

\begin{lem}\label{lemvintegral} 
\begin{enumerate}
\item Let $0\leq \alpha<1$ and let $q_0>3/(2-\alpha)$. 
We suppose {\bf (A)} and 
the estimate $($\ref{gradlqlrdlarge3}$)$ for all 
$3/(3-\alpha)<q\leq r<q_0$. 
Then the solution $($\ref{vrep}$)$ is of class 
\begin{align*}
v(t)\in L^s_{(1+|x|)^{-\alpha s},\sigma}(D)
\quad {\rm for}~t\in\R,\qquad 
\sup_{t\in\R}\|(1+|x|)^{-\alpha}v(t)\|_{s,D}<\infty
\end{align*}
for all $s\in(p_0,3/\alpha),$ where $p_0$ is defined by  
\begin{align}\label{p0def}
\frac{1}{p_0}=\frac{2}{3}-\frac{1}{q_0}.
\end{align}
\item Let $0<\alpha<1$ and let $q_0<3/\alpha$. 
We suppose {\bf (A)} and 
the estimate $($\ref{lqlrddual2}$)$ for all $1<q\leq r<q_0$.  
Then the solution $($\ref{vrep}$)$ is of class 
\begin{align}\label{plusalphas}
v(t)\in L^s_{(1+|x|)^{\alpha s},\sigma}(D)
\quad {\rm for~}t\in\R,\qquad 
\sup_{t\in\R}\|(1+|x|)^{\alpha}v(t)\|_{s,D}<\infty
\end{align}
for all $s\in(p_0,\infty),$ 
where $p_0$ is given by $($\ref{p0def}$)$. 
\end{enumerate}
\end{lem}
\begin{proof}
We first prove the assertion 1. 
For $\varphi\in C_{0,\sigma}^\infty(D),$ we find 
\begin{align}\label{vs1}
(v(t),\varphi)=\int_0^\infty
(e^{-\tau A}P_D\nabla\cdot F(t-\tau),\varphi)\,d\tau
=-\int_0^\infty (F(t-\tau),\nabla e^{-\tau A}\varphi)\,d\tau=
\int_0^1+\int_1^\infty.
\end{align} 
Let $s\in (p_0,3/\alpha),$  
then $s'\in(3/(3-\alpha),q_0).$  
We thus use (\ref{gradlqlrdlarge3}) to see 
\begin{align}
\left|\int_0^1 \right|&=
\left|\int_0^1 \big((1+|x|)^{-\alpha}F(t-\tau),(1+|x|)^\alpha
\nabla e^{-\tau A}\varphi\big)\,d\tau \right|\notag\\&\leq 
\int_0^1\|(1+|x|)^{-\alpha}F(t-\tau)\|_{s,D}
\|(1+|x|)^\alpha
\nabla e^{-\tau A}\varphi\|_{s',D}\,d\tau\leq 
CL|B_{r_0}(0)|^{\frac{1}{s}}\|(1+|x|)^\alpha\varphi\|_{s',D}
\label{vs2}
\end{align}
for all $t>0,\varphi\in C_{0,\sigma}^\infty(D),$ where 
$|B_{r_0}(0)|$ denotes the Lebesgue measure of $B_{r_0}(0)$ and 
$L=\sup_{t\in\R}\|F(t)\|_{\infty,D}.$ In view of 
\begin{align}\label{sp0q0}
-\frac{3}{2}\left(\frac{1}{s'}-\frac{1}{q_0}\right)-\frac{1}{2}
<-\frac{3}{2}\left(1-\frac{1}{p_0}-\frac{1}{q_0}\right)-\frac{1}{2}
=-1,
\end{align}
we take $r\in (s',q_0)$ such that
\begin{align}\label{sr}
-\frac{3}{2}\left(\frac{1}{s'}-\frac{1}{r}\right)-\frac{1}{2}<-1,
\end{align}
then  
(\ref{gradlqlrdlarge3}) yields 
\begin{align}
\left|\int_1^\infty \right|&=
\left|\int_1^\infty \big((1+|x|)^{-\alpha}F(t-\tau),(1+|x|)^\alpha
\nabla e^{-\tau A}\varphi\big)\,d\tau \right|\notag\\&\leq 
\int_1^\infty\|(1+|x|)^{-\alpha}F(t-\tau)\|_{r',D}
\|(1+|x|)^\alpha
\nabla e^{-\tau A}\varphi\|_{r,D}\,d\tau\leq 
CL|B_{r_0}(0)|^{\frac{1}{r'}}\|(1+|x|)^\alpha\varphi\|_{s',D}
\label{vs3}
\end{align}
for all $t>0,\varphi\in C_{0,\sigma}^\infty(D).$ 
From (\ref{vs1})--(\ref{vs3}) together with 
$L^{s'}_{(1+|x|)^{\alpha s'},\sigma}(D)^*
=L^s_{(1+|x|)^{-\alpha s},\sigma}(D)$, we conclude 
the assertion 1. Similarly, 
by applying (\ref{lqlrddual2}) to (\ref{vs1}), we get 
\begin{align*}
\left|\int_0^1 \right|&=
\left|\int_0^1 \big((1+|x|)^{\alpha}F(t-\tau),(1+|x|)^{-\alpha}
\nabla e^{-\tau A}\varphi\big)\,d\tau \right|\notag\\&\leq 
\int_0^1\|(1+|x|)^{\alpha}F(t-\tau)\|_{s,D}
\|(1+|x|)^{-\alpha}
\nabla e^{-\tau A}\varphi\|_{s',D}\,d\tau\\
&\leq 
C(1+r_0)^{\alpha}
L|B_{r_0}(0)|^{\frac{1}{s}}\|(1+|x|)^{-\alpha}\varphi\|_{s',D},
\notag\\
\left|\int_1^\infty \right|&=
\left|\int_1^\infty \big((1+|x|)^{\alpha}F(t-\tau),(1+|x|)^{-\alpha}
\nabla e^{-\tau A}\varphi\big)\,d\tau \right|\notag\\&\leq 
\int_1^\infty\|(1+|x|)^{\alpha}F(t-\tau)\|_{r',D}
\|(1+|x|)^{-\alpha}
\nabla e^{-\tau A}\varphi\|_{r,D}\,d\tau\\
&\leq C(1+r_0)^{\alpha}L|B_{r_0}(0)|^{\frac{1}{r'}}
\|(1+|x|)^{-\alpha}\varphi\|_{s',D}\notag
\end{align*}
for all $t>0,\varphi\in C_{0,\sigma}^\infty(D),$ where 
$r\in (s',q_0)$ fulfills (\ref{sr}). Moreover, we have  
$L^{s'}_{(1+|x|)^{-\alpha s'},\sigma}(D)^*
=L^s_{(1+|x|)^{\alpha s},\sigma}(D)$ due to 
$\alpha<1<3(1-1/p_0)<3(1-1/s)$, thus the assertion 2 also holds. 
The proof is complete. 
\end{proof}

We next prepare the following, which asserts 
the net force exerted on the obstacle by 
the fluid must vanish whenever the steady flow decays faster than 
the Stokes fundamental solution. 
This can be proved by using the asymptotic representation 
for $|x|\rightarrow\infty$ of the steady flow, see 
Chang and Finn \cite{chfi1961} and Galdi \cite{galdi2011}. 
Hishida \cite[Lemma 4.1]{hishida2011} also 
gave another independent proof by using the summability of 
the steady flow. Since the proof is same as the one of 
\cite[Lemma 4.1]{hishida2011}, 
we state the following without proof. 
\begin{lem}\label{lemnet}
Let $-1<\gamma<1$ and $F\in C_0^\infty(\overline{D})^{3\times 3}.$ 
Suppose $v$ is a smooth solution $($up to $\partial D$$)$ to 
\begin{align}\label{sta0}
-\Delta v+\nabla \pi=\nabla\cdot F,\quad 
\nabla\cdot v=0,\quad x\in D,\qquad 
v|_{\partial D}=0
\end{align}
of class 
$(1+|x|)^{\gamma}v\in L^{3/(1-\gamma)}(D).$ Then 
\begin{align*}
\int_{\partial D}\nu\cdot\{T(v,\pi)+F\}\,d\sigma=0,
\end{align*}
where $\nu$ stands for the outer unit normal to $\partial D,$
$T(v,\pi)=\nabla v+(\nabla v)^\top-\pi \I$
is the Cauchy stress tensor and $\I$ is the 
$3\times 3$-identity matrix.
\end{lem}

We are now in a position to prove Theorem \ref{thmoptimal}.\\
\noindent{\bf Proof of Theorem \ref{thmoptimal}.} \quad 
By the assertion 3 and 6 of Theorem \ref{lqlrd2}, 
we have (\ref{gradlqlrdlarge3}) 
for $t>0,f\in L^{q}_{(1+|x|)^{\alpha q}}(D)$ 
provided (\ref{rcritical}) and have (\ref{lqlrddual2})  
for $t>0,f\in L^{q}_{(1+|x|)^{-\alpha q}}(D)$ 
provided (\ref{alphar2}). 
It remains to prove the optimality of (\ref{rcritical}) and 
(\ref{alphar2}). 
We suppose (\ref{gradlqlrdlarge3}) for $3/(3-\alpha)<q\leq r<q_0$
with $q_0>3/(1-\alpha)$. 
For $F\in C_0^\infty(\overline{D})^{3\times 3},$ 
we apply the assertion 1 of Lemma \ref{lemvintegral} 
with $F(t)\equiv F,$ then the solution (\ref{vrep}) 
satisfies $(1+|x|)^{-\alpha}v(t)\in L^s(D)$ 
for $t\in\R$ and $s\in (p_0,3/\alpha)$, where 
\begin{align}\label{p0alpha}
\frac{1}{p_0}=\frac{2}{3}-\frac{1}{q_0}>\frac{1+\alpha}{3}. 
\end{align} 
Since the solution $v(t)$ has the same period as the one of $F(t)$ 
whenever $F(t)$ is periodic and 
since $F(t)\equiv F$ can be regarded 
as a periodic function with arbitrary period, 
$v(t)$ is a periodic solution with arbitrary period, 
namely $v(t)$ is a stationary solution. 
This observation together with (\ref{p0alpha}) 
yields that the problem (\ref{sta0}) 
always admits a solution $v$ satisfying 
$(1+|x|)^{-\alpha}v\in L^{3/(1+\alpha)}(D)$. 
But, due to Lemma \ref{lemnet}, this summability is possible only 
in a special case, which leads to contradiction.  
Similarly, if we suppose (\ref{lqlrddual2}) for $1<q\leq r<q_0$ 
with $q_0>3/(1+\alpha),$ then 
due to $p_0<3/(1-\alpha)$, 
the assertion 2 of Lemma \ref{lemvintegral} implies that 
the problem (\ref{sta0}) always admits a solution $v$ 
satisfying $(1+|x|)^{\alpha}v\in L^{3/(1-\alpha)}(D)$, 
which leads to contradiction. The proof is complete.
\qed

\begin{rmk}\label{rmks0optimal}
Let $1<q<\infty,0\leq \alpha<3/q$ and 
let $0\leq \varepsilon<\alpha/2$. 
By using the assertion 2 of Lemma \ref{lemvintegral}, 
we find that it is impossible to have 
\begin{align}\label{qinftyepsilon}
\|(1+|x|)^{-\alpha} S_{0}(t)P_{\R^3}f\|_{\infty,\R^3}
\leq Ct^{-\frac{3}{2q}+\varepsilon}\|(1+|x|)^{-\alpha}f\|_{q,\R^3}
\end{align}
for $t\geq 1$; therefore, the decay rate $t^{-3/(2q)+\alpha/2}$ is 
optimal in the sense that better rate $t^{-3/(2q)+\varepsilon}$ 
with some $0\leq \varepsilon<\alpha/2$ is impossible.
In fact, if we assume (\ref{qinftyepsilon}), 
then (\ref{sminusa2}) with $a=0,\beta=0$ is replaced by 
\begin{align*}
\|S_{0}(t)g_-\|_{W^{3,q}(D_R)}
\leq Ct^{-\frac{3}{2q}}(1+t)^{\varepsilon}
\sum_{i=0,2}
\|(1+|x|)^{-\alpha}\nabla^ig_-\|_{q,\R^3},
\end{align*}
thus the same argument as in the proof of 
the assertion 3 of Proposition \ref{local3} yields  
\begin{align*}
\|\partial_t
e^{-tA}P_Df\|_{q,D_R}+\|e^{-tA}P_Df\|_{W^{2,q}(D_R)}
\leq Ct^{-\frac{3}{2q}+\varepsilon}
\|(1+|x|)^{-\alpha}f\|_{q,D}
\end{align*}
for all $t\geq 3$ and $f\in L^q_{(1+|x|)^{-\alpha q}}(D)$.
From this estimate, we carry out 
the same calculation as in the proof of Theorem \ref{lqlrd2} 
to get (\ref{lqlrddual2})
for $t>0,f\in L^{q}_{(1+|x|)^{-\alpha q}}(D)$ whenever 
(\ref{alphar}) and 
\begin{align*}
1<q\leq r\leq \frac{3}{1+2\varepsilon}
\end{align*}
are fulfilled. Then the assertion 2 of Lemma \ref{lemvintegral}
with $q_0=3/(1+2\varepsilon)$ asserts that 
the problem (\ref{sta0}) always admits a solution $v$ satisfying
$(1+|x|)^{\alpha}v\in L^{3/(1-\alpha)}(D)$, which leads to 
contradiction.
\end{rmk}

\begin{rmk}\label{rmkoseen}

For the relation between summability of the solution 
\begin{align}\label{vrep2}
&v(t)=\int_{-\infty}^t e^{-(t-\tau)A_a}P_D\nabla\cdot F(\tau)\,d\tau
\end{align}
to the problem 
\begin{align*}
\frac{dv}{dt}+A_av=P_D\nabla\cdot F(t),\quad t\in\R
\end{align*}
and the $L^q$-$L^r$ estimates 
of the first derivative of the Oseen semigroup, 
we have the following. 
\begin{enumerate}
\item Given $a_0>0$, we assume $a\in(0,a_0]$. 
We suppose {\bf (A)} and 
the estimate 
\begin{align}\label{oseenhom}
\|\nabla e^{-tA_{-a}}P_Df\|_{r,D}
\leq Ct^{-\frac{3}{2}(\frac{1}{q}-\frac{1}{r})-\frac{1}{2}}
\|f\|_{q,D}
\end{align}
for $1<q\leq r<q_0$. 
Then the solution (\ref{vrep2}) is of class (\ref{plusalphas}) 
for all $s\in(p_0,\infty),$ where $p_0$ is defined by (\ref{p0def}).
\item Given $a_0>0$, we assume $a\in(0,a_0]$. 
Let $0<\alpha<1$ and $q_0>1$ satisfy $\alpha<3/q_0<2-\alpha$. 
We define $p_1$ by
\begin{align*}
\frac{1}{p_1}=\frac{2-\alpha}{3}-\frac{1}{q_0},
\end{align*} 
then $1<p_1<\infty$. We suppose {\bf (A)} and 
the estimate (\ref{lqlrddual3}) for $1<q\leq r<q_0$ and $t\geq 1$.  
Then the solution (\ref{vrep2}) is of class (\ref{plusalphas}) 
for all $s\in(p_1,\infty).$
\item Given $a_0>0$, we assume $a\in(0,a_0]$. 
Let $\alpha,\beta,\varepsilon>0$ and $q_0>1$ satisfy 
$\alpha+\beta<\min\{1,3/q_0\},\beta<1/q_0$ and 
\begin{align*} 
\frac{\max\{\alpha,2\beta\}+\beta}{3}<\frac{1}{q_0}< 
\frac{4-\alpha-\max\{\alpha,2\beta\}-4\varepsilon}{6},  
\end{align*} 
where we also assume $\alpha+8\beta<4$ if $\alpha\leq 2\beta$. 
We define $p_2$ by 
\begin{align*}
\frac{1}{p_2}=\frac{2}{3}-\frac{\alpha}{6}-\frac{\beta}{3}
-\frac{2}{3}\varepsilon-\frac{1}{q_0}\qquad 
\left(\text{resp.}~\frac{1}{p_2}=\frac{2}{3}-\frac{\alpha}{3}
-\frac{2}{3}\varepsilon-\frac{1}{q_0}\right)
\end{align*} 
if $\alpha\leq 2\beta$ (resp. $\alpha\geq 2\beta$), then 
$1<p_2<\infty$. We suppose {\bf (A)} and 
the estimate (\ref{lqlrddual10}) for $1<q\leq r<q_0$ and $t\geq 1$. 
Then the solution (\ref{vrep2}) is of class 
\begin{align*}
v(t)\in L^s_{(1+|x|)^{\alpha s}(1+|x|-x_1)^{\beta s},\sigma}(D)
\quad {\rm for~}t\in\R,\qquad 
\sup_{t\in\R}\|(1+|x|)^{\alpha}
(1+|x|-x_1)^{\beta}v(t)\|_{s,D}<\infty
\end{align*}
for all $s\in(p_2,\infty)$. 
\end{enumerate}
We in fact conclude the assertions from 
the proof of Lemma \ref{lemvintegral}, in which 
(\ref{sp0q0})--(\ref{sr}) are replaced by 
\begin{align*}
-\frac{3}{2}\left(\frac{1}{s'}-\frac{1}{q_0}\right)-\frac{1}{2}
<-\frac{3}{2}\left(1-\frac{1}{p_0}-\frac{1}{q_0}\right)-\frac{1}{2}
=-1,\quad
-\frac{3}{2}\left(\frac{1}{s'}-\frac{1}{r}\right)-\frac{1}{2}<-1
\end{align*}
or
\begin{align*}
-\frac{3}{2}\left(\frac{1}{s'}-\frac{1}{q_0}\right)-\frac{1}{2}
+\frac{\alpha}{2}
<-\frac{3}{2}\left(1-\frac{1}{p_1}-\frac{1}{q_0}\right)-\frac{1}{2}
+\frac{\alpha}{2}=-1,\quad
-\frac{3}{2}\left(\frac{1}{s'}-\frac{1}{r}\right)-\frac{1}{2}
+\frac{\alpha}{2}<-1
\end{align*}
or
\begin{align*}
-\frac{3}{2}\left(\frac{1}{s'}-\frac{1}{q_0}\right)&-\frac{1}{2}
+\frac{\alpha}{4}+\max\left\{\frac{\alpha}{4},
\frac{\beta}{2}\right\}+\varepsilon\\
&<-\frac{3}{2}\left(1-\frac{1}{p_2}-\frac{1}{q_0}\right)-\frac{1}{2}
+\frac{\alpha}{4}+\max\left\{\frac{\alpha}{4},
\frac{\beta}{2}\right\}+\varepsilon
=-1,\\
-\frac{3}{2}\left(\frac{1}{s'}-\frac{1}{r}\right)&-\frac{1}{2}
+\frac{\alpha}{4}+\max\left\{\frac{\alpha}{4},
\frac{\beta}{2}\right\}+\varepsilon<-1.
\end{align*}
Taking into account that 
stationary solution $v$ to the problem 
\begin{align*}
-\Delta v+a\partial_{x_1}v+\nabla\pi=\nabla\cdot F,\quad 
\nabla\cdot v=0,
\quad x\in D,\qquad 
v|_{\partial D}=-ae_1
\end{align*}
possesses decay structure 
$v=O((1+|x|)^{-1}(1+|x|-x_1)^{-1})$, 
Lemma \ref{finite} together with the assertion 1 suggests that 
(\ref{oseenhom}) may be possible for $1<q\leq r\leq 6$. 
This was conjectured by Hishida \cite[Section 5]{hishida2011}. 
Furthermore, from the assertion 2 and 3, 
it may be possible that 
(\ref{lqlrddual3}) for 
$1<q\leq r\leq 6/(1+\alpha)$ and that 
(\ref{lqlrddual10}) for 
$1<q\leq r\leq 6/(1+2\alpha-2\beta-4\varepsilon)$ 
(resp. $1<q\leq r\leq 6/(1+\alpha-4\varepsilon)$) 
if $\alpha\leq 2\beta$ (resp. $\alpha\geq 2\beta$). 
\end{rmk}

\section{Proof of Theorem \ref{stabilitythm} 
and Theorem \ref{attainthm}}
\quad We first collect the $L^q$-$L^r$ estimates of the 
Oseen semigroup, which will be used to 
prove Theorem \ref{stabilitythm} and Theorem \ref{attainthm}, 
in particular make use of the assertion 3 and 5 below 
in the proof of Theorem \ref{attainthm}. 
Based on the summability (\ref{sum1}) and (\ref{sum2}) below, 
we apply the assertion 3 instead 
of (\ref{gradlqlrdlarge5}) to treat $f_1,f_2$ 
(defined by (\ref{f1})--(\ref{f2})), see Remark \ref{f1rmk} below. 
\begin{prop}\label{proplqlr}
Given $a_0>0$ arbitrarily, we assume $a\in(0,a_0]$. 
\begin{enumerate}
\item Let $1<q\leq r\leq \infty ~(q\ne\infty)$ 
and let $\gamma,\delta\geq 0$ satisfy 
\begin{align*}
\delta<1-\frac{1}{q},\quad 
\gamma+\delta<3\left(1-\frac{1}{q}\right).
\end{align*} 
We set
\begin{align}\label{rhotildedef}
\widetilde{\rho}(x)
=(1+|x|)^{\gamma q}(1+|x|-x_1)^{\delta q}.
\end{align}
For every multi-index $k~(|k|\leq 1)$, there exists a constant 
$C=C(D,a_0,q,r,\gamma,\delta,k)$, independent of $a$, such that 
\begin{align}\label{lqlrdsmallq}
\|(1+|x|)^\gamma(1+|x|-x_1)^\delta\partial^k_{x}e^{-tA_a}P_Df\|_{r,D}
\leq 
Ct^{-\frac{3}{2}(\frac{1}{q}-\frac{1}{r})-\frac{|k|}{2}}
\|(1+|x|)^\gamma(1+|x|-x_1)^\delta f\|_{q,D}
\end{align}
for all $t\leq 3$ and $f\in L^{q}_{\widetilde{\rho}}(D)$.
\item 
Let $1<q\leq r\leq\infty~(q\ne\infty)$ 
and $\gamma,\delta\geq 0$ satisfy 
\begin{align}
\delta<\min\left\{1-\frac{1}{q},
\frac{1}{3}\right\},\quad 
\gamma+\delta<\min\left\{3\left(1-\frac{1}{q}\right),1\right\}.
\label{alphabeta6}
\end{align} 
Then there exists a constant 
$C(D,a_0,q,r,\gamma,\delta)$, 
independent of $a$, such that 
\begin{align}
\|(1+|x|)^\gamma(1+|x|-x_1)^\delta e^{-tA_a}P_Df\|_{r,D}
\leq 
C\sum_{i=1}^4
t^{-\frac{3}{2}(\frac{1}{q}-\frac{1}{r})+\widetilde{\eta}_i}
\|(1+|x|)^{\widetilde{\gamma}_i}
(1+|x|-x_1)^{\widetilde{\delta}_i} f\|_{q,D}
\label{lqlrdlargeq}
\end{align}
for all $t\geq 3$, $f\in L^{q}_{\widetilde{\rho}}(D)$, 
where $\widetilde{\rho}$ is given by $($\ref{rhotildedef}$)$ and 
$\widetilde{\gamma}_i,\widetilde{\delta}_i,\widetilde{\eta}_i$ 
are defined by
%\begin{equation}
%\begin{aligned}
\begin{align}
&(\widetilde{\gamma}_1,\widetilde{\gamma}_2,
\widetilde{\gamma}_3,\widetilde{\gamma}_4)=(\gamma,0,\gamma,0),
\quad
(\widetilde{\delta}_1,\widetilde{\delta}_2,\widetilde{\delta}_3,
\widetilde{\delta}_4)=(\delta,\delta,0,0)
\label{gammadeltatilde}\\
&(\widetilde{\eta}_1,\widetilde{\eta}_2,\widetilde{\eta}_3,
\widetilde{\eta}_4)=
\left(0,\gamma,
\displaystyle\frac{\delta}{2},
\gamma+\frac{\delta}{2}\right).\label{etatilde}
\end{align}
\item 
Let $1<q_i\leq 3~(i=1,2,3,4)$ and $\gamma,\delta\geq 0$ satisfy 
\begin{align}
&q_4\leq q_i\leq q_1\quad(i=2,3),\notag\\
&\gamma<\min\left\{3\left(1-\frac{1}{q_3}\right),1\right\},
\quad \delta<\min\left\{1-\frac{1}{q_2},
\frac{1}{3}\right\},\quad 
\gamma+\delta<\min\left\{3\left(1-\frac{1}{q_1}\right),1\right\}.
\label{alphabeta7}
\end{align} 
We set
\begin{align*}
\widetilde{\rho}_1(x)=(1+|x|)^{\gamma q_1}
(1+|x|-x_1)^{\delta q_1},\quad 
\widetilde{\rho}_2(x)=(1+|x|-x_1)^{\delta q_2},\quad 
\widetilde{\rho}_3(x)=(1+|x|)^{\gamma q_3}.
\end{align*}
Then there exists a constant 
$C(D,a_0,q_1,q_2,q_3,q_4,\gamma,\delta)$, 
independent of $a$, such that 
\begin{align}
\|(1+|x|)^\gamma(1+|x|-x_1)^\delta e^{-tA_a}P_Df\|_{3,D}
\leq 
C\sum_{i=1}^4
t^{-\frac{3}{2q_i}+\frac{1}{2}+\widetilde{\eta}_i}
\|(1+|x|)^{\widetilde{\gamma}_i}
(1+|x|-x_1)^{\widetilde{\delta}_i} f\|_{q_i,D}
\label{lqlrdlarge2}
\end{align}
for all $t\geq 3$, $f\in \displaystyle\bigcap^3_{i=1}
L^{q_i}_{\widetilde{\rho}_i}(D)\cap L^{q_4}(D)$, 
where
$\widetilde{\gamma}_i,
\widetilde{\delta}_i,\widetilde{\eta}_i$ are given by
$($\ref{gammadeltatilde}$)$--$($\ref{etatilde}$)$.
\item 
Let $1<q\leq 3$ 
and $\gamma,\delta\geq 0$ satisfy $($\ref{alphabeta6}$)$.
Then there exists a constant 
$C(D,a_0,q,\gamma,\delta)$, 
independent of $a$, such that 
\begin{align}
\|(1+|x|)^\gamma(1+|x|-x_1)^\delta \nabla e^{-tA_a}P_Df\|_{3,D}
\leq 
C\sum_{i=1}^4
t^{-\frac{3}{2q}+\widetilde{\eta}_i}
\|(1+|x|)^{\widetilde{\gamma}_i}
(1+|x|-x_1)^{\widetilde{\delta}_i} f\|_{q,D}
\label{lqlrdlargeq2}
\end{align}
for all $t\geq 3$, $f\in L^{q}_{\widetilde{\rho}}(D)$, 
where $\widetilde{\rho},\widetilde{\gamma}_i,
\widetilde{\delta}_i,\widetilde{\eta}_i$ are given by
$($\ref{rhotildedef}$)$, 
$($\ref{gammadeltatilde}$)$--$($\ref{etatilde}$).$
\item Let $\gamma,\delta\geq 0$ satisfy $\delta<2/3$ and 
$\gamma+\delta<2$. 
Then there exists a constant $C=C(a_0,\gamma,\delta)$, 
independent of $a$, such that 
\begin{align}
\|(1+|x|)^{\gamma}(1+|x|-x_1)^{\delta}
e^{-tA_a}P_D\div F\|_{3,D}
\leq Ct^{-\frac{1}{2}}
(1+t)^{\gamma+\frac{\delta}{2}}
\|(1+|x|)^{\gamma}(1+|x|-x_1)^{\delta} F\|_{3,D}
\label{lqlrddiv02}
\end{align}
for all $t>0$ and $F\in L^{3}_{(1+|x|)^{3\gamma}
(1+|x|-x_1)^{3\delta }}(D).$
\end{enumerate}
\end{prop}

\noindent In view of Proposition \ref{proplqlr}, it is 
reasonable to set 
\begin{align}\label{XYtdef}
[[v]]_{\gamma,\delta,t}:=\sum_{i=1}^4\big(
[v]_{3,\widetilde{\gamma}_i,\widetilde{\delta}_i,t}+
[v]_{\infty,\widetilde{\gamma}_i,\widetilde{\delta}_i,t}+
[\nabla v]_{3,\widetilde{\gamma}_i,\widetilde{\delta}_i,t}\big),\quad 
[[v]]_{t}:=[v]_{3,t}+[v]_{\infty,t}+[\nabla v]_{3,t},
\end{align}
for $\gamma,\delta\geq 0$ and $t\in(0,\infty]$, 
where $\widetilde{\gamma}_i,\widetilde{\delta}_i$ are 
given by (\ref{gammadeltatilde}) and 
\begin{align*}
&[v]_{q,\gamma,\delta,t}:=\sup_{0<\tau<t}
\tau^{\frac{1}{2}-\frac{3}{2q}}(1+\tau)^{-\gamma-\frac{\delta}{2}}
\|(1+|x|)^\gamma(1+|x|-x_1)^\delta v(\tau)\|_{q,D},\\
 &[\nabla v]_{q,\gamma,\delta,t}:=\sup_{0<\tau<t}
\tau^{1-\frac{3}{2q}}(1+\tau)^{-\gamma-\frac{\delta}{2}}
\|(1+|x|)^\gamma(1+|x|-x_1)^\delta \nabla v(\tau)\|_{q,D},\\
&[v]_{q,t}=[v]_{q,0,0,t},\quad 
[\nabla v]_{q,t}=[\nabla v]_{q,0,0,t}.
\end{align*}

In order to construct a solution in anisotropically weighted 
$L^q$ space, we extend the estimates by 
Enomoto and Shibata \cite[Lemma 3.1.]{ensh2005} and by 
the present author \cite[Lemma 3.3]{takahashi2021} 
to the following ones. 
\begin{lem}\label{closelem}
\begin{enumerate}
\item Let $\gamma,\delta$ satisfy 
\begin{align*}
\gamma\geq 0,\quad 0\leq\delta<\frac{1}{3},\quad \gamma+\delta<1 
\end{align*}
and let $\mu_i>0~(i=1,2,3,4)$ satisfy 
\begin{align}\label{mui}
\delta<\frac{1-\mu_i}{3},\qquad 
\gamma+\delta<1-\mu_i\qquad (i=1,3)
\end{align} 
and $($\ref{mu1}$)$. 
Suppose $0<a^{1/4}<\kappa_1$, where $\kappa_1$ is a constant in 
Proposition \ref{propsta} with $(\ref{staclass2})$--$(\ref{mu1})$,  
$($\ref{mui}$)$ and 
$u_s$ denotes the stationary solution obtained 
in Proposition \ref{propsta}.
Let $\psi$ be a function on $\R$ satisfying $(\ref{psidef})$ 
and set $M=\max_{t\in\R}|\psi'(t)|$.  
Given $u,v$ 
fulfilling $[[u]]_{\gamma,\delta,\infty},[[v]]_{\infty}<\infty,$ 
we set 
\begin{align*}
&E_1(u,v)(t)=\int_0^t e^{-(t-\tau)A_a}
P_D[u\cdot\nabla v](\tau)\,d\tau,\quad
E_2(u)(t)=\int_0^t 
e^{-(t-\tau)A_a}P_D[\psi(\tau)u\cdot\nabla u_s]\,d\tau,\\
&E_3(u)(t)=\int_0^t 
e^{-(t-\tau)A_a}P_D[\psi(\tau)u_s\cdot\nabla u]\,d\tau,\quad 
E_4(u)(t)=\int_0^t 
e^{-(t-\tau)A_a}P_D[(1-\psi(\tau))a
\partial_{x_1}u]\,d\tau.
\end{align*} 
Then there exists a constant $C$ independent of 
$u,v,\psi,a$ and $t$ such that 
\begin{align}
&[E_1(u,v)]_{3,\gamma,\delta,t}+[E_1(u,v)]_{\infty,\gamma,\delta,t}
+[\nabla E_1(u,v)]_{3,\gamma,\delta,t}
\leq C\left(\sum_{i=1}^4
[u]_{3,\widetilde{\gamma}_i,\widetilde{\delta}_i,t}^{\frac{1}{2}}
[u]_{\infty,\widetilde{\gamma}_i,
\widetilde{\delta}_i,t}^{\frac{1}{2}}\right)
[\nabla v]_{3,t},\label{i1est}\\
&[E_2(u)]_{3,\gamma,\delta,t}+[E_2(u)]_{\infty,\gamma,\delta,t}
+[\nabla E_2(u)]_{3,\gamma,\delta,t}\notag\\
&\quad \leq C\left(\sum_{i=1}^4
[u]_{\infty,\widetilde{\gamma}_i,\widetilde{\delta}_i,t}\right)
\big(\|\nabla u_s\|_{\frac{3}{2+\mu_3},D}
+\|\nabla u_s\|_{\frac{3}{2},D}+
\|\nabla u_s\|_{\frac{3}{2-\mu_4},D}\big),\label{i2est}\\
&[E_3(u)]_{3,\gamma,\delta,t}+[E_3(u)]_{\infty,\gamma,\delta,t}
+[\nabla E_3(u)]_{3,\gamma,\delta,t}\notag\\
&\quad \leq C\left(
\sum_{i=1}^4
[\nabla u]_{3,\widetilde{\gamma}_i,\widetilde{\delta}_i,t}\right)
\big(\|u_s\|_{\frac{3}{1+\mu_1},D}
+\|u_s\|_{3,D}+
\|u_s\|_{\frac{3}{1-\mu_2},D}\big)\label{i3est}\\
&[E_4(u)]_{3,\gamma,\delta,t}+[E_4(u)]_{\infty,\gamma,\delta,t}
+[\nabla E_4(u)]_{3,\gamma,\delta,t}
\leq Ca\left(
\sum_{i=1}^4
[\nabla u]_{3,\widetilde{\gamma}_i,\widetilde{\delta}_i,t}\right)
\label{i4est}
\end{align}
for $t\in(0,\infty]$, where
$\widetilde{\gamma}_i,\widetilde{\delta}_i$ are given by 
$($\ref{gammadeltatilde}$)$.
\item Let $\gamma,\delta$ satisfy 
\begin{align}\label{gammadelta02}
0\leq \gamma<\frac{1}{3},\quad 0\leq\delta<\frac{1}{3}
\end{align}
and let $\mu_i>0~(i=1,2,3,4)$ 
satisfy $($\ref{mu1}$)$--$($\ref{mu2}$)$. 
Suppose $0<a^{1/4}<\kappa_2$, where $\kappa_2$ is a constant in 
Proposition \ref{propsta} with $(\ref{staclass2})$--$($\ref{mu2}$)$ 
and $u_s$ denotes the stationary solution obtained 
in Proposition \ref{propsta}. 
Then $(1+|x|)^{\gamma}(1+|x|-x_1)^{\delta}u_s\in L^3(D)$ and 
$(1+|x|)^{\gamma}(1+|x|-x_1)^{\delta}
\nabla u_s\in L^{3/(2-\mu_4)}(D)$. 
Let $\psi$ be a function on $\R$ satisfying $(\ref{psidef})$ 
and set $M=\max_{t\in\R}|\psi'(t)|$. 
We set 
\begin{align}\label{f1f2def}
&F_1(t)=\int_0^t
e^{-(t-\tau)A_a}P_Df_1(\tau)\,d\tau,\quad 
F_2(t)=\int_0^t
e^{-(t-\tau)A_a}P_Df_2(\tau)\,d\tau,
\end{align}
where $f_1$ and $f_2$ are given by 
$($\ref{f1}$)$ and $($\ref{f2}$)$, respectively.
Then there exists a constant $C$ independent of 
$\psi,a$ and $t$ such that
\begin{align}
[F_1]_{3,\gamma,\delta,t}+[F_1]_{\infty,\gamma,\delta,t}
+[\nabla F_1]_{3,\gamma,\delta,t}
&\leq CM\|(1+|x|)^{\gamma}(1+|x|-x_1)^{\delta}u_s\|_{3,D}
\label{f1est}\\
[F_2]_{3,\gamma,\delta,t}+[F_2]_{\infty,\gamma,\delta,t}
+[\nabla F_2]_{3,\gamma,\delta,t}
&\leq C\|u_s\|_{\frac{3}{1-\mu_2},D}
\|(1+|x|)^{\gamma}(1+|x|-x_1)^{\delta}
\nabla u_s\|_{\frac{3}{2-\mu_4},D}\notag\\
&\qquad+Ca\|(1+|x|)^{\gamma}(1+|x|-x_1)^{\delta}
\nabla u_s\|_{\frac{3}{2-\mu_4},D}\label{f2est}
\end{align}
for $t\in(0,\infty].$ 
\end{enumerate}
\end{lem}

\begin{proof}
Since $L^\infty$ estimate is the same as the $L^3$ estimate of 
first derivative, we omit the estimate 
of $[\cdot]_{\infty,\gamma,\delta,t}$.
We first prove (\ref{i1est}). 
If $[[u]]_{\gamma,\delta,\infty}<\infty,$ 
then we find
$(1+|x|)^{\widetilde{\gamma}_i}
(1+|x|-x_1)^{\widetilde{\delta}_i} u(t)\in L^6(D)$ and 
\begin{align*}
\|&(1+|x|)^{\widetilde{\gamma}_i}
(1+|x|-x_1)^{\widetilde{\delta}_i} u(t)\|_{6,D}\leq 
t^{-\frac{1}{4}}(1+t)^{\widetilde{\gamma}+
\frac{\widetilde{\delta}}{2}}
[u]^{\frac{1}{2}}_{3,\widetilde{\gamma}_i,\widetilde{\delta}_i,t}
[u]^{\frac{1}{2}}_{\infty,\widetilde{\gamma}_i,\widetilde{\delta}_i,t}
\end{align*}
for all $t>0$, which together with (\ref{lqlrdsmallq}) yields 
\begin{align}
\int_0^t&\|(1+|x|)^\gamma(1+|x|-x_1)^\delta
\nabla ^ke^{-(t-\tau)A_a}P_D[u\cdot\nabla v](\tau)\|_{3,D}\,d\tau
\notag\\
&\leq C\int_0^t(t-\tau)^{-\frac{1}{4}-\frac{k}{2}}
\|(1+|x|)^{\gamma}(1+|x|-x_1)^{\delta} u(\tau)\|_{6,D}
\|\nabla v(\tau)\|_{3,D}\,d\tau\notag\\
&\leq C\int_0^t(t-\tau)^{-\frac{1}{4}-\frac{k}{2}}
\tau^{-\frac{3}{4}}(1+\tau)^{\gamma+\frac{\delta}{2}}\,d\tau\, 
[u]^{\frac{1}{2}}_{3,\gamma,\delta,t}
[u]^{\frac{1}{2}}_{\infty,\gamma,\delta,t}[\nabla v]_{3,t}\notag\\
&\leq Ct^{-\frac{k}{2}}
(1+t)^{\gamma+\frac{\delta}{2}}[u]^{\frac{1}{2}}_{3,\gamma,\delta,t}
[u]^{\frac{1}{2}}_{\infty,\gamma,\delta,t}[\nabla v]_{3,t}
\label{i1estleq}
\end{align}
for $t\leq 1$ and $k=0,1$. 
Furthermore, it follows from (\ref{lqlrdsmallq}), 
(\ref{lqlrdlargeq}) and (\ref{lqlrdlargeq2}) that 
\begin{align}
\int_0^t&\|(1+|x|)^\gamma(1+|x|-x_1)^\delta
\nabla ^ke^{-(t-\tau)A_a}P_D[u\cdot\nabla v](\tau)\|_{3,D}\,d\tau
\notag\\
&\leq C\sum_{i=1}^4\int_0^{t-1}
(t-\tau)^{-\frac{1}{4}-\frac{k}{2}+\widetilde{\eta}_i}
\|(1+|x|)^{\widetilde{\gamma}_i}
(1+|x|-x_1)^{\widetilde{\delta}_i} u(\tau)\|_{6,D}
\|\nabla v(\tau)\|_{3,D}\,d\tau\notag\\
&\qquad +C\int_{t-1}^t
(t-\tau)^{-\frac{1}{4}-\frac{k}{2}}
\|(1+|x|)^{\gamma}(1+|x|-x_1)^{\delta} u(\tau)\|_{6,D}
\|\nabla v(\tau)\|_{3,D}\,d\tau\notag\\
&\leq C\sum_{i=1}^4\int_0^{t-1}
(t-\tau)^{-\frac{1}{4}-\frac{k}{2}+\widetilde{\eta}_i}
\tau^{-\frac{3}{4}}
(1+\tau)^{\widetilde{\gamma}_i
+\frac{\widetilde{\delta}_i}{2}}\,d\tau\, 
[u]^{\frac{1}{2}}_{3,\widetilde{\gamma}_i,\widetilde{\delta}_i,t}
[u]^{\frac{1}{2}}_{\infty,\widetilde{\gamma}_i,
\widetilde{\delta}_i,t}[\nabla v]_{3,t}\notag\\
&\qquad +\int_{t-1}^t(t-\tau)^{-\frac{1}{4}-\frac{k}{2}}
\tau^{-\frac{3}{4}}(1+\tau)^{\gamma+\frac{\delta}{2}}\,d\tau\, 
[u]^{\frac{1}{2}}_{3,\gamma,\delta,t}
[u]^{\frac{1}{2}}_{\infty,\gamma,\delta,t}[\nabla v]_{3,t}\notag\\
&\leq C(1+t)^{\gamma+\frac{\delta}{2}}t^{-\frac{k}{2}}
\left(
\sum_{i=1}^4[u]_{3,\widetilde{\gamma}_i,
\widetilde{\delta}_i,t}^{\frac{1}{2}}
[u]_{\infty,\widetilde{\gamma}_i,
\widetilde{\delta}_i,t}^{\frac{1}{2}}\right)
[\nabla v]_{3,t}\label{i1estgeq}
\end{align}
for $t>1$ and $k=0,1$, where 
$\widetilde{\eta}_i$ is given by (\ref{etatilde}) 
and we have used 
$\widetilde{\eta}_i+\widetilde{\gamma}_i+\widetilde{\delta}_i/2
=\gamma+\delta/2$ for all $i=1,2,3,4.$ 
Combining (\ref{i1estleq}) and (\ref{i1estgeq}) 
asserts (\ref{i1est}) for all $t\in(0,\infty]$. 
By applying (\ref{lqlrdsmallq}), 
(\ref{lqlrdlargeq}) and (\ref{lqlrdlargeq2}), we have 
\begin{align}
\int_0^t&\|(1+|x|)^\gamma(1+|x|-x_1)^\delta
e^{-(t-\tau)A_a}P_D[\psi(\tau)u\cdot\nabla u_s](\tau)\|_{3,D}\,d\tau
\notag\\
&\leq C\int_0^t(t-\tau)^{-\frac{1}{2}}
\|(1+|x|)^{\gamma}(1+|x|-x_1)^{\delta} 
u(\tau)\|_{\infty,D}\|\nabla u_s\|_{\frac{3}{2},D}\,d\tau\notag\\
&\leq C\int_0^t(t-\tau)^{-\frac{1}{2}}
\tau^{-\frac{1}{2}}(1+\tau)^{\gamma+\frac{\delta}{2}}\,d\tau\, 
[u]_{\infty,\gamma,\delta,t}
\|\nabla u_s\|_{\frac{3}{2},D}
\leq C[u]_{\infty,\gamma,\delta,t}
\|\nabla u_s\|_{\frac{3}{2},D},\label{i2estleq}\\
\int_0^t&\|(1+|x|)^\gamma(1+|x|-x_1)^\delta\nabla
e^{-(t-\tau)A_a}P_D[\psi(\tau)u\cdot\nabla u_s](\tau)\|_{3,D}\,d\tau
\notag\\
&\leq C\int_0^t(t-\tau)^{-1+\frac{\mu_4}{2}}
\|(1+|x|)^{\gamma}(1+|x|-x_1)^{\delta} 
u(\tau)\|_{\infty,D}\|\nabla u_s\|_{\frac{3}{2-\mu_4},D}
\,d\tau\notag\\
&\leq C\int_0^t(t-\tau)^{-1+\frac{\mu_4}{2}}
\tau^{-\frac{1}{2}}
(1+\tau)^{\gamma+\frac{\delta}{2}}\,d\tau\, 
[u]_{\infty,\gamma,\delta,t}
\|\nabla u_s\|_{\frac{3}{2-\mu_4},D}
\leq Ct^{-\frac{1}{2}}[u]_{\infty,\gamma,\delta,t}
\|\nabla u_s\|_{\frac{3}{2-\mu_4},D}
\end{align}
for $t\leq 2$ and   
\begin{align}
\int_0^t&\|(1+|x|)^\gamma(1+|x|-x_1)^\delta
e^{-(t-\tau)A_a}P_D[\psi(\tau)u\cdot\nabla u_s](\tau)\|_{3,D}\,d\tau
\notag\\
&\leq C\sum_{i=1}^4\int_0^{t-1}
(t-\tau)^{-\frac{1}{2}+\widetilde{\eta}_i}\tau^{-\frac{1}{2}}
(1+\tau)^{\widetilde{\gamma}_i+\frac{\widetilde{\delta}_i}{2}}
\,d\tau\, [u]_{\infty,\widetilde{\gamma}_i,\widetilde{\delta}_i,t}
\|\nabla u_s\|_{\frac{3}{2},D}\notag\\
&\qquad +C\int_{t-1}^t(t-\tau)^{-\frac{1}{2}}
\tau^{-\frac{1}{2}}(1+\tau)^{\gamma+\frac{\delta}{2}}\,d\tau\, 
[u]_{\infty,\gamma,\delta,t}\|\nabla u_s\|_{\frac{3}{2},D}\notag\\
&\leq C(1+t)^{\gamma+\frac{\delta}{2}}
\left(
\sum_{i=1}^4
[u]_{\infty,\widetilde{\gamma}_i,\widetilde{\delta}_i,t}\right)
\|\nabla u_s\|_{\frac{3}{2},D},\\
\int_0^t&\|(1+|x|)^\gamma(1+|x|-x_1)^\delta
\nabla e^{-(t-\tau)A_a}
P_D[\psi(\tau)u\cdot\nabla u_s](\tau)\|_{3,D}\,d\tau
\notag\\
&\leq C\sum_{i=1}^4\int_0^{\frac{t}{2}}
(t-\tau)^{-1+\widetilde{\eta}_i}\tau^{-\frac{1}{2}}
(1+\tau)^{\widetilde{\gamma}_i+\frac{\widetilde{\delta}_i}{2}}
\,d\tau\, [u]_{\infty,\widetilde{\gamma}_i,\widetilde{\delta}_i,t}
\|\nabla u_s\|_{\frac{3}{2},D}\notag\\
&\qquad+C\sum_{i=1}^4\int_{\frac{t}{2}}^{t-1}
(t-\tau)^{-1-\frac{\mu_3}{2}+\widetilde{\eta}_i}\tau^{-\frac{1}{2}}
(1+\tau)^{\widetilde{\gamma}_i+\frac{\widetilde{\delta}_i}{2}}
\,d\tau\, [u]_{\infty,\widetilde{\gamma}_i,\widetilde{\delta}_i,t}
\|\nabla u_s\|_{\frac{3}{2+\mu_3},D}
\notag\\
&\qquad +C\int_{t-1}^t(t-\tau)^{-1+\frac{\mu_4}{2}}
\tau^{-\frac{1}{2}}(1+\tau)^{\gamma+\frac{\delta}{2}}\,d\tau\, 
[u]_{\infty,\gamma,\delta,t}
\|\nabla u_s\|_{\frac{3}{2-\mu_4},D}\notag\\
&\leq Ct^{-\frac{1}{2}}(1+t)^{\gamma+\frac{\delta}{2}}
\left(\sum_{i=1}^4
[u]_{\infty,\widetilde{\gamma}_i,\widetilde{\delta}_i,t}\right)
\big(\|\nabla u_s\|_{\frac{3}{2+\mu_3},D}
+\|\nabla u_s\|_{\frac{3}{2},D}+
\|\nabla u_s\|_{\frac{3}{2-\mu_4},D}\big)\label{i2estgeq}
\end{align}
for $t\geq 2.$ Here, we have used 
$L^{3/(2+\mu_3)}$-$L^3$ estimate of the first derivative because
the condition (\ref{alphabeta6}) with $q=3/(2+\mu_3)$ 
is accomplished by (\ref{mui}).
Collecting (\ref{i2estleq})--(\ref{i2estgeq}) yields 
(\ref{i2est}) for $t\in(0,\infty].$ 
By the same calculation, 
we can get (\ref{i3est}) for $t\in(0,\infty],$ 
where we apply $L^{3/(2+\mu_1)}$-$L^3$ estimate 
due to (\ref{mui}) with $i=1$. It follows that 
\begin{align*}
\int_0^t&\|(1+|x|)^\gamma(1+|x|-x_1)^\delta
\nabla^ke^{-(t-\tau)A_a}P_D[(1-\psi(\tau))a
\partial_{x_1}u]\|_{3,D}\,d\tau\\
&\leq Ca\int_0^t(t-\tau)^{-\frac{k}{2}}
\|(1+|x|)^\gamma(1+|x|-x_1)^\delta \nabla u(\tau)\|_{3,D}
\,d\tau\leq Ca[\nabla u]_{3,\gamma,\delta,t}
\end{align*}
for $t\leq 2,k=0,1$ and that 
\begin{align*}
\int_0^t&\|(1+|x|)^\gamma(1+|x|-x_1)^\delta
\nabla^ke^{-(t-\tau)A_a}P_D[(1-\psi(\tau))a
\partial_{x_1}u]\|_{3,D}\,d\tau\\
&\leq Ca\sum_{i=1}^4\int_0^1(t-\tau)^{-\frac{k}{2}
+\widetilde{\eta}_i}\tau^{-\frac{1}{2}}
(1+\tau)^{\widetilde{\gamma}_i+\frac{\widetilde{\delta}_i}{2}}
\,d\tau\,[\nabla u]_{3,\widetilde{\gamma}_i,\widetilde{\delta}_i,t}
\leq Cat^{-\frac{k}{2}}(1+t)^{\gamma+\frac{\delta}{2}}
\left(\sum_{i=1}^4
[\nabla u]_{3,\widetilde{\gamma}_i,\widetilde{\delta}_i,t}\right)
\end{align*}
for $t>2,k=0,1,$ which lead to (\ref{i4est}), 
thereby, we conclude the assertion 1. 
\par We next prove the assertion 2. If $\gamma$ and 
$\delta$ satisfy (\ref{gammadelta02}), then we find 
$\max\{1/(1-\delta),2/(1-\gamma)\}<3.$ By this and 
\begin{align*}
(1+|x|)^\gamma(1+|x|-x_1)^\delta |u_s(x)|\leq 
C(1+|x|)^{-(1-\gamma)}(1+|x|-x_1)^{-(1-\delta)},
\end{align*} 
Lemma \ref{finite} yields 
$(1+|x|)^\gamma(1+|x|-x_1)^\delta u_s\in L^3(D),$ from which it 
follows that 
\begin{align}
&\int_0^t\|(1+|x|)^\gamma(1+|x|-x_1)^\delta
\nabla^ke^{-(t-\tau)A_a}P_D\psi'(\tau)u_s\|_{3,D}\,d\tau\notag\\
&\leq CM\int_0^t(t-\tau)^{-\frac{k}{2}}\,d\tau\,
\|(1+|x|)^\gamma(1+|x|-x_1)^\delta u_s\|_{3,D}
\leq CM\|(1+|x|)^\gamma(1+|x|-x_1)^\delta u_s\|_{3,D}\label{f1near}
\end{align}
for $t\leq 2,k=0,1$ and that 
\begin{align*}
&\int_0^t\|(1+|x|)^\gamma(1+|x|-x_1)^\delta
\nabla^ke^{-(t-\tau)A_a}P_D\psi'(\tau)u_s\|_{3,D}\,d\tau\\
&\leq CM\sum_{i=1}^4\int_0^1(t-\tau)^{-\frac{k}{2}
+\widetilde{\eta}_i}\,d\tau\,
\|(1+|x|)^{\widetilde{\gamma}_i}
(1+|x|-x_1)^{\widetilde{\delta}_i} u_s\|_{3,D}\\
&\leq CMt^{-\frac{k}{2}}(1+t)^{\gamma+\frac{\delta}{2}}
\|(1+|x|)^\gamma(1+|x|-x_1)^\delta u_s\|_{3,D}
\end{align*}
for $t>2,k=0,1.$ Therefore, we conclude (\ref{f1est}). 
In view of $\mu_4\geq 1/2$ and (\ref{gammadelta02}), 
it turns out that 
\begin{align*}
\frac{3}{2-\mu_4}>\max\left\{\frac{1}{\frac{3}{2}-\delta},
\frac{2}{\frac{3}{2}-\gamma}\right\}=\frac{2}{\frac{3}{2}-\gamma}
\end{align*}
for all $\gamma,\delta$ satisfying (\ref{gammadelta02}), 
which together with Lemma \ref{finite} 
asserts $(1+|x|)^{\widetilde{\gamma}_i}
(1+|x|-x_1)^{\widetilde{\delta}_i}\nabla u_s\in L^{3/(2-\mu_4)}(D)$
for $i=1,2,3,4$. Due to this and (\ref{mu2}), we have 
\begin{align}
\int_0^t&\|(1+|x|)^{\gamma}
(1+|x|-x_1)^{\delta}\nabla^ke^{-(t-\tau)A_a}
P_D[\psi(\tau)(1-\psi(\tau))
(u_s\cdot\nabla u_s
+a\partial_{x_1}u_s)]\|_{3,D}\,d\tau
\nonumber\\&\leq 
C\int_0^{t}(t-\tau)^{\frac{\mu_2+\mu_4}{2}-1-\frac{k}{2}}\,d\tau\,
\|u_s\|_{\frac{3}{1-\mu_2},D}\|(1+|x|)^{\gamma}
(1+|x|-x_1)^{\delta}\nabla u_s\|_{\frac{3}{2-\mu_4},D}
\nonumber\\&\qquad+Ca
\int_0^{t}(t-\tau)^{-\frac{1}{2}+\frac{\mu_4}{2}-\frac{k}{2}}
\,d\tau\,\|(1+|x|)^{\gamma}
(1+|x|-x_1)^{\delta}\nabla u_s\|_{\frac{3}{2-\mu_4},D}\notag\\
&\leq C\|u_s\|_{\frac{3}{1-\mu_2},D}\|(1+|x|)^{\gamma}
(1+|x|-x_1)^{\delta}\nabla u_s\|_{\frac{3}{2-\mu_4},D}
+Ca\|(1+|x|)^{\gamma}(1+|x|-x_1)^{\delta}
\nabla u_s\|_{\frac{3}{2-\mu_4},D}\label{f21}
\end{align}
for $t\leq 2,k=0,1$ and 
\begin{align}
\int_0^t&\|(1+|x|)^{\gamma}
(1+|x|-x_1)^{\delta}\nabla^ke^{-(t-\tau)A_a}
P_D[\psi(\tau)\big(1-\psi(\tau)\big)
(u_s\cdot\nabla u_s
+a\partial_{x_1}u_s)]\|_{3,D}\,d\tau
\nonumber\\&\leq 
C\sum_{i=1}^4\int_0^{1}(t-\tau)^{\frac{\mu_2+\mu_4}{2}-1
-\frac{k}{2}+\widetilde{\eta}_i}\,d\tau\,
\|u_s\|_{\frac{3}{1-\mu_2},D}\|(1+|x|)^{\widetilde{\gamma}_i}
(1+|x|-x_1)^{\widetilde{\delta}_i}\nabla u_s\|_{\frac{3}{2-\mu_4},D}
\nonumber\\&\qquad+Ca\sum_{i=1}^4
\int_0^{1}(t-\tau)^{-\frac{1}{2}+\frac{\mu_4}{2}-\frac{k}{2}+
\widetilde{\eta}_i}\,d\tau
\,\|(1+|x|)^{\widetilde{\gamma}_i}
(1+|x|-x_1)^{\widetilde{\delta}_i}
\nabla u_s\|_{\frac{3}{2-\mu_4},D}\notag\\
&\leq Ct^{-\frac{k}{2}+\gamma+\frac{\delta}{2}}
\|u_s\|_{\frac{3}{1-\mu_2},D}\|(1+|x|)^{\gamma}
(1+|x|-x_1)^{\delta}\nabla u_s\|_{\frac{3}{2-\mu_4},D}\notag\\
&\qquad+Cat^{-\frac{k}{2}+\gamma+\frac{\delta}{2}}
\|(1+|x|)^{\gamma}(1+|x|-x_1)^{\delta}
\nabla u_s\|_{\frac{3}{2-\mu_4},D}\label{f22}
\end{align}
for $t\geq 2,k=0,1$, thereby, (\ref{f2est}) holds 
for $t\in(0,\infty]$. The proof is complete. 
\end{proof}

\begin{rmk}
Let $\gamma,\delta$ satisfy 
\begin{align*}
\gamma\geq 0,\quad 0<\delta<\frac{1}{3},\quad \gamma+\delta<1 
\end{align*}
and let $\mu_i>0~(i=1,2,3,4)$ satisfy (\ref{mu1}) 
and (\ref{mui}). 
Suppose $0<a^{1/4}<\kappa_1$, where $\kappa_1$ is a constant in 
Proposition \ref{propsta} with $(\ref{staclass2})$--$(\ref{mu1})$ and 
$($\ref{mui}$)$. We set 
\begin{align*}
\widetilde{E}_2(u)(t)=\int_0^t 
e^{-(t-\tau)A_a}P_D[u\cdot\nabla u_s]\,d\tau,\qquad
\widetilde{E}_3(u)(t)=\int_0^t 
e^{-(t-\tau)A_a}P_D[u_s\cdot\nabla u]\,d\tau,
\end{align*}  
then the same procedure 
as in the proof of (\ref{i2est}) and (\ref{i3est}) asserts 
\begin{align}
&[\widetilde{E}_2(u)]_{3,\gamma,\delta,t}
+[\widetilde{E}_2(u)]_{\infty,\gamma,\delta,t}
+[\nabla \widetilde{E}_2(u)]_{3,\gamma,\delta,t}\notag\\
&\quad \leq C\left(
\sum_{i=1}^4[u]_{\infty,\widetilde{\gamma}_i,
\widetilde{\delta}_i,t}\right)
\big(\|\nabla u_s\|_{\frac{3}{2+\mu_3},D}
+\|\nabla u_s\|_{\frac{3}{2},D}+
\|\nabla u_s\|_{\frac{3}{2-\mu_4},D}\big),\label{widei2est}\\
&[\widetilde{E}_3(u)]_{3,\gamma,\delta,t}
+[\widetilde{E}_3(u)]_{\infty,\gamma,\delta,t}
+[\nabla \widetilde{E}_3(u)]_{3,\gamma,\delta,t}\notag\\
&\quad \leq C\left(
\sum_{i=1}^4
[\nabla u]_{3,\widetilde{\gamma}_i,\widetilde{\delta}_i,t}\right)
\big(\|u_s\|_{\frac{3}{1+\mu_1},D}
+\|u_s\|_{3,D}+\|u_s\|_{\frac{3}{1-\mu_2},D}\big)\label{widei3est}
\end{align}
for $t\in(0,\infty]$ and $u$ 
such that $[[u]]_{\gamma,\delta,\infty}<\infty.$ 
\end{rmk}

We are now in a position to prove Theorem \ref{stabilitythm}. \\
\noindent{\bf Proof of Theorem \ref{stabilitythm}.}\quad 
Given $\alpha,\beta$ satisfying (\ref{alphabeta3}), we take 
$\mu_i~(i=1,2,3,4)$ fulfilling (\ref{mu1}) and  
\begin{align}\label{alphabetai}
\beta<\frac{1-\mu_i}{3},\quad 
\alpha+\beta<1-\mu_i\qquad (i=1,3), 
\end{align} 
and fix $\kappa_1>0$ in 
Proposition \ref{propsta} with (\ref{staclass2})--(\ref{mu1}) 
and (\ref{alphabetai}), then get the stationary solution $u_s$.  
We define the approximate solutions 
\begin{align*}
v_0(t)=e^{-tA_a}b,\quad 
v_{m+1}(t)=e^{-tA_a}b-
\int_0^t e^{-(t-\tau)A_a}P_D\Big[v_m\cdot\nabla v_m
+v_m\cdot\nabla u_s+u_s\cdot\nabla v_m\Big]d\tau
\end{align*}
for $m\geq 0.$ 
It follows from (\ref{i1est}), 
(\ref{widei2est})--(\ref{widei3est}) with 
$\gamma=\delta=0$, (\ref{usest}) and $0<a<1$ (see also 
Shibata \cite{shibata1999} and 
Enomoto and Shibata \cite{ensh2005} for the continuity and 
the condition near $t=0$) that 
\begin{align}
v_m\in Y_{0}:=\big\{v\in BC([0,\infty);L^3_{\sigma}(D))\mid
t^{\frac{1}{2}}v&\in BC((0,\infty);L^{\infty}(D)),
t^{\frac{1}{2}}\nabla v\in BC((0,\infty);L^3(D)),\nonumber\\
\lim_{t\rightarrow 0}&~t^{\frac{1}{2}}\big(
\|v(t)\|_{\infty,D}+\|\nabla v(t)\|_{3,D}\big)=0\big\}\label{y0def}
\end{align}
for $m\geq 0$ and that 
\begin{align*}
[[v_{m+1}]]_{\infty}\leq C_0\|b\|_{3,D}+C_1[[v_m]]_{\infty}^2
+C_2a^{\frac{1}{4}}[[v_m]]_{\infty}
\end{align*}
for $m\geq 0$, where $[[\cdot]]_{\infty}$ is given by 
(\ref{XYtdef}). Thus if 
\begin{align}\label{bsmall}
\|b\|_{3,D}<\frac{1}{16C_0C_1} 
\end{align}
and if 
\begin{align}\label{asmall}
a^{\frac{1}{4}}<\min\left\{\kappa_1,\frac{1}{2C_2}\right\},
\end{align}
then we have 
\begin{align*}
[[v_m]]_{\infty}\leq 4C_0\|b\|_{3,D}
\end{align*}
for $m\geq 0$ and get a solution $v\in Y_0$ with
\begin{align}
&[[v_m-v]]_{\infty}\rightarrow 0 
\quad{\rm as}~m\rightarrow\infty,\label{vmv}\\
&[[v]]_{\infty}\leq 4C_0\|b\|_{3,D}.\label{vapriori}
\end{align}
In order to prove that this solution enjoys 
(\ref{anisov})--(\ref{anisov2}) as $t\rightarrow\infty$, 
we first prove 
\begin{align}
&\|(1+|x|)^{\gamma_i}(1+|x|-x_1)^{\delta_i} v(t)\|_{q,D}
=O(t^{-\frac{1}{2}+\frac{3}{2q}+\gamma_i+\frac{\delta_i}{2}})
\label{bigo1}\\
&\|(1+|x|)^{\gamma_i}(1+|x|-x_1)^{\delta_i} \nabla v(t)\|_{3,D}
=O(t^{-\frac{1}{2}+\gamma_i+\frac{\delta_i}{2}})\label{bigo2}
\end{align} 
as $t\rightarrow\infty$.  
Applying (\ref{i1est}), (\ref{widei2est})--(\ref{widei3est}) to  
$(\gamma,\delta)=(\alpha,\beta),(0,\beta),(\alpha,0),(0,0)$ 
and summing up these lead us to  
\begin{align}
[[v_{m+1}]]_{\alpha,\beta,\infty}&\leq 
C_3\|(1+|x|)^\alpha(1+|x|-x_1)^\beta b\|_{3,D}
+C_4[[v_m]]_{\infty}[[v_m]]_{\alpha,\beta,\infty}
+C_5a^{\frac{1}{4}}[[v_m]]_{\alpha,\beta,\infty}\notag\\
&\leq C_3\|(1+|x|)^\alpha(1+|x|-x_1)^\beta b\|_{3,D}
+\big(4C_0C_4\|b\|_{3,D}
+C_5a^{\frac{1}{4}})[[v_m]]_{\alpha,\beta,\infty}
\label{vmxinfty}
\end{align}
for $m\geq 0$, where $[[\cdot]]_{\alpha,\beta,\infty}$ 
is given by (\ref{XYtdef}), thereby, we conclude 
$[[v_{m}]]_{\alpha,\beta,\infty}\leq 2C_3
\|(1+|x|)^\alpha(1+|x|-x_1)^\beta b\|_{3,D}$
for all $m\geq 0$ whenever 
\begin{align}\label{smallness}
4C_0C_4\|b\|_{3,D}+C_5a^{\frac{1}{4}}<\frac{1}{2}
\end{align}
is satisfied.  
This together with (\ref{vmv}) implies that if 
\begin{align}\label{bsmall2}
\|b\|_{3,D}<\min\left\{\frac{1}{16C_0C_1},\frac{1}{16C_0C_4}\right\}
\end{align}
and if 
\begin{align}\label{deltadef2}
a^{\frac{1}{4}}<\min\left\{\kappa_1,\frac{1}{2C_2},
\frac{1}{4C_5}\right\}=:\kappa
\end{align}
(see (\ref{bsmall}) and (\ref{asmall})),
then we get
\begin{align*}
[[v]]_{\alpha,\beta,\infty}\leq 
2C_3\|(1+|x|)^\alpha(1+|x|-x_1)^\beta b\|_{3,D},
\end{align*}
which yields (\ref{bigo1})--(\ref{bigo2}).
To derive (\ref{anisov}) with $q=3$, that is,  
\begin{align}\label{anisovn}
\|(1+|x|)^{\gamma_i}(1+|x|-x_1)^{\delta_i} v(t)\|_{3,D}
=o(t^{\gamma_i+\frac{\delta_i}{2}})
\end{align} 
as $t \rightarrow\infty$ for $i=1,2,3,4$, 
we first assume $b\in C_{0,\sigma}^\infty(D).$ 
Let $(\gamma,\delta)=(\alpha,\beta),(0,\beta),(\alpha,0),(0,0)$. 
Given $\mu_1,\mu_3$ satisfying (\ref{mu1}) and (\ref{alphabetai}), 
we set $\mu:=\min\{\mu_1,\mu_3\}<1,$ then estimates 
(\ref{lqlrdsmallq}) and (\ref{lqlrdlargeq}) yield 
\begin{align*}
\|(1+|x|)^\gamma(1+|x|-x_1)^\delta e^{-tA_a}b\|_{3,D}
&\leq C\sum_{i=1}^4t^{-\frac{\mu}{2}}
(1+t)^{\widetilde{\eta}_i}
\|(1+|x|)^{\widetilde{\gamma}_i}
(1+|x|-x_1)^{\widetilde{\delta}_i}b\|_{\frac{3}{1+\mu},D}\\
&\leq Ct^{-\frac{\mu}{2}}
(1+t)^{\gamma+\frac{\delta}{2}}
\|(1+|x|)^{\gamma}(1+|x|-x_1)^{\delta}b\|_{\frac{3}{1+\mu},D}
\end{align*} 
for $t>0$, where $\widetilde{\gamma}_i,\widetilde{\delta}_i$ and 
$\widetilde{\eta}_i$ are given by 
(\ref{gammadeltatilde}) and (\ref{etatilde}). 
Moreover, we have 
\begin{align}
\|&(1+|x|)^\gamma(1+|x|-x_1)^\delta E_1(v,v)(t)\|_{3,D}\notag\\
&\leq Ct^{-\frac{\mu}{2}}(1+t)^{\gamma+\frac{\delta}{2}}
[\nabla v]_{3,t}\sum_{i=1}^4\sup_{0<\tau<t}\tau^{\frac{\mu}{2}}
(1+\tau)^{-\widetilde{\gamma}_i-\frac{\widetilde{\delta}_i}{2}}
\|(1+|x|)^{\widetilde{\gamma}_i}
(1+|x|-x_1)^{\widetilde{\delta}_i}v(\tau)\|_{3,D},\label{e1mu}\\
&\|(1+|x|)^\gamma(1+|x|-x_1)^\delta \widetilde{E}_2(v)(t)\|_{3,D}
\leq C(1+t)^{-\frac{\mu_3}{2}+\gamma+\frac{\delta}{2}}
(\|\nabla u_s\|_{\frac{3}{2},D}+\|\nabla u_s\|_{\frac{3}{2+\mu_3},D})
\sum_{i=1}^4[v]_{\infty,\widetilde{\gamma}_i,\widetilde{\delta}_i,t},
\label{e2tildemu}\\
&\|(1+|x|)^\gamma(1+|x|-x_1)^\delta \widetilde{E}_3(v)(t)\|_{3,D}
\leq C(1+t)^{-\frac{\mu_1}{2}+\gamma+\frac{\delta}{2}}
(\|u_s\|_{3,D}+\|u_s\|_{\frac{3}{1+\mu_1},D})\sum_{i=1}^4
[\nabla v]_{3,\widetilde{\gamma}_i,\widetilde{\delta}_i,t}
\label{e3tildemu}
\end{align}
for $t>0$ due to 
\begin{align}
&\|(1+|x|)^\gamma(1+|x|-x_1)^\delta \widetilde{E}_2(v)(t)\|_{3,D}
\leq C(1+t)^{\gamma+\frac{\delta}{2}}
\|\nabla u_s\|_{\frac{3}{2},D}
\sum_{i=1}^4[v]_{\infty,\widetilde{\gamma}_i,\widetilde{\delta}_i,t},
\notag\\
&\|(1+|x|)^\gamma(1+|x|-x_1)^\delta \widetilde{E}_2(v)(t)\|_{3,D}
\leq Ct^{-\frac{\mu_3}{2}}(1+t)^{\gamma+\frac{\delta}{2}}
\|\nabla u_s\|_{\frac{3}{2+\mu_3},D}
\sum_{i=1}^4[v]_{\infty,\widetilde{\gamma}_i,\widetilde{\delta}_i,t},
\notag\\
&\|(1+|x|)^\gamma(1+|x|-x_1)^\delta \widetilde{E}_3(v)(t)\|_{3,D}
\leq C(1+t)^{\gamma+\frac{\delta}{2}}
\|u_s\|_{3,D}\sum_{i=1}^4
[\nabla v]_{3,\widetilde{\gamma}_i,\widetilde{\delta}_i,t},
\notag\\
&\|(1+|x|)^\gamma(1+|x|-x_1)^\delta \widetilde{E}_3(v)(t)\|_{3,D}
\leq Ct^{-\frac{\mu_1}{2}}(1+t)^{\gamma+\frac{\delta}{2}}
\|u_s\|_{\frac{3}{1+\mu_1},D}\sum_{i=1}^4
[\nabla v]_{3,\widetilde{\gamma}_i,\widetilde{\delta}_i,t}\notag
\end{align}  
for $t>0$, thereby,
\begin{align*}
\sup_{0<\tau<t}&\tau^{\frac{\mu}{2}}
(1+\tau)^{-\gamma-\frac{\delta}{2}}
\|(1+|x|)^{\gamma}
(1+|x|-x_1)^{\delta}v(\tau)\|_{3,D}\\
&\leq C\|(1+|x|)^{\gamma}
(1+|x|-x_1)^{\delta}b\|_{\frac{3}{1+\mu},D}\\
&\qquad+C[\nabla v]_{3,t}
\sum_{i=1}^4\sup_{0<\tau<t}\tau^{\frac{\mu}{2}}
(1+\tau)^{-\widetilde{\gamma}_i-\frac{\widetilde{\delta}_i}{2}}
\|(1+|x|)^{\widetilde{\gamma}_i}
(1+|x|-x_1)^{\widetilde{\delta}_i}v(\tau)\|_{3,D}\\
&\qquad+C(\|\nabla u_s\|_{\frac{3}{2},D}
+\|\nabla u_s\|_{\frac{3}{2+\mu_3},D})
\sum_{i=1}^4[v]_{\infty,\widetilde{\gamma}_i,
\widetilde{\delta}_i,t}
+C(\|u_s\|_{3,D}+\|u_s\|_{\frac{3}{1+\mu_1},D})
\sum_{i=1}^4[\nabla v]_{3,\widetilde{\gamma}_i,\widetilde{\delta}_i,t}
\\&\leq C\|(1+|x|)^{\gamma}
(1+|x|-x_1)^{\delta}b\|_{\frac{3}{1+\mu},D}\\
&\qquad+C[[v]]_{\infty}
\sum_{i=1}^4\sup_{0<\tau<t}\tau^{\frac{\mu}{2}}
(1+\tau)^{-\widetilde{\gamma}_i-\frac{\widetilde{\delta}_i}{2}}
\|(1+|x|)^{\widetilde{\gamma}_i}
(1+|x|-x_1)^{\widetilde{\delta}_i}v(\tau)\|_{3,D}\\
&\qquad+C(\|\nabla u_s\|_{\frac{3}{2},D}
+\|\nabla u_s\|_{\frac{3}{2+\mu_3},D}+
\|u_s\|_{3,D}+\|u_s\|_{\frac{3}{1+\mu_1},D})
[[v]]_{\alpha,\beta,\infty}
\end{align*}
for $t>0.$ 
Summing up this with 
$(\gamma,\delta)=(\alpha,\beta),(0,\beta),(\alpha,0),(0,0)$ 
and using (\ref{vapriori}) imply 
\begin{align*}
\sum_{i=1}^4&\sup_{0<\tau<t}\tau^{\frac{\mu}{2}}
(1+\tau)^{-\gamma_i-\frac{\delta_i}{2}}
\|(1+|x|)^{\gamma_i}
(1+|x|-x_1)^{\delta_i}v(\tau)\|_{3,D}\\
&\leq 4C_0C_6\|b\|_{3,D}
\sum_{i=1}^4\sup_{0<\tau<t}\tau^{\frac{\mu}{2}}
(1+\tau)^{-\gamma_i-\frac{\delta_i}{2}}
\|(1+|x|)^{\gamma_i}
(1+|x|-x_1)^{\delta_i}v(\tau)\|_{3,D}+C_7
\end{align*}
for $t>0,$ where $\gamma_i,\delta_i$ are 
given by (\ref{gammadeltaeta}) and 
$C_7=C_7(b,\alpha,\beta,[[v]]_{\alpha,\beta,\infty},
u_s,\mu_1,\mu_3)>0$ 
is independent of $t$. If we assume 
\begin{align}\label{epsilondef}
\|b\|_{3,D}<\min\left\{\frac{1}{16C_0C_1},\frac{1}{16C_0C_4},
\frac{1}{4C_0C_6}\right\}=:\varepsilon
\end{align}
(see (\ref{bsmall}) and (\ref{bsmall2})), then we have 
\begin{align*}
\sum_{i=1}^4\sup_{0<\tau<t}\tau^{\frac{\mu}{2}}
(1+\tau)^{-\gamma_i-\frac{\delta_i}{2}}
\|(1+|x|)^{\gamma_i}
(1+|x|-x_1)^{\delta_i}v(\tau)\|_{3,D}
\leq \frac{C_7}{1-4C_0C_6\|b\|_{3,D}}
\end{align*}
for all $t>0$, which yields 
\begin{align}\label{aa}
\|(1+|x|)^{\gamma_i}(1+|x|-x_1)^{\delta_i} v(t)\|_{3,D}
\leq Ct^{-\frac{\mu}{2}}
(1+t)^{\gamma_i+\frac{\delta_i}{2}}
\end{align}
for $t>0,i=1,2,3,4$ if $b\in C_{0,\sigma}^\infty(D)$. 
We in particular conclude (\ref{anisovn}) 
if $b\in C_{0,\sigma}^\infty(D)$. 
Let $b\in L^3_{(1+|x|)^{3\alpha}(1+|x|-x_1)^{3\beta},\sigma}(D)$ 
satisfy $\|b\|_{3,D}<\varepsilon.$ 
Since $C_{0,\sigma}^\infty(D)$ is dense 
in $L^3_{(1+|x|)^{3\alpha}(1+|x|-x_1)^{3\beta},\sigma}(D)$, we 
can take $\{b_m\}_{m=1}^\infty\subset C_{0,\sigma}^\infty(D)$ so 
that 
\begin{align}\label{bm}
\|(1+|x|)^\alpha(1+|x|-x_1)^\beta(b_m-b)\|_{3,D}\leq 
\frac{\varepsilon-\|b\|_{3,D}}{m},
\end{align}  
which implies $\|b_m\|_{3,D}<\varepsilon$ for $m\geq 1$. We thus  
get a solution $\widetilde{v}_m$ to the problem (\ref{NS42}) 
with $b=b_m$, then the same 
calculation as in (\ref{vmxinfty}) asserts 
\begin{align*}
\widetilde{v}_m(t)-v(t)&=e^{-tA_a}(b_m-b)
-E_1(\widetilde{v}_m,\widetilde{v}_m-v)
-E_1(\widetilde{v}_m-v,v)
-\widetilde{E}_2(\widetilde{v}_m-v)
-\widetilde{E}_3(\widetilde{v}_m-v),
\\
[[\widetilde{v}_m-v]]_{\alpha,\beta,\infty}&\leq 
C_3\|(1+|x|)^\alpha(1+|x|-x_1)^\beta (b_m-b)\|_{3,D}
+C_4([[\widetilde{v}_m]]_{\infty}+[[v]]_{\infty})
[[\widetilde{v}_m-v]]_{\alpha,\beta,\infty}
\\
&\qquad +C_5a^{\frac{1}{4}}[[\widetilde{v}_m-v
]]_{\alpha,\beta,\infty}\notag
\\
&\leq C_3\|(1+|x|)^\alpha(1+|x|-x_1)^\beta (b_m-b)\|_{3,D}
\\
&\qquad+\big(4C_0C_4(\|b_m\|_{3,D}+\|b\|_{3,D})+
C_5a^{\frac{1}{4}}\big)[[\widetilde{v}_m-v]]_{\alpha,\beta,\infty}. 
\end{align*} 
If (\ref{deltadef2}) and (\ref{epsilondef}) are fulfilled, then 
in view of $\|b_m\|_{3,D}<\varepsilon\leq 1/(16C_0C_4)$, 
it follows from (\ref{bm}) that 
\begin{align*}
[[\widetilde{v}_m-v]]_{\alpha,\beta,\infty}\leq
\frac{4C_3(\varepsilon-\|b\|_{3,D})}{m},
\end{align*} 
which implies 
\begin{align*}
\lim_{m\rightarrow\infty}\limsup_{t\rightarrow\infty} 
(1+t)^{-\gamma_i-\frac{\delta_i}{2}}\|(1+|x|)^{\gamma_i}
(1+|x|-x_1)^{\delta_i}(\widetilde{v}_m-v)(t)\|_{3,D}=0.
\end{align*}
Moreover, due to $b_m\in C_{0,\sigma}^\infty(D)$, 
we see from (\ref{aa}) that 
\begin{align*}
(1+t)^{-\gamma_i-\frac{\delta_i}{2}}\|(1+|x|)^{\gamma_i}
(1+|x|-x_1)^{\delta_i}\widetilde{v}_m(t)\|_{3,D}\leq C
t^{-\frac{\mu}{2}}\rightarrow 0
\end{align*}
as $t\rightarrow\infty$, thus get 
(\ref{anisovn}) if $b\in L_{(1+|x|)^{3\alpha}
(1+|x|-x_1)^{3\beta},\sigma}^3(D).$ 
\par We next prove 
\begin{align}\label{anisovinftygrad}
\|(1+|x|)^{\gamma_i}(1+|x|-x_1)^{\delta_i} v(t)\|_{\infty,D}+
\|(1+|x|)^{\gamma_i}(1+|x|-x_1)^{\delta_i} \nabla v(t)\|_{3,D}
=o(t^{-\frac{1}{2}+\gamma_i+\frac{\delta_i}{2}})
\end{align} 
as $t\rightarrow\infty$ for $i=1,2,3,4$. 
In view of (\ref{anisovn}), 
it suffices to prove 
\begin{align}
\|&(1+|x|)^\alpha(1+|x|-x_1)^\beta
v(t)\|_{\infty,D}+\|(1+|x|)^\alpha(1+|x|-x_1)^\beta
\nabla v(t)\|_{3,D}\notag\\
&\leq Ct^{-\frac{1}{2}}\sum_{i=1}^4t^{\eta_i}
\left\|(1+|x|)^{\gamma_i}(1+|x|-x_1)^{\delta_i}
v\left(\frac{t}{2}\right)\right\|_{3,D},
\label{inftygrad3alphabeta}\\
\|&(1+|x|)^\alpha
v(t)\|_{\infty,D}+\|(1+|x|)^\alpha\nabla v(t)\|_{3,D}
\leq Ct^{-\frac{1}{2}}
\left\|(1+|x|)^{\alpha}
v\left(\frac{t}{2}\right)\right\|_{3,D}
+Ct^{-\frac{1}{2}+\alpha}
\left\|v\left(\frac{t}{2}\right)\right\|_{3,D},
\label{inftygrad3alphabeta2}\\
\|&(1+|x|-x_1)^\beta v(t)\|_{\infty,D}+\|(1+|x|-x_1)^\beta
\nabla v(t)\|_{3,D}\notag\\
&\leq Ct^{-\frac{1}{2}}
\left\|(1+|x|-x_1)^{\beta}
v\left(\frac{t}{2}\right)\right\|_{3,D}
+Ct^{-\frac{1}{2}+\frac{\beta}{2}}
\left\|
v\left(\frac{t}{2}\right)\right\|_{3,D}\label{inftygrad3alphabeta3}
\end{align}
for $t>0$, where $\gamma_i,\delta_i,\eta_i$ are gived by
(\ref{gammadeltaeta}). We derive (\ref{inftygrad3alphabeta}) by 
following the similar argument to 
Enomoto and Shibata \cite{ensh2005}. 
For $t>T>0,$ we have 
\begin{align}\label{Tt}
v(t)=e^{-(t-T)A_a}v(T)-\int_T^t e^{-(t-\tau)A_a}P_D
\big[v\cdot\nabla v
+v\cdot\nabla u_s+u_s\cdot\nabla v\big]\,d\tau.
\end{align}
Let $(\gamma,\delta)=(\alpha,\beta),(0,\beta),(\alpha,0),(0,0)$. 
By the same calculation as in the proof of Lemma \ref{closelem}, 
we get 
\begin{align}
\int_T^t&\|(1+|x|)^\gamma(1+|x|-x_1)^\delta
e^{-(t-\tau)A_a}P_D[v\cdot\nabla v
+v\cdot\nabla u_s+u_s\cdot\nabla v](\tau)\|_{\infty,D}\notag\\&
\qquad +\|(1+|x|)^\gamma(1+|x|-x_1)^\delta
\nabla e^{-(t-\tau)A_a}P_D[v\cdot\nabla v
+v\cdot\nabla u_s+u_s\cdot\nabla v](\tau)\|_{3,D}\,d\tau
\notag\\
&\leq C(t-T)^{-\frac{1}{2}}t^{\gamma+\frac{\delta}{2}}\Big\{
\big(\sup_{T<\tau\leq t}\|v(\tau)\|_{3,D}\big)^{\frac{1}{2}}\,
\big(\sup_{T<\tau\leq t}(\tau-T)^{\frac{1}{2}}
\|v(\tau)\|_{\infty,D}\big)^{\frac{1}{2}}\notag\\
&\hspace{5cm}+\big(\|u_s\|_{\frac{3}{1+\mu_1},D}
+\|u_s\|_{3,D}+
\|u_s\|_{\frac{3}{1-\mu_2},D}\big)
\Big\}\sum_{i=1}^4
[\nabla v]_{3,\widetilde{\gamma}_i,\widetilde{\delta}_i,T,t}
\notag\\
&\qquad +C(t-T)^{-\frac{1}{2}}t^{\gamma+\frac{\delta}{2}}
\big(\|\nabla u_s\|_{\frac{3}{2+\mu_3},D}
+\|\nabla u_s\|_{\frac{3}{2},D}+
\|\nabla u_s\|_{\frac{3}{2-\mu_4},D}\big)
\sum_{i=1}^4
[v]_{\infty,\widetilde{\gamma}_i,\widetilde{\delta}_i,T,t},
\label{inftygrad3gammadelta}
\end{align}
where 
\begin{align*}
&[\nabla v]_{3,\gamma,\delta,T,t}:=
\sup_{T<\tau\leq t}(\tau-T)^{\frac{1}{2}}
\tau^{-\gamma-\frac{\delta}{2}}\|(1+|x|)^{\gamma}
(1+|x|-x_1)^{\delta}\nabla v(\tau)\|_{3,D},\\
&[v]_{\infty,\gamma,\delta,T,t}:=
\sup_{T<\tau\leq t}(\tau-T)^{\frac{1}{2}}
\tau^{-\gamma-\frac{\delta}{2}}
\|(1+|x|)^{\gamma}(1+|x|-x_1)^{\delta}v(\tau)\|_{\infty,D}.
\end{align*}
Applying (\ref{inftygrad3gammadelta}) with 
$(\gamma,\delta)=(\alpha,\beta),(0,\beta),(\alpha,0),(0,0)$ 
to (\ref{Tt}) and summing up these yield 
\begin{align*}
&\sum_{i=1}^4(t-T)^{\frac{1}{2}}
t^{-\gamma_i-\frac{\delta_i}{2}}
\big(\|(1+|x|)^{\gamma_i}
(1+|x|-x_1)^{\delta_i} v(t)\|_{\infty,D}+
\|(1+|x|)^{\gamma_i}(1+|x|-x_1)^{\delta_i}\nabla v(t)\|_{3,D}\big)\\
&\leq \sum_{i=1}^4(t-T)^{\frac{1}{2}}
t^{-\gamma_i-\frac{\delta_i}{2}}\big(\|(1+|x|)^{\gamma_i}
(1+|x|-x_1)^{\delta_i}e^{-(t-T)A_a}v(T)\|_{\infty,D}\\
&\hspace{7cm}+
\|(1+|x|)^{\gamma_i}(1+|x|-x_1)^{\delta_i}
\nabla e^{-(t-T)A_a}v(T)\|_{3,D}\big)\\
&\qquad+(C_4[[v]]_{\infty}+
C_5a^{\frac{1}{4}})\sum_{i=1}^4
\big([v]_{\infty,\gamma_i,\delta_i,T,t}
+[\nabla v]_{3,\gamma_i,\delta_i,T,t}\big)
\end{align*} 
for $t>T,$ thereby, it follows from 
(\ref{vapriori}), (\ref{smallness}) that 
\begin{align}
&\sum_{i=1}^4\big([v]_{\infty,\gamma_i,\delta_i,T,t}+
[\nabla v]_{3,\gamma_i,\delta_i,T,t}\big)\notag\\
&\leq 2\sum_{i=1}^4
\sup_{T<\tau\leq t}\Big\{(\tau-T)^{\frac{1}{2}}
\tau^{-\gamma_i-\frac{\delta_i}{2}}\big(\|(1+|x|)^{\gamma_i}
(1+|x|-x_1)^{\delta_i}\nabla e^{-(\tau-T)A_a}v(T)\|_{3,D}\notag\\
&\hspace{7cm}+\|(1+|x|)^{\gamma_i}
(1+|x|-x_1)^{\delta_i}e^{-(\tau-T)A_a}v(T)\|_{\infty,D}\big)\Big\}
\label{vesttT}
\end{align}
for $t>T.$ By taking the relation 
$t^{1/2}/\sqrt{2}\leq (t-T)^{1/2}$ for $t\geq 2T$ and 
$(\gamma_1,\delta_1)=(\alpha,\beta)$ into account and by applying 
(\ref{lqlrdsmallq}), (\ref{lqlrdlargeq}), (\ref{lqlrdlargeq2}) 
to (\ref{vesttT}), we obtain 
\begin{align*}
&\frac{1}{\sqrt{2}}\,t^{\frac{1}{2}-\alpha-\frac{\beta}{2}}
\big(\|(1+|x|)^{\alpha}
(1+|x|-x_1)^{\beta} v(t)\|_{\infty,D}+
\|(1+|x|)^{\alpha}(1+|x|-x_1)^{\beta}\nabla v(t)\|_{3,D}\big)\\
&\leq \sum_{i=1}^4\big([v]_{\infty,\gamma_i,\delta_i,T,t}+
[\nabla v]_{3,\gamma_i,\delta_i,T,t}\big)\\
&\leq 2\sum_{i=1}^4T^{-\gamma_i-\frac{\delta_i}{2}}
\sup_{T<\tau\leq t}(\tau-T)^{\frac{1}{2}}
\big\{\|(1+|x|)^{\gamma_i}
(1+|x|-x_1)^{\delta_i}e^{-(\tau-T)A_a}v(T)\|_{\infty,D}\\
&\hspace{7cm}+
\|(1+|x|)^{\gamma_i}(1+|x|-x_1)^{\delta_i}
\nabla e^{-(\tau-T)A_a}v(T)\|_{3,D}\big\}\\
&\leq C\Big[T^{-\alpha-\frac{\beta}{2}}
\|(1+|x|)^{\alpha}
(1+|x|-x_1)^{\beta}v(T)\|_{3,D}
+\big\{T^{-\alpha-\frac{\beta}{2}}(t-T)^\alpha
+T^{-\frac{\beta}{2}}\big\}\|(1+|x|-x_1)^{\beta}v(T)\|_{3,D}\\
&\quad\qquad +\big\{T^{-\alpha-\frac{\beta}{2}}(t-T)^\frac{\beta}{2}
+T^{-\alpha}\big\}\|(1+|x|)^{\alpha}v(T)\|_{3,D}\\
&\quad\qquad+\big\{T^{-\alpha-\frac{\beta}{2}}
(t-T)^{\alpha+\frac{\beta}{2}}+T^{-\alpha}(t-T)^{\alpha}
+T^{-\frac{\beta}{2}}(t-T)^{\frac{\beta}{2}}+1\big\}\|v(T)\|_{3,D}
\Big]
\end{align*}
for $t\geq 2T.$ We thus conclude (\ref{inftygrad3alphabeta}) 
by taking $T=t/2$. Similarly, by 
applying (\ref{inftygrad3gammadelta}) with 
$(\gamma,\delta)=(\alpha,0),(0,0)$ 
(resp. $(\gamma,\delta)=(0,\beta),(0,0)$) to (\ref{Tt}) 
and by summing up these, we have 
(\ref{inftygrad3alphabeta2}) (resp. (\ref{inftygrad3alphabeta3})). 
From (\ref{anisovn}) and (\ref{anisovinftygrad}), 
we complete the proof. \qed

In addition to Lemma \ref{closelem}, 
we make use of the assertion 3 and 5 of Proposition \ref{proplqlr} 
to complete the proof of Theorem \ref{attainthm}. 
\\
\noindent{\bf Proof of Theorem \ref{attainthm}.}\quad 
Given $\alpha,\beta$ satisfying (\ref{alphabeta5}), we take 
$\mu_i~(i=1,2,3,4)$ fulfilling
(\ref{mu1})--(\ref{mu2}) and 
\begin{align}\label{alphabetai2}
\beta<\frac{1-\mu_i}{3}\qquad (i=1,3),
\end{align} 
and fix $\kappa_2$ in 
Proposition \ref{propsta} with (\ref{staclass2})--(\ref{mu2}) and 
(\ref{alphabetai2}), then get the stationary solution $u_s$.
We define
\begin{align*}
&w_0(t)=0,\\
&w_{m+1}(t)=\int_0^t e^{-(t-\tau)A_a}P_D\big[-w_m\cdot\nabla w_m
-\psi(\tau)w_m\cdot\nabla u_s-\psi(\tau)u_s\cdot\nabla w_m
\nonumber\\
&\hspace{7cm}+(1-\psi(\tau))a\partial_{x_1}w_m
+f_1(\tau)+f_2(\tau)\big]d\tau
\end{align*}
for $m\geq 0$. We know from Lemma \ref{closelem} 
with $\gamma=\delta=0$ (see also Galdi, Heywood 
and Shibata \cite{gahesh1997}, 
the present author \cite{takahashi2021} 
for the continuity and the condition near $t=0$) 
that $w_m\in Y_0$ together with 
\begin{align}
&[[w_{m}]]_{\infty}\leq \widetilde{C}_1[[w_{m-1}]]_{\infty}^2
+\widetilde{C}_2a^\frac{1}{4}
[[w_{m-1}]]_{\infty}+\widetilde{C}_3(M+1)a^\frac{1}{4},\notag\\
&[[w_{m+1}-w_m]]_{\infty}\leq 
\{\widetilde{C}_1([[w_m]]_{\infty}+[[w_{m-1}]]_{\infty})+
\widetilde{C}_2a^\frac{1}{4}\}
[[w_m-w_{m-1}]]_{\infty}\nonumber
\end{align}
for all $m\geq 1$, 
where $Y_0$ and $[[\cdot]]_{\infty}$ are given by 
(\ref{y0def}) and (\ref{XYtdef}), respectively.
Thus if 
\begin{align}\label{asmall2}
(M+1)a^\frac{1}{4} <\min\left\{\kappa_2,
\frac{1}{2\widetilde{C}_2},
\frac{1}{16\widetilde{C}_1\widetilde{C}_3}\right\},
\end{align}
then we have 
\begin{align}
[[w_{m}]]_{\infty}\leq 4\widetilde{C}_3(M+1)a^\frac{1}{4}
\label{wmapriori}
\end{align}
for $m\geq 0$ and get a solution $w\in Y_0$ with 
\begin{align}
&[[w_m-w]]_{\infty}\rightarrow 0 
\quad{\rm as}~m\rightarrow\infty,\label{wmw}
\\
&[[w]]_{\infty}\leq 4\widetilde{C}_3(M+1)a^\frac{1}{4}.
\label{wapriori}
\end{align}
In order to prove that this solution enjoys 
(\ref{anisow})--(\ref{anisow2}) as $t\rightarrow\infty$, 
we first prove 
\begin{align}
&\|(1+|x|)^{\gamma_i}(1+|x|-x_1)^{\delta_i} w(t)\|_{q,D}
=O(t^{-\frac{1}{2}+\frac{3}{2q}+\gamma_i+\frac{\delta_i}{2}})
\label{bigo3}\\
&\|(1+|x|)^{\gamma_i}(1+|x|-x_1)^{\delta_i} \nabla w(t)\|_{3,D}
=O(t^{-\frac{1}{2}+\gamma_i+\frac{\delta_i}{2}})\label{bigo4}
\end{align} 
as $t\rightarrow\infty$.
In view of (\ref{alphabeta5}), (\ref{alphabetai2}), 
$0<\mu_1<1/2,0<\mu_3<1/4$, see (\ref{mu1}), we find 
\begin{align*}
\beta<\frac{1-\mu_i}{3},\quad 
\alpha+\beta<\frac{2-\mu_i}{3}<1-\mu_i\qquad (i=1,3), 
\end{align*}
thus apply (\ref{i1est})--(\ref{i4est}), 
(\ref{f1est})--(\ref{f2est}) with 
$(\gamma,\delta)=(\alpha,\beta),(0,\beta),(\alpha,0),(0,0)$  
and (\ref{wmapriori}) to obtain 
\begin{align*}
[[w_{m+1}]]_{\alpha,\beta,\infty}&\leq 
\widetilde{C}_4[[w_m]]_{\infty}[[w_m]]_{\alpha,\beta,\infty}
+\widetilde{C}_5a^{\frac{1}{4}}[[w_m]]_{\alpha,\beta,\infty}
+\widetilde{C}_6\\
&\leq(4\widetilde{C}_3\widetilde{C}_4+\widetilde{C}_5)
(M+1)a^{\frac{1}{4}}[[w_m]]_{\alpha,\beta,\infty}+\widetilde{C}_6
\end{align*}
for $m\geq 0$, where $[[\cdot]]_{\alpha,\beta,\infty}$ 
is given by (\ref{XYtdef}) and 
$\widetilde{C}_6=\widetilde{C}_6(u_s,a,M,\mu_2,\mu_4)>0$ is 
independent of $t$. From this, if 
\begin{align}
(M+1)a^{\frac{1}{4}}<\frac{1}
{2(4\widetilde{C}_3\widetilde{C}_4+\widetilde{C}_5)},\notag
\end{align} 
then we have  
$[[w_{m}]]_{\alpha,\beta,\infty}\leq 2\widetilde{C}_6$ 
for all $m\geq 0$, which together with (\ref{wmw}) 
implies 
\begin{align}\label{wapriori2}
[[w]]_{\alpha,\beta,\infty}\leq 2\widetilde{C}_6
\end{align}
if 
\begin{align}\label{assump2}
(M+1)a^\frac{1}{4} <\min\left\{\kappa_2,
\frac{1}{2\widetilde{C}_2},
\frac{1}{16\widetilde{C}_1\widetilde{C}_3},\frac{1}
{2(4\widetilde{C}_3\widetilde{C}_4+\widetilde{C}_5)}
\right\}
\end{align}
(see (\ref{asmall2})). We thus get (\ref{bigo3})--(\ref{bigo4}). 
We finally derive (\ref{anisow})--(\ref{anisow2}) by the 
similar argument due to the present author \cite{takahashi2021}.   
Since we find 
\begin{align*}
w(t)=e^{-(t-T)A_a}w(T)-\int_T^t e^{-(t-\tau)A_a}P_D\big[w\cdot\nabla w
+w\cdot\nabla u_s+u_s\cdot\nabla w\big]\,d\tau
\end{align*}
for $t>T\geq 1$, the same calculation as in the proof of 
(\ref{inftygrad3alphabeta})--(\ref{inftygrad3alphabeta3}) yields
\begin{align}
\|&(1+|x|)^\alpha(1+|x|-x_1)^\beta
w(t)\|_{\infty,D}+\|(1+|x|)^\alpha(1+|x|-x_1)^\beta
\nabla w(t)\|_{3,D}\notag\\
&\leq Ct^{-\frac{1}{2}}\sum_{i=1}^4t^{\eta_i}
\left\|(1+|x|)^{\gamma_i}(1+|x|-x_1)^{\delta_i}
w\left(\frac{t}{2}\right)\right\|_{3,D},
\label{inftygrad3alphabetaw1}\\
\|&(1+|x|)^\alpha
w(t)\|_{\infty,D}+\|(1+|x|)^\alpha\nabla w(t)\|_{3,D}
\leq Ct^{-\frac{1}{2}}
\left\|(1+|x|)^{\alpha}
w\left(\frac{t}{2}\right)\right\|_{3,D}
+Ct^{-\frac{1}{2}+\alpha}
\left\|w\left(\frac{t}{2}\right)\right\|_{3,D},
\label{inftygrad3alphabetaw2}\\
\|&(1+|x|-x_1)^\beta w(t)\|_{\infty,D}+\|(1+|x|-x_1)^\beta
\nabla w(t)\|_{3,D}\notag\\
&\leq Ct^{-\frac{1}{2}}
\left\|(1+|x|-x_1)^{\beta}
w\left(\frac{t}{2}\right)\right\|_{3,D}
+Ct^{-\frac{1}{2}+\frac{\beta}{2}}
\left\|w\left(\frac{t}{2}\right)\right\|_{3,D}
\label{inftygrad3alphabetaw3}
\end{align}
for $t\geq 2,$ where $\gamma_i,\delta_i,\eta_i$ are 
given by (\ref{gammadeltaeta}). Therefore, 
for each $i=1,2,3,4$, we prove by induction that 
\begin{align}\label{sharpl3}
\|(1+|x|)^{\gamma_i}(1+|x|-x_1)^{\delta_i}w(t)\|_{3,D}
=O(t^{-\sigma_k+\gamma_i+\frac{\delta_i}{2}})
\end{align} 
for $k\geq 1$, where 
$\sigma_k:=\min\{(k\mu)/2,1/4-\varepsilon\},\mu:=\min\{\mu_1,\mu_3\}$ 
and $\varepsilon>0$ is arbitrarily fixed. 
This implies (\ref{anisow}) with $q=3$, which combined with 
(\ref{inftygrad3alphabetaw1})--(\ref{inftygrad3alphabetaw3}) 
completes the proof. 
Let $(\gamma,\delta)=(\alpha,\beta),(0,\beta),(\alpha,0),(0,0).$ 
We use (\ref{lqlrddiv02}) with $F=e_1\otimes w$ to obtain 
\begin{align}
\int_0^t&\|(1+|x|)^\gamma(1+|x|-x_1)^\delta
e^{-(t-\tau)A_a}P_D[
(1-\psi(\tau))
a\partial_{x_1}w]\|_{3,D}d\tau\notag\\
&\leq Ca\int_0^{\min\{1,t\}}(t-\tau)^{-\frac{1}{2}}
(1+t-\tau)^{\gamma+\frac{\delta}{2}}\|(1+|x|)^\gamma
(1+|x|-x_1)^\delta w(\tau)\|_{3,D}\,d\tau\notag\\
&\leq Ca\int_0^{\min\{1,t\}}(t-\tau)^{-\frac{1}{2}}
(1+t-\tau)^{\gamma+\frac{\delta}{2}}\cdot
2^{\gamma+\frac{\delta}{2}}\,d\tau\,
[w]_{3,\gamma,\delta,t}
\leq Ca(1+t)^{-\frac{1}{2}+\gamma+\frac{\delta}{2}}
[w]_{3,\gamma,\delta,t}\label{div}
\end{align}
for $t>0.$ In view of (\ref{alphabeta5}), we 
have 
\begin{align*}
&\max\left\{\frac{1}{1-\beta},\frac{2}{1-\alpha}\right\}=
\frac{2}{1-\alpha},\qquad 
\max\left\{\frac{1}{1-\beta},2 \right\}=2,
\\
&
\max\left\{\frac{1}{\frac{3}{2}-\beta},
\frac{2}{\frac{3}{2}-\alpha}\right\}=
\frac{2}{\frac{3}{2}-\alpha},\qquad 
\max\left\{\frac{1}{\frac{3}{2}-\beta},\frac{4}{3}\right\}=
\frac{4}{3},
\end{align*}
thus it follows from Lemma \ref{finite} and (\ref{usest2}) that 
\begin{align*}
&(1+|x|)^\alpha(1+|x|-x_1)^\beta u_s
\in L^{q_1}(D)\quad \text{for}~q_1>\frac{2}{1-\alpha},\quad 
(1+|x|-x_1)^\beta u_s
\in L^{q_2}(D)\quad \text{for}~ q_2>2,
\\
&(1+|x|)^\alpha(1+|x|-x_1)^\beta \nabla u_s
\in L^{r_1}(D)\quad \text{for}~r_1>\frac{4}{3-2\alpha},\quad  
(1+|x|-x_1)^\beta \nabla u_s
\in L^{r_2}(D)\quad \text{for}~r_2>\frac{4}{3}. 
\end{align*}
This yields 
\begin{align}
&(1+|x|)^\gamma u_s,(1+|x|)^\gamma(1+|x|-x_1)^\delta u_s
\in L^{\frac{2}{1-\gamma-\nu_1}}(D),\quad 
u_s,(1+|x|-x_1)^\delta u_s
\in L^{\frac{2}{1-\nu_1}}(D),\label{sum1}
\\
&(1+|x|)^\gamma\nabla u_s,(1+|x|)^\gamma(1+|x|-x_1)^\delta \nabla u_s
\in L^{\frac{4}{3-2\gamma-\nu_2}}(D),\quad 
\nabla u_s,(1+|x|-x_1)^\delta \nabla u_s
\in L^{\frac{3}{2}}(D)\label{sum2}
\end{align}
for all $(\gamma,\delta)=(\alpha,\beta),(0,\beta),(\alpha,0),(0,0)$,  
where we have taken $\nu_1,\nu_2$ so that 
$0<\nu_1<(4\varepsilon)/3,0<\nu_2<(8\varepsilon)/3.$ 
From (\ref{sum1})--(\ref{sum2}) together with 
\begin{align*}
\min\left\{3\left(1-\frac{1-\gamma-\nu_1}{2}\right),1\right\}=1,\quad 
\min\left\{3\left(1-\frac{3-2\gamma-\nu_2}{4}\right),1\right\}
>\frac{2}{3}
\end{align*}
and (\ref{alphabeta5}), 
the condition (\ref{alphabeta7}) is accomplished by 
\begin{align*}
(q_1,q_2,q_3,q_4)=&\left(\frac{2}{1-\gamma-\nu_1},\frac{2}{1-\nu_1},
\frac{2}{1-\gamma-\nu_1},\frac{2}{1-\nu_1}\right),
\left(\frac{4}{3-2\gamma-\nu_2},\frac{3}{2},
\frac{4}{3-2\gamma-\nu_2},\frac{3}{2}\right).
\end{align*}
We then apply (\ref{lqlrdlarge2}) to obtain 
\begin{align*}
\int_0^t&\|(1+|x|)^\gamma(1+|x|-x_1)^\delta
e^{-(t-\tau)A_a}P_Df_1(\tau)\|_{3,D}d\tau\\
&\leq CM\int_0^{1}\Big[
(t-\tau)^{-\frac{1}{4}+\frac{3}{4}\nu_1+\gamma+\frac{\delta}{2}}
\|u_s\|_{\frac{2}{1-\nu_1},D}+
(t-\tau)^{-\frac{1}{4}+\frac{3}{4}\nu_1+\gamma}
\|(1+|x|-x_1)^\delta u_s\|_{\frac{2}{1-\nu_1},D}\\
&\qquad\qquad\qquad+(t-\tau)^{-\frac{1}{4}
+\frac{3}{4}\nu_1+\frac{3}{4}\gamma+\frac{\delta}{2}}
\|(1+|x|)^\gamma u_s\|_{\frac{2}{1-\gamma-\nu_1},D}\\
&\qquad\qquad\qquad+(t-\tau)^{-\frac{1}{4}+\frac{3}{4}\gamma
+\frac{3}{4}\nu_1}
\|(1+|x|)^\gamma(1+|x|-x_1)^\delta u_s\|_{\frac{2}{1-\gamma-\nu_1},D}
\Big]\,d\tau
\\
&\leq CMt^{-\frac{1}{4}+\varepsilon+\gamma+\frac{\delta}{2}}
\big(\|(1+|x|-x_1)^\delta u_s\|_{\frac{2}{1-\nu_1},D}+
\|(1+|x|)^\gamma
(1+|x|-x_1)^\delta u_s\|_{\frac{2}{1-\gamma-\nu_1},D}\big),
\\
\int_0^t&\|(1+|x|)^\gamma(1+|x|-x_1)^\delta
e^{-(t-\tau)A_a}P_D[\psi(\tau)(1-\psi(\tau))
a\partial_{x_1}u_s]\|_{3,D}\,d\tau
\\
&\leq Ca\int_0^{1}
\Big[(t-\tau)^{-\frac{1}{2}
+\gamma+\frac{\delta}{2}}
\|\nabla u_s\|_{\frac{3}{2},D}+
(t-\tau)^{-\frac{1}{2}+\gamma}
\|(1+|x|-x_1)^\delta\nabla u_s\|_{\frac{3}{2},D}
\\
&\qquad\qquad\qquad+(t-\tau)^{-\frac{5}{8}
+\frac{3}{8}\nu_2+\frac{3}{4}\gamma+\frac{\delta}{2}}
\|(1+|x|)^\gamma \nabla u_s\|_{\frac{4}{3-2\gamma-\nu_2},D}
\\
&\qquad\qquad\qquad+(t-\tau)^{-\frac{5}{8}+\frac{3}{8}\nu_2
+\frac{3}{4}\gamma}
\|(1+|x|)^\gamma(1+|x|-x_1)^\delta\nabla 
u_s\|_{\frac{4}{3-2\gamma-\nu_2},D}\Big]\,d\tau
\\
&\leq Cat^{-\frac{1}{2}+\gamma+\frac{\delta}{2}}\big(
\|(1+|x|-x_1)^\delta\nabla u_s\|_{\frac{3}{2},D}+
\|(1+|x|)^\gamma(1+|x|-x_1)^\delta
u_s\|_{\frac{4}{3-2\gamma-\nu_2},D}\big)
\end{align*}
for $t\geq 2$, which combined with (\ref{f1near})--(\ref{f21}) yield 
\begin{align}
\int_0^t&\|(1+|x|)^\gamma(1+|x|-x_1)^\delta
e^{-(t-\tau)A_a}P_Df_1(\tau)\|_{3,D}d\tau\notag\\
&\leq CM(1+t)^{-\frac{1}{4}+\varepsilon+\gamma+\frac{\delta}{2}}
\big(\|(1+|x|)^\gamma(1+|x|-x_1)^\delta u_s\|_{3,D}+
\|(1+|x|-x_1)^\delta u_s\|_{\frac{2}{1-\nu_1},D}\notag\\
&\hspace{6cm}+\|(1+|x|)^\gamma(1+|x|-x_1)
^\delta u_s\|_{\frac{2}{1-\gamma-\nu_1},D}\big),\label{f12}\\
\int_0^t&\|(1+|x|)^\gamma(1+|x|-x_1)^\delta
e^{-(t-\tau)A_a}P_D[\psi(\tau)(1-\psi(\tau))
a\partial_{x_1}u_s]\|_{3,D}
d\tau\notag\\
&\leq Ca(1+t)^{-\frac{1}{2}+\gamma+\frac{\delta}{2}}
\big(\|(1+|x|)^{\gamma}(1+|x|-x_1)^{\delta}
\nabla u_s\|_{\frac{3}{2-\mu_4},D}
+\|(1+|x|-x_1)^\delta\nabla u_s\|_{\frac{3}{2},D}
\notag
\\
&\hspace{6cm}
+\|(1+|x|)^\gamma(1+|x|-x_1)^\delta
u_s\|_{\frac{4}{3-2\gamma-\nu_2},D}\big)\label{div2}
\end{align}
for $t>0.$ Moreover, by the same calculation as in 
(\ref{e1mu})--(\ref{e3tildemu}), 
(\ref{f21})--(\ref{f22}), we have 
\begin{align}
&\|(1+|x|)^\gamma(1+|x|-x_1)^\delta E_1(w,w)(t)\|_{3,D}\notag\\
&\qquad \leq Ct^{-\frac{\mu}{2}}(1+t)^{\gamma+\frac{\delta}{2}}
[\nabla w]_{3,t}\sum_{i=1}^4\sup_{0<\tau<t}\tau^{\frac{\mu}{2}}
(1+\tau)^{-\gamma_i-\frac{\delta_i}{2}}
\|(1+|x|)^{\gamma_i}
(1+|x|-x_1)^{\delta_i}w(\tau)\|_{3,D}\\
&\|(1+|x|)^\gamma(1+|x|-x_1)^\delta E_2(w)(t)\|_{3,D}\notag\\
&\qquad\leq C(1+t)^{-\frac{\mu_3}{2}+\gamma+\frac{\delta}{2}}
(\|\nabla u_s\|_{\frac{3}{2},D}+\|\nabla u_s\|_{\frac{3}{2+\mu_3},D})
\sum_{i=1}^4[w]_{\infty,\widetilde{\gamma}_i,\widetilde{\delta}_i,t},
\\
&\|(1+|x|)^\gamma(1+|x|-x_1)^\delta E_3(w)(t)\|_{3,D}
\leq C(1+t)^{-\frac{\mu_1}{2}+\gamma+\frac{\delta}{2}}
(\|u_s\|_{3,D}+
\|u_s\|_{\frac{3}{1+\mu_1},D})\sum_{i=1}^4
[\nabla w]_{3,\widetilde{\gamma}_i,\widetilde{\delta}_i,t},
\label{e32}\\
\int_0^t&\|(1+|x|)^\gamma(1+|x|-x_1)^\delta
e^{-(t-\tau)A_a}P_D[\psi(\tau)(1-\psi(\tau))
u_s\cdot\nabla u_s]\|_{3,D}\,d\tau\notag\\&\leq 
C(1+t)^{\frac{\mu_2+\mu_4}{2}-1+\gamma+\frac{\delta}{2}}
\|u_s\|_{\frac{3}{1-\mu_2},D}\|(1+|x|)^{\gamma}
(1+|x|-x_1)^{\delta}\nabla u_s\|_{\frac{3}{2-\mu_4},D}\notag\\
&\leq 
C(1+t)^{-\frac{1}{4}+\varepsilon+\gamma+\frac{\delta}{2}}
\|u_s\|_{\frac{3}{1-\mu_2},D}\|(1+|x|)^{\gamma}
(1+|x|-x_1)^{\delta}\nabla u_s\|_{\frac{3}{2-\mu_4},D}
\label{f23}
\end{align}
for $t>0$, where we have used $(\mu_2+\mu_4)/2-1\leq -1/4$ 
due to (\ref{mu1}). 
Summing up (\ref{div}), (\ref{f12})--(\ref{f23}) with 
$(\gamma,\delta)=(\alpha,\beta),(0,\beta),(\alpha,0),(0,0)$
as well as using (\ref{wapriori}) lead to 
\begin{align*}
\sum_{i=1}^4&\sup_{0<\tau<t}\tau^{\frac{\mu}{2}}
(1+\tau)^{-\gamma_i-\frac{\delta_i}{2}}
\|(1+|x|)^{\gamma_i}
(1+|x|-x_1)^{\delta_i}w(\tau)\|_{3,D}\\
&\leq \widetilde{C}_7[\nabla w]_{3,t}
\sum_{i=1}^4\sup_{0<\tau<t}\tau^{\frac{\mu}{2}}
(1+\tau)^{-\gamma_i-\frac{\delta_i}{2}}
\|(1+|x|)^{\gamma_i}
(1+|x|-x_1)^{\delta_i}w(\tau)\|_{3,D}+\widetilde{C}_8\\
&\leq 4\widetilde{C}_3\widetilde{C}_7(M+1)a^\frac{1}{4}
\sum_{i=1}^4\sup_{0<\tau<t}\tau^{\frac{\mu}{2}}
(1+\tau)^{-\gamma_i-\frac{\delta_i}{2}}
\|(1+|x|)^{\gamma_i}
(1+|x|-x_1)^{\delta_i}w(\tau)\|_{3,D}+\widetilde{C}_8
\end{align*}
for $t>0,$ where $\widetilde{C}_8=\widetilde{C}_8(a,M,\alpha,\beta,w,
u_s,\mu_1,\mu_2,\mu_3,\mu_4)>0$ is independent of $t$ and  
$\mu=\min\{\mu_1,\mu_3\}$. 
We thus get 
\begin{align*}
\|(1+|x|)^{\gamma_i}(1+|x|-x_1)^{\delta_i}w(t)\|_{3,D}
\leq Ct^{-\frac{\mu}{2}}(1+t)^{\gamma_i+\frac{\delta_i}{2}}
\end{align*} 
for $t>0,i=1,2,3,4$ if 
\begin{align*}
(M+1)a^\frac{1}{4} <\min\left\{\kappa_2,
\frac{1}{2\widetilde{C}_2},
\frac{1}{16\widetilde{C}_1\widetilde{C}_3},\frac{1}
{2(4\widetilde{C}_3\widetilde{C}_4+\widetilde{C}_5)},
\frac{1}{4\widetilde{C}_3\widetilde{C}_7}
\right\}=:\widetilde{\kappa}
\end{align*}
(see (\ref{assump2})) is fulfilled, 
which implies (\ref{sharpl3}) with $k=1.$ 
\par 
Let $k\geq 2$ and suppose (\ref{sharpl3}) with 
$k-1$ for $i=1,2,3,4.$ 
Given $(\gamma,\delta)=(\alpha,\beta),(0,\beta),(\alpha,0),(0,0)$, 
by taking (\ref{wapriori2}) (near $t=0$),  
(\ref{inftygrad3alphabetaw1})--(\ref{inftygrad3alphabetaw3}) 
into account, we set 
\begin{align*}
&I_{k-1}(w):=
\sup_{\tau>0}\tau^{\frac{1}{2}}(1+\tau)^{\sigma_{k-1}}
\|\nabla w(\tau)\|_{3,D}<\infty,\\
&J_{k-1,i}(w):=\sup_{\tau>0}
\big(1+\tau)^{\sigma_{k-1}-\widetilde{\gamma}_i
-\frac{\widetilde{\delta}_i}{2}}
\|(1+|x|)^{\widetilde{\gamma}_i}
(1+|x|-x_1)^{\widetilde{\delta}_i}w(\tau)\|_{3,D}\\
&\hspace{3cm}+\sup_{\tau>0}\tau^{\frac{1}{2}}
(1+\tau)^{\sigma_{k-1}-\widetilde{\gamma}_i
-\frac{\widetilde{\delta}_i}{2}}
\big(\|(1+|x|)^{\widetilde{\gamma}_i}
(1+|x|-x_1)^{\widetilde{\delta}_i}w(\tau)\|_{\infty,D}\\
&\hspace{8cm}+
\|(1+|x|)^{\widetilde{\gamma}_i}
(1+|x|-x_1)^{\widetilde{\delta}_i}\nabla w(\tau)\|_{3,D}\big)<\infty
\end{align*} 
for $i=1,2,3,4$, where $\widetilde{\gamma}_i,
\widetilde{\delta}_i$ are given by (\ref{gammadeltatilde}).
 We use these to see
\begin{align}
\int_0^t&\|(1+|x|)^\gamma(1+|x|-x_1)^\delta
e^{-(t-\tau)A_a}P_D[w\cdot\nabla w](\tau)\|_{3,D}\,d\tau
\notag\\
&\leq C\sum_{i=1}^4\int_0^{t-1}
(t-\tau)^{-\frac{1}{2}+\widetilde{\eta}_i}
\tau^{-\frac{1}{2}}(1+\tau)^{-2\sigma_{k-1}+\widetilde{\gamma}_i
+\frac{\widetilde{\delta}_i}{2}}\,d\tau\notag\\
&\qquad\qquad\times I_{k-1}(w)\big(\sup_{\tau>0}
(1+\tau)^{\sigma_{k-1}-\widetilde{\gamma}_i
-\frac{\widetilde{\delta}_i}{2}}
\|(1+|x|)^{\widetilde{\gamma}_i}
(1+|x|-x_1)^{\widetilde{\delta}_i}w(\tau)\|_{3,D}\big)
\notag\\
&\qquad +C\int_{t-1}^t
(t-\tau)^{-\frac{1}{2}}\tau^{-\frac{1}{2}}(1+\tau)^{-2\sigma_{k-1}
+\gamma+\frac{\delta}{2}}\,d\tau\,I_{k-1}(w)
J_{k-1,4}(w)\notag\\
&\leq Ct^{-2\sigma_{k-1}+\gamma+\frac{\delta}{2}}
I_{k-1}(w)\sum_{i=1}^4J_{k-1,i}(w)\leq 
Ct^{-\sigma_{k}+\gamma+\frac{\delta}{2}}
I_{k-1}(w)\sum_{i=1}^4J_{k-1,i}(w),\notag\\
\int_0^t&\|(1+|x|)^\gamma(1+|x|-x_1)^\delta
e^{-(t-\tau)A_a}P_D[w\cdot\nabla u_s](\tau)\|_{3,D}\,d\tau
\notag\\
&\leq C\sum_{i=1}^4\int_0^{t-1}
(t-\tau)^{-\frac{1+\mu_3}{2}+\widetilde{\eta}_i}
\tau^{-\frac{1}{2}}(1+\tau)^{-\sigma_{k-1}+\widetilde{\gamma}_i
+\frac{\widetilde{\delta}_i}{2}}\,d\tau\,
\|\nabla u_s\|_{\frac{3}{2+\mu_3},D}\notag\\
&\qquad\qquad\times\sup_{\tau>0}\tau^{\frac{1}{2}}
(1+\tau)^{\sigma_{k-1}-\widetilde{\gamma}_i
-\frac{\widetilde{\delta}_i}{2}}
\|(1+|x|)^{\widetilde{\gamma}_i}
(1+|x|-x_1)^{\widetilde{\delta}_i}w(\tau)\|_{\infty,D}
\notag\\
&\qquad +C\int_{t-1}^t
(t-\tau)^{-\frac{1+\mu_3}{2}}
\tau^{-\frac{1}{2}}(1+\tau)^{-\sigma_{k-1}
+\gamma+\frac{\delta}{2}}\,d\tau\,
\|\nabla u_s\|_{\frac{3}{2+\mu_3},D}J_{k-1,4}(w)\notag\\
&\leq Ct^{-\frac{\mu_3}{2}-\sigma_{k-1}+\gamma+\frac{\delta}{2}}
\|\nabla u_s\|_{\frac{3}{2+\mu_3},D}
\sum_{i=1}^4J_{k-1,i}(w)
\leq Ct^{-\sigma_{k}+\gamma+\frac{\delta}{2}}
\|\nabla u_s\|_{\frac{3}{2+\mu_3},D}
\sum_{i=1}^4J_{k-1,i}(w)
\notag
\end{align}
and 
\begin{align}
\int_0^t&\|(1+|x|)^\gamma(1+|x|-x_1)^\delta
e^{-(t-\tau)A_a}P_D[u_s\cdot\nabla w](\tau)\|_{3,D}\,d\tau
\notag\\
&\leq C\sum_{i=1}^4\int_0^{t-1}
(t-\tau)^{-\frac{1+\mu_1}{2}+\widetilde{\eta}_i}
\tau^{-\frac{1}{2}}(1+\tau)^{-\sigma_{k-1}+\widetilde{\gamma}_i
+\frac{\widetilde{\delta}_i}{2}}\,d\tau\,
\|u_s\|_{\frac{3}{1+\mu_1},D}\notag\\
&\qquad\qquad\times\sup_{\tau>0}\tau^{\frac{1}{2}}
(1+\tau)^{\sigma_{k-1}-\widetilde{\gamma}_i
-\frac{\widetilde{\delta}_i}{2}}
\|(1+|x|)^{\widetilde{\gamma}_i}
(1+|x|-x_1)^{\widetilde{\delta}_i}\nabla w(\tau)\|_{3,D}
\notag\\
&\qquad +C\int_{t-1}^t
(t-\tau)^{-\frac{1+\mu_1}{2}}
\tau^{-\frac{1}{2}}(1+\tau)^{-\sigma_{k-1}
+\gamma+\frac{\delta}{2}}\,d\tau\,
\|u_s\|_{\frac{3}{1+\mu_1},D}J_{k-1,4}(w)\notag\\
&\leq Ct^{-\frac{\mu_1}{2}-\sigma_{k-1}+\gamma+\frac{\delta}{2}}
\|u_s\|_{\frac{3}{1+\mu_1},D}
\sum_{i=1}^4J_{k-1,i}(w)
\leq Ct^{-\sigma_{k}+\gamma+\frac{\delta}{2}}
\|u_s\|_{\frac{3}{1+\mu_1},D}
\sum_{i=1}^4J_{k-1,i}(w)
\notag
\end{align}
for $t\geq 2$. 
From these, (\ref{f12})--(\ref{div2}) and (\ref{f23}), we get 
(\ref{sharpl3}), which asserts (\ref{anisow}) with $q=3$.
The proof is complete.\qed

We end this paper with a comparison of the decay rate stated in 
Theorem \ref{attainthm} with the one derived by 
applying (\ref{gradlqlrdlarge5}), (\ref{gradlqlrdlarge6}).
\begin{rmk}\label{f1rmk}
We cannot derive (\ref{anisow})--(\ref{anisow2}) 
by using (\ref{gradlqlrdlarge5}). 
In fact, if $\alpha,\beta$ satisfy (\ref{alphabeta5}), 
then it follows from $\max\{1/(1-\beta),2/(1-\alpha)\}
=2/(1-\alpha)$, Lemma \ref{finite} and (\ref{usest2}) that 
$(1+|x|)^\alpha(1+|x|-x_1)^\beta u_s\in 
L^{2/(1-\alpha-\nu)}(D)$ with any small $\nu>0.$ 
Applying (\ref{gradlqlrdlarge5}) to $F_1$ implies 
\begin{align*}
\|&(1+|x|)^\alpha(1+|x|-x_1)^\beta F_1(t)\|_{3,D}\\
&\leq \int_0^t\|(1+|x|)^\alpha(1+|x|-x_1)^\beta
e^{-(t-\tau)A_a}P_Df_1(\tau)\|_{3,D}d\tau\\
&\leq CM\int_0^{1}(t-\tau)^{-\frac{3}{2}
(\frac{1-\alpha-\nu}{2}-\frac{1}{3})+\frac{\alpha}{4}
+\max\{\frac{\alpha}{4},\frac{\beta}{2}\}+\varepsilon}\,d\tau
\|(1+|x|)^\alpha(1+|x|-x_1)^\beta u_s\|_{\frac{2}{1-\alpha-\nu},D}
\\
&\leq CMt^{-\frac{1}{4}+\alpha
+\max\{\frac{\alpha}{4},\frac{\beta}{2}\}+\frac{3}{4}\nu+
\varepsilon}
\|(1+|x|)^\alpha(1+|x|-x_1)^\beta u_s\|_{\frac{2}{1-\alpha-\nu},D}
\end{align*}
for $t\geq 2$, thus we find by the same calculation as in the 
proof of Theorem \ref{attainthm} that the rate of 
$L^3_{(1+|x|)^{3\alpha}(1+|x|-x_1)^{3\beta}}$ norm of 
the solution is close to 
$-1/4+(5\alpha)/4(>-1/4+\alpha+\beta/2)$ if $\alpha>2\beta$.  
By the same reason, we employed (\ref{lqlrdlarge2}) 
instead of (\ref{gradlqlrdlarge6}) to get 
(\ref{anisow})--(\ref{anisow2}) with $\beta=0$.
\end{rmk}

%\bigskip
%\noindent{\bf Conflict of interest}~ 
%The author declares that there is no conflict of interest. 

\medskip
\begin{center}
\hspace{10cm}
Tomoki Takahashi\\
\hspace{10cm}
Graduate School of Mathematics,
\\
\hspace{10cm}
Nagoya University,
\\
\hspace{10cm}
Furo-cho, Chikusa-ku,
\\
\hspace{10cm}
Nagoya, 464-8602,
\\
\hspace{10cm}
Japan.
\hspace{10cm}
\\
\hspace{10cm}
e-mail: m17023c@math.nagoya-u.ac.jp
\end{center}

\end{document}